\documentclass[a4paper,10pt]{amsart}
\usepackage[a4paper, textwidth=15.5cm]{geometry}
\usepackage{amsmath,amsthm,amssymb,amsfonts,extarrows}
\usepackage{enumerate,amscd,amsxtra,MnSymbol}
\usepackage{mathrsfs}
\usepackage{color}
\usepackage[cmtip,all]{xy}
\usepackage{eucal}
\usepackage{bbold}
\usepackage{ upgreek }
\usepackage{paralist}
\usepackage{tikz-cd}

\theoremstyle{plein}
\newtheorem{theorem}{Theorem}[section]
\newtheorem{definition}[theorem]{Definition}
\newtheorem{definition*}{Definition}
\newtheorem*{theorem*}{Theorem}
\newtheorem{lemma}[theorem]{Lemma}
\newtheorem{proposition}[theorem]{Proposition}
\newtheorem*{proposition*}{Proposition}
\newtheorem{corollary}[theorem]{Corollary}
\newtheorem*{corollary*}{Corollary}

\newtheorem{conjecture*}{Conjecture}

\newtheorem{terminology}[theorem]{Terminology}

\theoremstyle{definition}

\newtheorem{name-1}[theorem]{Monoid objects}
\newtheorem{name-2}[theorem]{Algebra objects}

\newtheorem{warning}[theorem]{Warning}

\newtheorem{example}[theorem]{Example}
\newtheorem{remark}[theorem]{Remark}
\newtheorem{remark*}{Remark}

\newtheorem{construction}[theorem]{Construction}

\newtheorem{notation}[theorem]{Notation}

\newcommand{\bE}{{\mathbb E}}

\renewcommand{\P}{{\mathbb P}}

\newcommand{\mA}{{\mathcal A}}
\newcommand{\mB}{{\mathcal B}}
\newcommand{\mC}{{\mathcal C}}
\newcommand{\mD}{{\mathcal D}}
\newcommand{\mE}{{\mathcal E}}
\newcommand{\mF}{{\mathcal F}}

\newcommand{\mJ}{{\mathcal J}}

\newcommand{\mL}{{\mathcal L}}
\newcommand{\mM}{{\mathcal M}}
\newcommand{\mN}{{\mathcal N}}
\newcommand{\mO}{{\mathcal O}}
\newcommand{\mP}{{\mathcal P}}
\newcommand{\mQ}{{\mathcal Q}}
\newcommand{\mR}{{\mathcal R}}
\newcommand{\mS}{{\mathcal S}}

\newcommand{\mU}{{\mathcal U}}
\newcommand{\mV}{{\mathcal V}}
\newcommand{\mW}{{\mathcal W}}
\newcommand{\mX}{{\mathcal X}}
\newcommand{\mY}{{\mathcal Y}}
\newcommand{\mZ}{{\mathcal Z}}

\newcommand{\A}{{\mathrm A}}
\newcommand{\B}{{\mathrm B}}
\newcommand{\C}{{\mathrm C}}
\newcommand{\D}{{\mathrm D}}
\newcommand{\E}{{\mathrm E}}
\newcommand{\F}{{\mathrm F}}
\newcommand{\G}{{\mathrm G}}
\newcommand{\rH}{{\mathrm H}}

\newcommand{\K}{{\mathrm K}}
\renewcommand{\L}{{\mathrm L}}

\renewcommand{\P}{{\mathrm P}}

\newcommand{\R}{{\mathrm R}}
\newcommand{\rS}{{\mathrm S}}
\newcommand{\T}{{\mathrm T}}

\newcommand{\V}{{\mathrm V}}
\newcommand{\W}{{\mathrm W}}
\newcommand{\X}{{\mathrm X}}
\newcommand{\Y}{{\mathrm Y}}
\newcommand{\Z}{{\mathrm Z}}

\newcommand{\rc}{{\mathrm c}}

\newcommand{\bj}{{\mathrm j}}
\newcommand{\bi}{{\mathrm i}}
\newcommand{\m}{{\mathrm m}}
\newcommand{\bk}{{\mathrm k}}

\newcommand{\rt}{{\mathrm t}}

\newcommand{\q}{{\mathrm q}}
\newcommand{\g}{{\mathrm g}}
\newcommand{\n}{{\mathrm n}}

\newcommand{\op}{\mathrm{op}}

\newcommand{\Ho}{\mathsf{Ho}}

\newcommand{\s}{\mathsf{s}}

\newcommand{\Sp}{\mathrm{Sp}}

\newcommand{\y}{\mathsf{y}}
\newcommand{\f}{\mathsf{f}}

\newcommand{\colim}{\mathrm{colim}}

\newcommand{\LMod}{{\mathrm{LMod}}}
\newcommand{\RMod}{{\mathrm{RMod}}}

\newcommand{\ot}{\otimes}

\newcommand{\Ass}{  {\mathrm {   Ass  } }   }

\newcommand{\h}{\mathrm{h}}

\newcommand{\id}{\mathrm{id}}

\newcommand{\Cat}{\mathsf{Cat}}

\newcommand{\Set}{\mathsf{Set}}

\newcommand{\fl}{\mathrm{fl}}

\newcommand{\map}{\mathrm{map}}

\renewcommand{\Pr}{\mathrm{Pr}}

\newcommand{\Alg}{\mathrm{Alg}}

\newcommand{\Mon}{\mathrm{Mon}}
\newcommand{\Fun}{\mathrm{Fun}}
\newcommand{\tu}{{\mathbb 1}}
\newcommand{\Op}{{\mathrm{Op}}}
\newcommand{\cocart}{{\mathrm{cocart}}}

\newcommand{\cart}{{\mathrm{cart}}}

\newcommand{\Triv}{{\mathrm{Triv}}}
\newcommand{\Fin}{{\mathrm{\mF in}}}

\newcommand{\rev}{{\mathrm{rev}}}

\newcommand{\Mul}{{\mathrm{Mul}}}
\newcommand{\lan}{{\mathrm{lan}}}
\newcommand{\ev}{{\mathrm{ev}}}
\newcommand{\Ind}{{\mathrm{Ind}}}

\newcommand{\mon}{{\mathrm{mon}}}

\newcommand{\LM}{{\mathrm{LM}}}

\newcommand{\Cmon}{{\mathrm{Cmon}}}

\newcommand{\Env}{{\mathrm{Env}}}

\newcommand{\LinFun}{{\mathrm{LinFun}}}  
 
\newcommand{\act}{{\mathrm{act}}}  
\newcommand{\bZ}{{\mathbb{Z}}}

\newcommand{\BMod}{{\mathrm{BMod}}}  

\newcommand{\Act}{{\mathrm{Act}}}

\newcommand{\PreCat}{{\mathrm{PreCat}}}

\newcommand{\BM}{{\mathrm{BM}}}  
\newcommand{\cl}{{\mathrm{cl}}}  
\newcommand{\RM}{{\mathrm{RM}}}  
\newcommand{\Enr}{{\mathrm{Enr}}}

\newcommand{\Mor}{{\mathrm{Mor}}}

\newcommand{\gen}{{\mathrm{gen}}} 
\newcommand{\Quiv}{{\mathrm{Quiv}}} 
 
\newcommand{\LaxLinFun}{{\mathrm{LaxLinFun}}} 
\newcommand{\el}{{\mathrm{el}}}

\newcommand{\inert}{{\mathrm{int}}} 

\newcommand{\Lur}{{\mathrm{Lur}}} 

\usepackage{pdfpages}

\title{An equivalence between enriched $\infty$-categories and $\infty$-categories with weak action}

\begin{document}
	
\maketitle

\author{Hadrian Heine}

\begin{abstract}
	
We show that an $\infty$-category $\mathcal{M}$ with a closed left action
of a monoidal $\infty$-category $\mathcal{V}$ is completely determined by the  $\mathcal{V}$-valued graph of morphism objects equipped with the structure of a $\mathcal{V}$-enrichment in the sense of Gepner-Haugseng.
We prove a similar result when $\mathcal{M}$ is a $\mathcal{V}$-enriched $\infty$-category in the sense of Lurie, an operadic generalization of the notion of $\infty$-category with closed action.
Precisely, we prove that sending a $\mathcal{V}$-enriched $\infty$-category in the sense of Lurie to the $\mathcal{V}$-valued graph of morphism objects refines to an equivalence $\chi$ between the $\infty$-category of $\mathcal{V}$-enriched $\infty$-categories in the sense of Lurie and of Gepner-Haugseng.
Moreover if $\mathcal{V}$ is a presentably $\bE_{\bk+1}$-monoidal $\infty$-category for $1 \leq \bk \leq \infty$, we prove that $\chi$ restricts to a 
lax $\bE_\bk$-monoidal functor between the $\infty$-category of 
left $\mathcal{V}$-modules in $\Pr^\L$, the symmetric monoidal $\infty$-category of presentable $\infty$-categories, endowed with the relative tensor product,
and the tensor product of $\mathcal{V}$-enriched $\infty$-categories of Gepner-Haugseng.

As an application of our theory we construct a lax symmetric monoidal embedding of the $\infty$-category of small stable $\infty$-categories into the $\infty$-category of small spectral $\infty$-categories.
As a second application we produce a Yoneda-embedding for Lurie's notion of enriched $\infty$-categories.

\end{abstract}

\tableofcontents

\section{Motivation and overview}

\vspace{2mm} 
In \cite{GEPNER2015575} Gepner-Haugseng develop a theory of enriched $\infty$-categories
and show that many results from enriched category theory and homotopy theory hold in their framework. Among other results they prove the homotopy hypothesis \cite[Corollary 6.1.10]{GEPNER2015575} and the Baez-Dolan stabilization hypothesis \cite[Corollary 6.2.9.]{GEPNER2015575} of \cite{MR1355899}, show that $\infty$-categories enriched in spaces are modeled by Segal spaces 
\cite[Theorem 4.4.6.]{GEPNER2015575} and 
that the $\infty$-category of small $\infty$-categories enriched in any $\bE_{\n+1}$-monoidal $\infty$-category is $\bE_\n$-monoidal for $\n \geq 1$
\cite[Corollary 5.6.12.]{GEPNER2015575}.
These results prove the theory of Gepner-Haugseng as very powerful in structural and conceptual aspects. 

Despite of that it can be quite tricky to work with Gepner-Haugseng's definition of enriched $\infty$-categories: for example it was some time an open problem to construct an enriched Yoneda-embedding in the framework of Gepner-Haugseng.
This problem was solved by Hinich \cite{HINICH2020107129} by introducing a different but due to work of Macpherson \cite{MR4185309} equivalent definition of enriched $\infty$-categories \cite[Proposition 3.4.4.]{HINICH2020107129}.
In general it seems quite intricate to work with enriched $\infty$-categories of functors. 
Although homotopy-coherent enrichment can be rectified \cite[Theorem 1.1]{MR3402334}, enriched $\infty$-categories of functors are rarely modeled by 
their strictly enriched counterparts, what the basic example of enrichment in spaces already demonstrates:
the Bergner model structure \cite{MR2276611} on simplicial categories is not monoidal (see \cite{lurie.HTT} before A.3.4.9.) and so one cannot simply derive simplicial functor categories to produce a model for the $\infty$-category of functors. In fact, it is one of the main motivations for creating a theory of homotopy-coherent enrichment that there is no need for strict models.

Therefore one often avoids to deal with homotopy-coherent enrichment and uses
another much simpler considered structure as a substitute for homotopy-coherent enrichment: this structure is an $\infty$-category $\mC$ with a left action of a monoidal $\infty$-category $\mV$ \cite[Definition 4.2.1.19.]{lurie.higheralgebra} that is closed in the following sense:
for any object $\X$ of $\mC$ the action functor $(-) \ot \X: \mV \to \mC$
admits a right adjoint $\mC \to \mV$ sending an object $\Y$ of $\mC$
to an object $\Mor_\mC(\X,\Y) $ of morphisms $\X \to \Y$ in $\mC.$
These morphism objects make the homotopy category $\Ho(\mC)$ enriched in $\Ho(\mV)$, where the axioms of an action imply associativity and unitality of the enriched composition.
By work of Gepner-Haugseng \cite[Theorem 7.4.7.]{GEPNER2015575} and independently Hinich \cite[Proposition 6.3.1.]{HINICH2020107129} 
the induced enrichment of $\Ho(\mC)$ in $\Ho(\mV)$
refines to a homotopy-coherent enrichment of $\mC$ in $\mV$. This motivates viewing $\infty$-categories with closed left action as a substitute for
enriched $\infty$-categories. 

Moreover for convenience one likes to restrict to presentable $\infty$-categories with a closed left action of a presentably monoidal $\infty$-category.
These presentably left tensored $\infty$-categories admit a very rich theory:
by work of Lurie \cite[Proposition 4.8.1.15.]{lurie.higheralgebra} there is a closed symmetric monoidal structure on the $\infty$-category $\Pr^\L$ of presentable $\infty$-categories and left adjoint functors,
where left adjoint functors $\mC \ot \mD \to \mE$ from the tensor product of two presentable $\infty$-categories $\mC,\mD$ to a another one $\mE$ correspond to bilinear functors $\mC \times \mD \to \mE$, i.e. functors preserving small colimits in each component.
Pursuing the analogy between the tensor product for presentable $\infty$-categories and that for abelian groups one often considers the study of algebraic structures in $\Pr^\L$ as a kind of categorified higher linear algebra: (commutative) rings correspond to (symmetric) monoidal presentably $\infty$-categories
and left modules over a (commutative) ring correspond to presentably
left tensored $\infty$-categories.
This analogy is especially fruitful in the stable case, i.e. for stable presentable $\infty$-categories, and is omnipresent in the theory of higher-categorical invariants like higher derived Brauer groups \cite{MR3190610} and secondary K-theory \cite{MR3607274}. Here stable presentably left tensored $\infty$-categories serve as a tractable substitute for spectral presentable $\infty$-categories.

In this article we prove that $\infty$-categories with closed left $\mV$-action,
playing the role of a substitute for $\mV$-enriched $\infty$-categories, are in fact a model for a certain class of $\mV$-enriched $\infty$-categories, which we call tensored: 
a $\mV$-enriched $\infty$-category $\mD$ is tensored over $\mV$ if for any object $\V$ of $ \mV$ and $\X$ of $\mD$ there is an object $\V \ot \X$ in $\mD$, the tensor of $\V$ and $\X$, equipped with a morphism $\V \to \mD(\X, \V\ot \X)$ in $\mV$, where $ \mD(-,-) $ is the graph of $\mD$, such that for any object
$\Y $ of $\mD$ the canonical composition
$\mD(\V \otimes \X,\Y) \to \mD(\X,\Y)^{\mD(\X,\V \otimes \X)} \to \mD(\X,\Y)^\V$ is an equivalence, where the exponent is the internal hom of $\mV.$

Before we state our results, we need to say a word about the terminology
used by Gepner-Haugseng: Gepner-Haugseng associate to any $\mV$-enriched $\infty$-category a classifying Segal space \cite[Definition 5.1.1.]{GEPNER2015575},
which for $\mV$ the $\infty$-category of spaces gives an equivalence
between $\mV$-enriched $\infty$-categories and Segal spaces \cite[Theorem 4.4.6.]{GEPNER2015575}.
If the associated Segal space happens to be complete, they call the $\mV$-enriched $\infty$-category complete \cite[Definition 5.1.7.]{GEPNER2015575}.
Since Segal spaces are a model for $\infty$-precategories and complete Segal spaces are a model for $\infty$-categories,
Hinich \cite{HINICH2020107129} refers to $\mV$-enriched $\infty$-categories in the sense of Gepner-Haugseng as $\mV$-enriched $\infty$-precategories
and reserves the name $\mV$-enriched $\infty$-categories exclusively for the complete $\mV$-enriched $\infty$-precategories. We follow Hinich's terminology.
Moreover we call the (complete) Segal space associated to a $\mV$-enriched $\infty$-(pre)category the underlying $\infty$-(pre)category, and call a $\mV$-enriched $\infty$-(pre)category small if the associated Segal space is a simplicial object of small spaces. We prove the following theorem:
\begin{theorem}\label{1}
Let $\mV$ be a monoidal $\infty$-category, $ \LMod_\mV(\Cat_\infty)$ the $\infty$-category of small $\infty$-categories with left $\mV$-action and
$\Cat_\infty^{\mV}$ the $\infty$-category of small $\mV$-enriched $\infty$-categories. 
Sending an $\infty$-category with closed left $\mV$-action to its associated $\mV$-enriched $\infty$-category promotes to an equivalence 
$$ \LMod_\mV(\Cat_\infty)^\cl \simeq \Cat_\infty^{\mV,\mathrm{ten}}$$
between the full subcategory $ \LMod_\mV(\Cat_\infty)^\cl \subset \LMod_\mV(\Cat_\infty)$
of $\infty$-categories with closed left $\mV$-action
and a non-full subcategory $\Cat_\infty^{\mV,\mathrm{ten}} \subset \Cat_\infty^{\mV}$ of tensored $\mV$-enriched $\infty$-categories. 
\end{theorem}

Before we explain how to prove this theorem, we
will give several specializations, generalizations and enhancements of this theorem.
Restricting to presentably left tensored $\infty$-categories we obtain the following corollary,
where we call a $\mV$-enriched functor $\F: \mC \to \mD$ a left adjoint if
for any object $\Y$ of $\mD$ the presheaf $\X \mapsto \mD(\F(\X), \Y)$ on $\mC$ is representable. 
\begin{theorem}\label{1.1}
Let $\mV$ be a presentably monoidal $\infty$-category, $\widehat{\Cat}_\infty^\mV$ the $\infty$-category of large $\mV$-enriched $\infty$-categories, and $\Pr^\L_\mV \subset \widehat{\Cat}_\infty^\mV$ the subcategory of tensored $\mV$-enriched $\infty$-categories whose underlying $\infty$-category is presentable, and left adjoint $\mV$-enriched functors.
Sending a presentably left tensored $\infty$-category to its associated $\mV$-enriched $\infty$-category promotes to an equivalence 
$$ \LMod_\mV(\Pr^\L) \simeq \Pr^\L_\mV. $$ 
\end{theorem}

If $\mV$ is a presentably $\bE_{\n+1}$-monoidal $\infty$-category for $\n \geq 1$, Gepner-Haugseng \cite[Proposition 4.3.10.]{GEPNER2015575} and Hinich \cite[Corollary 3.5.3.]{HINICH2020107129} endow $\widehat{\Cat}_\infty^\mV$ with an $\bE_{\n}$-monoidal structure.
Also the $\infty$-category $ \LMod_\mV(\Pr^\L)$ is $\bE_\n$-monoidal,
where the tensor product of two $\mV$-modules $\mC,\mD$ in $\Pr^\L$ is the relative tensor product $\mC \ot_\mV \mD$ computed as the colimit of a derived Bar-construction 
$ \mathrm{Bar}(\mC,\mV,\mD)$, a simplicial object in $\Pr^\L$,
whose $\n$-th term is $\mC \ot \mV^{\ot \n} \ot \mD$.
We prove the following enhancement of Theorem \ref{1.1}:

\begin{theorem}(Theorem \ref{corrr})
Let $\mV$ be a presentably $\bE_{\n+1}$-monoidal $\infty$-category for $1 \leq \n \leq \infty$. 
	
The composition
$$ \LMod_\mV(\Pr^\L) \simeq \Pr^\L_\mV \subset \widehat{\Cat}_\infty^\mV$$ is lax $\bE_\n$-monoidal.
	
\end{theorem}

Lurie \cite[Proposition 4.8.2.18.]{lurie.higheralgebra} constructs from the symmetric monoidal structure on $\Pr^\L$ a symmetric monoidal structure on the full subcategory $\Pr^\L_\mathrm{st}$ of stable presentable $\infty$-categories. The tensor unit is the $\infty$-category of spectra $\Sp,$
which therefore inherits a symmetric monoidal structure.
This symmetric monoidal structure on $\Pr^\L_\mathrm{st}$ is determined by the requirement that the forgetful functor $\LMod_\Sp(\Pr^\L) \to \Pr^\L$ restricts to a symmetric monoidal equivalence $ \LMod_\Sp(\Pr^\L) \simeq \Pr^\L_\mathrm{st}$ \cite[Proposition 4.8.2.10.]{lurie.higheralgebra}.
We get the following corollary:

\begin{corollary}
	
There is a lax symmetric monoidal inclusion
$$ \Pr^\L_\mathrm{st} \hookrightarrow \widehat{\Cat}_\infty^\Sp $$
from stable presentable $\infty$-categories to large spectral $\infty$-categories.
	
\end{corollary}

Looking at Theorem \ref{1} the question arises what happens on the left hand side of the equivalence if one takes all of $\Cat_\infty^\mV$ on the right hand side.
In \cite[Definition 4.2.1.12, 4.2.1.25, 4.2.1.28]{lurie.higheralgebra} Lurie gives several generalizations of the notion of an $\infty$-category with closed left action of a monoidal $\infty$-category $\mV$, which he calls $\infty$-categories enriched in $\mV$, pseudo-enriched in $\mV$ and weakly enriched in $\mV$ going from the most special to the most general notion.
We refer to $\infty$-categories weakly enriched in $\mV$ in Lurie's sense as $\infty$-categories weakly left tensored over $\mV$ since we will use Lurie's term in another meaning. 

To understand how these notions extend the concept of an $\infty$-category with closed left action let us recall that a left action of a monoidal $\infty$-category $\mV$ on an $\infty$-category $\mC$
gives rise to a non-symmetric $\infty$-operad $\mO$ whose colors are the objects of $\mV$ and $\mC$, and whose space of multimorphisms 
$\Mul_\mO(\V_1,...,\V_\n,\X;\Y) $ for $\V_1,...,\V_\n \in \mV$, $\n \geq 0$ and $\X,\Y \in 
\mV \coprod \mC$ is equivalent to the mapping space $\map_\mU(\V_1 \ot ... \ot \V_\n \ot \X \to \Y)$ for $\mU=\mV,\mC$, respectively.
For $\mV$ and $\mC$ contractible we write $\LM$ for $\mO$. So $\LM$ is a non-symmetric $\infty$-operad with two colors $\mathfrak{a}, \mathfrak{m}$.

By functorility the non-symmetric $\infty$-operad $\mO$ comes equipped with
a map of non-symmetric $\infty$-operads $\mO \to \LM,$ 
which in fact completely determines the left action of $\mV$ on $\mC$:
for any object $\V$ of $\mV$ and $\X $ of $\mC$ one can reconstruct $\V \ot \X$
as the object corepresenting the functor $\mC \to \mS, \Y \mapsto \Mul_\mO(\V,\X;\Y)$.
In fact an arbitrary map of non-symmetric $\infty$-operads $\mO \to \LM$ is induced by a left action of the fiber $\mV:= \mO_{\mathfrak{a}}$ on the fiber $\mC:= \mO_{\mathfrak{m}}$ if and only if the following three conditions hold:

\begin{enumerate}
\item The fiber $\mV=\mO_{\mathfrak{a}}$, which is canonically a non-symmetric $\infty$-operad, is a monoidal $\infty$-category.
\item For any objects $\V_1,...,\V_\n $ of $\mV$ for $\n \geq 0$ and $\X,\Y$ of $ \mC$ the natural map 
$$ \Mul_\mO(\V_1 \ot ... \ot \V_\n,\X;\Y) \to \Mul_\mO(\V_1,...,\V_\n,\X;\Y) $$ is an equivalence.
\item For any object $\V$ of $\mV$ and $\X $ of $\mC$ there is an object $\V \ot \X$ and an equivalence $$\Mul_\mO(\V_1,...,\V_\n, \V \ot \X;\Y) \simeq \Mul_\mO(\V_1,...,\V_\n, \V,\X;\Y)$$
natural in objects $\V_1,...,\V_\n$ of $\mV$ for $\n \geq 0 $ and $\Y$ of $\mC$. 
	
\end{enumerate}
If (3) holds, there are canonical equivalences
$$\Mul_\mO(\V, \W,\X;\Y) \simeq \Mul_\mO(\V, \W \ot \X;\Y) \simeq \map_\mC(\V \ot (\W \ot \X),\Y), $$
$$ \Mul_\mO(\V \ot \W, \X;\Y) \simeq \map_\mC((\V \ot \W) \ot \X,\Y).$$
so that the natural map $\alpha: \Mul_\mO(\V \ot \W, \X;\Y) \to \Mul_\mO(\V, \W,\X;\Y)$
represents a morphism $\beta: \V \ot (\W \ot \X)\to (\V \ot \W) \ot \X $ in $\mC.$
If also (1) and (2) hold, $\alpha$ and so $\beta$ are equivalences providing associativity.

To arrive at the notion of weakly left tensored $\infty$-category one drops all three conditions: a weakly left tensored $\infty$-category is by definition a map of non-symmetric $\infty$-operads $\mO \to \LM$. We say that $\mO \to \LM$
exhibits $\mC:= \mO_{\mathfrak{m}}$ as weakly left tensored over $\mV:=\mO_{\mathfrak{a}}.$ 
If condition (1) and (2) hold, we say that $\mO \to \LM$ exhibits $\mC$ as pseudo-enriched in $\mV.$ 
If only condition (3) holds, we say that $\mO \to \LM$ exhibits $\mC$ as locally left tensored over $\mV.$
We say that $\mO \to \LM$ exhibits $\mC$ as enriched in $\mV$ if for any object $\X $ of $\mC$ and $\Y$ of $\mD$ there is an object $\Mor_\mC(\X,\Y)$ of $\mV$ and an equivalence $$\Mul_\mO(\V_1,...,\V_\n, \X;\Y) \simeq \Mul_\mV(\V_1,...,\V_\n, \Mor_\mC(\X,\Y))$$ natural in objects $\V_1,...,\V_\n$ of $\mV$ for $\n \geq 0 $.
This is an extension of Lurie's definition \cite[Definition 4.2.1.28]{lurie.higheralgebra} from enrichment in a monoidal $\infty$-category to enrichment in any non-symmetric $\infty$-operad.

In total we obtain the following commutative diagram, where all squares are pullback squares:

\begin{equation*} 
\begin{xy}
\xymatrix{
& \{\underset{\text{$\infty$-categories}}{\text{closed left tensored}}\} \ar[rd]^{ }\ar[ld]^{ }
\\
\{\underset{\text{$\infty$-categories}}{\text{left tensored}}\}\ar[d]\ar[rd] && \{\underset{\text{in a monoidal $\infty$-category}}{\text{$\infty$-categories enriched}}\} \ar[d] \ar[ld]
\\
\{\underset{\text{$\infty$-categories}}{\text{locally left tensored}}\}  \ar[rd] & \{\underset{\text{$\infty$-categories}}{\text{pseudo enriched}}\} \ar[d] & \{\underset{\text{in an $\infty$-operad}}{\text{$\infty$-categories enriched}}\} \ar[ld] 
\\ & \{\underset{\text{$\infty$-categories}}{\text{weakly left tensored}}\}
}
\end{xy} 
\end{equation*}
In this work we identify the different classes of weakly left tensored $\infty$-categories of the last diagram with classes of enriched $\infty$-categories in the sense of Gepner-Haugseng, which we explain now:

For every $\infty$-category $\mC$ weakly left tensored over a non-symmetric $\infty$-operad $\mV$, objects $\X,\Y$ of $\mC$ and a natural $\n \geq 0$
there is a presheaf $ \V_1,...,\V_\n \mapsto \Mul_\mO(\V_1,...,\V_\n,\X;\Y)$ on $\mV^{\times \n}$. For varying $\n \geq 0$ these presheaves combine to a single presheaf $ \Mul_\mO(-,..,-,\X;\Y)$ on a certain $\infty$-category $\Env(\mV)$, the enveloping monoidal $\infty$-category of $\mV,$ which was intensively studied by Lurie \cite[2.2.4.]{lurie.higheralgebra}. 

The $\infty$-category $\Env(\mV)$ consists of finite sequences of objects of $\mV$ and
carries a monoidal structure given by concatenation of sequences.
Identifying an object of $\mV$ with a sequence of length one defines an embedding of non-symmetric $\infty$-operads $\mV \subset \Env(\mV)$, which exhibits $\Env(\mV)$ as the initial monoidal $\infty$-category under $\mV$ giving $\Env(\mV)$ the name.
The presheaves $\Mul_\mO(-,..,-,\X;\Y)$ for $\X,\Y \in \mC$ make the homotopy-category $\Ho(\mC)$ enriched in $\Ho(\mP\Env(\mV))$, where $\mP\Env(\mV)$ is the $\infty$-category of presheaves on $\Env(\mV)$ endowed with the Day-convolution monoidal structure.

If $\mV$ is a monoidal $\infty$-category, there is a lax monoidal embedding $\mP\mV \subset \mP\Env(\mV)$ sending a presheaf $\F$ on $\mV$
to the presheaf $ (\V_1,...\V_\n) \mapsto \F(\V_1 \ot ...\ot \V_\n)$. 
If $\mC$ is pseudo-enriched in a monoidal $\infty$-category $\mV$, by condition (2) the presheaf $ \Mul_\mO(-,..,-,\X;\Y) $ on $\Env(\mV)$ is the image of
the presheaf $\V \mapsto \Mul_\mO(\V,\X;\Y) $ on $\mV$. So in this case the homotopy-category $\Ho(\mC)$ is enriched in $\Ho(\mP\mV).$
If $\mC$ is enriched in a non-symmetric $\infty$-operad $\mV$, for any $\X, \Y \in \mC$ the presheaf $\Mul_\mO(-,..,-,\X;\Y)$ on $\Env(\mV)$ is represented by an object of
$\mV$. So in this case $\Ho(\mC)$ is enriched in $\Ho(\mV).$

For any $\infty$-category $\mC$ weakly left tensored over a non-symmetric $\infty$-operad $\mV$ we lift the enrichment of $\Ho(\mC)$ in $\Ho(\mP\Env(\mV))$ to an enrichment of $\mC$ in $\mP\Env(\mV)$
in the sense of Gepner-Haugseng, which we call a weak enrichment of $\mC$ in $\mV$, (Theorem \ref{pqp}).
We extend Theorem \ref{1} to the following one:

\begin{theorem}\label{2}
Let $\mV$ be a small non-symmetric $\infty$-operad, $ \omega\LMod_\mV$ the $\infty$-category of small $\infty$-categories weakly left tensored over $\mV$ and $\omega\Cat_\infty^{\mV}$ the $\infty$-category of small $\mP\Env(\mV)$-enriched $\infty$-categories.
Sending a small $\infty$-category weakly left tensored over $\mV$ to the associated $\mP\Env(\mV)$-enriched $\infty$-category promotes to a canonical equivalence 
$$ \chi: \omega\LMod_\mV \simeq \omega\Cat_\infty^{\mV},$$
under which the following structures correspond, where enrichment on the right hand side is by Gepner-Haugseng:	
	
\begin{tabular}[h]{l|c}
$\infty$-categories pseudo-enriched in $\mV$ & $\mP\mV$-enriched $\infty$-categories \\
\hline
$\infty$-categories enriched in $\mV$ & $\mV$-enriched $\infty$-categories \\
\hline
$\infty$-categories locally left tensored over $\mV$ & $\mP\Env(\mV)$-enriched $\infty$-categories tensored over $\mV$ \\
\hline
$\infty$-categories tensored over $\mV$ & $\mP\mV$-enriched $\infty$-categories tensored over $\mV$ \\
\hline
$\infty$-categories with closed left action of $\mV$ & $\mV$-enriched $\infty$-categories tensored over $\mV$ \\
\end{tabular} 
	
\end{theorem}

Theorem \ref{2} is a corollary of the following one involving enriched $\infty$-precategories:

\begin{theorem}\label{3}(Theorem \ref{trfd})
	
Let $\mV$ be a small non-symmetric $\infty$-operad and
$\omega\mathrm{Pre}\Cat_\infty^{\mV}$ the $\infty$-category of small $\mP\Env(\mV)$-enriched $\infty$-precategories in the sense of Gepner-Haugseng.
Sending a small $\infty$-category weakly left tensored over $\mV$ to its underlying $\infty$-category weakly enriched in $\mV$ promotes to a functor
$$ \chi: \omega\LMod_\mV \to \omega\Cat_\infty^{\mV} \subset \omega\mathrm{Pre}\Cat_\infty^{\mV}$$
that admits a left adjoint $\L.$
For any small $\mV$-enriched $\infty$-precategory $\mC$ the unit $\mC \to \chi(\L(\mC))$ is an equivalence if and only if $\mC$ is a $\mV$-enriched $\infty$-category.	
	
\end{theorem}

Theorem \ref{2} implies that the functor $ \chi: \omega\LMod_\mV \to \omega\Cat_\infty^{\mV}$ is fully faithful.
We prove a stronger result.
For any $\infty$-categories $\mC,\mD$ weakly left tensored over $\mV$
the collection of functors $\mC \to \mD$ compatible with the weak left $\mV$-action,
which we call lax $\mV$-linear functors, forms an $\infty$-category $ \LaxLinFun_{\mV}(\mC,\mD)$.
Similarly, for any monoidal $\infty$-category $\mW$ compatible with small colimits
the $\mW$-enriched functors $\chi(\mC) \to \chi(\mD)$ form an $\infty$-category $\Fun^{\mW}(\chi(\mC),\chi(\mD))$.

We prove the following theorem:
\begin{theorem}\label{00}
	
For any non-symmetric $\infty$-operad $\mV$ and $\infty$-categories $\mC,\mD$ weakly left tensored over $\mV$ there is a canonical equivalence
$$\LaxLinFun_\mV(\mC,\mD) \simeq \Fun^{\mP\Env(\mV)}(\chi(\mC),\chi(\mD))$$
that induces on maximal subspaces the map induced by $\chi.$	
\end{theorem}

Theorem \ref{00} suggests that the functor $\chi: \omega\LMod_\mV \to \omega\Cat_\infty^{\mV}$ is a functor of $(\infty,2)$-categories,
a $\Cat_\infty$-enriched functor, where the morphism object of two $\infty$-categories $\mC, \mD$ weakly left tensored over $\mV$ is $\LaxLinFun_\mV(\mC,\mD)$ and the morphism object of two $\mP\Env(\mV)$-enriched $\infty$-categories $\mA, \mB$ is $\Fun^{\mP\Env(\mV)}(\mA,\mB)$.
In fact the $\infty$-category $\omega\LMod_\mV$ carries a closed left $\Cat_\infty$-action, where the morphism object of two $\infty$-categories $\mC, \mD$ weakly left tensored over $\mV$ is $\LaxLinFun_\mV(\mC,\mD)$
(Lemma \ref{2enr}).
On the other hand for any monoidal $\infty$-category $\mW$ compatible with small colimits there is a canonical left action of the $\infty$-category of spaces $\mS $ endowed with the cartesian structure on the $\infty$-category $\mW$, where the left action is compatible with the monoidal structures.
The latter action induces a left action of $\Cat_\infty$ on the $\infty$-category 
$\Cat_\infty^\mW$ (Construction \ref{CoN}).
Since $\mP\Env(\mV)$ is a monoidal $\infty$-category compatible with small colimits, the $\infty$-category $\omega\Cat_\infty^{\mV}$ carries a left $\Cat_\infty$-action. 

We prove the following theorem: 

\begin{theorem}\label{4} (Theorem \ref{PP})
Let $\mV$ be a small non-symmetric $\infty$-operad.
The functor $$\chi: \omega\LMod_\mV \to \omega\Cat_\infty^{\mV}$$ refines to a $\Cat_\infty$-linear equivalence.
\end{theorem}
Theorem \ref{4} implies that the left $\Cat_\infty$-action on
$\omega\Cat_\infty^{\mV}$ is closed, and Theorems \ref{3}, \ref{00} guarantee that 
for any $\mP\Env(\mV)$-enriched $\infty$-categories $\mA, \mB$ the $\infty$-category $\Fun^{\mP\Env(\mV)}(\mA, \mB)$ is the morphism object of $\omega\Cat_\infty^{\mV}$ (Corollary \ref{uuij}).
Let $\Cat_\infty^{\mV, \Lur} \subset \omega\LMod_\mV$ be the full subcategory spanned by those $\infty$-categories $\mC$ weakly left tensored over $\mV$ that exhibit $\mC$ as $\mV$-enriched. 

\begin{theorem}\label{99} (Corollary \ref{QQ})
Let $\mW$ be a monoidal $\infty$-category compatible with small colimits.
The $\Cat_\infty$-linear functor $\chi: \omega\LMod_\mW \to \omega\Cat_\infty^{\mW}$ restricts to a $\Cat_\infty$-enriched equivalence
$$\chi: \Cat_\infty^{\mW, \Lur} \to \Cat_\infty^{\mW}. $$

\end{theorem}

Theorem \ref{99} guarantees that the left $\Cat_\infty$-action on $ \Cat_\infty^{\mW}$ is closed, Theorems \ref{3}, \ref{00} imply that 
for any $\mW$-enriched $\infty$-categories $\mA, \mB$ the $\infty$-category $\Fun^{\mW}(\mA, \mB)$ is the morphism object of $\Cat_\infty^{\mW}$
(Corollary \ref{uui}).

Moreover we strongly refine Theorem \ref{00} via the following notion of weakly bitensored $\infty$-category relaxing the notion of bitensored $\infty$-category:
like any left tensored $\infty$-category also any bitensored $\infty$-category has an underlying non-symmetric $\infty$-operad, which comes equipped with a map to the non-symmetric $\infty$-operad $\BM$ underlying the biaction of the contractible $\infty$-category on itself.

A weakly bitensored $\infty$-category is by definition a map of non-symmetric $\infty$-operads $\mO \to \BM.$
For any non-symmetric $\infty$-operads $\mV,\mW$, any $\infty$-category $\mC$ weakly left 
tensored over $\mV$ and $\infty$-category $\mD$ weakly bitensored over $\mV,\mW$
the $\infty$-categories $\LaxLinFun_\mV(\mC,\mD),$
$\Fun^{\mP\Env(\mV)}(\chi(\mC),\chi(\mD))$ are weakly right tensored over $\mW$
(Proposition \ref{lehmmm} and Notation \ref{exxxx}) and we prove the following theorem:

\begin{theorem}\label{rat}(Theorem \ref{werf})
For any non-symmetric $\infty$-operads $\mV,\mW$ and $\infty$-categories $\mC$ weakly left tensored over $\mV$ and $\mD$ weakly bitensored over $\mV,\mW$
the canonical equivalence
$$\LaxLinFun_\mV(\mC,\mD) \simeq \Fun^{\mP\Env(\mV)}(\chi(\mC),\chi(\mD))$$
is compatible with the weak right $\mW$-action.	 
\end{theorem}

Before we explain how we construct the adjunction $ \L:  \omega\mathrm{Pre}\Cat_\infty^{\mV} \rightleftarrows \omega\LMod_\mV:\chi$ and prove Theorem \ref{00}, we like to mention two applications of our theory.
The first application is a comparison between stable and spectral $\infty$-categories,
the second is a construction of enriched Yoneda-embedding in Lurie's model
of enriched $\infty$-categories.

We start with the comparison result: in \cite{article} Blumberg-Gepner-Tabuada develop a theory of derived non-commutative motives giving rise to a universal characterization of higher algebraic K-theory. 
One important tool in their framework is a comparison \cite[Theorem 4.22.]{article} between stable and spectral $\infty$-categories, i.e. $\infty$-categories enriched in the $\infty$-category of spectra $\Sp$ endowed with the smash product.
The authors model the $\infty$-category of small spectral $\infty$-categories
by the Dwyer-Kan model structure on categories enriched in an appropriate model for symmetric spectra, which is possible by \cite[Theorem 1.1]{MR3402334}.
Using this model they construct a functor $\Cat_\infty^\Sp \to \Cat_\infty^{\mathrm{st}}$ associating to any spectral $\infty$-category 
a universal stable $\infty$-category, the stable closure.
Their main theorem \cite[Theorem 4.22.]{article} says that the functor $\Cat_\infty^\Sp \to \Cat_\infty^{\mathrm{st}}$ admits a fully faithful right adjoint.

We give a direct proof for this statement without need to pass to strict models.
Moreover we prove a monoidal version: the $\infty$-category $\Cat_\infty^\Sp$
carries a canonical symmetric monoidal structure.
By Lemma \ref{lmo} the $\infty$-category $\Cat_\infty^{\mathrm{st}}$ carries a symmetric monoidal structure whose tensor unit is the $\infty$-category of compact spectra
and such that exact functors $\mC \ot \mD \to \mE$ correspond to biexact functors $\mC \times \mD \to \mE.$
We prove the following theorem:

\begin{theorem}\label{wqp} (Theorem \ref{spectrall})
There is a canonical lax symmetric monoidal embedding
$$ \Cat_\infty^\mathrm{st} \hookrightarrow \Cat_\infty^\Sp$$
of the $\infty$-category of small stable $\infty$-categories into the
$\infty$-category of small spectral $\infty$-categories.
The essential image precisely consists of the spectral $\infty$-categories $\mD$ 
that are stable and such that the enriched hom-functor $\mD^\op \times \mD \to \Sp$ is biexact.	 
\end{theorem}

Theorem \ref{wqp} gives another description of the tensor product $\mC \otimes^{\mathrm{st}} \mD$ of two small stable $\infty$-categories $\mC,\mD$. By Theorem \ref{wqp} we may identify stable $\infty$-categories with certain spectral $\infty$-categories, which we denote by the same name. If we write $\mC \otimes^\Sp \mD$ for the spectral tensor product, by Theorem \ref{wqp} there is a canonical spectral functor $\mC \ot^\Sp \mD \to \mC \otimes^{\mathrm{st}} \mD $, which we show to exhibit $ \mC \otimes^{\mathrm{st}} \mD $ as the stable closure of $\mC \otimes^{\Sp} \mD$ as a consequence of Corollary \ref{symrea}. 

The second application of our theory is a construction of enriched Yoneda-embedding in Lurie's model. For that we construct for any $\mV$-enriched $\infty$-category $\mM$ in the sense of Lurie an opposite enriched $\infty$-category $\mM^\op$ (Notation \ref{notori}), in which source and target of the morphism objects are reversed, which is enriched in the opposite non-symmetric $\infty$-operad $\mV^\rev$. Our construction purely uses Lurie's model of enrichment and corresponds under $\chi$ to the opposite
enriched $\infty$-category of \cite[6.2.1.]{HINICH2020107129} (Corollary \ref{uuu}). 
Based on that we construct a $\mV$-enriched Yoneda-embedding $\mM \to \LaxLinFun_{\mV^\rev}(\mM^\op, \mV)$ purely in terms of Lurie's definition of enrichment (Theorem \ref{ooo}, Corollary \ref{uuuzz}).

In the following we explain how we prove Theorem \ref{00}.
In \cite{HINICH2020107129} Hinich constructs for every small space $\X$ and presentably monoidal $\infty$-category $\mV$ a monoidal structure on the $\infty$-category of functors $\X \times \X \to \mV $ 
and defines $\mV$-enriched $\infty$-precategories with space of objects $\X$ as associative algebras for this monoidal structure \cite[Definition 3.1.1]{HINICH2020107129}.
The tensor product on functors $\X \times \X \to \mV $ is a categorification of the product of matrices sending two functors $\F,\G: \X \times \X \to \mV$ to the functor
$\A,\C \mapsto \colim_{\B \in \X} \F(\B,\C) \ot \G(\A,\B)$.
Therefore the multiplication map of an associative algebra
$\mC: \X \times \X \to \mV$ corresponds to a family of compatible maps
$ \mC(\B,\C) \ot \mC(\A,\B) \to \mC(\A,\C)$ in $\mV$ for $\B \in \X$
equipped with coherence data, and so looks like a model of a $\mV$-enriched $\infty$-precategory in the sense of Gepner-Haugseng.
This was made formal by work of Macpherson \cite{MR4185309}.
An alternative construction of this tensor product was given by Haugseng \cite[Corollary 3.6.4.]{Rune} who directly identified associative algebras for this tensor product with $\mV$-enriched $\infty$-precategories in the sense of Gepner-Haugseng.

The monoidal structure on functors $\X \times \X \to \mV$ is very powerful
since it acts in various ways: a closed left (right) $\mV$-action on a cocomplete $\infty$-category $\mM$ gives rise to a closed left (right) action of the $\infty$-category of functors $\X \times \X \to \mV$ on the $\infty$-category of functors $\X \to \mM$ \cite[3.4.2.]{HINICH2020107129}, Proposition \ref{xre}.
The left action sends functors $\rH : \X \times \X \to \mV, \F: \X \to \mM$ 
to the functor $\Z \mapsto \colim_{\Y \in \X} \rH(\Y,\Z) \ot \F(\Y).$ 
Thus a left action of a $\mV$-enriched $\infty$-precategory $\mC$ with space of objects $\X$ on a functor $ \F: \X \to \mM$ corresponds to a collection of
maps $\mC(\Y,\Z) \to \Mor_\mM(\F(\Y), \F(\Z)) $ in $\mV$ for $\Y,\Z \in \X$ 
equipped with coherence data, and so looks like a $\mV$-enriched functor from $\mC $ to $\mM.$
Similarly, if $\mM$ is right tensored over $\mV$, a right action of an $\infty$-precategory $\mC$ on a functor $ \F: \X \to \mM$ looks like a $\mV$-enriched presheaf on $\mC$ with values in $\mM.$ 

Following this intuition Hinich defines the $\infty$-category $\Fun^\mV(\mC,\mM) $ of $\mV$-enriched functors $\mC \to \mM$ as the $\infty$-category of left $\mC$-modules in $\Fun(\X,\mM)$ \cite[6.1.]{HINICH2020107129}.
Similarly, he defines the $\infty$-category $\mP_\mV(\mC,\mM)$ of $\mV$-enriched presheaves on $\mC$ with values in $\mM$ as the $\infty$-category of right $\mC$-modules in $\Fun(\X,\mM)$ \cite[6.2.]{HINICH2020107129}. 
Since for $\mM=\mV$ the right action of the $\infty$-category of functors $\X \times \X \to \mV$ on the $\infty$-category of functors $\X \to \mV$
is compatible with the left diagonal $\mV$-action, the
$\infty$-category $\mP_\mV(\mC):= \mP_\mV(\mC,\mV)$ of right $\mC$-modules is left tensored over $\mV$. Therefore we can use $\mP_\mV(\mC)$ to turn enrichment in the sense of Gepner-Haugseng to enrichment in the sense of Lurie:
for any $\infty$-category $\mC$ weakly enriched in $\mV$ with small space of objects $\X$ we write $\L(\mC) \subset \mP_{\mP\Env(\mV)}(\mC) $ for the essential image of the $\mP\Env(\mV) $-enriched Yoneda-embedding.
Because $ \mP_{\mP\Env(\mV)}(\mC)$ is left tensored over $\mP\Env(\mV),$ the $\infty$-category $\L(\mC)$ is weakly left tensored over $\mV.$

To define $\chi$ we introduce a variant of the $\infty$-category $\Fun^\mV(\mC,\mM)$ of $\mV$-enriched functors $\mC \to \mM$ when $\mV$ is a non-symmetric $\infty$-operad
and $\mC$ is weakly enriched in $\mV.$
For that we use a variation of the enveloping monoidal $\infty$-category
$\Env(\mV)$ associated to a monoidal $\infty$-category $\mV:$
like $\Env(\mV)$ is initial among monoidal $\infty$-categories into which $\mV$ embeds, for any $\infty$-category $\mM$ weakly left tensored over a non-symmetric $\infty$-operad $\mV$ there is an initial $\infty$-category $\L\Env(\mM)$ left tensored over $\Env(\mV)$ into which $\mM$ embeds (Lemma \ref{gre}). Taking presheaves we get an $\infty$-category $\mP\L\Env(\mM)$ closedly left tensored over $\mP\Env(\mV)$, which we call the enveloping $\infty$-category with closed left action. 
Then we define the $\infty$-category $\Fun^\mV(\mC,\mM)$ of $\mV$-enriched functors $\mC$ to $\mM$ as the $\infty$-category of 
$\mP\Env(\mV)$-enriched functors $\mC$ to $\mP\L\Env(\mM)$
that carry objects of $\X$ to $ \mM \subset \mP\L\Env(\mM)$.
Via $\Fun^\mV(\mC,\mM)$ we define $\chi(\mM)$ by the following theorem:
\begin{theorem}\label{prroqa} (Theorem \ref{proooq})
Let $\mV$ be a small non-symmetric $\infty$-operad and $\mM$ a small $\infty$-category weakly left tensored over $\mV$.
The presheaf $$(\omega\mathrm{Pre}\Cat_\infty^{\mV})^\op \to \mS, \ \mC \mapsto \Fun^{\mV}(\mC, \mM)^\simeq $$ is represented by some $\infty$-precategory $\chi(\mM)$ weakly enriched in $\mV.$

\end{theorem}

Theorem \ref{prroqa} describes a universal property of the $\infty$-category $\Fun^{\mV}(\mC,\mM)$ of $\mV$-enriched functors. 
As a second step we prove a universal property of the $\infty$-category of $\mV$-enriched presheaves:  
\begin{theorem}\label{th2} (Theorem \ref{tgre})
Let $\mV$ be a presentably monoidal $\infty$-category and $\mM$ an $\infty$-category 
left tensored over $\mV$ that admits small colimits preserved by the left action in both variables.
	
For any $\mV$-enriched $\infty$-precategory $\mC$ with small space of objects $\X$ there is a canonical equivalence
\begin{equation}\label{Yoneet}
\LinFun_\mV^\L(\mP_\mV(\mC),\mM) \simeq \Fun^\mV(\mC,\mM)
\end{equation}	
over $\Fun(\X,\mM),$
where the left hand side is the $\infty$-category of left adjoint $\mV$-linear functors $\mP_\mV(\mC) \to \mM$.
	
\end{theorem}

In \cite[Theorem 4.8.4.1.]{lurie.higheralgebra} Lurie proves a universal property of the $\infty$-category $\RMod_\A(\mV)$ of right $\A$-modules in any presentably monoidal $\infty$-category $\mV$:
there is a canonical equivalence
\begin{equation}\label{rmoo}
\LinFun_\mV^\L(\RMod_\A(\mV),\mM) \simeq \LMod_\A(\mM),
\end{equation}
which is Theorem \ref{th2} for $\X$ contractible.

In \cite{HINICH2020107129} Hinich constructs an enriched Yoneda-embedding.
Under equivalence (\ref{Yoneet}) the identity of $\mP_{\mV}(\mC)$ corresponds to another candidate of enriched Yoneda-embedding.
In \cite{hinich2021colimits} Hinich proves a similar universal property for his
version of enriched Yoneda-embedding, which identifies both constructions.

By Theorem \ref{th2} for any non-symmetric $\infty$-operad $\mV$,
$\infty$-category $\mM$ weakly left tensored over $\mV$ and $\infty$-precategory $\mC$ weakly enriched in $\mV$ there is a canonical equivalence 
\begin{equation*}\label{eqwc}
\LinFun_{\mP\Env(\mV)}^\L(\mP_{\mP\Env(\mV)}(\mC),\mP\L\Env(\mM)) \simeq \Fun^{\mP\Env(\mV)}(\mC,\mP\L\Env(\mM))
\end{equation*}
that restricts to an equivalence 
\begin{equation}\label{eqyy}
\LinFun_{\mP\Env(\mV)}^{\L}(\mP_{\mP\Env(\mV)}(\mC),\mP\L\Env(\mM))' \simeq \Fun^{\mV}(\mC,\mM),
\end{equation}
where the left hand side is the full subcategory spanned by the left adjoint
$\mP\Env(\mV)$-linear functors sending representable enriched presheaves to $\mM.$ 
We combine equivalence (\ref{eqyy}) with the following theorem:
\begin{theorem}\label{propo1} (Theorem \ref{cooo})
Let $\mV$ be a small non-symmetric $\infty$-operad and $\mC$ a small $\infty$-precategory weakly enriched in $\mV$.
The embedding $$\L(\mC) \subset \mP_{\mP\Env(\mV)}(\mC)$$
exhibits $\mP_{\mP\Env(\mV)}(\mC)$ as enveloping $\infty$-category with closed left $\mP\L\Env(\mV)$-action associated to $\L(\mC).$	 
\end{theorem}
Theorem \ref{propo1} shows one advantage of weakly enriched $\infty$-categories over enriched $\infty$-categories:
for $\infty$-categories $\mC$ weakly enriched in a non-symmetric $\infty$-operad $\mV$ it
is not reasonable to consider the $\infty$-category of $\mV$-enriched presheaves
on $\mC$ since $\mC$ does not embed into the former one. 
Instead one should consider $\mP\Env(\mV)$-enriched presheaves on $\mC$
playing the role of $\mV$-enriched presheaves in the world of weakly enriched
$\infty$-categories. Theorem \ref{propo1} says that under the equivalence of Theorem \ref{2} 
the $\infty$-category of $\mP\Env(\mV)$-enriched presheaves on $\mC$
identifies with the enveloping $\infty$-category with closed left action, which has a much simpler description. 

By Theorem \ref{propo1} and the universal property of the enveloping $\infty$-category with closed left action there is a canonical equivalence
\begin{equation}\label{eqy}
\LinFun_{\mP\Env(\mV)}^{\L}(\mP_{\mP\Env(\mV)}(\mC),\mP\L\Env(\mM))' \simeq \LaxLinFun_\mV(\L(\mC),\mM).
\end{equation}
Combining equivalences \ref{eqyy} and \ref{eqy} we obtain for any small non-symmetric $\infty$-operad $\mV$, small $\infty$-category $\mM$ weakly left tensored over $\mV$ and small $\infty$-precategory $\mC$ weakly enriched in $\mV$ a canonical equivalence 
\begin{equation}\label{eqayx}
\LaxLinFun_\mV(\L(\mC),\mM) \simeq \Fun^{\mV}(\mC,\mM).
\end{equation}
In view of Theorem \ref{prroqa} equivalence \ref{eqayx} shows that $\chi$ admits a left adjoint, where the unit $\mC \to \chi(\L(\mC))$ at an $\infty$-precategory $\mC$
weakly enriched in $\mV$ is induced by the enriched Yoneda-embedding.
Since the enriched Yoneda-embedding is fully faithful in the enriched sense
(Proposition \ref{yooon}), the unit induces an essentially surjective map on spaces of objects and an equivalence on morphism objects and so is an equivalence if and only if $\mC$ is complete \cite[Corollary 5.2.8.]{GEPNER2015575}.

\vspace{2mm}

\subsubsection{Overview}

In the following we give an overview over the content.

\vspace{1mm}
In section \ref{opera} we lay the foundations for the later sections by introducing several types of $\infty$-operads (Definition \ref{opap}) and generalized versions of those (Definition \ref{opep}):
\begin{itemize}
\item non-symmetric $\infty$-operads that were studied by Lurie \cite[Definition 4.1.3.2.]{lurie.higheralgebra} under the name planar $\infty$-operads as the non-symmetric counterpart of Lurie's model of $\infty$-operads. 
\item generalized non-symmetric $\infty$-operads that were studied by Gepner-Haugseng
\cite[Definition 3.1.3.]{GEPNER2015575} and in the classical context by Leinster 
\cite{leinster_2004}, \cite{Leinster2002} and Shulman \cite{Shulman2009} as the most general place, categories can be enriched in.
\item (generalized) $\LM,$-$\RM$- and $\BM$-operads,
where $\LM,\RM, \BM$ are (approximations to) the $\infty$-operads
governing left, right and bimodules, that encode weak left, right and biactions of one or two (generalized) $\infty$-operads on some $\infty$-category, which were studied by Lurie \cite[Definition 4.2.1.12.]{lurie.higheralgebra} 
as the general context, in which left, right and bimodules are defined.	
\end{itemize}
We treat all these variants of (generalized) $\infty$-operads on one footage
under the name (generalized) $\mO$-operads (Definition \ref{opep}, \ref{opap}).
Moreover we consider families of (generalized) $\mO$-operads (Definition \ref{famm}) that allow us to make constructions functorial, in which (generalized) $\mO$-operads appear.
In section \ref{weak action} we introduce weakly left, right and bitensored
$\infty$-categories (Definition \ref{wla}, \ref{bla})
as a model for generalized $\LM$-,$\RM$- and $\BM$-operads (Propositions \ref{proo}, \ref{bproo}) that seems generally easier to work with.
In subsection \ref{LinF} we study morphisms of $\infty$-categories weakly left tensored over a non-symmetric $\infty$-operad $\mV$, which we call lax $\mV$-linear functors.
We prove that the $\infty$-category of lax $\mV$-linear functors to an $\infty$-category weakly bitensored over non-symmetric $\infty$-operads $\mV,\mW$ is canonically weakly right tensored over $\mW$ (Proposition \ref{lehmmm}), 
a result needed to prove Theorem \ref{rat} (Theorem \ref{werf}).
In subsection \ref{envob} we define and study enveloping $\infty$-categories left, right and bitensored over the enveloping monoidal $\infty$-category. 
In subsection \ref{Enpse} we define enriched and pseudo-enriched $\infty$-categories in the sense of Lurie (Definition \ref{Enr} and \ref{Lu}) as certain weakly left tensored $\infty$-categories.

In section \ref{extr} we introduce weakly enriched $\infty$-categories and construct the functor $\chi$ relating weakly enriched $\infty$-categories to weakly left tensored
$\infty$-categories. We construct $\chi$ via Theorem \ref{pqp} using a theory of Day-convolution for generalized $\mO$-operads developed in section \ref{genDay} and using Lurie's theory
of endomorphism algebras \cite[4.7.1]{lurie.higheralgebra}, which we adapt to
weakly left tensored $\infty$-categories (Proposition \ref{urr}).
In section \ref{univpr} we prove Theorem \ref{th2} (Theorem \ref{tgre}),
a universal property of enriched presheaves.
In section \ref{eqqiv} we prove Theorem \ref{propo1} (Theorem \ref{cooo})
identifying the $\infty$-category of enriched presheaves with the enveloping
$\infty$-category with closed left action associated to the full subcategory of representable presheaves.
From Theorems \ref{pqp}, \ref{tgre} and  \ref{cooo} we deduce our main Theorems \ref{trfd}
and \ref{werf}.
In section \ref{moneq} we study the monoidality of the functor $\chi$
and prove Theorem \ref{corrr}.
In section \ref{EnrYoneda} we construct an enriched Yoneda-embedding 
in the model of Lurie's definition of enriched $\infty$-categories
(Theorem \ref{ooo}).
In section \ref{2Fu} we prove that the functor $\chi$ is a functor of
$(\infty,2)$-categories (Theorem \ref{PP}).
In section \ref{stabspec} we prove Theorem \ref{wqp} (Theorem \ref{spectrall})
providing a lax symmetric monoidal embedding of the $\infty$-category 
of small stable $\infty$-categories into the $\infty$-category 
of small spectral $\infty$-categories.

In section \ref{genDay} we construct Day-convolution for generalized
$\mO$-operads extending Hinich's Day-convolution for $\mO$-operads.
Hinich \cite{HINICH2020107129} develops Day-convolution for $\mO$-operads to construct a monoidal structure on quivers whose associative algebras are enriched $\infty$-precategories with fixed object space.
We produce an extension of Hinich's construction to generalized $\mO$-operads.
Working with generalized $\mO$-operads 
requires more technology but makes constructions easier since Gepner-Haugseng's generalized non symmetric $\infty$-operad governing enrichment is by far simpler than its very combinatorial operadic model.

\subsubsection{Related work}
Gepner-Haugseng \cite[Theorem 7.4.7.]{GEPNER2015575} associate to any $\infty$-category with closed action an underyling enriched $\infty$-category and conjecture that this process extends to an equivalence between Lurie's definition of enriched $\infty$-categories and Gepner-Haugseng's definition \cite[7.2.14.]{GEPNER2015575}.
Hinich \cite[6.3.]{HINICH2020107129} describes an alternative way to produce an enriched $\infty$-category from any $\infty$-category with closed action and proves that his procedure preserves the space of objects and underlying graph. 

In \cite{berman2020enriched} Berman proves structural results about
$\infty$-categories of enriched presheaves and proves a duality between 
enriched presheaves and enriched copresheaves \cite[Theorem 1.7]{berman2020enriched}. Moreover he conjectures Theorem \ref{tgre}
\cite[Remark 1.8.]{berman2020enriched}.
In a recent manuscript \cite{hinich2021colimits} Hinich proves a monoidal version of Theorem \ref{tgre} that universally endows the $\infty$-category of $\mV$-enriched presheaves in any symmetric monoidal $\mV$-category with a symmetric monoidal structure.
From his result he obtains a partial version of Theorem \ref{2} identifying 
enriched $\infty$-categories in the sense of Gepner-Haugseng with a subcategory
of Lurie-enriched $\infty$-categories \cite[Corollary 3.4.1.]{hinich2021colimits}.

\subsection{Acknowledgements}

I thank Hongyi Chu, Markus Spitzweck and Paula Verdugo for helpful discussions. 
Moreover I like to thank David Gepner, Rune Haugseng and Vladimir Hinich for laying the foundations on which I build my work. 
I also thank Rune Haugseng for carefully reading an early version.
Most parts of this work were arising during my postdoc in 2019/2020 in the group of Ieke Moerdijk at Utrecht University. I thank Ieke Moerdijk for his hospitality during this time.

\subsection{Notation and terminology}

We fix a hierarchy of Grothendieck universes whose objects we call small, large, very large, etc.
We call a space small, large, etc. if the the set of path components and all homotopy groups are so for any choice of base point. We call an $\infty$-category small, large, etc. if the maximal subspace and all mapping spaces are so.

We write 
\begin{itemize}
\item $\Set$ for the category of small sets.
\item $\Delta$ for (a skeleton of) the category of finite, non-empty, partially ordered sets and order preserving maps, whose objects we denote by $[\n] = \{0 < ... < \n\}$ for $\n \geq 0$.
\item $\mS$ for the $\infty$-category of small spaces.
\item $ \Cat_\infty$ for the $\infty$-category of small $\infty$-categories.
	
\end{itemize}
We often indicate $\infty$-categories of large objects by $\widehat{(-)}$, for example we write $\widehat{\mS}, \widehat{\Cat}_\infty$ for the $\infty$-categories of large spaces, large $\infty$-categories, respectively.
\vspace{1mm}

\subsubsection{Functor $\infty$-categories}

The homotopy category of $\Cat_\infty$ is cartesian closed.
For any small $\infty$-categories $\mC,\mD$ we write $\Fun(\mC,\mD)$
for the internal hom, the $\infty$-category of functors $\mC \to \mD$.

\vspace{1mm}
For any $\infty$-category $\K$ and functors $\alpha: \mC \to \rS,\beta: \mD \to \rS$ we write
$\Fun_\rS(\mC,\mD)$ for the fiber of the functor $\Fun(\mC,\mD) \to \Fun(\mC,\rS)$
over $\alpha$ and $\mD^\K$ for the pullback of the functor $\Fun(\K,\mD) \to \Fun(\K,\rS)$ along the diagonal functor $\rS \to \Fun(\K,\rS)$. There are canonical equivalences
$$\Fun(\K, \Fun_\rS(\mC,\mD)) \simeq \Fun_\rS(\K \times \mC,\mD) \simeq \Fun_\rS(\mC, \mD^\K).$$

\vspace{1mm}
For any $\infty$-category $\mC$ containing objects $\A, \B$ we write
\begin{itemize}
\item $\mC_{/\B}:= \{\B\} \times_{\Fun(\{1\},\mC)} \Fun([1],\mC)$ for the $\infty$-category of objects over $\B$,
\item $\mC(\A,\B)$ for the space of maps $\A \to \B$ in $\mC$ defined as the fiber of
$\mC_{/\B} \to \mC$ over $\A.$
\item $\Ho(\mC)$ for the homotopy category of $\mC$,
\item $\mC^{\triangleleft}, \mC^{\triangleright}$ for the $\infty$-categories arising from $\mC$ by adding an initial, final object, respectively,
\item $\mC^\simeq $ for the maximal subspace in $\mC$.
\end{itemize}

\subsubsection{Inclusions and embeddings}

We often call a fully faithful functor $\mC \to \mD$ an embedding.
We call a functor $\phi: \mC \to \mD$ an inclusion
or say that $\phi$ exhibits $\mC$ as a subcategory of $\mD$ if for any $\infty$-category $\mB$ the induced map
$\Cat_\infty(\mB,\mC) \to \Cat_\infty(\mB,\mD)$ is an embedding.
A functor $\phi: \mC \to \mD$ is an inclusion if and only if it induces an embedding on maximal subspaces and on all mapping spaces.
We say that an inclusion $\phi: \mC \to \mD$ exhibits $\mC$ as a wide subcategory of $\mD$ if $\phi$ is essentially surjective.

A functor $\phi: \mC \to \mD$ is an inclusion (embedding) if and only if the 
induced functor $\Ho(\phi):\Ho(\mC) \to \Ho(\mD)$ is an inclusion (embedding) and the canonical functor
$\mC \to \Ho(\mC) \times_{\Ho(\mD)} \mD$ is an equivalence.
In this case $\phi$ is essentially determined by $\mD$ and $\Ho(\phi): \Ho(\mC) \to \Ho(\mD).$

\subsubsection{Join construction}\label{join}

We will often use the join of two objects $[\n],[\m] \in \Delta$, which we define as $[\n] \ast [\m]:= [\n+\m+1]$.
There are canonical embeddings $[\n] \subset [\n]\ast[\m], [\m] \subset [\n]\ast[\m]$ sending $[\n]\cong \{0,...,\n\} \subset [\n+\m+1], [\m]\cong \{\n+1,...,\n+\m+1\} \subset [\n+\m+1]$.
The join is functorial in a way that these embeddings are natural.
Moreover we set $[\n] \ast \emptyset = \emptyset \ast [\n] := [\n].$

\subsubsection{Relative cocartesian fibrations}
Let $\mE \subset \Fun([1], \rS)$ a full subcategory. 
We call a functor $\phi: \mC \to \rS$ a cocartesian fibration relative to
$\mE$ if the pullback $\phi': [1]\times_\rS \mC \to [1]$ along any functor
$[1]\to \rS$ corresponding to a morphism of $\mE$ is a cocartesian fibration
and the projection $ [1]\times_\rS \mC \to \mC$
sends $\phi'$-cocartesian morphisms to $\phi$-cocartesian morphisms.

For any cocartesian fibrations $\mC \to \rS, \mD \to \rS$ relative to
$\mE$ a map $\mC \to \mD$ of cocartesian fibrations relative to $\mE$ is a functor $\mC \to \mD$ over $\rS$ that preserves cocartesian lifts of morphisms of $\mE.$

\subsubsection{Relative adjunctions}\label{rell}

For any $\infty$-category $\rS$ and functor $\F:\mC \to \mD$ over $\rS$
we call a right adjoint $\G: \mD \to \mC$ of $\F$ a right adjoint relative to $\rS$ if the unit and counit of the adjunction $\F \dashv \G$ are sent to equivalences by the functors $\mC \to \rS, \mD \to \rS.$
If $\F$ is a map of locally 
cocartesian fibrations over $\rS$, then $\F$ admits a right adjoint relative to $\rS$ if and only if for any $\s \in \rS$ the induced functor
$\mC_\s \to \mD_\s$ on the fiber over $\s$ admits a right adjoint
\cite[Proposition 7.3.2.6.]{lurie.higheralgebra}.

\section{Generalized $\Ass$-, $\LM$-, $\RM$- and $\BM$-operads}
\label{opera}

In this section we define generalized non-symmetric $\infty$-operads
and weak actions of these called generalized $\LM$-, $\RM$- and $\BM$-operads.

For a detailed treatment of generalized non-symmetric $\infty$-operads we recommend \cite{GEPNER2015575} section 3.

\begin{notation}\label{ooop}
Let $\Ass:= \Delta^\op$ be the category of finite non-empty totally ordered sets
and order preserving maps.
We call a map $[\n] \to [\m]$ in $\Ass$
\begin{itemize}
\item inert if it corresponds to a map of $\Delta$ of the form $[\m] \simeq \{\bi, \bi+1,..., \bi+\m \} \subset [\n]$ for some $\bi \geq 0.$
\item standard inert if it is inert and $\m=1$.
\item active if it corresponds to a map of $\Delta$, which preserves the minimum and maximum.
\end{itemize}

\end{notation}

\begin{remark}\label{facto}
Every morphism in $\Ass$ uniquely factors as an inert morphism followed by an active one:
if $\f: [\m] \to [\n]$ is a map in $\Delta$, we can factor $\f$ as $[\m] \xrightarrow{\bar{\f}} \{\f(0), ..., \f(\m)\} \subset [\n].$	
\end{remark}

\begin{remark}\label{not}
For every $\n \geq 2$ there are $\n$-many standard inert morphisms $[1] \simeq \{ \bi-1, \bi \} \subset [\n]$ in $\Delta$ for $1 \leq \bi \leq \n$ that are part of a colimit diagram in $\Delta$:
\begin{equation}\label{dia}
(\{0,1\} \leftarrow \{1\} \to \{1,2\} \leftarrow \{2\} \to ... \leftarrow \{\n-1\} \to \{\n-1,\n\} ) \to [\n].
\end{equation}
\end{remark}

\begin{notation}Let $ \BM:= (\Delta_{/[1]})^\op$ and $ \LM \subset \BM,\RM \subset \BM$ the full subcategories of maps $[\n] \to [1]$ sending 0 to 0 and at most one object to 1, sending $\n$ to 1 and at most one object to 0, respectively.
\end{notation}

\begin{remark}\label{elem}
There is a canonical equivalence $\Ass \times [1] \simeq \LM$ corresponding to
the natural transformation sending $[\n]$ to $[\n] \ast \emptyset \to [\n] \ast [0] $
over $[0]*[0]=[1]$, and similar for $\RM$. 
\end{remark}
\begin{remark}
The right fibration $\Delta_{/[1]} \to \Delta$ is opposite to a left fibration
$\BM \to \Ass$ that restricts to cocartesian fibrations $\LM \to \Ass, \ \RM \to \Ass$ relative to the collection of inert morphisms with discrete fibers.
We call a morphism of $\LM, \RM, \BM$ inert, active if its image in $\Ass$ is.
\end{remark}

\begin{remark}\label{rightactt}
The maps $\{0\} \subset [1], \{1\} \subset [1]$ yield two embeddings
$\Ass \subset \BM$ over $\Ass$, which we call left and right embedding,
that factor through $\LM,\RM$, respectively,
and which are right fibrations relative to the collection of active morphisms.

\end{remark}

\begin{notation}\label{nooor}

Let $\psi: \mO \to \Ass $ be a cocartesian fibration relative to the collection of inert morphisms. We call a morphism of $\mO$ (standard) inert if it is $\psi$-cocartesian and its image in $\Ass$ is (standard) inert. We call a morphism of $\mO$ active if its image in $\Ass$ is active. 

\end{notation}

\begin{definition}\label{opep}
Let $\psi: \mO \to \Ass$ be a cocartesian fibration relative to the collection of inert morphisms. 

\vspace{1mm}

A generalized $\mO$-operad is a cocartesian fibration $\phi: \mC \to \mO$ relative to the collection of inert morphisms such that the following conditions are satisfied:

\begin{enumerate}
\item (Segal condition)
For every diagram 
\begin{equation*}\label{diaa}
\X \to (\X_{0,1} \to \X_{1} \leftarrow \X_{1,2} \to \X_{2} \leftarrow ... \to \X_{\n-1} \leftarrow \X_{\n-1,\n})
\end{equation*} in $\mO$, where all morphisms are $\psi$-cocartesian, lifting diagram (\ref{dia}) the induced functor $$\alpha: \mC_\X \to \mC_{\X_{0,1}} \times_{ \mC_{\X_{1}}} \mC_{\X_{1,2}} \times_{ \mC_{\X_{2}} } ... \times_{\mC_{\X_{\n-1}}} \mC_{\X_{\n-1,\n}} $$ is an equivalence.

\vspace{1mm}
\item For every diagram 
\begin{equation*}\label{diaa}
\Y \to (\Y_{0,1} \to \Y_{1} \leftarrow \Y_{1,2} \to \Y_{2} \leftarrow ... \to \Y_{\n-1} \leftarrow \Y_{\n-1,\n})
\end{equation*} in $\mC$, where all morphisms are $\phi$-cocartesian, lifting diagram (\ref{dia}) and every $\Z \in \mC$ the square 
\begin{equation}\label{geno}
\begin{xy}
\xymatrix{
\mC(\Z, \Y) \ar[d]
\ar[r]^{}
& \mC(\Z,\Y_{0,1}) \times_{\mC(\Z,\Y_{1})} ... \times_{\mC(\Z,\Y_{\n-1}) } \mC(\Z,\Y_{\n-1,\n}) \ar[d]^{} 
\\
\mO(\phi(\Z), \phi(\Y)) \ar[r]^{}  &  \mO(\phi(\Z),\phi(\Y_{0,1}))\times_{\mO(\phi(\Z),\phi(\Y_{1}))} ... \times_{\mO(\phi(\Z),\phi(\Y_{\n-1}))} \mO(\phi(\Z),\phi(\Y_{\n-1,\n}))}
\end{xy}
\end{equation}
is a pullback square.

\end{enumerate}

\end{definition}

\begin{remark}
(2) implies that the functor $\alpha$ in (1) is fully faithful
by taking the fiber over the identity of $\X=\phi(\Z) $ in square (\ref{geno}).
So if (2) holds, (1) holds if and only if $\alpha$ is essentially surjective.
\end{remark}

\begin{remark}\label{simpp}
In (1) and (2) of Definition \ref{opep} one can replace diagrams of inert morphisms lying over diagram (\ref{dia})
by diagrams of inert morphisms lying over diagrams in $\Delta$ of the form
\begin{equation*}
\begin{xy}
\xymatrix{
[0] \ast \emptyset  \ar[d]^{} \ar[r]^{ }
&[0] \ast [0] \ar[d]^{}  & \ar[l] \emptyset \ast [0] \ar[d]^{} 
\\ [\n] \ast \emptyset
\ar[r]^{}  & [\n] \ast [\m]  & \ar[l] \emptyset \ast [\m],
}
\end{xy} 
\end{equation*}
where the map $[0] \to [\n]$ takes $\n$ and the map $[0] \to [\m]$ takes $0.$ 

\end{remark}

\begin{notation}
We call a generalized $\mO$-operad $\mC \to \mO$ small if $\mC$ and $\mO$ are small $\infty$-categories. 
\end{notation}

We consider generalized $\mO$-operads for $\mO \in \{\Ass,\LM,\RM,\BM\}.$
Generalized $\Ass$-operads are generalized non-symmetric $\infty$-operads 
in the sense of \cite[Definition 3.1.12.]{GEPNER2015575}. We refer to generalized $\Ass$-operads as generalized $\infty$-operads as we rarely deal with their symmetric variant.

\begin{notation}
For $\mO =\Ass$ we use the following notation: we write $\mV^\ot \to \Ass$ for a generalized $\infty$-operad and call $\mV:= \mV^\ot_{[1]}$ the underlying $\infty$-category. We call objects of $\mV$ the colors of $\mV^\ot \to \Ass.$
\end{notation}

\begin{remark}
The left fibration $\BM \to \Ass$ is a generalized $\infty$-operad. This implies that generalized $\BM$-operads are precisely generalized $\infty$-operads over $\BM.$	
\end{remark}

\begin{remark}\label{reali}
For $\mO \in \{\Ass,\LM,\RM,\BM\}$ every diagram in $\mO$ lifting diagram (\ref{dia}) is a limit diagram. Therefore condition (2) of Definition \ref{opep} is equivalent to the condition that
for every diagram 
\begin{equation*}
\Y \to (\Y_{0,1} \to \Y_{1} \leftarrow \Y_{1,2} \to \Y_{2} \leftarrow ... \to \Y_{\n-1} \leftarrow \Y_{\n-1,\n})
\end{equation*} in $\mC$, where all morphisms are $\phi$-cocartesian, lifting diagram (\ref{dia}) 
and every $\Z \in \mC$ the canonical map 
$$\mC(\Z,\Y) \to \mC(\Z,\Y_{0,1}) \times_{\mC(\Z,\Y_{1})} ... \times_{\mC(\Z,\Y_{\n-1}) } \mC(\Z,\Y_{\n-1,\n})$$ is an equivalence.
\end{remark}

Using Remarks \ref{simpp}, \ref{reali} the definition of generalized $\BM$-operad takes the following form:

\begin{lemma}\label{visio}

A cocartesian fibration $\phi: \mC \to \BM$ relative to the collection of inert morphisms
is a generalized $\BM$-operad if and only if the following conditions are satisfied:
\begin{enumerate}
\item The pullbacks $\Ass \times_\BM \mC \to \Ass$
along both embeddings $\Ass \subset \BM$ are generalized $\infty$-operads.

\vspace{1mm}
\item For every $\n \geq 0$ the diagram in $\BM^\op$
\begin{equation*}\label{qavbl}
\begin{xy}
\xymatrix{
[0] \ast \emptyset  \ar[d]^{} \ar[r]^{ }
&[0] \ast [0] \ar[d]^{}  & \ar[l] \emptyset \ast [0] \ar[d]^{} 
\\ [\n] \ast \emptyset
\ar[r]^{}  & [\n] \ast [\m]  & \ar[l] \emptyset \ast [\m],
}
\end{xy} 
\end{equation*}
where the map $[0] \to [\n]$ takes $\n$ and the map $[0] \to [\m]$ takes $0,$ induces an equivalence 
$$ \mC_{[\n] \ast [\m]} \to \mC_{[\n] \ast \emptyset} \times_{\mC_{[0] \ast \emptyset}} \mC_{[0] \ast [0]} \times_{\mC_{\emptyset \ast [0] }} \mC_{\emptyset \ast [\m]}.$$

\item Every commutative diagram 
\begin{equation*}\label{agfa}
\begin{xy}
\xymatrix{
\Y' \ar[d] & \Y \ar[l] \ar[d]^{} \ar[r]^{ }
&\Y'' \ar[d]^{} 
\\\Y''''& \Y'''\ar[l]
\ar[r]^{}  & \Y'''''
}
\end{xy} 
\end{equation*}
in $\mC$ of $\phi$-cocartesian morphisms lying over diagram (\ref{qavbl})
induces for every $\X \in \mC$ an equivalence
$$ \mC(\X,\Y) \to \mC(\X,\Y') \times_{\mC(\X,\Y'''') } \mC(\X,\Y''') \times_{\mC(\X,\Y''''')} \mC(\X,\Y'').$$
\end{enumerate}

\end{lemma}

\begin{definition}\label{opap}
Let $\psi: \mO \to \Ass$ be a cocartesian fibration relative to the collection of inert morphisms. 

\begin{itemize}	
\item An $\mO$-operad is a generalized $\mO$-operad $\phi: \mC \to \mO$ such that the functor $\mC_{[0]} \to \mO_{[0]}$ is an equivalence.

\vspace{1mm}
\item A (generalized) $\mO$-monoidal $\infty$-category is a (generalized) $\mO$-operad $\phi: \mC \to \mO $ that is a cocartesian fibration.

\vspace{1mm}
\item A (generalized) corepresentable $\mO$-operad is a (generalized) $\mO$-operad $\phi: \mC \to \mO $ that is a locally cocartesian fibration.
\end{itemize}

\end{definition}

\begin{remark}
A cocartesian fibration $\phi: \mC \to \mO $ is a generalized $\mO$-operad
if and only if condition (1) of Definition \ref{opep} holds since for any cocartesian fibration (1) implies (2). 	
\end{remark}

\begin{notation} We say that a corepresentable $\mO$-operad is compatible with small colimits if for every $\X \in \mO_{[1]}$ the fiber $\mC_\X$ admits small colimits and for any active morphism $\X \to \Y$ in $\mO$, where $\Y \in \mO_{[1]}$, the induced functor $\prod_{\bi=1}^\n \mC_{\X_\bi} \simeq \mC_\X \to \mC_\Y$ preserves small colimits component-wise. 
We will also use this terminology for different kind of colimits like
$\kappa$-small colimits for a regular cardinal $\kappa.$
\end{notation}	

\begin{definition}
A presentably $\mO$-monoidal $\infty$-category is an $\mO$-monoidal $\infty$-category 
$\mC \to \mO$ compatible with small colimits such that for every $\X \in \mO_{[1]}$ the fiber $\mC_\X$ is presentable.
\end{definition}

\begin{lemma}\label{ope} Let $\psi: \mO \to \Ass$ be a cocartesian fibration relative to the collection of inert morphisms. 
A cocartesian fibration $\phi:\mC \to \mO$ relative to the collection of inert morphisms
is an $\mO$-operad if and only if the following conditions are satisfied:

\begin{enumerate}
\item For any $\X \in \mO$ lying over $[\n]$ the standard inert morphisms $\X \to \X_{\bi}$ for $1 \leq \bi \leq \n$ induce an equivalence $\mC_\X \to \prod_{\bi=1}^\n \mC_{\X_\bi}$ and $\mC_\X$ is contractible if $\X$ lies over $[0].$

\item For any $\Y \in \mC$ lying over $[\n]$ and $\Z \in \mC$ the induced map
$$ \mC(\Z,\Y)\to \mO(\phi(\Z),\phi(\Y)) \times_{\prod_{\bi=1}^\n \mO(\phi(\Z),\phi(\Y_{\bi}))} \prod_{\bi=1}^\n \mC(\Z,\Y_{\bi})$$ is an equivalence. 
\end{enumerate}	

\end{lemma}

\begin{proof}
This follows from the fact that the following two conditions are equivalent:
\begin{enumerate}
\item The functor $\mC_{[0]} \to \mO_{[0]}$ is an equivalence.
\item For any  $\Z \in \mO_{[0]}$ the fiber $\mC_\Z$ is contractible and for any $\X \in \mC_{[\m]}$ for $\m \geq 0$ and $\Y \in \mC_{[0]}$ the map $\mC(\X,\Y) \to \mO(\phi(\X),\phi(\Y))$ is an equivalence.
\end{enumerate}	
(2) implies (1) since the functors $\mO_{[0]} \to \mO, \mC_{[0]} \to \mC$ are fully faithful as $[0]$ is the initial object of $\Ass$.
(1) implies (2) because $\phi$ is a map of cocartesian fibrations relative to the collection of inert morphisms and all maps $\f: [0] \to [\m]$ in $\Delta$ are inert
so that the fiber over $\f$ of the map
$\mC(\X,\Y) \to \mO(\phi(\X),\phi(\Y))$ is equivalent to
$\mC_{[0]}(\X',\Y) \to \mO_{[0]}(\phi(\X'),\phi(\Y))$, where $\X \to \X'$ is a  cocartesian lift of $\f.$	
\end{proof}

\begin{example}
$\Ass$-operads are non-symmetric $\infty$-operads in the sense of \cite[Definition 3.1.3.]{GEPNER2015575}. We refer to $\Ass$-operads as $\infty$-operads because we mainly deal with non-symmetric $\infty$-operads. 
\end{example}
\begin{example}
$\Ass$-monoidal $\infty$-categories agree with monoidal $\infty$-categories.
Generalized $\Ass$-monoidal $\infty$-categories precisely classify Segal objects in $\Cat_\infty.$

\end{example}

For the proofs of the next lemmas we fix the following notation:
let $\mO \to \Ass$ be a cocartesian fibration relative to the collection of inert morphisms and $\X \in \mO.$ Let $\mO^\inert_{\X/} \subset \mO_{\X/}$ be the full subcategory of inert morphisms whose target lies over $[0],[1]\in \Ass.$ 

\begin{lemma}\label{qbg}Let $\phi: \mC \to \mO$ be a generalized $\mO$-operad.
A morphism $\rho:\A \to \B$ of $\mC$ is locally $\phi$-cocartesian if for every
inert morphism $\B \to \C$ in $\mC$, where $\C$ lies over $[0],[1]\in \Ass$,
the composition $\A \xrightarrow{\rho} \B \to \C$ is locally $\phi$-cocartesian.
\end{lemma}

\begin{proof}
For any $\Z \in \mC$ lying over $\T:={\phi(\B)}$ the map
$ \mC_{\T}(\B,\Z) \to \{\phi(\rho)\} \times_{\mO(\phi(\A),\T)} \mC(\A,\Z)$
identifies with the canonical map
$$\lim_{(\sigma: \T \to \T') \in \mO^\inert_{\T/}} \mC_{\T'}(\sigma_*(\B),\sigma_*(\Z)) \to \lim_{(\sigma: \T \to \T') \in \mO^\inert_{\T/}} \{\sigma \circ \phi(\rho)\} \times_{\mO(\phi(\A),\T')} \mC(\A,\sigma_*(\Z)).$$
 
\end{proof}

\begin{lemma}\label{reyqa} Let $\phi: \mC \to \mO$ be a generalized $\mO$-operad and locally cocartesian fibration relative to the collection of active morphisms whose target lies over $[0],[1]\in \Ass$. Assume that $\mC_{\X}$ is a space for any $\X \in \mO_{[0]}$.
Then $\phi$ is a corepresentable generalized $\mO$-operad.
Moreover a morphism $\rho:\A \to \B$ of $\mC$ is locally $\phi$-cocartesian if and only if for every inert morphism $\beta: \B \to \C$ in $\mC$, where $\C$ lies over $[0],[1]\in \Ass$, the composition $\A \xrightarrow{\rho} \B \xrightarrow{\beta} \C$ is locally $\phi$-cocartesian.
In particular, the composition of a locally $\phi$-cocartesian morphism followed by an inert morphism is locally $\phi$-cocartesian.

\end{lemma}

\begin{proof}
Observe that $\phi$ is a locally cocartesian fibration relative to the collection of morphisms whose target lies over $[0],[1]\in \Ass$
as any map of $\mO$ factors as an inert map followed by an active one.

Let $\Y \in \mC$ and $\alpha: \phi(\Y) \to \Z$ a morphism of $\mO$.
Then for every $(\sigma:\Z \to \Z')\in \mO^\inert_{\Z/}$ there is a locally $\phi$-cocartesian lift $\rho_\sigma: \Y \to (\sigma \circ \alpha)_*(\Y) $ of $\sigma \circ \alpha.$
For any inert morphism $\tau: \Z' \to \Z''$ in $\mO$ for $\Z''\in \mO_{[0]} $ the composition 
$\Y \xrightarrow{\rho_\sigma} (\sigma \circ \alpha)_*(\Y) \to \tau_*((\sigma \circ \alpha)_*(\Y))$
is locally $\phi$-cocartesian since $\mC_{\Z''}$ is a space. So by the uniqueness of locally $\phi$-cocartesian lifts the morphisms $\rho_\sigma$ are compatible and so give a lift $\rho: \Y \to \Y'$ of $\alpha$ such that $\rho_\sigma$ factors as $\Y \xrightarrow{\rho} \Y' \xrightarrow{} \sigma_*(\Y')$ using that $\mC \to \mO$ is a generalized $\infty$-operad.
By Lemma \ref{qbg} the morphism $\rho$ is locally $\phi$-cocartesian.
Hence $\phi$ is a locally cocartesian fibration.

The characterization of locally $\phi$-cocartesian morphisms follows from the construction of locally $\phi$-cocartesian lifts and Lemma $\ref{qbg}.$
	
\end{proof}

\begin{remark}\label{Reym}
	
If $\phi:\mC \to \mO$ is an $\mO$-operad, $\phi$ is a locally cocartesian fibration relative to the collection of active morphisms whose target lies over $[0]\in \Ass$. 
Indeed, any active map $\alpha: \X \to \Y$ in $\mO$ whose target lies over $[0]$ belongs to $\mO_{[0]}$ as there is only one active map $[\n]\to [0]$ in $\Ass$. As $\phi_{[0]}: \mC_{[0]} \to \mO_{[0]}$ is an equivalence, for any $\X' \in \mC$ lying over $\X$ there is a locally $\phi$-cocartesian lift $\X'\to \Y'$ of $\alpha$.
\end{remark}

\begin{corollary}\label{repmon}

A generalized corepresentable $\mO$-operad $\phi: \mC \to \mO$ is a generalized $\mO$-monoidal $\infty$-category if
and only if for any locally $\phi$-cocartesian lifts $\beta: \Z \to \Y, \ \alpha: \Y \to \X$ of active morphisms of $\mO$,
where $\X$ lies over $[0],[1] \in \Ass $, the 
composition $\alpha \circ \beta$ is locally $\phi$-cocartesian.
\end{corollary}

\begin{proof}
By \cite[Proposition 2.4.2.8.]{lurie.HTT} $\phi: \mC \to \mO$ is a cocartesian fibration if and only if for any locally $\phi$-cocartesian morphisms $\beta: \Z \to \Y, \ \alpha: \Y \to \X$ in $\mC$ the composition $\alpha \circ \beta$ is locally $\phi$-cocartesian.
By Lemma \ref{reyqa} the morphism $\alpha \circ \beta$ is locally $\phi$-cocartesian if for every inert morphism $\gamma: \X \to \X'$ in $\mC$, where $\X'$ lies over $[0],[1]\in \Ass$, the composition 
$ \gamma \circ \alpha \circ \beta$ is locally $\phi$-cocartesian.
By the same lemma $\gamma \circ \alpha$ is locally $\phi$-cocartesian.
So we can assume that $\X$ lies over $[0],[1]\in \Ass$.

The map $\beta$ factors as an inert morphism
$\beta': \Z \to \Y'$ followed by an active one $\beta'': \Y' \to \Y$.
As $\beta'$ is $\phi$-cocartesian and $\beta$ is locally $\phi$-cocartesian, $\beta'': \Y' \to \Y$ is locally $\phi$-cocartesian, similarly for $\alpha.$
The map $\alpha' \circ \beta'': \Y' \to \Y \to \X'$ factors as
an inert morphism $\kappa: \Y' \to \T $ followed by an active one $\rho: \T \to \X'$.
By Lemma \ref{reyqa} the map $\alpha' \circ \beta''$ is locally $\phi$-cocartesian.
Since $\kappa$ is $\phi$-cocartesian, $\rho$ is locally $\phi$-cocartesian.
The map $\alpha \circ \beta$ factors as
$ \alpha'' \circ \alpha' \circ \beta'' \circ \beta'=  \alpha'' \circ \rho \circ \kappa \circ \beta'.$
Since $\kappa, \beta'$ and so $\kappa \circ \beta'$ are $\phi$-cocartesian,
it is enough to see that $ \alpha'' \circ \rho$ is locally $\phi$-cocartesian.
This holds by assumption.

\end{proof}

There is the following relationship between $\LM,\RM,\BM$-operads and their generalized version:
\begin{lemma}

A generalized $\LM$-operad $\mC \to \LM$ is a $\LM$-operad if and only if the pullback $\Ass \times_\LM \mC \to \Ass$ is an $\infty$-operad, and similar for $\RM.$
A generalized $\BM$-operad $\mC \to \BM$ is a $\BM$-operad if and only if the pullbacks $\Ass \times_\BM \mC \to \Ass$ along both embeddings $\Ass \subset \BM$ are $\infty$-operads.

\end{lemma}

\begin{proof}
The fibers $\LM_{[0]}, \RM_{[0]}$ are contractible and the fiber $\BM_{[0]}$
has two elements, each of which lies in the image of one of the two embedding $\Ass \subset \BM.$
\end{proof}

There is the following duality between (generalized) $\LM$ and $\RM$-operads:

\begin{remark}\label{inv}
There is a canonical involution on $\Delta$ sending $[\n]$ to $[\n]$ and a map $\alpha: [\n] \to [\m]$ in $\Delta$ to the map $$[\n] \simeq [\n] \xrightarrow{\alpha} [\m] \simeq [\m],$$ where the bijection $[\n] \simeq [\n]$ sends $\bi$ to $\n-\bi$ and similar for $\m.$
The involution on $\Delta$ yields involutions on $\Ass=\Delta^\op$ and $\BM=(\Delta_{/[1]})^\op$. The involution on $\BM$ restricts to an equivalence
$$\LM \simeq \Ass \times [1] \simeq \RM \simeq \Ass \times [1]$$
that is the product of the involution on $\Ass$ with the identity of $[1].$

The involution on $\Ass$ yields involutions on $\Op_\infty, \Op^\gen_\infty$ denoted by $(-)^\rev$, the involution on $\BM$ yields involutions on $\Op^\BM_\infty, \Op^{\BM,\gen}_\infty$
that restrict to equivalences $$\Op^\LM_\infty \simeq \Op^\RM_\infty, \  \Op^{\LM,\gen}_\infty \simeq \Op^{\RM,\gen}_\infty$$
and cover the involutions on $\Op_\infty, \Op^\gen_\infty,$ respectively.
\end{remark}

We also consider maps of generalized $\mO$-operads:
 
\begin{definition}
Let $\mO \to \Ass$ be a cocartesian fibration relative to the collection of inert morphisms

\begin{itemize}
\item Let $\mC \to \mO,\mD \to \mO$ be generalized $\mO$-operads.
A $\mC$-algebra in $\mD$ or map of generalized $\mO$-operads $\mC \to \mD$ is a functor $\mC \to \mD$ over $\mO$ preserving cocartesian lifts of inert morphisms.

\item Let $\mC \to \mO,\mD \to \mO$ be generalized corepresentable $\mO$-operads.
An $\mO$-monoidal functor $\mC \to \mD$ is a map $\mC \to \mD$ of cocartesian fibrations over $\mO$.

\item Let $\mC \to \mO,\mD \to \mO$ be generalized corepresentable $\mO$-operads.
A lax $\mO$-monoidal functor $\mC \to \mD$ is a $\mC$-algebra in $\mD.$

\end{itemize}

\begin{notation}
We call a map of $\mO$-operads $\mC \to \mD$ an embedding (inclusion) if the functor $\mC \to \mD$ is an embedding (inclusion).
\end{notation}

\end{definition}
\begin{notation}
We say that an $\mO$-monoidal functor $\mC \to \mD$ preserves small colimits 
(admits a right adjoint) if for any $\X \in \mO$ the induced functor $\mC_\X \to \mD_\X$
does. 
\end{notation}

\begin{remark}
As a consequence of \cite[Proposition 7.3.2.6.]{lurie.higheralgebra} an $\mO$-monoidal functor admits a right adjoint if and only if it admits a right adjoint relative to $\mO.$
\end{remark}
\begin{notation}
Let $\mC \to \mO,\mD \to \mO$ be generalized $\mO$-operads. Let $\Alg_{\mC/\mO}(\mD) \subset \Fun_\mO(\mC,\mD)$ be the full subcategory of $\mC$-algebras in $\mD$ and let $\Alg_\mO(\mD):= \Alg_{\mO/\mO}(\mD)$.

\vspace{1mm}
Let $\mC \to \mO,\mD \to \mO$ be generalized $\mO$-monoidal $\infty$-categories.
Let $\Fun^\ot_\mO(\mC, \mD) \subset \Alg_{\mC/\mO}(\mD)$ be the full subcategory of $\mO$-monoidal functors $\mC \to \mD.$

\end{notation}

\begin{notation}
For $\mO=\Ass, \LM, \RM,\BM$ and a generalized $\mO$-operad $\mC \to \mO$ we call $\mO$-algebras in $\mC$ associative algebras in $\mC$, left modules, right modules and bimodules in $\mC$, respectively. 
We set $$\Alg(\mC):= \Alg_\Ass(\mC),\ \LMod(\mC):=\Alg_\LM(\mC),\ \RMod(\mC):=\Alg_\RM(\mC),\ \BMod(\mC):=\Alg_\BM(\mC).$$

\end{notation}

\begin{notation}
There are forgetful functors $$\LMod(\mC) \to \Alg(\mC),\ \RMod(\mC) \to \Alg(\mC),\ \BMod(\mC) \to \Alg(\mC)\times \Alg(\mC)$$ whose fibers over associative algebras
$\A,\B$ we denote by $\LMod_\A(\mC), \ \RMod_\B(\mC), \ \BMod_{\A,\B}(\mC).$
\end{notation}

\begin{notation}\label{notabene}
Let $\mO \to \Ass$ be a cocartesian fibration relative to the collection of inert morphisms. Let
\begin{itemize}
\item $\Cat_{\infty / \mO}^{\mathrm{inert}} \subset \Cat_{\infty / \mO}$
be the subcategory of cocartesian fibrations relative to the collection of inert morphisms of $\mO$ and functors over $\mO$ preserving cocartesian lifts of inert morphisms.
\item $\Op_\infty^\mO \subset \Op_\infty^{\mO, \gen} \subset \Cat_{\infty / \mO}^{\mathrm{inert}}$ be the full subcategories of (generalized) $\mO$-operads.
\item $\Op_\infty^{\mO, \mon} \subset\Cat_{\infty / \mO}^{\mathrm{inert}} $ be
the (non-full) subcategory of $\mO$-monoidal $\infty$-categories and $\mO$-monoidal functors.
\item $\Op_\infty^{\mO,\mon,\rc\rc} \subset \widehat{\Op}_\infty^{\mO,\mon}$
be the subcategory of $\mO$-monoidal $\infty$-categories compatible with small colimits and small colimits preserving $\mO$-monoidal functors.
\end{itemize}
For $\mO=\Ass$ we drop $\mO$ from the notation.

\end{notation}

If the tensor product is the product, the definition of algebra gets easier:

\begin{definition}
Let $\mO \to \Ass$ be a cocartesian fibration relative to the collection of inert morphisms,
$\psi: \mC \to \mO$ a generalized $\mO$-operad and $\mD$ an $\infty$-category with finite products.
We call a functor $\alpha: \mC \to \mD$ a $\mC$-monoid in $\mD$ if for every 
$\X \in \mC$ lying over $[\n]\in \Ass$ the standard inert morphisms
$\X \to \X_\bi$ for $1 \leq \bi \leq \n$ in $\mC$ induce an equivalence
$\alpha(\X) \to \prod_{\bi=1}^\n \alpha(\X_\bi).$
\end{definition}
\begin{notation}
Let $\Mon_{\mC}(\mD) \subset \Fun(\mC,\mD) $ be the full subcategory of $\mC$-monoids in $\mD$. For $\mC=\mO=\Ass$ we omit $\mC$ from the notation.
\end{notation}

\begin{example}
Note that $\mO$-monoids in $\Cat_\infty$ are classified by $\mO$-monoidal $\infty$-categories. This gives a canonical equivalence
$ \Mon_{\mO}(\Cat_\infty) \simeq \Op_\infty^{\mO,\mon}.$	
\end{example}

\begin{example}\label{stru} Let $\mD$ be an $\infty$-category with finite products.
The canonical equivalence $\LM \simeq \Ass \times [1]$ induces an embedding
$ \Mon_{\LM}(\mD) \subset \Fun([1],\Fun(\Ass,\mD)) $ whose essential image consists of those natural transformations
$\F \to \G$ such that for any $[\n] \in \Ass$ the map $[0] \simeq \{\n\} \subset [\n]$ induces an equivalence $\F([\n]) \to \G([\n]) \times \F([0]) .$
\end{example}

\begin{proposition}\label{cacart}
Let $\mD$ be an $\infty$-category with finite products.
There is a monoidal $\infty$-category $\mD^\times \to \Ass$
and $\mD^\times$-monoid in $ \mD$ such that for any cocartesian fibration $\mO \to \Ass$ relative to the collection of inert morphisms and generalized $\mO$-operad $\mC \to \mO$ the
functor $\mO \times_\Ass \mD^\times \to \mD^\times \to \mD$ induces an equivalence 
$$ \Alg_{\mC/\mO}(\mO \times_\Ass \mD^\times) \to \Mon_{\mC}(\mD).$$
We call $\mD^\times \to \Ass$ the cartesian monoidal structure on $\mD.$

\end{proposition}

\begin{proof}
Like the proof of \cite[Proposition 2.4.1.7.]{lurie.higheralgebra}.	
\end{proof}

\begin{remark}
Let $\Triv \subset \Ass$ be the wide subcategory of inert morphisms.
The functor $\Triv \subset \Ass$ is an $\infty$-operad and for any 
$\infty$-operad $\mB^\ot \to \Ass$ and $\infty$-category $\mD$ having finite products the forgetful functors
$\Alg_{\Triv/\Ass}(\mB) \to \mB_{[1]}, \Mon_\Triv(\mD) \to \mD$ are equivalences.
So we get a canonical equivalence
$\mD^\times_{[1]} \simeq \Alg_{\Triv/\Ass}(\mD^\times) \to \Mon_{\Triv}(\mD) \simeq \mD.$
\end{remark}
\begin{example}\label{strai}
There is a canonical equivalence
$ \Alg_{\mO}(\mO \times_\Ass \Cat_\infty^\times) \simeq \Mon_{\mO}(\Cat_\infty) \simeq \Op_\infty^{\mO,\mon}.$	
\end{example}

We will also often use the following modification of $\Cat_\infty^\times:$
	
\begin{notation}\label{coli}
Let $(\Cat_\infty^{{\rc\rc}})^\ot \subset \widehat{\Cat}_\infty^\times $
be the suboperad whose colors are the $\infty$-categories having small colimits and 
whose active morphisms $\mC \to \mD$ lying over a map $[\n] \to [1]$ in $\Ass$
correspond to functors $\mC_1 \times ...\times \mC_\n \to \mD $ that preserve small colimits component-wise.
Let $(\Pr^\L)^\ot \subset (\Cat_\infty^{{\rc\rc}})^\ot $ be
the full suboperad whose colors are the presentable $\infty$-categories.
\end{notation}	
By \cite[Corollary 4.8.1.4.]{lurie.higheralgebra} the $\infty$-operad $(\Cat_\infty^{{\rc\rc}})^\ot \to \Ass $
is a monoidal $\infty$-category and by \cite[Corollary 4.8.1.15.]{lurie.higheralgebra} the $\infty$-operad $(\Pr^\L)^\ot \to \Ass $ is a monoidal $\infty$-category such that the embedding $(\Pr^\L)^\ot \subset (\Cat_\infty^{{\rc\rc}})^\ot $ is a monoidal functor.

\begin{remark}\label{colimon}
 	
The canonical equivalence
$ \Alg_{\mO}(\mO \times_\Ass \widehat{\Cat}_\infty^\times) \simeq \Mon_{\mO}(\widehat{\Cat}_\infty) \simeq \widehat{\Op}_\infty^{\mO,\mon}$
restricts to an equivalence
$ \Alg_{\mO}(\mO \times_\Ass \Cat_\infty^{{\rc\rc}}) \simeq \Op_\infty^{\mO,\mon,\rc\rc}.$

\end{remark}

By \cite[Proposition 4.8.1.3.]{lurie.higheralgebra} the inclusion $(\Cat_\infty^{{\rc\rc}})^\ot \subset \widehat{\Cat}_\infty^\times $
admits a left adjoint relative to $\Ass$ 
that by \cite[Theorem 5.1.5.6.]{lurie.HTT} sends any small $\infty$-category $\mB$ to $\mP(\mB)=\Fun(\mB^\op,\mS),$
where the unit $\mB \to \mP(\mB)$ is the Yoneda-embedding.
Hence the induced inclusion $$\Op_\infty^{\mO,\mon,\rc\rc} \simeq \Alg_{\mO}(\mO \times_\Ass (\Cat_\infty^{{\rc\rc}})^\ot) \subset \Alg_{\mO}(\mO \times_\Ass \widehat{\Cat}_\infty^\times) \simeq \Mon_{\mO}(\widehat{\Cat}_\infty) \simeq \widehat{\Op}_\infty^{\mO,\mon}$$ admits a left adjoint.
This implies the following proposition:
\begin{proposition}\label{presday}
Let $\mO \to \Ass$ be a cocartesian fibration relative to the collection of inert morphisms and $\mC \to \mO$ a small $\mO$-monoidal $\infty$-category.
There is an $\mO$-monoidal $\infty$-category $\mP\mC \to \mO$ compatible with small colimits and an $\mO$-monoidal embedding $\mC \to \mP\mC$
that induces on the fiber over every $\X \in \mO$
the Yoneda-embedding $\mC_\X \to \mP(\mC_\X)$ and such that for any $\mO$-monoidal $\infty$-category $\mD \to \mO$ compatible with small colimits the induced functor $ \Fun_\mO^{\ot,\L}(\mP\mC, \mD) \to \Fun_\mO^{\ot}(\mC, \mD)$ is an equivalence, where the left hand side is the full subcategory of left adjoint $\mO$-monoidal functors. 
\end{proposition}

\begin{notation}Let $\mV^\ot \to \Ass$ be a small monoidal $\infty$-category. We write $\mP(\mV)^\ot $ for $\mP\mV^\ot.$\end{notation}

We also need the following variant: 
\begin{notation} Let $\kappa$ be a regular cardinal. For any small $\infty$-category $\mC$ that admits $\kappa$-small colimits let $\Ind_\kappa(\mC) \subset \mP(\mC)$ be the full subcategory of functors preserving $\kappa$-small limits. For any $\infty$-category $\mC$ that admits small colimits let $\widehat{\Ind}(\mC) \subset \widehat{\mP}(\mC)$ be the full subcategory of functors preserving small limits. \end{notation}

\begin{remark}\label{neos}Let $\mC$ be a small $\infty$-category that admits $\kappa$-small colimits. By construction the Yoneda-embedding $\y:\mC \to \Ind_\kappa(\mC)$ preserves $\kappa$-small colimits. As $\kappa$-small limits commute with  $\kappa$-filtered colimits, $\Ind_\kappa(\mC)$ is closed in $\mP(\mC)$ under $\kappa$-filtered colimits. Moreover $\Ind_\kappa(\mC)$ is a localization of $\mP(\mC)$ with respect to the set of maps $\{ \colim(\y \circ \rH) \to \y(\colim(\rH)) \mid \rH:\K \to \mC, \ \K \ \kappa\text{-small} \}$. Hence $\Ind_\kappa(\mC)$ is presentable. \end{remark}

\begin{notation}\label{Indi}
Let $\kappa$ be a regular cardinal. For any small $\mO$-monoidal $\infty$-category $\mC \to \mO$ compatible with $\kappa$-small colimits let $\Ind_\kappa\mC \subset \mP\mC$ be the full suboperad spanned by objects of $\Ind_\kappa(\mC_\X) \subset \mP(\mC_\X) \simeq \mP\mC_\X$ for some $\X \in \mO.$	For any $\mO$-monoidal $\infty$-category $\mC \to \mO$ compatible with small colimits let $\widehat{\Ind}\mC \subset \widehat{\mP}\mC$ be the full suboperad spanned by objects of $\widehat{\Ind}(\mC_\X) \subset \widehat{\mP}(\mC_\X) \simeq \widehat{\mP}\mC_\X$ for some $\X \in \mO.$	\end{notation}

\begin{remark}\label{nos}Let $\mC \to \mO$ be a small $\mO$-monoidal $\infty$-category compatible with $\kappa$-small colimits. By the description of generating local equivalences $\Ind_\kappa\mC \subset \mP\mC$ is a localization relative to $\mO$ (\cite[Remark 2.2.4.12]{lurie.higheralgebra}) so that $\Ind_\kappa\mC \to \mO$ is an $\mO$-monoidal $\infty$-category compatible with small colimits. Again by the description of generating local equivalences for any $\mO$-monoidal $\infty$-category $\mD \to \mO$ compatible with small colimits the induced functor $\Fun_\mO^{\ot,\L}(\Ind_\kappa\mC, \mD) \to \Fun_\mO^{\ot,\kappa}(\mC, \mD)$ is an equivalence, where the right hand side is the full subcategory of $\mO$-monoidal functors preserving fiberwise $\kappa$-small colimits.
\end{remark}

\begin{notation}Let $\mV^\ot \to \Ass$ be a small monoidal $\infty$-category compatible with $\kappa$-small colimits. We write $\Ind_\kappa(\mV)^\ot $ for $\Ind_\kappa\mV^\ot.$
Let $\mV^\ot \to \Ass$ be a monoidal $\infty$-category compatible with small colimits. We write $\widehat{\Ind}(\mV)^\ot $ for $\widehat{\Ind}\mV^\ot.$
\end{notation}

We also consider families of $\mO$-operads relative to an $\infty$-category $\rS$ for any cocartesian fibration $\mO \to \Ass$ relative to the collection of inert morphisms
by applying Definition \ref{opep} to the projection $\mO \times \rS \to \mO \to \Ass$, which is a cocartesian fibration relative to the collection of inert morphisms:

\begin{definition}\label{famm}

Let $\rS$ be an $\infty$-category and $\mO \to \Ass$ a cocartesian fibration relative to the collection of inert morphisms.

\begin{itemize}
\item A $\rS$-family of (generalized) $\mO$-operads is a (generalized) $\mO \times \rS$-operad.

\vspace{1mm}

\item A $\rS$-family of (generalized) $\mO$-monoidal $\infty$-categories is a $\rS$-family
$\phi: \mC \to \mO \times \rS $ of (generalized) $\mO$-operads
such that $\phi$ is a map of cocartesian fibrations over $\mO.$

\vspace{1mm}

\item A cocartesian (cartesian) $\rS$-family of (generalized) $\mO$-operads/ (generalized) $\mO$-monoidal $\infty$-categories is
a $\rS$-family $\phi: \mC \to \mO \times \rS $ of (generalized) $\mO$-operads/ (generalized) $\mO$-monoidal $\infty$-categories such that $\phi $ is a map of cocartesian (cartesian) fibrations over $\rS.$

\vspace{1mm}

\item A bicartesian $\rS$-family of presentably $\mO$-monoidal $\infty$-categories
is a cocartesian and cartesian $\rS$-family $\phi: \mC \to \mO \times \rS $ of $\mO$-monoidal $\infty$-categories such that for any $\s \in \rS$ the functor $\mC_\s \to \mO$ is a presentably $\mO$-monoidal $\infty$-category.
\end{itemize}

\end{definition}

\begin{remark}\label{recogn}
Let $\rS$ be an $\infty$-category and $\mO \to \Ass$ a cocartesian fibration relative to the collection of inert morphisms.
A functor $\mV^\ot \to \mO \times \rS$ exhibits $\mV^\ot$ as a $\rS$-family of generalized $\mO$-operads if and only if it is a map of generalized $\mO$-operads.
This gives an equivalence $\Op_{\infty}^{\mO \times \rS, \gen} \simeq (\Op_\infty^{\mO,\gen})_{/\mO \times \rS}.$

More generally, for any $\infty$-category $\T$ a functor $\mV^\ot \to \mO \times \rS \times \T$ exhibits $\mV^\ot$ as a $\rS \times \T$-family of generalized $\mO$-operads if and only if it is a map of $\rS$ families of generalized $\mO$-operads.

\end{remark}

Associated to families of $\infty$-operads is a fibered version of $\infty$-category of algebras:

\begin{notation}
Let $\rS$ be an $\infty$-category and $\mO \to \Ass$ a cocartesian fibration relative to the collection of inert morphisms. 
We write $\Alg^\rS_{\mO}(-): \Op_\infty^{\mO \times \rS,\gen} \to \Cat_{\infty/\rS} $ for the right adjoint of the functor $$(-)\times \mO: \Cat_{\infty/\rS} \to  \Op_\infty^{\mO \times \rS,\gen}$$ that exists 
by \cite[Theorem B.4.2.]{lurie.higheralgebra} using that the projection $\mO \times \rS \to \rS$ is a cocartesian fibration.
\end{notation}

\begin{remark}\label{diagol} 
Let $\mD \to \rS \times \mO$ be a $\rS$-family of generalized $\mO$-operads.
By the defining universal property there is an equivalence
$\Alg_\mO^\rS(\mD) \simeq \rS \times_{\Alg_\mO(\rS \times \mO)} \Alg_\mO(\mD) $,
where the pullback uses the diagonal functor $\delta:\rS \to \Alg_\mO(\rS \times \mO) \subset \Fun_\mO(\mO, \rS \times \mO) \simeq \Fun(\mO,\rS).$
In particular, if $\mO$ is weakly contractible (so that $\delta$ is an equivalence), $\Alg_\mO^\rS(\mD)$ is the full subcategory of $ \Alg_\mO(\mD)$
spanned by the $\mO$-algebras $\alpha$
such that $\mO \xrightarrow{\alpha}\mD \to \rS$ inverts every morphism.
This holds for example for $\mO\in \{\Ass, \LM,\RM,\BM\}.$	
\end{remark}

\begin{lemma}\label{fibb}
Let $\rS' \to \rS$ be a functor, $\mO' \to \mO$ a map of cocartesian fibrations relative to the collection of inert morphisms of $\Ass$ and $\mD \to \rS\times \mO$ a $\rS$-family of generalized $\mO$-operads and $\Y \in \mO.$

\begin{enumerate}
\item There is a canonical equivalence $$\rS' \times_\rS \Alg^\rS_{\mO}(\mD) \simeq \Alg^{\rS'}_{\mO}(\rS' \times_\rS\mD).$$ 
In particular, for any $\s \in \rS$ the fiber $\Alg^\rS_{\mO}(\mD)_\s$ is
$\Alg_{\mO}(\mD_\s).$

\vspace{1mm}
\item There are functors over $\rS$: $$\Alg^\rS_{\mO}(\mD) \to \mD_\Y, \ \Alg^\rS_{\mO}(\mD) \to \Alg^\rS_{\mO'}(\mO' \times_\mO \mD)$$

\vspace{1mm}
\item Let $\mE \subset \Fun([1],\rS)$ be a full subcategory. If $\mD \to \rS\times \mO$ is a map of (co)cartesian fibrations
relative to $\mE$, the functor $\Alg^\rS_{\mO}(\mD) \to \rS$ is a (co)cartesian fibration relative to $\mE$ and the functors of (2) are maps of (co)cartesian fibrations relative to $\mE$.

\end{enumerate}

\end{lemma}

\begin{proof}

The map in (1) is represented by the map
$$ \Fun_{\rS'}(\T, \rS' \times_\rS \Alg^\rS_{\mO}(\mD)) \simeq \Fun_{\rS}(\T,\Alg^\rS_{\mO}(\mD)) \to
\Fun_{\rS'}(\T, \Alg^{\rS'}_{\mO}( \rS' \times_\rS \mD)) $$
that is the restriction of the canonical equivalence
$$\Fun_{\mO \times \rS}(\mO \times \T, \mD) \simeq \Fun_{\mO \times \rS'}(\mO \times \T, \rS' \times_\rS \mD)$$
for any functor $\T \to \rS'$.	
The first functor in (2) is represented by the map
$$ \Fun_\rS(\T, \Alg^\rS_{\mO}(\mD)) \hookrightarrow \Fun_{\mO \times \rS}(\mO \times \T, \mD) \to \Fun_{\mO \times \rS}(\{\Y \} \times \T, \mD) \simeq \Fun_\rS(\T, \mD_\Y)$$
for any functor $\T \to \rS$. The second functor in (2) is represented by the map $$ \Fun_\rS(\T, \Alg^\rS_{\mO}(\mD)) \to \Fun_\rS(\T, \Alg^\rS_{\mO'}(\mO' \times_\mO \mD)) $$ that is the restriction of the canonical map
$$\Fun_{\mO \times \rS}(\mO \times \T, \mD) \to \Fun_{\mO \times \rS}(\mO' \times \T, \mD) \simeq \Fun_{\mO' \times \rS}(\mO' \times \T, \mO' \times_\mO \mD)$$
for any functor $\T \to \rS$.
(3) follows from Lemma \ref{flafla}.

\end{proof}

\begin{notation}\label{modol}
Let $\rS$ be an $\infty$-category, $\mO \in \{\Ass,\LM,\RM,\BM\}$
and $\mC \to \rS\times \mO$ a $\rS$-family of generalized $\mO$-operads.	
We set $$\Alg^\rS(\mC):= \Alg^\rS_\Ass(\mC),\ \LMod^\rS(\mC):=\Alg^\rS_\LM(\mC),\ \RMod^\rS(\mC):=\Alg^\rS_\RM(\mC),\ \BMod^\rS(\mC):=\Alg^\rS_\BM(\mC)$$
and omit $\rS$ from the notation if $\rS$ is contractible.	
\end{notation}

\begin{example}\label{exemp}
For any (cocartesian, cartesian) $\rS$-family of generalized $\LM$-operads $\mD \to \rS \times \LM$ there are canonical maps $$\LMod^\rS(\mD) \to \Alg^\rS(\Ass \times_\LM \mD), \ \LMod^\rS(\mD) \to \mD_{[0]*[0]} $$ of (cocartesian, cartesian) fibrations over $\rS$.

\end{example}

\section{Weakly tensored $\infty$-categories}
\label{weak action}
In this section we give a different model for $\LM$-, $\RM$- and $\BM$-operads, 
which we call weakly left, right and bitensored $\infty$-categories, and which seems generally easier to work with.
In \cite[4.2.2.]{lurie.higheralgebra} Lurie introduces so called simplicial models for
left, right and bitensored $\infty$-categories, which are a special case of
our notion of weakly left, right and bitensored $\infty$-categories.
In Propositions \ref{proo} and \ref{bproo} we show that weakly left, right and bitensored $\infty$-categories are a model for generalized $\LM-,\RM-,\BM$-operads, respectively.
We define enriched and pseudo-enriched $\infty$-categories in the sense of Lurie (Definition \ref{Enr} and \ref{Lu}) as certain weakly left tensored $\infty$-categories.

\subsection{Weakly left, right and bitensored $\infty$-categories}

\begin{notation}
We say that a morphism in $ \Ass$ preserves the minimum (maximum) if it corresponds to a map $[\m] \to [\n]$ in $\Delta$ sending $0$ to $0$
(sending $\m$ to $\n$).
We call a morphism of $\Ass \times \Ass$ inert if it is inert in both components.

\end{notation}

\begin{remark}\label{flpqmmmm}
Every morphism in $\Ass$ can be uniquely factored in an inert morphism that preserves the maximum (minimum) followed by a morphism that preserves the minimum (maximum).
	
\end{remark}

\begin{definition}\label{bla}
Let $\mV^\ot \to \Ass, \mW^\ot \to \Ass$ be generalized $\infty$-operads and $\phi=(\phi_1,\phi_2): \mM^\circledast \to \mV^\ot \times \mW^\ot $ a map of cocartesian fibrations relative to the collection of inert morphisms of $\Ass \times \Ass$ whose first component preserves the maximum and whose second component preserves the minimum.

\vspace{2mm}	

We call $\phi: \mM^\circledast \to \mV^\ot \times \mW^\ot $ an $\infty$-category weakly bitensored over $(\mV, \mW)$ (or an $\infty$-category with weak biaction over $(\mV,\mW)$) if the following conditions hold:
\begin{enumerate}
\item for every $\n,\m \geq 0$ the map $[0]\simeq \{\n\} \subset [\n]$ in the first component and the map $[0]\simeq \{0\} \subset [\m]$ in the second component induce an equivalence
$$ \theta: \mM^\circledast_{[\n][\m]} \to \mV^\ot_{[\n]} \times_{\mV^\ot_{[0]}}  \mM^\circledast_{[0][0]}  \times_{\mW^\ot_{[0]}}  \mW^\ot_{[\m]},$$
\vspace{1mm}
\item for every $\X,\Y \in \mM^\circledast$ lying over $([\m'], [\n']), ([\m], [\n]) \in \Ass \times \Ass$ the cocartesian lift $\Y \to \Y'$ of the map $[0]\simeq \{\m\} \subset [\m]$ and $[0]\simeq \{0\} \subset [\n]$ induces an
equivalence
$$\hspace{11mm} \mM^\circledast(\X,\Y) \to \mV^\ot(\phi_1(\X),\phi_1(\Y))\times_{\mV^\ot(\phi_1(\X),\phi_1(\Y'))} \mM^\circledast(\X,\Y') \times_{\mW^\ot(\phi_2(\X),\phi_2(\Y'))} \mW^\ot(\phi_2(\X),\phi_2(\Y)).$$
	
\end{enumerate}

\end{definition}

\begin{remark}\label{porto}
(2) implies that the functor $\theta$ in (1) is fully faithful
by taking the fiber over the identity of $[\n],[\m] \in \Ass \times \Ass $ in (2).
So if (2) holds, (1) holds if and only if $\theta$ is essentially surjective.

Moreover observe that $\theta$ is essentially surjective
if and only if for every $\V \in \mV^\ot$ lying over $[\n]$ and $\W \in \mW^\ot$ lying over $[\m]$ the cocartesian lifts $\V \to \V',\W \to \W'$ of the maps $[0]\simeq \{\n\} \subset [\n], [0]\simeq \{0\} \subset [\m]$ induce an essentially surjective functor
$\mM^\circledast_{\V,\W} \to \mM^\circledast_{\V',\mW'}.$
\end{remark}

\begin{notation}
	
For any $\infty$-category $\phi: \mM^\circledast \to \mV^\ot \times \mW^\ot$ with weak biaction we call $\mM:=  \mM^\circledast_{[0],[0]} $ the underlying $\infty$-category of $\phi$ and say that $\phi$ exhibits $\mM$ as weakly bitensored over $\mV,\mW$.
	
\end{notation}

\begin{example}\label{Exq}
Let $\mV^\ot \to \Ass$ be a generalized $\infty$-operad.
We write $\mV^\circledast \to \Ass \times \Ass$ for the pullback of $\mV^\ot \to \Ass$ along the functor $\Ass \times \Ass \to \Ass, \ ([\n],[\m]) \mapsto [\n]\ast [\m]$.
The two functors $\Ass \times \Ass \times [1] \to \Ass$ corresponding to the natural transformations $(-)\ast \emptyset \to (-)\ast (-), \ \emptyset \ast (-) \to (-)\ast (-)$ send the morphism $\id_{[\n],[\m]}, 0 \to 1$ to
an inert one and so give rise to functors
$\mV^\circledast \to \mV^\ot \times \Ass, \ \mV^\circledast \to \Ass \times \mV^\ot $ over $\Ass\times \Ass$. The resulting functor $\mV^\circledast \to \mV^\ot \times \mV^\ot$ is an $\infty$-category weakly bitensored over $\mV,\mV.$

\end{example}

\begin{example}\label{jhll}
For every $\infty$-category $\K$ and generalized $\infty$-operad $\mV^\ot \to \Ass$
the projection $ \mV^\ot \times \K \times \mW^\ot \to \mV^\ot \times \mW^\ot $
is an $\infty$-category weakly bitensored over $\mV,\mW.$
	
\end{example}

\begin{remark}\label{remosq}
Let $\mV^\ot \to \Ass, \mW^\ot \to \Ass$ be generalized $\infty$-operads,
$\phi: \mM^\circledast\to \mV^\ot \times \mW^\ot$ a weakly bitensored $\infty$-category and $\mN^\circledast \hookrightarrow \mM^\circledast$ a fully faithful functor.
Assume that the fiber transports of $\phi$ along inert morphisms of $\Ass \times \Ass$ whose first component preserves the maximum and whose second component preserves the minimum, preserve $\mN^\circledast $.
Assume that for every cocartesian lift $\X \to \Y$ in $\mM^\circledast$ of the maps $[0]\simeq \{\n\} \subset [\n], [0]\simeq \{0\} \subset [\m]$, where $\Y \in \mN^\circledast$, the object $\X$ belongs to $\mN^\circledast$.

The restriction $\mN^\circledast \hookrightarrow \mM^\circledast\to \mV^\ot \times \mW^\ot$ is a weakly bitensored $\infty$-category. Condition (2) of Definition \ref{bla} is inherited from $ \phi: \mM^\circledast \to \mV^\ot \times \mW^\ot$. So Condition (1) of Definition \ref{bla} follows from Remark \ref{porto}.

\end{remark}

\begin{example}
Let $\mM^\circledast \to \mV^\ot \times \mW^\ot$ be a weakly bitensored $\infty$-category and $\mN \subset \mM$ a full subcategory.
Let $\mN^\circledast \subset \mM^\circledast$ the full subcategory of $\X \in \mM^\circledast$ lying over some
$([\n],[\m])\in \Ass \times \Ass$ corresponding to a family $\V_1,...,\V_\n \in \mV, \Y \in \mN, \W_1,...,\W_\m \in \mW.$
By Remark \ref{remosq} the restriction $\mN^\circledast \subset \mM^\circledast \to \mV^\ot \times \mW^\ot$ is an $\infty$-category with weak biaction, where $\mN^\circledast_{[0][0]} \simeq \mN$.
We call $\mN^\circledast \to \mV^\ot \times \mW^\ot$ the full subcategory of $ \mM^\circledast $ with weak biaction spanned by $\mN.$

\end{example}

\begin{notation}\label{empp}
Let $\emptyset^\ot \subset \Ass$ be the full subcategory spanned by $[0] \in \Ass.$
Then $\emptyset^\ot$ is contractible and $\emptyset^\ot \subset \Ass$ is an $\infty$-operad such that $ \emptyset^\ot_{[0]} $ contractible and $\emptyset^\ot_{[\n]}$ is empty for $\n >0$.
For any generalized $\infty$-operad $\mV^\ot \to \Ass$
there is a canonical equivalence $\Alg(\emptyset,\mV) \simeq \mV_{[0]}^\ot$
so that $\emptyset^\ot \to \Ass$ is the initial $\infty$-operad.
\end{notation}

We will often consider $\infty$-categories weakly bitensored over $\mV, \emptyset$
and $ \emptyset,\mV$, which we call $\infty$-categories weakly left (right) tensored over 
$\mV.$
In this case the definition of $\infty$-category with weak biaction simplifies to the following one:
\begin{definition}\label{wla}
Let $\mV^\ot \to \Ass$ be a generalized $\infty$-operad and $\phi: \mM^\circledast \to \mV^\ot $ a map of cocartesian fibrations relative to the collection of inert morphisms of $\Ass$ that preserve the maximum.

\vspace{1mm}	

We call $\phi: \mM^\circledast \to \mV^\ot $ an $\infty$-category weakly left tensored over $\mV$ (or an $\infty$-category with weak left $\mV$-action) if the following conditions hold:
\begin{enumerate}
\item for every $\n \geq 0$ the map $[0]\simeq \{\n\} \subset [\n]$ in $\Delta$ induces an equivalence
$$ \theta: \mM^\circledast_{[\n]} \to \mV^\ot_{[\n]} \times_{\mV^\ot_{[0]}} \mM^\circledast_{[0]}.$$

\item for every $\X,\Y \in \mM^\circledast$ lying over $[\m], [\n] \in \Ass$ the cocartesian lift $\Y \to \Y'$ of the map $[0]\simeq \{\n\} \subset [\n]$ in $\Delta$ induces a pullback square
\begin{equation*} 
\begin{xy}
\xymatrix{
\mM^\circledast(\X,\Y)  \ar[d]^{} \ar[r]^{ }
& \mM^\circledast(\X,\Y') \ar[d]^{} 
\\ \mV^\ot(\phi(\X),\phi(\Y)) 
\ar[r]^{}  & \mV^\ot(\phi(\X),\phi(\Y')). 
}
\end{xy} 
\end{equation*}
\end{enumerate}
\end{definition}

There is a dual description of $\infty$-categories with weak right action.

\begin{definition}Let $\phi: \mM^\circledast \to \mV^\ot \times \mW^\ot$ be an  $\infty$-category with weak biaction.
	
\begin{itemize}
\item We say that $\phi: \mM^\circledast \to \mV^\ot \times \mW^\ot$ exhibits $\mM$ as bitensored over $\mV,\mW$ if $\phi$ is a map
of cocartesian fibrations over $\Ass \times \Ass$.

\item We say that $\phi: \mM^\circledast \to \mV^\ot \times \mW^\ot$ exhibits $\mM$ as left tensored over $\mV$ if $\phi$ is a map of cocartesian fibrations relative to the collection of morphisms of $\Ass \times \Ass$ whose second component is inert and preserves the minimum.

\item We say that $\phi: \mM^\circledast \to \mV^\ot \times \mW^\ot$ exhibits $\mM$ as right tensored over $\mW$ if $\phi$ is a map of cocartesian fibrations relative to the collection of morphisms of $\Ass \times \Ass$ whose first component is inert and preserves the maximum. 

\item We say that an $\infty$-category $\phi: \mM^\circledast \to \mV^\ot$
weakly left tensored over $\mV$ exhibits $\mM$ as left tensored over $\mV$
if $\phi$ is a map of cocartesian fibrations over $\Ass$.

\end{itemize}

\end{definition} 

\begin{remark}\label{remuu}

Let $\phi: \mM^\circledast \to \mV^\ot \times \mW^\ot$ be a map of cocartesian fibrations
over $\Ass \times \Ass.$
Then $\phi: \mM^\circledast \to \mV^\ot \times \mW^\ot$ exhibits $\mM$ as bitensored over $\mV,\mW$ if and only if the cocartesian fibrations $\mV^\ot \to \Ass, \mW^\ot \to \Ass$ 
are generalized monoidal $\infty$-categories and condition (1) of Definition \ref{bla} holds. Condition (2) is then automatic. 
\end{remark}

\begin{remark}\label{locally}
An $\infty$-category with weak biaction $\phi: \mM^\circledast \to \mV^\ot \times \mW^\ot$ exhibits $\mM$ as left tensored over $\mV$ if and only if
the functor $\mM^\circledast \to \mV^\ot$ is a map of cocartesian fibrations
over $\Ass$ and every morphism that is cocartesian with respect to the functor
$\mM^\circledast \to \mV^\ot \to \Ass$ is inverted by the functor $\mM^\circledast \to \mW^\ot.$ 
If $\phi: \mM^\circledast \to \mV^\ot \times \mW^\ot$ exhibits $\mM$ as left tensored over $\mV$, the underlying $\infty$-category with weak left action
exhibits $\mM$ as left tensored over $\mV$.
But the converse is not true.

\end{remark}

\begin{example}\label{Exq2}
Let $\mV^\ot \to \Ass$ be a generalized monoidal $\infty$-category.
The weakly bitensored $\infty$-category $\mV^\circledast= (\Ass \times \Ass) \times_\Ass \mV^\ot \to \mV^\ot \times \mV^\ot$ of Example \ref{Exq} exhibits $\mV$ as bitensored over $\mV,\mV.$
	
\end{example}

\begin{definition}Let $\phi: \mM^\circledast \to \mV^\ot$ be an $\infty$-category left tensored over a monoidal $\infty$-category.
\begin{itemize}
\item We say that $\phi: \mM^\circledast \to \mV^\ot$ exhibits $\mM$ as left tensored over $\mV$ compatible with small colimits if $\mV^\ot \to \Ass$ is a monoidal $\infty$-category compatible with small colimits, $\mM$ has small colimits and for any $\V \in \mV, \X \in \mM$
the functors $(-) \ot \X: \mV \to \mM, \V \ot (-): \mM \to \mM$ preserve small colimits.
\vspace{1mm}

\item We say that $\phi: \mM^\circledast \to \mV^\ot$ is an $\infty$-category with closed left action if for any $\V \in \mV, \X \in \mM$ the functor $(-) \ot \X: \mV \to \mM$ admits a right adjoint.
\end{itemize}	
\end{definition}

Moreover we define families of weakly left tensored / bitensored $\infty$-categories using Remark \ref{recogn}:

\begin{definition}\label{famwe}
Let $\rS\in \Cat_\infty$ and $\mV^\ot \to \rS \times \Ass,\mW^\ot \to \rS \times \Ass $ be $\rS$-families of generalized $\infty$-operads.
\begin{itemize}
\item A $\rS$-family of weakly left tensored $\infty$-categories is a weakly left tensored $\infty$-category $\mM^\circledast \to \mV^\ot.$

\item A $\rS$-family of weakly bitensored $\infty$-categories is a functor $\mM^\circledast \to \mV^\ot \times_\rS \mW^\ot$ such that the composition
$\mM^\circledast \to \mV^\ot \times_\rS \mW^\ot \to \mV^\ot \times \mW^\ot$
is a weakly bitensored $\infty$-category.
	
\item A (co)cartesian $\rS$-family of weakly left tensored / bitensored $\infty$-categories is a $\rS$-family of weakly left tensored / bitensored $\infty$-categories that is a map of (co)cartesian fibrations over $\rS.$
\end{itemize}		
\end{definition}

\begin{notation}\label{mult}(Multi-morphism spaces)
	
\begin{itemize}
\item Let $\mV^\ot \to \Ass$ be an $\infty$-operad and $\V_1,..., \V_{\n},\W \in \mV$ for some $\n \geq 0 $. We write
$$\Mul_{\mV}(\V_1,..., \V_\n;\W)$$ for the full subspace of $\mV^\ot(\V,\W)$ spanned by the morphisms $\V \to \W$ in $\mV^\ot$ lying over the active morphism $[1] \to [\n]$ in $\Delta,$ where $\V \in \mV^\ot_{[\n]} \simeq \mV^{\times \n}$ corresponds to $(\V_1,..., \V_\n) $.
		
\vspace{0,5mm}
		
\item Let $\mM^\circledast \to \mV^\ot \times \mW^\ot$ be an $\infty$-category weakly bitensored over $\infty$-operads and $\V_1,..., \V_{\n} \in \mV, \W_1, ..., \W_\m \in \mW, \ \X, \Y \in \mM$ for $\n, \m \geq 0$. We write
$$\Mul_{\mM}(\V_1,..., \V_\n,\X,\W_1,...,\W_\m; \Y)$$ for the full subspace of $\mM^\circledast(\Z,\Y)$ spanned by the morphisms $\Z \to \Y$ in $\mM^\circledast$ lying over the maps $[0] \simeq \{0 \} \subset [\n], [0] \simeq \{\m \} \subset [\m]$ in $\Delta,$ where $\Z \in \mM_{[\n],[\m]}^\circledast \simeq \mV^{\times \n} \times \mM \times  \mW^{\times \n}$ corresponds to $(\V_1,..., \V_\n,\X, \W_1,...,\W_\m) $.
		
\end{itemize}
\end{notation}

\begin{remark}
Let $\mV^\ot \to \Ass$ be an $\infty$-operad and $\V_1,..., \V_{\n},\X, \W_1,..., \W_{\m}, \Y \in \mV$ for some $\n,\m \geq 0 $. 
Let $\mM^\circledast:= \mV^\circledast$ be the weak $\mV,\mV$-biaction of $\mV$ on itself (Notation \ref{Exq}).
There is a canonical equivalence 
$$\Mul_{\mM}(\V_1,..., \V_\n,\X,\W_1,...,\W_\m; \Y) \simeq \Mul_{\mV}(\V_1,..., \V_\n,\X,\W_1,...,\W_\m; \Y).$$	

If $(\V_1,..., \V_\n,\X,\W_1,...,\W_\m)$ correspond to $\Z \in  \mV^\circledast_{[\n],[\m]} \simeq \mV^\ot_{[\n]*[\m]},$ this equivalence is the restriction of the map
$ \mV^\circledast(\Z,\Y)\simeq \mV^\ot(\Z,\Y)$ induced by the projection
$ \mV^\circledast= (\Ass \times \Ass) \times_\Ass \mV^\ot \to \mV^\ot.$ 
	
\end{remark}

Next we define morphisms of small weakly left, right and bitensored $\infty$-categories:

\begin{definition}\label{linmapp}
Let $\phi: \mM^\circledast \to \mV^\ot \times \mW^\ot, \phi':\mM'^\circledast \to \mV'^\ot \times \mW'^\ot $ be $\infty$-categories with weak biaction.

We call a commutative square of $\infty$-categories over $\Ass \times \Ass$ 
\begin{equation*} 
\begin{xy}
\xymatrix{
\mM^\circledast  \ar[d]^{\phi} \ar[rr]^{\gamma}
&&\mM'^\circledast \ar[d]^{\phi'} 
\\ \mV^\ot \times \mW^\ot
\ar[rr]^{\rho}  && \mV'^\ot \times \mW'^\ot
}
\end{xy} 
\end{equation*}
a map of weakly bitensored $\infty$-categories if $\gamma$ preserves cocartesian lifts of inert morphisms of $\Ass \times \Ass$ whose first component preserves the maximum and whose second component preserves the minimum.

We call a map of weakly bitensored $\infty$-categories
\begin{itemize}
\item a lax $\mV,\mW$-linear functor if $\rho$ is the identity. 

\item an embedding if the functors $\gamma$ and $\rho$ are embeddings of $\infty$-categories.
\end{itemize}
\end{definition}

\begin{definition}

Let $\phi: \mM^\circledast \to \mV^\ot \times \mW^\ot, \phi':\mM'^\circledast \to \mV^\ot \times \mW^\ot $ be $\infty$-categories with weak biaction.

We call a lax $\mV,\mW$-linear functor $ \mM^\circledast \to \mM'^\circledast$
\begin{itemize}
\item a $\mV$-linear functor if it preserves cocartesian lifts of morphisms of $\Ass \times \Ass$ whose second component preserves the minimum and $\phi, \phi'$ exhibit
$\mM, \mM'$ as left tensored over $\mV.$
\vspace{1mm}
\item a $\mW$-linear functor if it preserves cocartesian lifts of morphisms of $\Ass \times \Ass$ whose first component preserves the maximum and $\phi, \phi'$ exhibit
$\mM, \mM'$ as right tensored over $\mW.$
\vspace{1mm}
\item a $\mV, \mW$-linear functor if it preserves cocartesian lifts of morphisms of $\Ass \times \Ass$ and $\phi, \phi'$ exhibit $\mM, \mM'$ as bitensored over $\mV,\mW.$
\end{itemize}

\end{definition} 

\begin{notation}
If $\mW^\ot \simeq \emptyset^\ot$, we call a (lax) $\mV,\mW$-linear functor a (lax) $\mV$-linear functor and similar for weak right actions.
\end{notation}

\begin{notation} Let
$ \Cat_{\infty / \Ass \times \Ass}^{\max\min} \subset \Cat_{\infty / \Ass \times \Ass}$ be the subcategory of cocartesian fibrations relative to the collection of inert morphisms whose first component preserves the maximum and whose second component preserves the minimum.
Let $$\omega\BMod \subset \omega\BMod^\gen \subset (\Op_\infty^\gen \times \Op_\infty^\gen ) \times_{\Cat_{\infty / \Ass \times \Ass}^{\max\min} } \Fun([1], \Cat_{\infty / \Ass \times \Ass}^{\max\min}) $$ be the full subcategories of $\infty$-categories with weak biaction over (generalized) $\infty$-operads.
Let $$\BMod \subset \BMod^\gen \subset (\Op_\infty^{\mon,\gen} \times \Op_\infty^{\mon,\gen}) \times_{\Cat_{\infty / \Ass \times \Ass}^{\cocart} } \Fun([1], \Cat_{\infty / \Ass \times \Ass}^\cocart) $$ the full subcategory of $\infty$-categories with biaction over (generalized) monoidal $\infty$-categories.
		
\end{notation}

\begin{example}\label{abbml}
	
For every $\infty$-category $\K$ the unique functor $\K \to \emptyset^\ot \times \emptyset^\ot $ exhibits
$\K$ as weakly bitensored over $\emptyset, \emptyset.$
The forgetful functor $$\omega\BMod_{\emptyset, \emptyset} \subset \Cat_{\infty /\emptyset^\ot \times \emptyset^\ot} \simeq \Cat_\infty, \mM^\circledast \to \emptyset^\ot \times \emptyset^\ot \mapsto \mM \simeq \mM^\circledast$$ is an equivalence inverse to the functor sending $\K$ to $\K \to \emptyset^\ot \times \emptyset^\ot $.
	
\end{example}

\begin{notation} Let $\Cat_{\infty/\Ass}^{\max} \subset \Cat_{\infty/\Ass}$ be the subcategory of cocartesian fibrations relative to the collection of inert morphisms preserving the maximum.

Let 
$$\omega\LMod \subset \omega\LMod^\gen \subset \Op_\infty^\gen \times_{\Fun(\{1\},\Cat_{\infty/\Ass}^{\max}) }  
\Fun([1],\Cat_{\infty/\Ass}^{\max})$$ be the full subcategories
of $\infty$-categories with weak left action of (generalized) $\infty$-operads.

Let $$\LMod \subset \LMod^\gen \subset \Fun([1],\Cat_{\infty/\Ass}^{\cocart})$$ be the full subcategory of $\infty$-categories with left action of a (generalized) monoidal $\infty$-category.

Let $$\LMod^{\rc\rc} \subset \widehat{\LMod}$$ be the subcategory of $\infty$-categories with left action compatible with small colimits and functors preserving small colimits.
\end{notation}

\begin{remark}

There is a canonical equivalence $$\{ \emptyset^\ot \} \times_{([0] \times \Cat_{\infty / \Ass})} (\Cat_{\infty / \Ass} \times \Cat_{\infty / \Ass}) \times_{\Cat_{\infty / \Ass \times \Ass} } \Fun([1], \Cat_{\infty / \Ass \times \Ass}) \simeq \Fun([1], \Cat_{\infty / \Ass}) $$
that restricts to equivalences
$$ \omega\BMod^\gen \times_{\Op^\gen_{\infty}} \{\emptyset^\ot\}\simeq \omega\LMod^\gen, \ \omega\BMod \times_{\Op_{\infty}} \{\emptyset^\ot\}\simeq \omega\LMod$$
over $\Op_{\infty}^\gen$ giving embeddings
$\omega\LMod^\gen \subset \omega\BMod^\gen, \omega\LMod \subset \omega\BMod.$
\end{remark}

\begin{notation}
For any generalized $\infty$-operads $\mV^\ot \to \Ass, \mW^\ot \to \Ass$ let
\begin{itemize}
\item $\Cat^{{\max\min}}_{\infty/\mV^\ot \times \mW^\ot} \subset \Cat_{\infty/\mV^\ot\times \mW^\ot}$
be the subcategory of cocartesian fibrations relative to the collection of
cocartesian lifts in $\mV^\ot \times \mW^\ot$ of inert morphisms in $\Ass \times \Ass$ whose first component preserves the maximum and whose second component preserves the minimum.

\item $ \Cat^{{\max}}_{\infty/\mV^\ot} \subset \Cat_{\infty/\mV^\ot}$
be the subcategory of cocartesian fibrations relative to the collection of
cocartesian lifts in $\mV^\ot$ of inert morphisms of $\Ass$ that preserves the maximum.
\end{itemize}

\end{notation}

\begin{notation}
Evaluation at the target restricts to forgetful functors $$ \omega\BMod^\gen \to \Op_\infty^\gen \times \Op_\infty^\gen,\ \omega\BMod \to \Op_\infty \times \Op_\infty, $$$$ \BMod^\gen \to \Op_\infty^{\mon,\gen} \times \Op_\infty^{\mon,\gen}, \  \BMod \to \Op_\infty^\mon \times \Op_\infty^\mon, $$ 
whose fibers over generalized $\infty$-operads (generalized monoidal $\infty$-categories)
$\mV^\ot \to \Ass, \mW^\ot \to \Ass$ we denote by
$$ \omega\BMod_{\mV,\mW}, \BMod_{\mV,\mW} \subset \Cat^{{\max\min}}_{\infty/\mV^\ot \times \mW^\ot}. $$
\end{notation}

\begin{notation}
Similarly, evaluation at the target restricts to forgetful functors $$ \omega\LMod^\gen \to \Op_\infty^\gen,\ \omega\LMod \to \Op_\infty,\ \LMod^\gen \to \Op_\infty^{\mon,\gen}, \  \LMod \to \Op_\infty^\mon $$ whose fiber over a generalized $\infty$-operad (generalized monoidal $\infty$-category)
$\mV^\ot \to \Ass$ we denote by
$$ \omega\LMod_{\mV}, \LMod_{\mV} \subset \Cat^{\mathrm{max}}_{\infty/\mV^\ot}. $$
\end{notation}

\begin{remark}\label{modules}
By Example \ref{stru} there is a canonical equivalence
$ \LMod \simeq \Mon_\LM(\Cat_\infty) \simeq \LMod(\Cat_\infty)$
over $\Op_{\infty}^\mon \simeq \Mon(\Cat_\infty)\simeq \Alg(\Cat_\infty)$,
where $\Cat_\infty$-carries the cartesian structure.
Applied to a larger universe the latter equivalence restricts to an 
equivalence $ \LMod^{\rc\rc} \simeq \LMod(\Cat^{\rc\rc}_\infty)$
over $\Op_{\infty}^{\mon, \rc\rc} \simeq \Alg(\Cat^{\rc\rc}_\infty).$

For any monoidal $\infty$-category $\mV^\ot \to \Ass$ corresponding to an associative algebra in $\Cat_\infty$ there is an induced equivalence
$ \LMod_\mV \simeq \LMod_\mV(\Cat_\infty).$

\end{remark}

Next we study functoriality of these fibers with respect to $ \mV^\ot \to \Ass, \mW^\ot \to \Ass$ (Proposition \ref{bica}).

For that note that both functors evaluating at the target 
$$ \Fun([1], \Cat_{\infty / \Ass \times \Ass}^{\max\min}) \to \Cat_{\infty / \Ass \times \Ass}^{\max\min}, \ \Fun([1], \Cat_{\infty/\Ass \times \Ass}^\cocart) \to \Cat_{\infty/\Ass \times \Ass}^\cocart$$ are bicartesian fibrations since $ \Cat_{\infty / \Ass \times \Ass}^{\max\min}, \Cat_{\infty/\Ass \times \Ass}^\cocart $ admit pullbacks. 

Moreover for the proof of the next proposition we use the following lemma:

\begin{lemma}\label{remul}
	
Let $\psi: \mB \to\mD, \phi: \mC \to \mD$ be cartesian fibrations and $\mB \to \mC$ a
map of cartesian fibrations over $\mD$ that is fully faithful.
If for every $\X \in \mD$ the induced functor
$ \{\X\} \times_\mD \mB \subset \{\X\} \times_\mD \mC$ admits a left adjoint
and $\phi$ is a cocartesian fibration, then $\psi$ is a cocartesian fibration.
\end{lemma}
\begin{proof}
Let $\alpha: \X \to \Y$ be a morphism of $\mD$ and
$[1] \times_\mD \mC \to [1]$ the pullback along the functor $[1]\to \mD$
corresponding to $\alpha.$
The bicartesian fibration $[1] \times_\mD \mC \to [1]$ classifies an 
adjunction $\alpha_!: \mC_\X \rightleftarrows \mC_\Y: \alpha^*$.
Since the induced functor $[1] \times_\mD \mB \to [1] \times_\mD \mC$
is a map of cartesian fibrations over $[1]$,
the functor $\alpha^*: \mC_\Y \to \mC_\X$ restricts to a functor
$\mB_\Y \to \mB_\X$. In other words the functor $\alpha^*$
preserves local objects for the corresponding localizations.
This is equivalent to ask that the left adjoint
$\alpha_!$ preserves local equivalences for the corresponding localizations.
By \cite[Lemma 2.2.4.11.]{lurie.HTT} this guarantees that $\psi$ is a cocartesian fibration if $\phi$ is a cocartesian fibration.
\end{proof}

\begin{proposition}\label{bica}
The canonical functors $$  \omega\BMod^\gen \to \Op_\infty^\gen \times \Op_\infty^\gen,\ 
\BMod^\gen \to \Op_\infty^{\mon,\gen} \times \Op_\infty^{\mon,\gen} $$ 
are bicartesian fibrations and the embeddings
$$\omega\BMod^\gen \subset (\Cat_{\infty / \Ass}^{\mathrm{inert}} \times \Cat_{\infty / \Ass}^{\mathrm{inert}}) \times_{\Cat_{\infty / \Ass \times \Ass}^{\max\min} } \Fun([1], \Cat_{\infty / \Ass \times \Ass}^{\max\min}), $$$$ \BMod^\gen \subset (\Cat_{\infty / \Ass}^{\cocart} \times \Cat_{\infty / \Ass}^{\cocart}) \times_{\Cat_{\infty / \Ass \times \Ass}^{\cocart} } \Fun([1], \Cat_{\infty / \Ass \times \Ass}^\cocart) $$ preserve cartesian morphisms and admit left adjoints relative to $ \Op^\gen_\infty \times \Op^\gen_\infty, \Op_\infty^{\mon,\gen} \times \Op_\infty^{\mon,\gen} $.

\end{proposition}
\begin{proof}
The bicartesian fibrations 
$$(\Cat_{\infty / \Ass}^{\mathrm{inert}} \times \Cat_{\infty / \Ass}^{\mathrm{inert}}) \times_{\Cat_{\infty / \Ass \times \Ass}^{\max\min} } \Fun([1], \Cat_{\infty / \Ass \times \Ass}^{\max\min}) \to \Cat_{\infty / \Ass}^{\mathrm{inert}} \times \Cat_{\infty / \Ass}^{\mathrm{inert}},$$
$$(\Cat_{\infty / \Ass}^{\cocart} \times \Cat_{\infty / \Ass}^{\cocart}) \times_{\Cat_{\infty / \Ass \times \Ass}^{\cocart} } \Fun([1], \Cat_{\infty / \Ass \times \Ass}^\cocart) \to \Cat_{\infty / \Ass}^{\cocart} \times \Cat_{\infty / \Ass}^{\cocart}$$
restrict to cartesian fibrations
$ \omega\BMod^\gen \to \Op_\infty^\gen \times \Op_\infty^\gen, \BMod^\gen \to \Op_\infty^{\mon,\gen} \times \Op_\infty^{\mon,\gen} $ with the same cartesian morphisms, respectively.

By Lemma \ref{remul} it is enough to check that the full embeddings
$$\omega\BMod^\gen \subset (\Cat_{\infty / \Ass}^{\mathrm{inert}} \times \Cat_{\infty / \Ass}^{\mathrm{inert}}) \times_{\Cat_{\infty / \Ass \times \Ass}^{\max\min} } \Fun([1], \Cat_{\infty / \Ass \times \Ass}^{\max\min}), $$$$\BMod^\gen \subset (\Cat_{\infty / \Ass}^{\cocart} \times \Cat_{\infty / \Ass}^{\cocart}) \times_{\Cat_{\infty / \Ass \times \Ass}^{\cocart} } \Fun([1], \Cat_{\infty / \Ass \times \Ass}^\cocart)$$
induce on the fiber over every pair of generalized $\infty$-operads (generalized monoidal $\infty$-categories) $\mV^\ot \to \Ass, \mW^\ot \to \Ass$ a localization.
Let $\mE \subset \Fun([1],\mV^\ot \times \mW^\ot)$ be the full subcategory spanned by all cocartesian lifts in $\mV^\ot \times \mW^\ot$.
Let $\Xi$ be the collection of functors $[0]^{\triangleleft} \simeq [1] \to \mV^\ot \times \mW^\ot$ determining a pair of inert morphisms $\V \to \V'$ in $\mV^\ot$ and
$\W \to \W'$ in $\mW^\ot$ lying over the maps $[0]\simeq \{\n\}\subset [\n]$,
$[0]\simeq \{0\}\subset [\m]$, respectively.
By Remark \ref{porto} the full subcategories $$\omega\BMod_{\mV,\mW} \subset (\Cat^{\max\min}_{\infty /\Ass \times \Ass})_{/\mV^\ot \times \mW^\ot} \simeq \Cat^{\max\min}_{\infty/\mV^\ot \times \mW^\ot}, \BMod^\gen_{\mV,\mW} \subset (\Cat^\cocart_{\infty /\Ass \times \Ass})_{/\mV^\ot \times \mW^\ot} \simeq \Cat^{\mE}_{\infty/\mV^\ot \times \mW^\ot}  $$ are the full subcategories
of $\Xi$-fibered objects in the sense of \cite[Definition B.0.19.]{lurie.higheralgebra}, which are localizations as a consequence of \cite[Theorem B.0.20.]{lurie.higheralgebra}.

\end{proof}

In the following we identify generalized $\LM$-operads with weakly left tensored $\infty$-categories:

\begin{proposition}\label{proo}
There is an equivalence
$\Cat_{\infty/\LM}^\mathrm{inert} \simeq \Cat_{\infty/\Ass}^{\mathrm{inert}} \times_{\Fun(\{1\},\Cat_{\infty/\Ass}^{\max}) } \Fun([1],\Cat_{\infty/\Ass}^{\max})$
that restricts to equivalences
$$ \Op^{\LM,\gen}_{\infty} \simeq \omega\LMod^\gen, \hspace{3mm} \Op^{\LM}_{\infty} \simeq \omega\LMod, \hspace{3mm} \Op^{\LM, \mon}_{\infty} \simeq \LMod. $$
\end{proposition}

\begin{proof}
Via the equivalence $\LM \simeq \Ass \times [1]$
classifying a functor by a cocartesian fibration gives an equivalence
\begin{equation}\label{epp}
(\Cat^\cocart_{\infty/[1]})_{/\LM}\simeq \Fun([1], \Cat_{\infty/\Ass}).\end{equation}
By Lemma \ref{llem} (1) the inclusion $ \Cat_{\infty/\LM}^\mathrm{inert} \subset \Cat_{\infty/\LM} $ lands in $(\Cat^\cocart_{\infty/[1]})_{/\LM}\subset (\Cat_{\infty/[1]})_{/\LM} \simeq \Cat_{\infty/\LM}$.
By Lemma \ref{llem} equivalence (\ref{epp}) restricts to an equivalence
$$\Cat_{\infty/\LM}^\mathrm{inert} \simeq \Cat_{\infty/\Ass}^{\mathrm{inert}} \times_{\Fun(\{1\},\Cat_{\infty/\Ass}^{\max})} \Fun([1],\Cat_{\infty/\Ass}^{\max}) \subset \Fun([1],\Cat_{\infty/\Ass}^{\max}) \subset \Fun([1],\Cat_{\infty/\Ass})$$ further restricting to the claimed equivalences.

\end{proof}

We used the following lemma:
\begin{lemma}\label{llem}
A functor $\alpha: \mO \to \LM \simeq \Ass \times [1]$ is a cocartesian fibration
relative to the collection of inert morphisms if and only if the following three conditions are satisfied:
\begin{enumerate}
\item The functor $\alpha$ is a map of cocartesian fibrations over $[1]$.
\vspace{1mm}
\item The pullback $ \mV^\ot:= (\Ass \times\{1\}) \times_{(\Ass \times [1])} \mO \to \Ass $ is a cocartesian fibration relative to the collection of inert morphisms.
\vspace{1mm}
\item The pullback $ \mM^\circledast:= (\Ass \times\{0\}) \times_{(\Ass \times [1])} \mO \to \Ass $ is a cocartesian fibration relative to the collection of inert morphisms preserving the maximum and the functor $\phi: \mM^\circledast \to \mV^\ot$ over $\Ass$ classified by $\alpha$ via (1) preserves cocartesian lifts of inert morphisms preserving the maximum.

\end{enumerate}

For cocartesian fibrations $\mO \to \LM,\mO'\to \LM$ relative to the collection of inert morphisms a functor $\theta: \mO \to \mO'$ over $\LM$ is a map of cocartesian fibrations relative to the collection of inert morphisms if and only if
$\theta$ is a map of cocartesian fibrations over $[1]$
classifying a commutative square of $\infty$-categories
\begin{equation*} 
\begin{xy}
\xymatrix{
\mM^\circledast  \ar[d]^{\phi} \ar[r]^{\alpha }
&\mN^\circledast \ar[d]^{\phi'}
\\ \mV^\ot
\ar[r]^{\beta}  & \mW^\ot,
}
\end{xy} 
\end{equation*} 
over $\Ass$ such that $\beta$ is a map of cocartesian fibrations relative to the collection of inert morphisms and $\alpha$ is a map of cocartesian fibrations relative to the collection of inert maps preserving the maximum.

\end{lemma}

\begin{proof}
	
First note that a morphism in $\LM$ is inert if and only if its corresponding morphism in $\Ass \times [1]$ satisfies one of the following conditions:
\begin{itemize}
\item its first component is inert and its second component is the unique map to the final object.
\item its first component is inert and preserves the maximum and its second component is the identity of the initial object.
\end{itemize}
	
Thus a functor $\alpha: \mO \to \Ass \times [1]$ is a cocartesian fibration relative to the collection of inert morphisms if and only if
the following conditions are satisfied:
\begin{enumerate}
\item $\alpha $ is a cocartesian fibration relative to the collection of morphisms of the form $(\f,\g)$, where $\f$ is an equivalence. 
\item $\alpha $ is a cocartesian fibration relative to the collection of morphisms of the form $(\f,\g)$, where $\g$ is the identity of 1 and $\f$ is 
an inert morphism.
\item $\alpha $ is a cocartesian fibration relative to the collection of morphisms of the form $(\f,\g)$, where $\g$ is the identity of 0 and $\f$ is 
an inert morphism that preserves the maximum.
\end{enumerate} 

Now observe that the first, second, third point is equivalent to the
first, second, third point of the statement, respectively.
The second part of the statement follows immediately from the first one.

\end{proof}

Proposition \ref{proo} also identifies $\rS$-families of $\LM$-operads with $\rS$-families of weakly left tensored $\infty$-categories for every $\infty$-category $\rS.$

The $\infty$-category of $\rS$-families of small weakly left tensored $\infty$-categories is by definition the pullback
$\omega\LMod^\gen \times_{\Op_\infty^\gen} \Op^{\Ass \times \rS,\gen}_\infty$
of the forgetful functor $\omega\LMod^\gen \to \Op_\infty^\gen$
along the forgetful functor $\Op^{\Ass \times \rS,\gen}_\infty \simeq (\Op_\infty^{\gen})_{/\Ass \times \rS}\to \Op_\infty^\gen$ using Remark \ref{recogn}. Note that
\begin{equation}\label{euter}
\omega\LMod^\gen \times_{\Op_\infty^\gen} \Op^{\Ass \times \rS,\gen}_\infty
\simeq \omega\LMod^\gen \times_{\Op_\infty^\gen} (\Op_\infty^{\gen})_{/\Ass \times \rS}\simeq \omega\LMod^\gen_{/(\id:\rS \times \Ass\to \rS \times \Ass)} .\end{equation}

\begin{corollary}
Let $\rS$ be an $\infty$-category. There is a canonical equivalence 
$$\Op^{\rS \times \LM,\gen}_{\infty} \simeq \omega\LMod^\gen \times_{\Op_\infty^\gen} \Op^{\rS \times \Ass,\gen}_\infty$$
between $\rS$-families of generalized $\LM$-operads and  $\rS$-families of weakly left tensored $\infty$-categories.
\end{corollary}

\begin{proof}
The generalized
$\LM$-operad $\rS \times \LM \to \LM$ classifies the weakly left tensored
$\infty$-category $\id: \rS \times \Ass\to \rS \times \Ass.$

The equivalence 
$\Op^{\LM,\gen}_{\infty} \simeq \omega\LMod^\gen$
of Proposition \ref{proo} induces an equivalence
$$\Op^{\rS \times \LM,\gen}_{\infty} \simeq (\Op^{\LM,\gen}_{\infty})_{/\rS \times \LM} \simeq \omega\LMod^\gen_{/(\id:\rS \times \Ass\to \rS \times \Ass)} \simeq \omega\LMod^\gen \times_{\Op_\infty^\gen} \Op^{\rS \times \Ass,\gen}_\infty,$$
where the first equivalence is by Remark \ref{recogn} and the last one is
the inverse of equivalence (\ref{euter}).
\end{proof}

Next we identify generalized $\BM$-operads with weakly bitensored $\infty$-categories.
For that we use the category $\Ass^{\triangleright} $ opposite to the category of finite totally ordered sets. 

\begin{notation}\label{aact}
We call a morphism of $\Ass^{\triangleright}$ inert if it is an inert morphism of $\Ass$ or its target is the final object.
We call a morphism in $\Ass^{\triangleright}$ active if it is an active morphism of $\Ass$ or it is the identity of the final object. 

\end{notation}

\begin{lemma}\label{zuwas}
Let $\phi: \X \to \rS$ be a functor.
\begin{enumerate}
\item The induced functor $\phi^{\triangleright}: \X^{\triangleright} \to \rS^{\triangleright}$ is a cocartesian fibration relative to the collection
of morphisms $\{ \A \to \infty \mid \A \in \rS^{\triangleright} \}.$

\item If $\phi: \X \to \rS$ is a cocartesian (left) fibration, the functor $\phi^{\triangleright}: \X^{\triangleright} \to \rS^{\triangleright}$ is a cocartesian (left) fibration, where a morphism of $\X^{\triangleright}$ is $\phi^{\triangleright}$-cocartesian if and only if it is a $\phi$-cocartesian morphism of $\X$ or it is the unique map to the final object of $\X^{\triangleright}$.
\end{enumerate}
\end{lemma}

\begin{proof}
(1) holds since the unique map in $\X^{\triangleright}$ to the final object is $\phi^{\triangleright}$-cocartesian.
(2) follows from (1) and the fact that the embedding $\X \subset \X^{\triangleright}$ sends $\phi$-cocartesian morphism to $\phi^{\triangleright}$-cocartesian morphisms.
\end{proof}

\begin{remark}\label{zuwas2}
By Lemma \ref{zuwas} the cocartesian fibration $\rS \to [0]$
induces a cocartesian fibration 
$ \rS^{\triangleright} \to [0]^{\triangleright} = [1]$,
which classifies the unique functor $\rS \to [0].$	
	
\end{remark}

By Lemma \ref{zuwas} the left fibration $\BM \to \Ass$ gives rise to a left fibration $ \BM^{\triangleright} \to \Ass^{\triangleright}$.
We call a morphism of $\BM^{\triangleright}$ inert (active) if its image in $\Ass^{\triangleright}$ is inert (active).
We have the following lemma:
\begin{lemma}\label{BMM}
There is a canonical equivalence 
$\Ass^{\triangleright}  \times \Ass^{\triangleright}  \simeq \BM^{\triangleright}$
that restricts to equivalences
$$(\Ass^{\triangleright} \times \{ \emptyset \to \infty \}) \setminus \{(\infty,\emptyset) \}\simeq \LM^{\triangleright},\ \Ass \times \{ \emptyset \to \infty \} \simeq \LM,$$
$$(\{ \emptyset \to \infty \} \times \Ass^{\triangleright}) \setminus \{(\emptyset, \infty) \} \simeq \RM^{\triangleright}, \ \{\emptyset  \to \infty \} \times \Ass \simeq \RM.$$
\begin{proof}
The desired equivalence is opposite to the canonical equivalence
\begin{equation*}\label{tzb}
\Delta^{\triangleleft} \times  \Delta^{\triangleleft} \simeq (\Delta^{\triangleleft})_{/[1]} = (\Delta_{/[1]})^{\triangleleft} , \ ([\n],[\m]) \mapsto [\n]\ast[\m]
\end{equation*}
inverse to the functor
$(\Delta^{\triangleleft})_{/[1]} \to \Delta^{\triangleleft} \times  \Delta^{\triangleleft}$ sending a map $\phi: [\n]\to [1]$ in $\Delta^{\triangleleft} $
to its fibers $(\phi^{-1}(0),\phi^{-1}(1)).$
\end{proof}	

\end{lemma}

\begin{remark}\label{char}
 
By Lemma \ref{BMM} there is a canonical equivalence $\BM^{\triangleright} \simeq \Ass^{\triangleright} \times \Ass^{\triangleright}. $
A morphism of $\BM^{\triangleright} \simeq \Ass^{\triangleright} \times \Ass^{\triangleright} $ is inert if and only if both of its components in $ \Ass^{\triangleright}$ are inert and one of the following conditions are satisfied:
\begin{itemize}
\item the first or second component is the unique map to the final object.
\item both components lie in $\Ass$ and the first component preserves the maximum and the second component preserves the minimum.
\end{itemize}

\end{remark}

\begin{notation}
	
Let $\Omega \subset \Fun([1],  \Ass^{\triangleright} \times \Ass^{\triangleright})$ be the full subcategory
of pairs of morphisms $(\f,\g)$ of one of the following types: 
\begin{itemize}
\item $\f$ is the unique morphism
$\A \to \infty$  for some $\A \in \Ass^{\triangleright}$ and $\g$ is an equivalence, 
\item $\f$ is an equivalence and $\g$ is the unique morphism
$\A \to \infty$ for some $\A \in   \Ass^{\triangleright}.$
\end{itemize}
\end{notation}
To any functor $\mO \to \BM$ we associate the functors
$$\mM^\circledast:=(\mO^{\triangleright})_{|{\Ass \times \Ass}} \to \Ass \times \Ass, \ \mV^\ot:=(\mO^{\triangleright})_{|{\{\infty\} \times \Ass}} \to \Ass, \ \mW^\ot:=(\mO^{\triangleright})_{\mid \Ass \times {\{\infty\}}} \to \Ass. $$

If the functor $\mO^{\triangleright}\to \BM^{\triangleright}$
is a cocartesian fibration relative to $\Omega$, the canonical natural transformation from the inclusion
$\Ass \times \Ass \subset \Ass^{\triangleright} \times \Ass^{\triangleright} $ to
$\Ass \times \Ass \to \{\infty\} \times \Ass \subset \Ass^{\triangleright} \times \Ass^{\triangleright} $ 
induces a functor $\mM^\circledast \to \mV^\ot \times \Ass, \mM^\circledast \to \Ass \times \mW^\ot$ over $\Ass \times \Ass$.
Similarly, the canonical natural transformation from the inclusion
$\Ass \times \Ass \subset \Ass^{\triangleright} \times \Ass^{\triangleright} $ to
$\Ass \times \Ass \to \Ass \times \{\infty\} \subset \Ass^{\triangleright} \times \Ass^{\triangleright} $ induces a functor $ \mM^\circledast \to \Ass \times \mW^\ot$ over $\Ass \times \Ass$.
In total we obtain a functor $\mM^\circledast \to \mV^\ot \times \mW^\ot$ over $\Ass \times \Ass.$

\begin{lemma}\label{luxa}
Let $\phi: \mO \to \BM$ be a functor such that the induced functor $\phi^{\triangleright}: \mO^{\triangleright}\to \BM^{\triangleright}$ is a cocartesian fibration relative to $\Omega$.
The functor $\phi^{\triangleright}$ is a cocartesian fibration relative to the collection of inert morphisms
if and only if the following conditions hold:
\begin{enumerate}
\item The functors $\mV^\ot \to \Ass, \ \mW^\ot \to \Ass$ are cocartesian fibrations relative to the collection of inert morphisms.

\item The functor $\mM^\circledast \to \Ass \times \Ass$ is a cocartesian fibration relative to the collection of inert morphisms whose first component preserves the maximum and whose second component preserves the minimum.

\item The functor $\mM^\circledast \to \mV^\ot \times \mW^\ot$ over $\Ass \times \Ass$ is a map of cocartesian fibrations relative to the collection of inert morphisms of $\Ass \times \Ass$ whose first component preserves the maximum and whose second component preserves the minimum.
\end{enumerate}
\end{lemma}

\begin{proof}
Since the functor $\phi^{\triangleright}:\mO^{\triangleright}\to \BM^{\triangleright}$ is a cocartesian fibration relative to $\Omega$, by Remark \ref{char}
the functor $\phi^{\triangleright}$ is a cocartesian fibration relative to the collection of inert morphisms
if and only if the following conditions hold:
\begin{itemize}
\item The functor $\phi^{\triangleright}$ is a cocartesian fibration relative to the collection of pairs of inert morphisms of $\Ass^{\triangleright}$ 
whose first or second component is the identity of the final object.

\item The functor $\phi^{\triangleright}$ is a cocartesian fibration relative to the collection $\mE$ of pairs of inert morphisms of $\Ass$ 
whose first component preserves the maximum and whose second component preserves the minimum.
\end{itemize}
The first of these conditions is equivalent to (1) because the embeddings $\mV^\ot, \mW^\ot \subset \mO^{\triangleright} $ preserve cocartesian lifts of inert morphisms.
Condition (2) is equivalent to ask that $\phi^{\triangleright}$ is a locally cocartesian fibration relative to $\mE$ and that every locally $\phi^{\triangleright} $-cocartesian lift of a morphism of $\mE$ is cocartesian
with respect to the pullback
$(\mO^{\triangleright})_{|{\Ass \times \Ass}} \to \Ass \times \Ass.$
If (1) and (2) hold, (3) is equivalent to ask that every locally $\phi^{\triangleright} $-cocartesian lift of a morphism of $\mE$, is
$\phi^{\triangleright} $-cocartesian.

\end{proof}

\begin{remark}\label{akka}
By Lemma \ref{visio} a cocartesian fibration $\mO \to \BM$ relative to the collection of inert morphisms is a (generalized) $\BM$-operad / (generalized) $\BM$-monoidal $\infty$-category if and only if the functor $\mM^\circledast \to \mV^\ot \times \mW^\ot$ is an $\infty$-category weakly bitensored / bitensored over (generalized) $\infty$-operads.
\end{remark}

\begin{terminology}
Let $\mO \to \BM$ be a generalized $\BM$-operad.
We call the functor $\mM^\circledast \to \mV^\ot \times \mW^\ot$ the $\infty$-category with weak biaction underlying $\mO \to \BM$.	
\end{terminology}

Sending a cocartesian fibration $\mO \to \BM$ relative to the collection of inert morphisms
to $\mM^\circledast \to \mV^\ot \times \mW^\ot$
defines a functor 
$$\zeta:\Cat_{\infty/\BM}^{\mathrm{inert}} \to (\Cat_{\infty / \Ass}^{\mathrm{inert}} \times \Cat_{\infty / \Ass}^{\mathrm{inert}}) \times_{\Cat_{\infty / \Ass \times \Ass}^{\max\min} } \Fun([1], \Cat_{\infty / \Ass \times \Ass}^{\max\min}). $$ 

We have the following proposition:

\begin{proposition}\label{bproo}
The functor $\zeta$ is an equivalence.
\end{proposition}
By Remark \ref{akka} the functor $\zeta$ restricts to equivalences 
$$\Op^{\BM,\gen}_{\infty} \simeq \omega\BMod^\gen, \hspace{3mm} \Op^\BM_{\infty} \simeq \omega\BMod, \hspace{3mm} \Op^{\BM,\mon}_{\infty} \simeq \BMod. $$

\begin{proof}
By Corollary \ref{addd} adding a final object $\mO \mapsto \mO^\triangleright$ defines an embedding $ \Cat_{\infty/\BM}^{\mathrm{inert}} \subset \Cat_{\infty/\BM^{\triangleright}}^{\mathrm{inert}}$
whose essential image is the full subcategory $ (\Cat_{\infty/\BM^{\triangleright}}^{\mathrm{inert}})'$ of cocartesian fibrations relative to the collection of inert morphisms of $\BM^{\triangleright} \simeq \Ass^{\triangleright} \times \Ass^{\triangleright} $ whose fiber over the final object is contractible.

Let $\xi $ be the functor $\Cat^{\Omega}_{\infty/ \BM^{\triangleright}} \to (\Cat_{\infty / \Ass} \times \Cat_{\infty / \Ass}) \times_{\Cat_{\infty / \Ass \times \Ass}} \Fun([1], \Cat_{\infty / \Ass \times \Ass})$  
that sends a cocartesian fibration $\mC \to \BM^{\triangleright} \simeq \Ass^{\triangleright} \times \Ass^{\triangleright}$ relative to $\Omega$ to
$$\mV^\ot:=\mC_{|{\{\infty\} \times \Ass}} \to \Ass, \ \mW^\ot:=\mC_{\mid \Ass \times {\{\infty\}}} \to \Ass, \ \mM^\circledast:=\mC_{|{\Ass \times \Ass}} \to \mV^\ot \times \mW^\ot. $$

Let $(\Cat^{\Omega}_{\infty/ \BM^{\triangleright}})' \subset \Cat^{\Omega}_{\infty/ \BM^{\triangleright}} $ be the full subcategory of cocartesian fibrations $\mC \to \BM^{\triangleright}$ relative to $ \Omega$ whose fiber over the final object is contractible. Note that $(\Cat_{\infty/\BM^{\triangleright}}^{\mathrm{inert}})'  \subset(\Cat^{\Omega}_{\infty/ \BM^{\triangleright}})'.$
By Lemma \ref{easy} the functor $\xi$ restricts to an equivalence
$$(\Cat^{\Omega}_{\infty/ \BM^{\triangleright}})' \simeq (\Cat_{\infty / \Ass} \times \Cat_{\infty / \Ass}) \times_{\Cat_{\infty / \Ass \times \Ass}} \Fun([1], \Cat_{\infty / \Ass \times \Ass}).$$
By Lemma \ref{luxa} this equivalence restricts to an equivalence
$$ (\Cat_{\infty/\BM^{\triangleright}}^{\mathrm{inert}})' \simeq (\Cat_{\infty / \Ass}^{\mathrm{inert}} \times \Cat_{\infty / \Ass}^{\mathrm{inert}}) \times_{\Cat_{\infty / \Ass \times \Ass}^{\max\min} } \Fun([1], \Cat_{\infty / \Ass \times \Ass}^{\max\min})$$
and $\zeta$ factors as 
$ \Cat_{\infty/\BM}^{\mathrm{inert}} \simeq (\Cat_{\infty/\BM^{\triangleright}}^{\mathrm{inert}})' \simeq (\Cat_{\infty / \Ass}^{\mathrm{inert}} \times \Cat_{\infty / \Ass}^{\mathrm{inert}}) \times_{\Cat_{\infty / \Ass \times \Ass}^{\max\min} } \Fun([1], \Cat_{\infty / \Ass \times \Ass}^{\max\min}).$

\end{proof}

For the proof of Proposition \ref{bproo} we used a few lemmas, for which we need the following notation:

\begin{notation}
For $\infty$-categories $\rS,\T$ let $\Omega \subset \Fun([1],  \rS^\triangleright \times \T^\triangleright)$ be the full subcategory
of pairs of morphisms $(\f,\g)$ of one of the following types: 
\begin{itemize}
\item $\f$ is the unique morphism
$\A \to \infty$  for some $\A \in \rS$ and $\g$ is an equivalence, 
\item $\f$ is an equivalence and $\g$ is the unique morphism
$\A \to \infty$ for some $\A \in \T.$
\end{itemize}
\end{notation}

Let $(\Cat_{\infty/{\rS^{\triangleright}\times \T^{\triangleright}}}^{\Omega})' \subset \Cat_{\infty/{\rS^{\triangleright}\times \T^{\triangleright}}}^{\Omega}$
be the full subcategory spanned by the cocartesian fibrations relative to $\Omega$ whose fiber over the final object is contractible.

\begin{lemma}\label{easy}
Let $\rS,\T$ be $\infty$-categories.	
There is a canonical equivalence $$\Phi:(\Cat_{\infty/{\rS^{\triangleright}\times \T^{\triangleright}}}^{\Omega})' \to (\Cat_{\infty / \rS} \times \Cat_{\infty / \T}) \times_{\Cat_{\infty / \rS \times \T}} \Fun([1],\Cat_{\infty / \rS \times \T}), $$
that sends $ \X \to \rS^{\triangleright} \times \T^{\triangleright} $ to
$(\X_{\mid\rS \times \{\infty\} }, \X_{\mid\{\infty\} \times \T}, \X_{\mid \rS \times \T} \to \X_{\mid\rS \times \{\infty\} } \times \X_{\mid\{\infty\} \times \T} ).$

\end{lemma}	

\begin{proof}
	
Observe that the equivalence 	
$ (\Cat_{\infty/[1] \times [1]})_{/\rS^\triangleright \times \T^{\triangleright}} \simeq  \Cat_{\infty/\rS^{\triangleright} \times \T^{\triangleright}},$
where we view $\rS^\triangleright \times \T^{\triangleright} $ over $[1] \times [1]$ via the functor $\rS^\triangleright \times \T^{\triangleright} \to [0]^\triangleright \times[0]^\triangleright =[1] \times [1],$
restricts to an equivalence
$ (\Cat^\cocart_{\infty/[1] \times [1]})_{/\rS^\triangleright \times \T^{\triangleright}} \simeq \Cat_{\infty/{\rS^{\triangleright}\times \T^{\triangleright}}}^{\Omega}$
that further restricts to an equivalence
$ ((\Cat^\cocart_{\infty/[1] \times [1]})_{/\rS^\triangleright \times \T^{\triangleright}})' \simeq (\Cat_{\infty/{\rS^{\triangleright}\times \T^{\triangleright}}}^{\Omega})',$
where the left hand side is the full subcategory of $(\Cat^\cocart_{\infty/[1] \times [1]})_{/\rS^\triangleright \times \T^{\triangleright}}$ spanned by the maps of cocartesian fibrations over $[1] \times [1]$ that induce an equivalence on the fiber over $(1,1).$
	
Let $\phi: [1] \times [1] \to \Cat_\infty$ be the functor that sends
$(1,1)$ to the final $\infty$-category,
sends $\{0\} \times [1] $ to $\rS \times \T \to \rS$ and
$[1] \times \{0\}$ to $\rS \times \T \to \T$.
The equivalence
$ \Cat^\cocart_{\infty/[1] \times [1]}\simeq \Fun([1] \times [1],\Cat_{\infty})$
gives rise to an equivalence 
$ (\Cat^\cocart_{\infty/[1] \times [1]})_{/\rS^\triangleright \times \T^\triangleright} \simeq \Fun([1] \times [1],\Cat_{\infty})_{/\phi}$
that restricts to an equivalence
$ ((\Cat^\cocart_{\infty/[1] \times [1]})_{/\rS^\triangleright \times \T^\triangleright})' \simeq \Fun([1] \times [1],\Cat_{\infty})'_{/\phi},$
where $\Fun([1] \times [1],\Cat_{\infty})'$ is the full subcategory
of $\Fun([1] \times [1],\Cat_{\infty})$ spanned the functors that send
$(1,1)$ to the final $\infty$-category.

The embedding 
$\bj: \Lambda^2_0 \simeq \{(0,0)\to (0,1), (0,0) \to (1,0) \} \subset [1] \times [1]$
gives rise to an adjunction
$\bj_*: \Fun([1] \times [1],\Cat_{\infty}) \rightleftarrows \Fun(\Lambda^2_1,\Cat_{\infty}):\bj_*,$
where the right adjoint is fully faithful and takes the right Kan extension.
By construction of the right Kan extension 
this adjunction restricts to an adjunction
$\Fun([1] \times [1],\Cat_{\infty})' \rightleftarrows \Fun(\Lambda^2_0,\Cat_{\infty}).$
As the left adjoint is conservative, it is an equivalence.
The left adjoint equivalence induces an equivalence 
$\Fun([1] \times [1],\Cat_{\infty})'_{/\phi} \simeq \Fun(\Lambda^2_0,\Cat_{\infty})_{/\sigma},$
where $\sigma: \Lambda_0^2 \to \Cat_\infty$ is the functor that sends
$0 \to 1, 0 \to 2$ to the projections $\q_\rS:\rS \times \T \to \rS, \q_\T: \rS \times \T \to \T,$ respectively. 
The functor
$$ \rho_\rS: \Theta_\rS:= \Cat_{\infty / \rS} \times_{\Fun(\{1\},\Cat_{\infty / \rS \times \T})} \Fun([1],\Cat_{\infty / \rS \times \T}) \to \Fun([1],\Cat_{\infty})_{/\q_\rS}$$ sending $(\mA \to \rS, \mC \to \mA \times \T) $ to $ \mC \to \mA \times \T \to \mA$ is an equivalence.
This follows from the fact that $\rho$ is a map of cocartesian fibrations over
$ \Cat_{\infty / \rS}$, where we view the target of $\rho$ over $ \Cat_{\infty / \rS}$ via evaluation at the target, that induces fiberwise equivalences.
Indeed, evaluation at the target is a cocartesian fibration
whose cocartesian morphisms are those that are inverted by evaluation at the source. Hence $\rho$ is a map of cocartesian fibrations over
$ \Cat_{\infty / \rS}$. Moreover $\rho$ induces on the fiber over any functor $\mC \to \rS$
the canonical equivalence $(\Cat_{\infty / \rS \times \T})_{/\mC \times \T} \simeq \Cat_{\infty / \mC \times \T}.$
Via $\rho_\rS, \rho_\T$ we obtain a canonical equivalence
$$\Fun(\Lambda^2_0,\Cat_{\infty})_{/\sigma} \simeq  \Fun([1],\Cat_{\infty})_{/\q_\rS} \times_{\Fun(\{0\},\Cat_{\infty})_{/ \rS \times \T}}
\Fun([1],\Cat_{\infty})_{/\q_\T} \simeq \Theta_\rS \times_{\Fun(\{0\},\Cat_{\infty / \rS \times \T})} \Theta_\T.$$

Finally, we observe that there is a canonical equivalence
$$ \Theta_\rS \times_{\Fun(\{0\},\Cat_{\infty / \rS \times \T})} \Theta_\T \simeq (\Cat_{\infty / \rS} \times \Cat_{\infty / \T}) \times_{\Fun(\{1\},\Cat_{\infty / \rS \times \T})} \Fun([1],\Cat_{\infty / \rS \times \T})$$
sending $(\mA \to \rS, \mC \to \mA \times \T, \mB \to \T, \mC \to \rS \times \mB)$
to $(\mA \to \rS, \mB \to \T, \mC \to (\mA \times \T) \times_{(\rS \times \T)}  (\rS \times \mB) \simeq \mA \times \mB).$

\end{proof}

\begin{lemma}
Let $\rS$ be an $\infty$-category, $ \Y \to \rS$ a functor and $\X \to \rS^{\triangleright}$ a functor such that $\X$ has a final object lying over the final object of $\rS^{\triangleright}$. The canonical functor
$\phi: \Fun_{\rS^{\triangleright}}(\Y^{\triangleright},\X) \to
\Fun_\rS(\Y, \X_{\mid \rS})$ is an equivalence.

\end{lemma}

\begin{proof}
The functor $\phi$ is inverse to the functor
$\Fun_\rS(\Y,\X_{\mid \rS}) \to \Fun_{\rS^{\triangleright}}(\Y^{\triangleright},(\X_{\mid \rS})^{\triangleright}) \to \Fun_{\rS^{\triangleright}}(\Y^{\triangleright},\X) $
induced by the functor $(\X_{\mid \rS})^{\triangleright} \to \X$ over $\rS^{\triangleright}$ whose pullback to $\rS$ is the identity of $ \X_{\mid \rS}$ and whose image of the final object of $(\X_{\mid \rS})^{\triangleright}$ is the final object of $\X.$

\end{proof}

\begin{corollary}\label{addd}
Let $\rS$ be an $\infty$-category and $\Theta \subset \Cat_{\infty / \rS^{\triangleright}}$ the full subcategory of functors $\X \to \rS^{\triangleright}$ such that $\X$ has a final object lying over the final object of $\rS^{\triangleright}$.
Restricting to $\rS \subset \rS^{\triangleright}$ 
defines a functor $(-)_{\mid \rS}: \Theta \subset \Cat_{\infty / \rS^{\triangleright}} \to \Cat_{\infty / \rS}$ that admits a fully faithful left adjoint with colocal objects 
the functors $\X \to \rS^{\triangleright}$ whose fiber over the final object of $\rS^{\triangleright}$ is contractible.

\end{corollary}

\begin{proof}
Let $\lambda: \X \to \rS^{\triangleright}$ belong to $\Theta$. The functor $(\X_{\mid \rS})^{\triangleright} \to \X$ over $\rS^{\triangleright}$ is an equivalence if and only if the fiber of $\lambda$ over the final object of $\rS^{\triangleright}$ is contractible.

\end{proof}
Let $\rS$ be an $\infty$-category and $\sigma: \Ass \times \Ass \times \rS \to \Ass \times \rS \times \Ass \times \rS$ the weakly bitensored $\infty$-category induced by the diagonal, which is classified by the generalized $\BM$-operad $\rS \times \BM \to \BM$.

The $\infty$-category of $\rS$-families of small weakly bitensored $\infty$-categories is by definition $\omega\BMod_{/\sigma}^\gen$ 
since a $\rS$-family of weakly bitensored $\infty$-categories is equivalently given by a weakly bitensored $\infty$-category over $\sigma$ using Remark \ref{recogn}. Since $\Op^{\rS \times \BM,\gen}_{\infty} \simeq (\Op^{\BM,\gen}_{\infty})_{/\rS \times \BM}$ by Remark \ref{recogn}, we get the following corollary:

\begin{corollary}
Let $\rS$ be an $\infty$-category. 
There is a canonical equivalence $$\Op^{\rS \times \BM,\gen}_{\infty} \simeq \omega\BMod^\gen_{/\sigma}$$
between $\rS$-families of generalized $\BM$-operads and $\rS$-families of weakly bitensored $\infty$-categories.

\end{corollary}

By Example \ref{strai} Proposition \ref{bproo} gives the following corollary:

\begin{corollary}
There is a canonical equivalence 
$\BMod(\Cat_\infty) \simeq \Op^{\BM, \mon}_{\infty} \simeq \BMod.$
\end{corollary}

For any generalized $\BM$-operad $\mO \to \BM$ one can consider bimodules in $\mO$. 
If $\mO \to \BM$ classifies a weakly bitensored $\infty$-category 
$\phi: \mM^\circledast \to \mV^\ot \times \mW^\ot, $ there is a description of bimodules in terms of $\phi:$
\begin{notation}
For any $\infty$-category $\phi: \mM^\circledast \to \mV^\ot \times \mW^\ot $ with weak biaction 
let $$\BMod(\mM) \subset \Fun_{\Ass \times \Ass}(\Ass \times \Ass,\mM^\circledast) $$ the full subcategory of functors $\theta: \Ass \times \Ass \to \mM^\circledast$ over $\Ass \times \Ass$ with the following properties:
\begin{enumerate}
\item $\alpha$ sends inert morphisms whose first component preserves the maximum and
whose second component preserves the minimum to cocartesian lifts.
\item The composition $\Ass \times \Ass \xrightarrow{\theta} \mM^\circledast \xrightarrow{\phi} \mV^\ot \times \mW^\ot$ sends any inert morphism to a cocartesian lift.
\end{enumerate}
\end{notation}

\begin{notation}\label{not1}
The functor $\phi$ gives rise to a functor $$\Fun_{\Ass \times \Ass}(\Ass \times \Ass,\mM^\circledast) \to \Fun_\Ass(\Ass,\mV^\ot) \times \Fun_\Ass(\Ass,\mW^\ot) $$ restricting to a functor $\BMod(\mM) \to \Alg(\mV) \times \Alg(\mW),$
whose fiber over $\A \in \Alg(\mV), \B \in \Alg(\mW)$ we denote by $\BMod_{\A,\B}(\mM)$.
\end{notation}

Proposition \ref{bproo} implies the following corollary:
\begin{corollary}\label{motivo}
Let $\mO \to \BM$ be a generalized $\BM$-operad classifying a weakly bitensored $\infty$-category $\mM^\circledast:= \mO^\triangleright_{|{\Ass \times \Ass}} \to \mV^\ot \times \mW^\ot. $ 
The canonical functor $$\Fun_\BM(\BM,\mO)\simeq \Fun_{\Ass^{\triangleright} \times \Ass^{\triangleright}}(\Ass^{\triangleright} \times \Ass^{\triangleright},\mO^\triangleright) \to \Fun_{\Ass \times \Ass}(\Ass \times \Ass,\mM^\circledast) $$ restricts to an equivalence
$ \BMod(\mO) \to \BMod(\mM)$ over $\Alg(\mV) \times \Alg(\mW)$.
\end{corollary}

\begin{notation}
Let $\rS$ be an $\infty$-category and $\mM^\circledast \to \mV^\ot \times_\rS \mW^\ot$ a $\rS$-family of weakly bitensored $\infty$-categories.
Let $\BMod^\rS(\mM) \subset \BMod(\mM)$ be the pullback of 
$\BMod(\mM) \to \Alg(\mV) \times \Alg(\mW)$ to the full subcategory  $\Alg^\rS(\mV) \times \Alg^\rS(\mW) \subset \Alg(\mV) \times \Alg(\mW)$ (see Remark \ref{diagol}).

\end{notation}

\begin{remark}
The functor $\mM^\circledast \to \mV^\ot \times_\rS \mW^\ot$ induces a functor $\BMod^\rS(\mM) \to \Alg^\rS(\mV) \times_\rS \Alg^\rS(\mW)$.
For any $\s \in \rS$ there is a canonical equivalence $\BMod^\rS(\mM)_\s \simeq \BMod(\mM_\s)$ over $\Alg(\mV_\s) \times \Alg(\mW_\s).$
\end{remark}

Remark \ref{diagol} implies the following corollary:
\begin{corollary}\label{motivo}
Let $\mO \to \rS \times \BM$ be a $\rS$-family of generalized $\BM$-operads classifying a $\rS$-family of weakly bitensored $\infty$-categories $\mM^\circledast\to \mV^\ot \times \mW^\ot. $ 
The canonical equivalence
$ \BMod(\mO)\simeq \BMod(\mM)$ over $\Alg(\mV)\times \Alg(\mW)$ restricts to an equivalence
$ \BMod^\rS(\mO)\simeq \BMod^\rS(\mM)$ over $\Alg^\rS(\mV)\times_\rS \Alg^\rS(\mW)$.
\end{corollary}

\vspace{2mm}

The same is true for generalized $\LM$-operads:
\begin{notation}
For any $\infty$-category $\phi: \mM^\circledast \to \mV^\ot $ with weak left action 
let $$\LMod(\mM) \subset \Fun_\Ass(\Ass,\mM^\circledast) $$ be the full subcategory spanned by the functors $\alpha: \Ass \to \mM^\circledast$ over $\Ass$ 
with the following properties:
\begin{enumerate}
\item $\alpha$ sends inert morphisms preserving the maximum to a cocartesian lift.
\item The composition $\Ass \xrightarrow{\alpha} \mM^\circledast \xrightarrow{\phi} \mV^\ot$ sends any inert morphism to a cocartesian lift.
\end{enumerate}
\end{notation}
\begin{notation}\label{not2}
The functor $\phi$ gives rise to a functor $\Fun_\Ass(\Ass,\mM^\circledast) \to \Fun_\Ass(\Ass,\mV^\ot)$ restricting to a functor $\LMod(\mM) \to \Alg(\mV),$
whose fiber over $\A \in \Alg(\mV)$ we denote by $\LMod_\A(\mM)$.
\end{notation}
Proposition \ref{proo} gives the following corollary:
\begin{corollary}\label{motiv}
Let $\mO \to \LM$ be a generalized $\LM$-operad classifying a weakly left tensored $\infty$-category 
$\mM^\circledast :=(\Ass \times \{0\}) \times_\LM \mO \to \mV^\ot.$
The embedding $\Ass \times \{0\} \subset  \Ass \times [1] \simeq \LM$ 
induces a functor $$\Fun_\LM(\LM,\mO) \to \Fun_\LM(\Ass,\mO) \simeq \Fun_\Ass(\Ass,\mM^\circledast)$$ that restricts to an equivalence $ \LMod(\mO) \to \LMod(\mM)$ over $\Alg(\mV)$.
\end{corollary}

\begin{notation}
Let $\rS$ be an $\infty$-category and $\mM^\circledast \to \mV^\ot $ a $\rS$-family of weakly left tensored $\infty$-categories. Let $\LMod^\rS(\mM) \subset \LMod(\mM)$ be the pullback of $\LMod(\mM) \to \Alg(\mV)$
to the full subcategory $\Alg^\rS(\mV) \subset \Alg(\mV)$ (see Remark \ref{diagol}).
	
\end{notation}

\begin{remark}
For any $\s \in \rS$ there is a canonical equivalence $\LMod^\rS(\mM)_\s \simeq \LMod(\mM_\s)$ over $\Alg(\mV_\s).$
\end{remark}

\begin{corollary}\label{motivo}
Let $\mO \to \rS \times \LM$ be a $\rS$-family of generalized $\LM$-operads classifying a $\rS$-family of weakly left tensored $\infty$-categories $\mM^\circledast\to \mV^\ot.$
The canonical equivalence
$ \LMod(\mO)\simeq \LMod(\mM)$ over $\Alg(\mV)$ restricts to an equivalence
$ \LMod^\rS(\mO)\simeq \LMod^\rS(\mM)$ over $\Alg^\rS(\mV).$
\end{corollary}

\begin{lemma}\label{2enr}
Let $\mV^\ot \to \Ass, \mW^\ot \to \Ass$ be generalized $\infty$-operads.
The $\infty$-category $\omega\BMod_{\mV,\mW}$ is left tensored over
$\Cat_\infty$. 
The left action sends an $\infty$-category $\K$ and an $\infty$-category with weak biaction
$\mM^\circledast \to \mV^\ot \times \mW^\ot$ to $\K \times \mM^\circledast \to \mV^\ot \times \mW^\ot$ and so restricts to $\BMod_{\mV,\mW}$ if $\mV^\ot \to \Ass, \mW^\ot \to \Ass$ are monoidal $\infty$-categories.
\end{lemma}
\begin{proof}  
By Example \ref{jhll} the functor $\Cat_\infty \to \Cat_{\infty/\mV^\ot \times \mW^\ot}, \ \K \mapsto \mV^\ot \times \K \times \mW^\ot$ induces a functor $\xi: \Cat_\infty \to \omega\BMod_{\mV,\mW}$ that preserves finite products and so promotes to a monoidal functor
$ \Cat_\infty^\times \to \omega\BMod_{\mV,\mW}^\times$ (\ref{cacart}).
By Example \ref{Exq2} we may view $\omega\BMod_{\mV,\mW}$ as left tensored over itself provided by the left tensored $\infty$-category
$\omega\BMod_{\mV,\mW}^\circledast \to \omega\BMod_{\mV,\mW}^\times.$
The pullback $ \Cat_\infty^\times \times_{\omega\BMod_{\mV,\mW}^\times}\omega\BMod_{\mV,\mW}^\circledast \to \Cat_\infty^\times$ exhibits $\omega\BMod_{\mV,\mW}$ as left tensored over $ \Cat_\infty,$
where the left action sends an $\infty$-category $\K$ and an $\infty$-category with weak biaction $\mM^\circledast \to \mV^\ot \times \mW^\ot$ to $\K \times \mM^\circledast \to \mV^\ot \times \mW^\ot$.
If $\mV^\ot \to \Ass, \mW^\ot \to \Ass$ are monoidal $\infty$-categories,
$\xi$ factors through $\BMod_{\mV,\mW}$ and $ \Cat_\infty^\times \times_{\BMod_{\mV,\mW}^\times} \BMod_{\mV,\mW}^\circledast \to \Cat_\infty^\times$ exhibits $\BMod_{\mV,\mW}$ as left tensored over $ \Cat_\infty$. So the inclusion $\BMod_{\mV,\mW} \subset \omega \BMod_{\mV,\mW}$
is $\Cat_\infty$-linear.
\end{proof}

\begin{notation}
Let $\Cat^{\max}_{\infty /\Ass \times \Ass}, \Cat^{\min}_{\infty /\Ass \times \Ass} \subset \Cat_{\infty /\Ass \times \Ass}$ be the subcategories of cocartesian fibrations relative to the collection of inert morphisms whose first component preserves the maximum, whose second component preserves the minimum, respectively.
\end{notation}

\begin{proposition}
Let $\mV^\ot \to \Ass, \mW^\ot \to \Ass$ be generalized $\infty$-operads.
\begin{enumerate}
\item If $\mW^\ot \to \Ass$ is a monoidal $\infty$-category, there is a canonical embedding $$ \RMod_\mW(\omega\LMod_\mV) \hookrightarrow (\Cat^{\max}_{\infty /\Ass \times \Ass})_{/\mV^\ot \times \mW^\ot}$$
with essential image the $\infty$-categories $\mM^\circledast \to \mV^\ot \times \mW^\ot$ with weak biaction that exhibit $\mM$ as right tensored over $\mW$.
\vspace{1mm}

\item If $\mV^\ot \to \Ass$ is monoidal $\infty$-category, there is a canonical embedding $$ \LMod_\mV(\omega\RMod_\mW) \hookrightarrow (\Cat^{\min}_{\infty /\Ass \times \Ass})_{/\mV^\ot \times \mW^\ot}$$
with essential image the $\infty$-categories $\mM^\circledast \to \mV^\ot \times \mW^\ot$ with weak biaction that exhibit $\mM$ as left tensored over $\mV$.

\vspace{1mm}
\item If $\mV^\ot \to \Ass,\mW^\ot \to \Ass$ are monoidal $\infty$-categories, the
$\Cat_\infty$-linear inclusion $\LMod_\mV \subset \omega\LMod_\mV$ induces an
inclusion $\RMod_\mW(\LMod_\mV) \subset \RMod_\mW(\omega\LMod_\mV) \hookrightarrow (\Cat^{\max}_{\infty /\Ass \times \Ass})_{/\mV^\ot \times \mW^\ot}$
that restricts to an equivalence $\RMod_\mW(\LMod_\mV) \simeq \BMod_{\mV,\mW}.$

\vspace{1mm}
\item If $\mV^\ot \to \Ass,\mW^\ot \to \Ass$ are monoidal $\infty$-categories, the
$\Cat_\infty$-linear inclusion $\RMod_\mW \subset \omega\RMod_\mW$ induces an
inclusion $\LMod_\mV(\RMod_\mW) \subset \LMod_\mV(\omega\RMod_\mW) \hookrightarrow (\Cat^{\min}_{\infty /\Ass \times \Ass})_{/\mV^\ot \times \mW^\ot}$
that restricts to an equivalence $\LMod_\mV(\RMod_\mW) \simeq \BMod_{\mV,\mW}.$
\end{enumerate}	
\end{proposition}

\begin{proof}
(1) By Example \ref{stru} there is a canonical embedding
$$\RMod_\mW(\omega\LMod_\mV) \simeq \RMod_{\mV^\ot \times \mW}(\omega\LMod_\mV^\times)\subset\RMod_{\mV^\ot \times \mW}((\Cat^{\max}_{\infty/\mV^\ot})^\times)\subset (\Cat^{\max}_{\infty /\Ass \times \Ass})_{/\mV^\ot \times \mW^\ot} $$
with essential image the $\infty$-categories $\mM^\circledast \to \mV^\ot \times \mW^\ot$ with weak biaction that exhibit $\mM$ as right tensored over $\mW$.
(2) is similar and (3) and (4) are clear.
\end{proof}

\subsection{$\infty$-categories of lax linear functors}
\label{LinF}
\begin{notation}\emph{}
For $\infty$-categories with weak biaction $\phi: \mM^\circledast \to \mV^\ot \times \mW^\ot, \phi':\mM'^\circledast \to \mV^\ot \times \mW^\ot $ let
\begin{itemize}
\item $ \LaxLinFun_{\mV,\mW}(\mM,\mM') \subset \Fun_{\mV^\ot \times \mW^\ot}(\mM^\circledast,\mM'^\circledast)$ be the full subcategory of lax $\mV,\mW$-linear functors.

\vspace{1mm}
\item $ \LinFun_{\mV,\mW}(\mM,\mM') \subset \LaxLinFun_{\mV,\mW}(\mM,\mM')$ be the full subcategory of $\mV,\mW$-linear functors if $\phi$, $\phi'$ are $\infty$-categories with biaction.
\end{itemize}

\end{notation}

\begin{remark}\label{enrrr}
For any $\K\in \Cat_\infty$ and $\infty$-categories $\mM^\circledast \to \mV^\ot \times \mW^\ot, \mM'^\circledast \to \mV^\ot \times \mW^\ot$ with weak biaction the equivalence
$\Fun(\K, \Fun_{\mV^\ot \times \mW^\ot}(\mM^\circledast,\mM'^\circledast)) \simeq \Fun_{\mV^\ot \times \mW^\ot}(\K \times \mM^\circledast, \mM'^\circledast)$
restricts to an equivalence
$$\Fun(\K,\LaxLinFun_{\mV,\mW}(\mM,\mM')) \simeq \LaxLinFun_{\mV,\mW}(\K \times \mM,\mM').$$
\end{remark}
\begin{notation}\label{diagor}
For any $\infty$-category $\K$ and $\infty$-category $\mM^\circledast \to \mV^\ot \times \mW^\ot$ with weak biaction let $$(\mM^\circledast)^\K:=(\mV^\ot \times \mW^\ot) \times_{\Fun(\K,\mV^\ot \times \mW^\ot)} \Fun(\K,\mM^\circledast) \to \mV^\ot \times \mW^\ot $$
be the pullback along the diagonal functor, which is an $\infty$-category with weak biaction
that endows $\Fun(\K, \mM)$ with the diagonal weak biaction. The canonical equivalence
$ \Fun_{\mV^\ot \times \mW^\ot}(\K \times \mM^\circledast, \mM'^\circledast) \simeq \Fun_{\mV^\ot \times \mW^\ot}(\mM^\circledast, (\mM'^\circledast)^\K)$
restricts to an equivalence
$$\LaxLinFun_{\mV,\mW}(\K \times \mM,\mM') \simeq \LaxLinFun_{\mV,\mW}(\mM,\mM'^\K).$$
\end{notation}

\begin{definition}
Let $\phi: \mM^\circledast \to \mV^\ot \times \mW^\ot, \phi':\mM'^\circledast \to \mV^\ot \times \mW^\ot $ be $\infty$-categories with biaction.

We say that a $\mV,\mW$-linear functor $\mM^\circledast \to \mM'^\circledast $ preserves small colimits (admits a right adjoint) if the underlying functor $\mM \to \mM'$ preserves small colimits (admits a right adjoint).

\end{definition}

\begin{remark}\label{relright}
Let $\phi: \mM^\circledast \to \mV^\ot \times \mW^\ot, \phi':\mM'^\circledast \to \mV^\ot \times \mW^\ot $ be $\infty$-categories with biaction.
\cite[Proposition 7.3.2.6.]{lurie.higheralgebra} implies the following: a $\mV,\mW$-linear functor $\mM^\circledast \to \mM'^\circledast $ admits a right adjoint
if and only if it admits a right adjoint relative to $\mV^\ot \times \mW^\ot,$
where the right adjoint relative to $\mV^\ot \times \mW^\ot$ is a lax $\mV,\mW$-linear functor $\mM'^\circledast \to \mM^\circledast.$

\end{remark}

\begin{notation}
For $\infty$-categories $\phi:\mM^\circledast \to \mV^\ot, \phi':\mM'^\circledast \to \mV^\ot$ with weak left action let
\begin{itemize}
\item $ \LaxLinFun_\mV(\mM,\mM') \subset \Fun_{\mV^\ot}(\mM^\circledast,\mM'^\circledast)$ the full subcategory of lax $\mV$-linear functors.
\vspace{1mm}
\item $\LinFun^\L_\mV(\mM,\mM') \subset \LinFun_\mV(\mM,\mM') \subset \LaxLinFun_{\mV}(\mM,\mM')$ the full subcategories of $\mV$-linear functors (preserving small colimits) in case where $\phi$, $\phi'$ are $\infty$-categories with left action.

\end{itemize} 
\end{notation}

Remark \ref{modules} and \cite[Corollary 4.2.4.7.]{lurie.higheralgebra} imply the following remark:
\begin{remark}\label{lefta}
Let $\phi:\mM^\circledast \to \mV^\ot$ be an $\infty$-category with left action.
The functor $\LinFun_\mV(\mV,\mM) \to \mM$ evaluating at the tensor unit of $\mV$
is an equivalence whose inverse lifts the functor
$\mM \to \Fun(\mV,\mM)$ adjoint to the functor $\ot: \mV \times \mM \to \mM.$	
\end{remark}

\begin{lemma}\label{fghhnml}

Let $\mV^\ot \to \Ass, \mW^\ot \to \Ass$ be $\infty$-operads and $\mM^\circledast \to \mV^\ot \times \mW^\ot,\phi: \mN^\circledast \to \mV^\ot \times \mW^\ot$ be $\infty$-categories with weak biaction such that $\phi$ is a locally cocartesian fibration.
\begin{enumerate}
\item If $\mN$ has small colimits and for every $\V \in \mV, \W \in \mW$ the functor $\V \ot (-) \ot \W: \mN \to \mN$ preserves small colimits, $ \LaxLinFun_{\mV,\mW}(\mM, \mN) $ has small colimits and the forgetful functor $$ \LaxLinFun_{\mV,\mW}(\mM, \mN) \to \Fun(\mM, \mN)$$ preserves small colimits.
\item If for every $\V \in \mV, \W \in \mW$ the functor $\V \ot (-) \ot \W: \mN \to \mN$ is accessible, $\mM^\circledast$ is small and $\mN$ is accessible, the $\infty$-category $ \LaxLinFun_{\mV,\mW}(\mM, \mN)$ is accessible.

\item If $ \mM^\circledast \to \mV^\ot \times \mW^\ot $ is an $\infty$-category with biaction, the same holds for $ \LinFun_{\mV,\mW}(\mM, \mN) $.

\end{enumerate}

\end{lemma}

\begin{proof}

By definition $\LinFun_{\mV,\mW}(\mM, \mN),\ (\LaxLinFun_{\mV,\mW}(\mM, \mN)) \subset \Fun_{\mV^\ot \times \mW^\ot}(\mM^\circledast, \mN^\circledast)$ are the full subcategories of functors $\mM^\circledast \to \mN^\circledast$
over $\mV^\ot \times \mW^\ot$ preserving cocartesian lifts of morphisms of $\Ass \times \Ass$ (whose first component preserves the maximum and whose second component preserves the minimum).
By assumption the fibers of $\phi$ admit small colimits (are accessible)
and the fiber transports preserve small colimits (are accessible).
Hence (1), (2), (3) follow from \cite[Proposition 5.4.7.11.]{lurie.HTT}.
\end{proof}	

\begin{remark}
Let $\mV^\ot \to \Ass,\mW^\ot \to \Ass$ be generalized $\infty$-operads.
The functor $$\Cat_{\infty / \mV^\ot} \times \Cat_{\infty / \mW^\ot} \to \Cat_{\infty / \mV^\ot \times \mW^\ot}, \ (\mM^\circledast \to \mV^\ot, \mN^\circledast \to \mW^\ot) \mapsto \mM^\circledast \times \mN^\circledast \to \mV^\ot \times \mW^\ot$$
restricts to a functor
$$\omega\LMod_{\mV}\times \omega\RMod_{\mW} \to \omega\BMod_{\mV,\mW}. $$
If $\mV^\ot \to \Ass,\mW^\ot \to \Ass$ are monoidal $\infty$-categories, respectively,
this functor restricts to functors $$ \LMod_{\mV}\times \omega\RMod_{\mW} \to \LMod_{\mV}(\omega\RMod_{\mW}), \ \omega\LMod_{\mV}\times \RMod_{\mW} \to \RMod_{\mW}(\omega\LMod_{\mV}), $$$$ \LMod_{\mV}\times \RMod_{\mW} \to \BMod_{\mV,\mW}.$$

\end{remark}

We have the following proposition:
\begin{proposition}

For any generalized $\infty$-operads $\mV^\ot \to \Ass,\mW^\ot \to \Ass$ the functor
$$\omega\LMod_{\mV}\times \omega\RMod_{\mW} \to \omega\BMod_{\mV,\mW}$$
admits componentwise right adjoints.

If $\mV^\ot \to \Ass,\mW^\ot \to \Ass$ are monoidal $\infty$-categories, respectively,
the functors $$ \LMod_{\mV}\times \omega\RMod_{\mW} \to \LMod_{\mV}(\omega\RMod_{\mW}), \ \omega\LMod_{\mV}\times \RMod_{\mW} \to \RMod_{\mW}(\omega\LMod_{\mV}), $$$$ \LMod_{\mV}\times \RMod_{\mW} \to \BMod_{\mV,\mW}$$ admit componentwise right adjoints.

\end{proposition}

To prove this proposition we fix the following terminology:
\begin{definition}\label{flat}
We call a functor $\mC \to \rS$ flat if 
$ \mC \times_\rS (-): \Cat_{\infty / \rS} \to \Cat_{\infty / \mC}$ 
admits a right adjoint. 
\end{definition}

\begin{remark}
Evidently, the opposite of a flat functor is flat. By \cite[Example B.3.11.]{lurie.higheralgebra} every cocartesian fibration is flat and so dually every cartesian fibration is flat.
\end{remark}

\begin{remark}
Let $\mC \to \rS, \rS' \to \rS$ be functors and $\mC':= \rS' \times_\rS \mC.$
By the pasting law the functor $ \Cat_{\infty / \rS'} \xrightarrow{\mC' \times_{\rS'}(-)} \Cat_{\infty / \mC'} \to  \Cat_{\infty / \mC}$ factors as 
$ \Cat_{\infty / \rS'} \to \Cat_{\infty / \rS} \xrightarrow{\mC \times_\rS (-)}  \Cat_{\infty / \mC}.$

So the pullback of a flat functor is flat.

\end{remark}

\begin{notation}
For functors $\mC \to \T, \T \to \rS$ such that the composition $\mC \to \T \to \rS$ is flat, the functor 
$(-)\times_\rS \mC:\Cat_{\infty / \rS} \xrightarrow{ } \Cat_{\infty / \mC} \to \Cat_{\infty / \T}$ admits a right adjoint that we denoty by
$ \Fun_\T^{\rS}(\mC,-).$ 

If $\rS, \T $ are contractible, we drop $\rS$, $\T$ from the notation.
\end{notation}

\begin{remark}\label{pullbor}
For any functors $\mD \to \T$ and $\rS' \to \rS$ there is a canonical equivalence 
$ \rS' \times_\rS  \Fun_\T^{\rS}(\mC,\mD) \simeq \Fun_{\rS' \times_\rS \T}^{\rS'}( \rS' \times_\rS \mC, \rS' \times_\rS \mD)$ specifying the fibers of the functor $\Fun_\T^{\rS}(\mC,\mD) \to \rS.$
\end{remark}

\begin{lemma}\label{flafla}
Let $\alpha: \T \to \rS$ be a cocartesian fibration relative to a subcategory $\delta \subset \rS$ and $\delta_\T\subset \T$ the subcategory of $\alpha$-cocartesian lifts of morphisms of $\delta.$ 
For any cocartesian fibration $\mC \to \T$ relative to $\delta_\T$ and cartesian fibration
$\mD \to \T$ relative to $\delta_\T$ the functor $\psi: \Fun_{\T}^{\rS}(\mC,\mD) \to \rS$ is a cartesian fibration relative to $\delta$. 

A morphism of $\Fun_{\T}^{\rS}(\mC,\mD)$ lying over one of $\delta$ corresponding to a functor $ \gamma: [1] \times_\rS \mC \to [1] \times_\rS \mD$ over $[1]\times_\rS \T$ is $\psi$-cartesian if and only if $\gamma$ sends cocartesian lifts of the map $0\to 1$ to cartesian lifts.

\end{lemma}
\begin{proof}
This follows from \cite[Theorem B.4.2.]{lurie.higheralgebra}.	
\end{proof}

\begin{lemma}\label{innerho}
Let $\mM^\circledast \to \mV^\ot, \mO^\circledast \to \mW^\ot, \mN^\circledast \to \mV^\ot \times \mW^\ot$ be $\infty$-categories with weak left, right, biaction. 

\begin{enumerate}
\item The functor $\alpha: \Fun^{\mW^\ot}_{\mV^\ot \times \mW^\ot}(\mM^\circledast \times \mW^\ot, \mN^\circledast) \to \mW^\ot $ is an $\infty$-category with weak right action and for any $\X \in \mM^\circledast$ lying over $\Y \in \mV^\ot$ the canonical functor
$$ \kappa: \Fun^{\mW^\ot}_{\mV^\ot \times \mW^\ot}(\mM^\circledast \times \mW^\ot, \mN^\circledast) \to \mN^\circledast_\Y $$ preserves cocartesian lifts of inert morphisms of $\Ass$ that preserve the minimum.

\vspace{1mm}	
\item The functor $\beta: \Fun^{\mV^\ot}_{\mV^\ot \times \mW^\ot}(\mV^\ot \times \mO^\circledast, \mN^\circledast) \to \mV^\ot $ is an $\infty$-category with weak left action and for any $\X \in \mO^\circledast$ lying over $\Y \in \mW^\ot$ the canonical functor
$$ \tau: \Fun^{\mV^\ot}_{\mV^\ot \times \mW^\ot}(\mV^\ot \times \mO^\circledast, \mN^\circledast) \to \mN^\circledast_\Y $$ preserves cocartesian lifts of
inert morphisms of $\Ass$ that preserve the maximum.

\vspace{1mm}

\item If $\mN^\circledast \to\mV^\ot \times \mW^\ot$ exhibits $\mN$ as right tensored over $\mW$, then $\alpha$ is a right tensored $\infty$-category and for any $\X \in \mM^\circledast$ lying over $\Y \in \mV^\ot$ the map $\kappa$ preserves cocartesian morphisms over $\Ass.$

\vspace{1mm}

\item If $\mN^\circledast \to\mV^\ot \times \mW^\ot$ exhibits $\mN$ as left tensored over $\mV$, then $\beta$ is a left tensored $\infty$-category and for any $\X \in \mO^\circledast$ lying over $\Y \in \mW^\ot$ the functor $\tau$ preserves cocartesian morphisms over $\Ass.$

\end{enumerate}
\end{lemma}

\begin{proof}
We prove (1) and (3), statements (2) and (4) are similar.
By \cite[Theorem B.4.2.]{lurie.higheralgebra} (see Lemma \ref{flafla}) the functors $\alpha$, $\kappa$ preserve cocartesian lifts of inert morphisms of $\Ass$ that preserve the minimum.

Next we prove that $\alpha$ satisfies (1) in Definition \ref{wla}.
This follows from the canonical equivalence 
$$\Fun^{\mW^\ot}_{\mV^\ot \times \mW^\ot}(\mM^\circledast \times \mW^\ot, \mN^\circledast)_{[\n]} \simeq \Fun^{\mW_{[\n]}^\ot}_{\mV^\ot \times \mW^\ot_{[\n]}}(\mM^\circledast \times \mW_{[\n]}^\ot, \mN_{[\n]}^\circledast) \simeq $$$$\Fun^{\mW_{[\n]}^\ot}_{\mV^\ot \times \mW^\ot_{[\n]}}(\mM^\circledast \times \mW_{[\n]}^\ot, \mN_{[0]}^\circledast\times_{\mW^\ot_{[0]}} \mW^\ot_{[\n]}) \simeq \Fun^{\mW_{[0]}^\ot}_{\mV^\ot \times \mW^\ot_{[0]}}(\mM^\circledast \times \mW_{[0]}^\ot, \mN_{[0]}^\circledast)\times_{\mW^\ot_{[0]}} \mW^\ot_{[\n]}.$$

By \cite[Theorem B.4.2.]{lurie.higheralgebra}
under the assumptions of (3) the functors $\alpha,\kappa$ are maps of cocartesian fibrations over $\Ass$.
So by Remark \ref{remuu} the functor $\alpha$ is a right tensored $\infty$-category. This proves (3). 

\vspace{1mm}
We complete the proof by showing (1). The lax $\mV,\mW$-linear embedding $\mN^\circledast \subset \mQ^\circledast:=\mV^\ot 
\times_{\Env(\mV)^\ot} \B\Env(\mN)^\circledast \times_{\Env(\mW)^\ot} \mW^\ot$ yields an embedding
$\Fun^{\mW^\ot}_{\mV^\ot \times \mW^\ot}(\mM^\circledast \times \mW^\ot, \mN^\circledast) \subset \Fun^{\mW^\ot}_{\mV^\ot \times \mW^\ot}(\mM^\circledast \times \mW^\ot, \mQ^\circledast) $ over $\mW^\ot$ that preserves cocartesian lifts of inert morphisms of $\Ass$ that preserve the minimum.

Since $\mV^\ot \times_{\Env(\mV)^\ot} \B\Env(\mN)^\circledast \to \Env(\mW)^\ot$
exhibits $ \B\Env(\mN)$ as right tensored over $\Env(\mW)$,
by (3) $\Fun^{\Env(\mW)^\ot}_{\mV^\ot \times \Env(\mW)^\ot}(\mM^\circledast \times \Env(\mW)^\ot, \mV^\ot \times_{\Env(\mV)^\ot} \B\Env(\mN)^\circledast) \to \Env(\mW)^\ot$ 
is a right tensored $\infty$-category.
Thus $$ \Fun^{\mW^\ot}_{\mV^\ot \times \mW^\ot}(\mM^\circledast \times \mW^\ot, \mQ^\circledast) \simeq \Fun^{\Env(\mW)^\ot}_{\mV^\ot \times \Env(\mW)^\ot}(\mM^\circledast \times \Env(\mW)^\ot, \mV^\ot \times_{\Env(\mV)^\ot} \B\Env(\mN)^\circledast)\times_{\Env(\mW)^\ot} \mW^\ot \to \mW^\ot$$
is a weakly right tensored $\infty$-category.
Remark \ref{remosq} implies that $\alpha: \Fun^{\mW^\ot}_{\mV^\ot \times \mW^\ot}(\mM^\circledast \times \mW^\ot, \mN^\circledast) \to \mW^\ot $
is a weakly right tensored $\infty$-category as $\alpha$ satisfies (1) of Definition \ref{wla}. This proves (1).
 
\end{proof}

By Lemma \ref{innerho} any cocartesian lift $\Y \to \Y'$ in $ \mW^\ot$ of the inert map $[0] \simeq \{0\} \subset [\n]$ yields an equivalence
$$\Fun^{\mW^\ot}_{\mV^\ot \times \mW^\ot}(\mM^\circledast \times \mW^\ot, \mN^\circledast)_{\Y} \simeq \Fun^{\mW^\ot}_{\mV^\ot \times \mW^\ot}(\mM^\circledast \times \mW^\ot, \mN^\circledast)_{\Y'} \simeq \Fun_{\mV^\ot}(\mM^\circledast, \mN^\circledast_{\Y'}).$$

\begin{notation}\label{bbbb}
Let $\mM^\circledast \to \mV^\ot, \mO^\circledast \to \mW^\ot, \mN^\circledast \to \mV^\ot \times \mW^\ot$ be $\infty$-categories with weak left, right, biaction, respectively.
Let $$\LaxLinFun_\mV(\mM,\mN)^\circledast \subset \Fun^{\mW^\ot}_{\mV^\ot \times \mW^\ot}(\mM^\circledast \times \mW^\ot, \mN^\circledast) $$ be the full subcategory spanned by the objects of 
$$\Fun^{\mW^\ot}_{\mV^\ot \times \mW^\ot}(\mM^\circledast \times \mW^\ot, \mN^\circledast)_{\Y} \simeq \Fun^{\mW^\ot}_{\mV^\ot \times \mW^\ot}(\mM^\circledast \times \mW^\ot, \mN^\circledast)_{\Y'} \simeq \Fun_{\mV^\ot}(\mM^\circledast, \mN^\circledast_{\Y'})$$
for some $\Y \in \mW^\ot$ that are lax $\mV$-linear functors $\mM^\circledast \to \mN^\circledast_{\Y'}.$
Similarly, we define $\LaxLinFun_\mW(\mO,\mN)^\circledast \subset \Fun^{\mV^\ot}_{\mV^\ot \times \mW^\ot}(\mV^\ot \times \mO^\circledast, \mN^\circledast) \to \mV^\ot.$
If $\mV^\ot= \emptyset^\ot$, we write $\Fun(\mM,\mN)^\circledast$ for $\LaxLinFun_\mV(\mM,\mN)^\circledast.$	
\end{notation}

\begin{notation}
Let $\mM^\circledast \to \mV^\ot$ be an $\infty$-category with left action and
$\mN^\circledast \to \mV^\ot \times \mW^\ot$ an $\infty$-category with biaction that exhibits $\mN$ as left tensored over $\mV.$
Let $$\LinFun_\mV(\mM,\mN)^\circledast \subset \LaxLinFun_\mV(\mM,\mN)^\circledast$$
be the full subcategory with weak right $\mW$-action spanned by the objects of 
$\LaxLinFun_\mV(\mM,\mN)^\circledast_{\Y}$
for some $\Y \in \mW^\ot$ that are $\mV$-linear functors $\mM^\circledast \to \mN^\circledast_{\Y'}.$
Similarly, we define $\LinFun_\mW(\mO,\mN)^\circledast \subset \LaxLinFun_\mW(\mO,\mN)^\circledast $ for any $\infty$-category $ \mO^\circledast \to \mW^\ot$ with right action and an $\infty$-category $ \mN^\circledast \to \mV^\ot \times \mW^\ot$ with biaction that exhibits $\mN$ as right tensored over $\mW.$

\end{notation}

Lemma \ref{innerho} immediately implies the following proposition:

\begin{proposition}\label{innerhost}
Let $\mM^\circledast \to \mV^\ot, \mO^\circledast \to \mW^\ot, \mN^\circledast \to \mV^\ot \times \mW^\ot$ be $\infty$-categories with weak left, right, biaction, respectively.

\begin{enumerate}
\item The functor $\LaxLinFun_\mV(\mM,\mN)^\circledast \to \mW^\ot $ is an $\infty$-category with weak right action.

\vspace{1mm}	
\item The functor $\LaxLinFun_\mW(\mO,\mN)^\circledast \to \mV^\ot $ is an $\infty$-category with weak left action.

\vspace{1mm}

\item If $\mM^\circledast \to \mV^\ot$ is an $\infty$-category with left action and
$\mN^\circledast \to \mV^\ot \times \mW^\ot$ exhibits $\mN$ as left tensored over $\mV,$
the functor $\LinFun_\mV(\mM,\mN)^\circledast \to \mW^\ot $ is an $\infty$-category with weak right action.

\vspace{1mm}	
\item If $\mO^\circledast \to \mW^\ot$ is an $\infty$-category with right action and
$\mN^\circledast \to \mV^\ot \times \mW^\ot$ exhibits $\mN$ as right tensored over $\mW,$
the functor $\LinFun_\mW(\mO,\mN)^\circledast \to \mV^\ot $ is an $\infty$-category with weak left action.

\vspace{1mm}

\item If the functor $\mN^\circledast \to\mW^\ot$ exhibits $\mN$ as right tensored over $\mW$, then
$\LaxLinFun_\mV(\mM,\mN)^\circledast \to \mW^\ot$ is an $\infty$-category with right action
and for any $\X \in \mM^\circledast$ lying over $\Y \in \mV^\ot$ the functor
$\LaxLinFun_\mV(\mM,\mN)^\circledast\to \mN^\circledast_\Y $ preserves cocartesian morphisms over $\Ass,$ and similar for $\LaxLinFun_\mW(\mO,\mN)^\circledast \to \mV^\ot.$

\vspace{1mm}
\item If $\mM^\circledast \to \mV^\ot, \mN^\circledast \to \mV^\ot \times \mW^\ot$ are $\infty$-categories with left, biaction, respectively,
the functor $\LinFun_\mV(\mM,\mN)^\circledast \to \mW^\ot $ is an $\infty$-category with right action and for any $\X \in \mM^\circledast$ lying over $\Y \in \mV^\ot$ the functor
$\LinFun_\mV(\mM,\mN)^\circledast\to \mN^\circledast_\Y $ preserves cocartesian morphisms over $\Ass,$ and similar for $\LinFun_\mW(\mO,\mN)^\circledast \to \mV^\ot.$

\end{enumerate}
\end{proposition}

Proposition \ref{innerhost} and Lemma \ref{fghhnml} imply the following corollary:

\begin{corollary}\label{comppres}
Let $\mM^\circledast \to \mV^\ot$ be a weakly left tensored $\infty$-category and $\mN^\circledast \to \mV^\ot \times \mW^\ot$ a bitensored $\infty$-category such that $\mN$ has small colimits and for every $\V \in \mV$ the functor $\V \ot (-): \mN \to \mN$ preserves small colimits.

Then $\LaxLinFun_\mV(\mM,\mN)$ has small colimits, $\LaxLinFun_\mV(\mM,\mN)^\circledast \to \mW^\ot$ is an $\infty$-category with right action and $\LaxLinFun_\mV(\mM,\mN)^\circledast \to \mN^\circledast_{[1]} $ 
is $\mW$-linear and preserves small colimits.
So if $ \mN^\circledast_{[1]} \to \mW^\ot$ is compatible with small colimits, $\LaxLinFun_\mV(\mM,\mN)^\circledast \to \mW^\ot$ is compatible with small colimits.
\end{corollary}
Remark \ref{flafla} implies the following propositions:

\begin{proposition}\label{lehmmm} 
Let $\mM^\circledast \to \mV^\ot, \mO^\circledast \to \mW^\ot, \mN^\circledast \to \mV^\ot \times \mW^\ot$ be $\infty$-categories with weak left, right, biaction, respectively.
The canonical equivalence
$$ \Fun_{\mW^\ot}(\mO^\circledast, \Fun^{\mW^\ot}_{\mV^\ot \times \mW^\ot}(\mM^\circledast \times \mW^\ot, \mN^\circledast)) \simeq \Fun_{\mV^\ot \times \mW^\ot}(\mM^\circledast \times \mO^\ot, \mN^\circledast)$$
restricts to an equivalence
\begin{equation}\label{equino}
\LaxLinFun_{\mW}(\mO, \LaxLinFun_\mV(\mM,\mN)) \simeq \LaxLinFun_{\mV, \mW}(\mM \times \mO, \mN).
\end{equation} 

The canonical equivalence
$$ \Fun_{\mV^\ot}(\mM^\circledast,\Fun^{\mV^\ot}_{\mV^\ot \times \mW^\ot}(\mV^\ot \times \mO^\circledast, \mN^\circledast)) \simeq \Fun_{\mV^\ot \times \mW^\ot}(\mM^\circledast \times \mO^\ot, \mN^\circledast)$$
restricts to an equivalence
$$ \LaxLinFun_{\mV}(\mM, \LaxLinFun_\mW(\mO,\mN)) \simeq \LaxLinFun_{\mV,\mW}(\mM \times \mO, \mN).$$
\end{proposition}
\begin{proposition}\label{partial}
Let $\mM^\circledast \to \mV^\ot, \mO^\circledast \to \mW^\ot, \mN^\circledast \to \mV^\ot \times \mW^\ot$ be $\infty$-categories with weak left, right, biaction, respectively.

\begin{enumerate}
\item If $ \mO^\circledast \to \mW^\ot$ is an $\infty$-category with right action and $\mN^\circledast \to \mV^\ot \times \mW^\ot$ exhibits $\mN$ as right tensored over $\mW$, equivalence (\ref{equino}) restricts to an embedding 
$$ \LinFun_{\mW}(\mO, \LaxLinFun_\mV(\mM,\mN)) \hookrightarrow \LaxLinFun_{\mV, \mW}(\mM \times \mO, \mN) $$
with essential image the $\mW$-linear functors. 

\vspace{1mm}
\item If $ \mM^\circledast \to \mV^\ot$ is an $\infty$-category with left action and $\mN^\circledast \to \mV^\ot \times \mW^\ot$ exhibits $\mN$ as left tensored over $\mV$, equivalence (\ref{equino}) restricts to an embedding 
$$ \LaxLinFun_{\mW}(\mO, \LinFun_\mV(\mM,\mN)) \hookrightarrow \LaxLinFun_{\mV, \mW}(\mM \times \mO, \mN) $$
with essential image the $\mV$-linear functors. 

\vspace{1mm}

\item If $\mM^\circledast \to \mV^\ot, \mO^\circledast \to \mW^\ot, \mN^\circledast \to \mV^\ot \times \mW^\ot$ are $\infty$-categories with left, right, biaction, respectively,
equivalence (\ref{equino}) restricts to an equivalence $$ \LinFun_{\mW}(\mO, \LinFun_\mV(\mM,\mN)) \simeq \LinFun_{\mV, \mW}(\mM \times \mO, \mN).$$
\end{enumerate}
\end{proposition}

\vspace{0,1mm}

\begin{remark}\label{llll}
Let $\mM^\circledast \to \mV^\ot$ be an $\infty$-category with left action
and $\mW^\ot \to \Ass$ a monoidal $\infty$-category. 
By Remark \ref{enrrr} there is an adjunction $\mM^\circledast \times (-) : \Cat_\infty \rightleftarrows \omega\LMod_\mV: \LaxLinFun_\mV(\mM,-) $,
where the left adjoint is $\Cat_\infty$-linear as a consequence of Lemma \ref{2enr} and Remark \ref{lefta}.

Hence this adjunction induces an adjunction
$$\mM^\circledast \times (-) :\RMod_\mW \simeq \RMod_\mW(\Cat_\infty) \rightleftarrows \RMod_\mW(\omega\LMod_\mV): \LaxLinFun_\mV(\mM,-).$$
The left adjoint is the restriction of the functor
$\mM^\circledast \times (-): \Cat_{\infty/\mW^\ot} \to \Cat_{\infty/\mV^\ot \times \mW^\ot}$. So by adjointness the right adjoint identifies with the functor
of \ref{partial} (1) with the same name.

\end{remark}

The following lemma follows immediately from the universal property shown in Proposition \ref{lehmmm}:

\begin{lemma}\label{leman}
Let $\mM^\circledast \to \mV^\ot, \mM'^\circledast \to \mV^\ot, \mN^\circledast \to \mV^\ot \times \mW^\ot$ be $\infty$-categories with weak left, biaction, respectively.
\begin{enumerate}

\item For any map $\mW'^\ot \to \mW^\ot$ of generalized $\infty$-operads
there is a canonical equivalence
$$ \mW'^\ot \times_{\mW^\ot} \LaxLinFun_\mV(\mM,\mN)^\circledast \simeq \LaxLinFun_\mV(\mM,\mW' \times_\mW \mN)^\circledast$$
of $\infty$-categories weakly right tensored over $\mW'$.
\vspace{1mm}	

\item For any map $\mV'^\ot \to \mV^\ot $ of generalized $\infty$-operads there is a canonical map
$$ \LaxLinFun_\mV(\mM,\mN)^\circledast \to
\LaxLinFun_{\mV'}(\mV' \times_{\mV}\mM, \mV' \times_{\mV} \mN)^\circledast$$
of $\infty$-categories weakly right tensored over $\mW.$

\vspace{1mm}	
\item There is a canonical map 
$$ \LaxLinFun_\mV(\mM,\mN)^\circledast \to 
\Fun(\LaxLinFun_\mV(\mM', \mM),\LaxLinFun_\mV(\mM', \mN))^\circledast $$
of $\infty$-categories weakly right tensored over $\mW$.
\end{enumerate}

\end{lemma}

\begin{remark}\label{motis}
For every weakly bitensored $\infty$-category $\mN^\circledast \to \mV^\ot \times \mW^\ot $, weakly left tensored $\infty$-category $\mM^\circledast \to \mV^\ot$ and $\A \in \Alg(\mV),\B \in \Alg(\mW)$	we defined $\infty$-categories
$\BMod_{\A,\B}(\mN)$ (Notation \ref{not1}) and $ \LMod_\A(\mM)$ (Notation \ref{not2}).
There are canonical equivalences
$$\BMod_{\A,\B}(\mN) \simeq \LaxLinFun_{\Ass_{[0]} \times \Ass_{[0]}}(\Ass_{[0]} \times \Ass_{[0]}, \mN_{\A,\B}), $$$$\LMod_\A(\mM) \simeq \LaxLinFun_{\Ass_{[0]}}(\Ass_{[0]}, \mM_\A),$$ 
where $\mN_{\A,\B}^\circledast:= (\Ass \times \Ass) \times_{(\mV^\ot \times \mW^\ot)} \mN^\circledast, \mM_\A^\circledast:= \Ass \times_{\mV^\ot} \mM^\circledast$ are the pullbacks along $\A \times \B$, $\A.$\end{remark}
This motivates the following notation:
\begin{notation}\label{rightact}
Let $\mM^\circledast \to \mV^\ot \times \mW^\ot $ be an $\infty$-category with weak biaction and $\A\in \Alg(\mV).$
We set $$\LMod_\A(\mM)^\circledast:= \LaxLinFun_{\Ass_{[0]}}(\Ass_{[0]}, \mM_\A)^\circledast \in \omega\BMod_{\emptyset,\mW} \simeq \omega\RMod_\mW.$$
Similarly, for any $\B \in \Alg(\mW)$ we define 
$$\RMod_\B(\mM)^\circledast:= \LaxLinFun_{\Ass_{[0]}}(\Ass_{[0]}, \mM_\B)^\circledast \in \omega\BMod_{\mV,\emptyset} \simeq \omega\LMod_\mV.$$
\end{notation}

\begin{remark}\label{coala}
For any $\infty$-category $\mM^\circledast \to \mV^\ot \times \mW^\ot $ with weak biaction, $\infty$-category $\mN^\circledast \to \mW^\ot$ with weak right action
and $\A\in \Alg(\mV)$ by definition there is a canonical equivalence
$$ \LaxLinFun_{\mW}(\mN,\LMod_\A(\mM))\simeq \LaxLinFun_{\Ass, \mW}(\Ass_{[0]} \times \mN, \mM_\A).$$
Similarly, for any $\infty$-category $\mN^\circledast \to \mV^\ot$ with weak left action
and $\B \in \Alg(\mW)$ there is an equivalence
$$ \LaxLinFun_{\mV}(\mN,\RMod_\B(\mM))\simeq \LaxLinFun_{\mV,\Ass}(\mN \times \Ass_{[0]}, \mM_\B).$$
\end{remark}
Remark \ref{coala} implies the following one:
\begin{remark}
Let $\mM^\circledast \to \mV^\ot \times \mW^\ot $ be a weakly bitensored $\infty$-category and $\A\in \Alg(\mV), \B \in \Alg(\mW)$.
There is a canonical equivalence
$$\LMod_\A(\mM)^\circledast \simeq (\RMod_{\A^\rev}(\mM^\rev)^\circledast)^\rev$$ of $\infty$-categories weakly right tensored over $\mW$.

\end{remark}

Lemma \ref{leman} implies the following remark:
\begin{remark}\label{jbfigh}
Let $\mM^\circledast \to \mV^\ot$ be a weakly left tensored $\infty$-category, $\mN^\circledast \to \mV^\ot \times \mW^\ot$ a weakly bitensored $\infty$-category and $\A\in \Alg(\mV).$ There is a canonical map 
$$ \LaxLinFun_\mV(\mM,\mN)^\circledast \to
\Fun(\LMod_\A(\mM), \LMod_\A(\mN))^\circledast$$
of $\infty$-categories weakly right tensored over $\mW.$
\end{remark}

\begin{proposition}\label{compright}
Let $\mM^\circledast \to \mV^\ot \times \mW^\ot $ be a weakly bitensored $\infty$-category and $\A \in \Alg(\mV),\B \in \Alg(\mW).$
There is a canonical equivalence
$$\RMod_\B(\LMod_\A(\mM)) \simeq \BMod_{\A,\B}(\mM)$$
compatible with the forgetful functors to $\mM.$

\end{proposition}

\begin{proof}
Lemma \ref{leman} (1) provides a canonical equivalence
$\LMod_\A(\mM)^\circledast_\B \simeq \LMod_\A(\mM_\B)^\circledast$
of $\infty$-categories weakly right tensored over $[0].$
So we obtain an equivalence
$$\RMod_\B(\LMod_\A(\mM))= \LaxLinFun_{\Ass_{[0]}}(\Ass_{[0]}, \LMod_\A(\mM)_\B) \simeq \LaxLinFun_{\Ass_{[0]}}(\Ass_{[0]},\LMod_\A(\mM_\B)). $$
Remarks \ref{coala} and \ref{motis} provide a canonical equivalence
$$\LaxLinFun_{\Ass_{[0]}}(\Ass_{[0]},\LMod_\A(\mM_\B)) \simeq \LaxLinFun_{\Ass_{[0]}, \Ass_{[0]}}(\Ass_{[0]} \times \Ass_{[0]}, \mM_{\A,\B}) \simeq \BMod_{\A,\B}(\mM).$$

\end{proof}

\begin{lemma}\label{monart}
Let $\mM^\circledast \to \mV^\ot \times \mW^\ot$ be an $\infty$-category with biaction and $\A, \B$ associative algebras in $\mV,\mW ,$ respectively.
The forgetful functor $\nu_\mM: \BMod_{\A,\B}(\mM) \to \mM$ is monadic,
where the left adjoint sends $\X \in \mM$ to $\A \ot \X \ot \B$.
\end{lemma}

\begin{proof}
By Proposition \ref{compright} there is a canonical equivalence $\BMod_{\A,\B}(\mM) \simeq \LMod_\A(\RMod_\B(\mM))$ compatible with the forgetful functors to $\mM.$
This implies the existence and description of the left adjoint.
By the $\infty$-categorical Barr-Beck theorem \cite[Theorem 4.7.3.5]{lurie.higheralgebra} it remains to see that $\BMod_{\A,\B}(\mM) $ admits colimits of $\nu_\mM$-split simplicial objects and that these colimits are preserved by $\nu_\mM.$
For that we can reduce to the case, where $\mM^\circledast \to \mV^\ot \times \mW^\ot$ is an $\infty$-category with biaction compatible with small colimits
since $\nu_\mM$ is the pullback of $\nu_{\mP(\mM)}$. 
In this case by Lemma \ref{fghhnml} the $\infty$-category $ \BMod_{\A,\B}(\mM) \simeq \LMod_\A(\RMod_\B(\mM))$ admits small colimits preserved by $\nu_\mM.$
\end{proof}

\begin{notation}\label{switchh}
For any generalized $\infty$-operads $\mV^\ot \to \Ass,\mW^\ot \to \Ass $ let  $$\theta_{\mV,\mW}: \mV^\ot \times_\Ass (\mW^\ot)^\rev \to \mV^\ot \times (\mW^\ot)^\rev \simeq \mV^\ot \times \mW^\ot $$
be the composition of the projection and the functor induced by the 
equivalence $(\mW^\ot)^\rev \simeq \mW^\ot$.

Taking pullback along the functor $\theta_{\mV,\mW}$ induces a functor
$(\theta_{\mV,\mW})^*: \omega\BMod_{\mV,\mW} \to \omega\LMod_{\mV \times \mW^\rev}.$ 
\end{notation}

\begin{proposition}\label{eqaybbn}
Let $\mM^\circledast \to \mV^\ot \times \mW^\ot$ be an $\infty$-category weakly bitensored over $\infty$-operads and $\A, \B$ be associative algebras in $\mV,\mW ,$ respectively.
The canonical functor $$\kappa_\mM: \BMod_{\A,\B}(\mM) \to \LMod_{\A,\B^\rev}(\theta_{\mV,\mW}^\ast(\mM))$$
is an equivalence.

\end{proposition}

\begin{proof}
Since the functor $\kappa_\mM$ is the pullback of the functor $\kappa_{\B\Env(\mM)}$, we can reduce to the case where $\mM^\circledast \to \mV^\ot \times \mW^\ot$ is a weakly bitensored $\infty$-category. In this case by Lemma \ref{monart} the forgetful functor
$\BMod_{\A,\B}(\mM) \to \mM$ is monadic and admits a left adjoint sending $\X \in \mM$ to $\A \ot \X \ot \B$.
The forgetful functor $\LMod_{\A,\B^\rev}(\theta_{\mV,\mW}^\ast(\mM)) \to \mM$ is monadic, too, where the left adjoint sends $\X \in \mM$ to $\A \ot \X \ot \B$.
The functor $\kappa_\mM$ commutes with the forgetful functors and preserves the left adjoints and so is an equivalence by \cite[Corollary 4.7.3.16.]{lurie.higheralgebra}.

\end{proof}

\subsection{Enveloping $\infty$-categories with left action}
\label{envob}
In this subsection we prove that any weakly left tensored $\infty$-category universally embeds into a left tensored $\infty$-category reducing questions about weakly left tensored $\infty$-categories to questions about left tensored $\infty$-categories.

\vspace{1mm}

For any cocartesian fibration $\mO \to \Ass$ relative to the collection of inert morphisms let $\Act(\mO) \subset \Fun([1],\mO)$ be the full subcategory of active morphisms.

\begin{notation}\label{ene}
Let $\mO \to \Ass$ be a cocartesian fibration relative to the collection of inert morphisms such that $\mO_{[0]}$ is a space and $\mC \to \mO$ a generalized $\mO$-operad.
We set $$\Env_\mO(\mC):= \Act(\mO) \times_{\Fun(\{0\}, \mO)} \mC.$$
\end{notation}

We call the functor \begin{equation}
\Env_\mO(\mC) \to \Act(\mO) \to \Fun(\{1\},\mO) 
\end{equation}
evaluating at the target the enveloping generalized $\mO$-monoidal $\infty$-category of $\mC \to \mO$.
If $\mC \to \mO$ is an $\mO$-operad, we call the enveloping generalized $\mO$-monoidal $\infty$-category of $\mC \to \mO$ the enveloping $\mO$-monoidal $\infty$-category. The next Proposition \ref{unb} justifies this terminology.

The diagonal embedding $\mO \subset \Act(\mO)$ 
induces an embedding $\mC \subset \Env_\mO(\mC)=\Act(\mO) \times_{\Fun(\{0\}, \mO)} \mC.$

\begin{proposition}\label{unb}
Let $\mO \to \Ass$ be a cocartesian fibration relative to the collection of inert morphisms such that $\mO_{[0]}$ is a space.

\begin{enumerate}
\item The functor $\Env_\mO(\mC) \to \Act(\mO) \to \Fun(\{1\}, \mO) $
is a generalized $\mO$-monoidal $\infty$-category and the embedding $\mC \subset \Env_\mO(\mC)$ is a map of generalized $\mO$-operads.

\vspace{1mm}
	
\item The functor $\Env_\mO(\mC) \to \mO$ is an $\mO$-monoidal $\infty$-category if $\mC \to \mO$ is an $\mO$-operad.

\vspace{1mm}

\item For any generalized $\mO$-monoidal $\infty$-category $\mD \to \mO$
the functor $\rho: \Alg_{\Env_\mO(\mC) / \mO}(\mD) \to \Alg_{\mC/ \mO}(\mD)$
admits a fully faithful left adjoint taking values in the full subcategory of $\mO$-monoidal functors.
So $\rho$ restricts to an equivalence
$\Fun_\mO^{\ot}(\Env_\mO(\mC), \mD) \simeq \Alg_{ \mC/ \mO}(\mD).$
\end{enumerate}

\end{proposition} 
 
\begin{proof}
(1) follows from Remark \ref{lababop}.
(2): For every $\Z \in \mO_{[0]} $ there is a canonical equivalence
$\Env_\mO(\mC)_\Z \simeq (\mO_{[0]})_{/\Z} \times_{\mO_{[0]}} \mC_{[0]} \simeq  \mC_{\Z}$ as $\mO_{[0]}$ is a space.
(3) is by Proposition \ref{Envelo} (1).

\end{proof}

We have the following important lemmas:
Lemma \ref{envvo} immediately implies the following:

\begin{lemma}\label{looocx}
Let $\mC \to \mO$ be a generalized $\mO$-monoidal $\infty$-category.
The embedding $\mC \subset \Env_\mO(\mC)$ admits a left adjoint relative to $\mO$
(an $\mO$-monoidal left adjoint).

\end{lemma}	

We apply the next lemma to the left and right embeddings $\Ass \subset \BM$ (Remark \ref{rightact}) and the induced embeddings $\Ass \subset \LM,\Ass \subset \RM$:

\begin{lemma}\label{gre}
Let $\theta: \mO' \to \mO$ be a map of cocartesian fibrations relative to the collection of inert morphisms of $\Ass.$ 
If $\theta$ is a right fibration relative to the collection of active morphisms,
the functor
$$\Env_{\mO'}(\mO' \times_\mO \mC) \to \mO' \times_\mO \Env_\mO(\mC)$$ 
over $\mO'$ is an equivalence.
\begin{proof}
The functor $\theta: \mO' \to \mO$ yields a functor $\rho: \Act(\mO') \to \Fun(\{1\}, \mO') \times_{\Fun(\{1\}, \mO)} \Act(\mO)$ that is an equivalence if $\theta$ is a right fibration relative to the collection of active morphisms.
In this case $\rho$ gives rise to an equivalence
$$\mO' \times_\mO \Env_\mO(\mC) = \mO' \times_{\Fun(\{1\}, \mO)} (\Act(\mO) 
\times_{\Fun(\{0\}, \mO)} \mC) \simeq \Act(\mO') \times_{\Fun(\{0\}, \mO')} (\mO' \times_\mO \mC)= \Env_{\mO'}(\mO' \times_\mO \mC)$$
inverse to the canonical functor.
\end{proof}
\end{lemma}

\begin{notation}
Let $\mV^\ot \to \Ass$ be an $\infty$-operad. We write $\Env(\mV)^\ot \to \Ass$ for the enveloping monoidal $\infty$-category.
	
\end{notation}

\begin{notation}
Let $\text{Min},\text{Max} \subset \Fun([1],\Ass)$ be the full subcategories of morphisms
preserving the minimum (maximum).
\end{notation}
\begin{remark}\label{rqqqyzz}
Every morphism in $\Ass$ uniquely factors as an inert morphism followed by an active morphism, which is a morphism preserving the minimum and maximum.
Similarly, every morphism in $\Ass$ uniquely factors as an inert morphism preserving the maximum (minimum) followed by a morphism preserving the minimum (maximum).
By \cite[Lemma 5.2.8.19.]{lurie.HTT} this guarantees that the embeddings $\Act:=\Act(\Ass), \text{Min},\text{Max} \subset \Fun([1],\Ass)$ admit left adjoints, where a morphism in $\Fun([1],\Ass) $ with local target is a local equivalence
if and only if its image in $\Ass$ under evaluation at the source is inert,
is inert and preserves the maximum, is inert and preserves the minimum, respectively, and its image in $\Ass$ under evaluation at the target is an equivalence.
\end{remark}

\begin{notation}
Let $\mM^\circledast \to \mV^\ot$ be a weakly left tensored $\infty$-category, $\mN^\circledast \to \mW^\ot$ a weakly right tensored $\infty$-category and $\mO^\circledast \to \mV^\ot \times \mW^\ot$ a weakly bitensored $\infty$-category. We set	
\begin{itemize}	
\item $\L\Env(\mM)^\circledast:=\text{Min} \times_{\Fun(\{0\},\Ass)} \mM^\circledast. $ 

\vspace{1mm}
\item $\R\Env(\mN)^\circledast:=\text{Max} \times_{\Fun(\{0\},\Ass)} \mN^\circledast. $
		
\vspace{1mm}
\item $\B\Env(\mO)^\circledast:= (\text{Min} \times \text{Max}) \times_{\Fun(\{0\},\Ass \times \Ass)} \mO^\circledast.$ 
\end{itemize}
\end{notation}

\begin{remark}
Remark \ref{rqqqyzz} implies that for any weakly left tensored $\infty$-category 
$\mM^\circledast \to \mV^\ot$ the embedding $$\text{Min} \times_{\Fun(\{0\},\Ass)} \mM^\circledast\subset \Fun([1],\Ass) \times_{\Fun(\{0\},\Ass)} \mM^\circledast$$
admits a left adjoint relative to $\Fun(\{1\},\Ass)$.
This guarantees that the functor $$\L\Env(\mM)^\circledast \to\Fun(\{1\},\Ass)$$ is a cocartesian fibration. Remark \ref{rqqqyzz} implies that for any weakly bitensored $\infty$-category 
$\mM^\circledast \to \mV^\ot \times \mW^\ot$ the embedding $$(\text{Min} \times \text{Max}) \times_{\Fun(\{0\},\Ass \times \Ass)} \mM^\circledast\subset \Fun([1],\Ass \times \Ass) \times_{\Fun(\{0\},\Ass \times \Ass)} \mM^\circledast$$
admits a left adjoint relative to $\Fun(\{1\},\Ass \times \Ass)$.
This guarantees that the functor $$\B\Env(\mM)^\circledast \to\Fun(\{1\},\Ass \times \Ass)$$ is a cocartesian fibration. 
\end{remark}

\begin{remark}
Let $\mM^\circledast \to \mV^\ot$ be a weakly left tensored $\infty$-category.
There is a canonical map $$\L\Env(\mM)^\circledast = \mathrm{Min} \times_{\Fun(\{0\},\Ass)} \mM^\circledast
\subset \Fun([1],\Ass) \times_{\Fun(\{0\},\Ass)} \mM^\circledast \to $$$$ \Fun([1],\Ass) \times_{\Fun(\{0\},\Ass)} \mV^\ot \to \Env(\mV)^\ot= \Act \times_{\Fun(\{0\},\Ass)} \mV^\ot$$
of cocartesian fibrations over $\Ass$, where the last functor arises from the localization functor of \ref{facto}.

\vspace{1mm}	

Let $\mM^\circledast \to \mV^\ot \times \mW^\ot$ be an weakly bitensored $\infty$-category.
There is a canonical map
$$\B\Env(\mM)^\circledast = (\mathrm{Min} \times \mathrm{Max}) \times_{\Fun(\{0\},\Ass \times \Ass)} \mM^\circledast 
\subset \Fun([1],\Ass\times\Ass) \times_{\Fun(\{0\},\Ass \times \Ass)} \mM^\circledast \to $$$$ (\Fun([1],\Ass) \times_{\Fun(\{0\},\Ass)} \mV^\ot) \times (\Fun([1],\Ass) \times_{\Fun(\{0\},\Ass)} \mW^\ot) \to $$$$\Env(\mV)^\ot \times \Env(\mW)^\ot= (\Act \times_{\Fun(\{0\},\Ass)} \mV^\ot) \times (\Act \times_{\Fun(\{0\},\Ass)} \mW^\ot)$$
of cocartesian fibrations over $\Ass\times \Ass, $ where the last functor is induced by the localization functors.
	
\end{remark}

\begin{proposition}\label{rqqqy}\emph{} 
\begin{enumerate}
\item Let $\mM^\circledast \to \mV^\ot$ be a weakly left tensored $\infty$-category classified by a generalized $\LM$-operad $\mC \to \LM.$
Then $\L\Env(\mM)^\circledast \to \Env(\mV)^\ot$
is a left tensored $\infty$-category classified by $\Env_\LM(\mC) \to \LM.$

\vspace{1mm}
\item Let $\mM^\circledast \to \mW^\ot$ be a weakly right tensored $\infty$-category classified by a generalized $\RM$-operad $\mC \to \RM.$
Then $\R\Env(\mM)^\circledast \to \Env(\mW)^\ot$
is a right tensored $\infty$-category classified by $\Env_\RM(\mC) \to \RM.$

\vspace{1mm}
\item Let $\mM^\circledast \to \mV^\ot \times \mW^\ot$ be a weakly bitensored $\infty$-category 
classified by a generalized $\BM$-operad $\mC \to \BM.$
Then $\B\Env(\mM)^\circledast \to \Env(\mV)^\ot \times \Env(\mW)^\ot$
is a bitensored $\infty$-category classified by $\Env_\BM(\mC) \to \BM.$
\end{enumerate}
\end{proposition}

\begin{proof}
	
(1): The embedding $\Ass \simeq \Ass \times \{0\} \subset \Ass \times [1] \simeq \LM$ induces an equivalence $$\mathrm{Min} \simeq \Fun(\{1\},\Ass) \times_{\Fun(\{1\},\LM)}\Act(\LM).$$
Moreover the functor 
$\mathrm{Min} \simeq (\Ass \times \{0\}) \times_{\LM}\Act(\LM) \to (\Ass \times \{1\}) \times_{\LM}\Act(\LM) \simeq \Act$
factors as $\mathrm{Min} \subset \Fun([1],\Ass) \to \Act$.
So there is a canonical equivalence
$$(\Ass \times \{0\}) \times_{\LM} \Env_{\LM}(\mC) \simeq (\Ass \times \{0\}) \times_{\Fun(\{1\},\LM)} \Act(\LM) \times_{\Fun(\{0\},\LM)} \mC $$$$\simeq \L\Env(\mM)^\circledast= \mathrm{Min} \times_{\Fun(\{0\},\Ass)} \mM^\circledast$$
over $\Ass$ that is compatible with the functor to
$$(\Ass \times \{1\}) \times_{\LM} \Env_{\LM}(\mC) \simeq (\Ass \times \{1\})  \times_{\Fun(\{1\},\LM)} \Act(\LM) \times_{\Fun(\{0\},\LM)} \mC \simeq \Env(\mV)^\ot.$$

(2) is similar to (1).
(3): The embedding $\Ass \times \Ass \subset \Ass^{\triangleright} \times \Ass^{\triangleright} \simeq \BM^{\triangleright}$ induces an equivalence $$\mathrm{Min} \times \mathrm{Max} \simeq \Fun(\{1\},\Ass \times \Ass) \times_{\Fun(\{1\},\BM^{\triangleright})}\Act(\BM^{\triangleright})$$
and the functor 
$\mathrm{Min} \times \mathrm{Max} \simeq \Fun(\{1\},\Ass \times \Ass) \times_{\Fun(\{1\},\BM^{\triangleright})}\Act(\BM^{\triangleright}) \to $$$ (\Ass \times \{\infty\}) \times_{\Fun(\{1\},\BM^{\triangleright})}\Act(\BM^{\triangleright}) \times (\{\infty\} \times \Ass) \times_{\Fun(\{1\},\BM^{\triangleright})}\Act(\BM^{\triangleright}) \simeq \Act\times \Act$$
factors as $\mathrm{Min} \times \mathrm{Max} \subset \Fun([1],\Ass) \times \Fun([1],\Ass) \to \Act \times \Act$.
Since $$\Env_{\BM}(\mC)^\triangleright \simeq \Act(\BM^{\triangleright}) \times_{\Fun(\{0\},\BM^{\triangleright})} \mC^\triangleright $$(see Notation \ref{aact}), there is a canonical equivalence
$$ (\Ass \times \Ass) \times_{\BM^{\triangleright}} \Env_{\BM}(\mC)^\triangleright \simeq (\Ass \times \Ass) \times_{\Fun(\{1\},\BM^{\triangleright})} \Act(\BM^{\triangleright}) \times_{\Fun(\{0\},\BM^{\triangleright})} \mC^\triangleright $$$$\simeq \B\Env(\mM)^\circledast=(\mathrm{Min} \times \mathrm{Max}) \times_{\Fun(\{0\},\Ass \times \Ass)} \mM^\circledast$$
over $\Ass \times \Ass$ that is compatible with the functor to
$$ \Ass \times_{\BM} \Env_{\BM}(\mC) \simeq \Ass \times_{\Fun(\{1\},\BM)} \Act(\BM) \times_{\Fun(\{0\},\BM)} \mC \simeq \Env(\mV)^\ot $$
and similar for the right embedding.

\end{proof}
 
\begin{corollary}\label{bienv}
Let $\mM^\circledast \to \mV^\ot$ be a weakly left tensored $\infty$-category
and $\mN^\circledast \to \mW^\ot$ a weakly right tensored $\infty$-category.
The embedding $\mM^\circledast \times \mN^\circledast \subset \L\Env(\mM)^\circledast \times \R\Env(\mN)^\circledast$
of weakly bitensored $\infty$-categories induces an equivalence of $\infty$-categories bitensored over $\Env(\mV),\Env(\mW)$:
$$ \B\Env(\mM \times \mN)^\circledast \simeq \L\Env(\mM)^\circledast \times \R\Env(\mN)^\circledast.$$	
\end{corollary}

\begin{notation}
Let $\mM^\circledast \to \mV^\ot \times \mW^\ot$ be a weakly bitensored $\infty$-category.
We set $$\L\Env(\mM)^\circledast:= (\mathrm{Min} \times \Ass) \times_{\Fun(\{0\},\Ass \times \Ass)} \mM^\circledast \to \Fun(\{1\},\Ass \times \Ass)$$ so that $\L\Env(\mM)^\circledast_{[0]} \simeq \L\Env(\mM_{[0]})^\circledast$,
where we take the fiber over $[0]$ in the second factor of $\Ass$.
\end{notation}

There is a canonical functor $ \L\Env(\mM)^\circledast = \mathrm{Min} \times_{\Fun(\{0\},\Ass)} \mM^\circledast \to \Fun([1],\Ass) \times_{\Fun(\{0\},\Ass)} \mV^\ot $$$ \to \Env(\mV)^\ot = \Act \times_{\Fun(\{0\},\Ass)} \mV^\ot$$
over $\Fun(\{1\},\Ass), $ where the last functor is induced by the localization $\Fun([1],\Ass) \to \Act$. 

\vspace{1mm}
There is a functor $\L\Env(\mM)^\circledast \to \mM^\circledast \to \mW^\ot$ and embeddings $\mM^\circledast \subset \L\Env(\mM)^\circledast \subset \B\Env(\mM)^\circledast$ induced by the embeddings $\Ass \subset \mathrm{Min}, \Ass \subset \text{Max}.$
\begin{remark}
Note that $\L\Env(\mM)^\circledast \subset \B\Env(\mM)^\circledast$ is the full weakly bitensored subcategory spanned by $\L\Env(\mM)$ and $\mW.$
So $\L\Env(\mM)^\circledast \to \Env(\mV)^\ot \times \mW^\ot$ is a
weakly bitensored $\infty$-category and the embeddings $\mM^\circledast \subset \L\Env(\mM)^\circledast,  \L\Env(\mM)^\circledast \subset \B\Env(\mM)^\circledast$ are maps of weakly bitensored $\infty$-categories.
\end{remark}
 
\begin{lemma}\label{ripre}
Let $\mM^\circledast \to \mV^\ot \times \mW^\ot$ be a weakly bitensored $\infty$-category that exhibits $\mM$ as right tensored over $\mW.$
The embedding $\L\Env(\mM)^\circledast \subset \B\Env(\mM)^\circledast$ admits a left adjoint relative to $\Ass \times \Ass.$ 
A morphism is a local equivalence if and only if it lies over an equivalence in $\Fun(\{1\},\Ass \times \Ass)$ and lies over a cocartesian lift in $\mM^\circledast$ of a morphism in $\Fun(\{0\},\Ass \times \Ass)$ whose first component is an equivalence and whose second component preserves the maximum.

\end{lemma}

\begin{proof}
Since the diagonal embedding $\Ass \subset \Fun([1],\Ass)$ admits a left adjoint relative to $\Fun(\{1\},\Ass), $ 
the embeddings $\Ass \subset \text{Min}, \Ass \subset \text{Max} $ also do.
Thus the embedding $(\mathrm{Min} \times \Ass) \times_{\Fun(\{0\},\Ass \times \Ass)} \mM^\circledast \subset (\mathrm{Min} \times \text{Max}) \times_{\Fun(\{0\},\Ass \times \Ass)} \mM^\circledast$ admits a left adjoint relative to $\Ass \times \Ass$ using that $\mM^\ot \to \Ass$ is a cocartesian fibration via projection to the second factor.

\end{proof}

\begin{remark}
Lemma \ref{ripre} guarantees that for any weakly bitensored $\infty$-category $\mM^\circledast \to \mV^\ot \times \mW^\ot$ that exhibits $\mM$ as right tensored over $\mW$ the functor $\L\Env(\mM)^\circledast \to \Env(\mV)^\ot \times \mW^\ot$ is a bitensored $\infty$-category and the embedding $\mM^\circledast \subset \L\Env(\mM)^\circledast$ is $\mW$-linear.
\end{remark}

\noindent
The weakly bitensored $\infty$-category $\L\Env(\mM)^\circledast \to \Env(\mV)^\ot \times \mW^\ot$ has the following universal property:

\begin{proposition}
Let $\mM^\circledast \to \mV^\ot \times \mW^\ot, \mN \to \Env(\mV)^\ot \times \mW^\ot$ be weakly bitensored $\infty$-categories.

\begin{enumerate}
\item If $\mN^\circledast \to \Env(\mV)^\ot\times \mW^\ot$ exhibits $\mN$ as left tensored over $\Env(\mV)$ the canonical functor 
$$ \LaxLinFun_{\Env(\mV),\mW}(\L\Env(\mM),\mN) \to \LaxLinFun_{\mV,\mW}(\mM,\mN)$$
admits a fully faithful left adjoint that lands in $\LaxLinFun_{\Env(\mV),\mW}(\L\Env(\mM),\mN)'$, the full subcategory of $\Env(\mV)$-linear functors. 

In particular, the functor $ \LaxLinFun_{\Env(\mV),\mW}(\L\Env(\mM),\mN)' \to \LaxLinFun_{\mV,\mW}(\mM,\mN)$ is an equivalence. 

\vspace{1mm}
\item If $\mN^\circledast \to \Env(\mV)^\ot \times \mW^\ot$ is a bitensored $\infty$-category and $\mM^\circledast \to \mV^\ot \times \mW^\ot$ exhibits $\mM$ as right tensored over $\mW$, the former equivalence restricts to an equivalence 
$$ \alpha: \LinFun_{\Env(\mV),\mW}(\L\Env(\mM),\mN) \simeq \LaxLinFun_{\mV,\mW}(\mM,\mN)'', $$
where the right hand side is the full subcategory of lax $\mV,\mW$-linear functors 
that are $\mW$-linear.

\end{enumerate} 

\end{proposition}

\begin{proof}
(1) follows from Proposition \ref{Envelo} and Example \ref{envort}.
(2): The functor $\alpha$ in (2) factors as
$$\LinFun_{\Env(\mV),\mW}(\L\Env(\mM),\mN) \hookrightarrow \LinFun_{\Env(\mV),\Env(\mW)}(\B\Env(\mM),\mN) \simeq \LaxLinFun_{\mV,\mW}(\mM,\mN),$$
where the first functor is induced by the localization functor
$ \B\Env(\mM)^\circledast \to \L\Env(\mM)^\circledast$ of Lemma \ref{ripre}.
So the claim follows from the fact that a $\Env(\mV),\Env(\mW) $-linear functor 
$  \B\Env(\mM)^\circledast \to \mN^\circledast$ inverts local equivalences if and only if its restriction
$ \mM^\circledast_{[0]} \to \mN^\circledast_{[0]}$ is $\mW $-linear.

\end{proof}

\begin{corollary}\label{envdecom}
Let $\mM^\circledast \to \mV^\ot \times \mW^\ot$ be a weakly bitensored $\infty$-category. There is an equivalence
$$\R\Env(\L\Env(\mM))^\circledast \simeq \B\Env(\mM)^\circledast$$
of $\infty$-categories bitensored over $\Env(\mV),\Env(\mW).$

\end{corollary}

In the following we introduce a variant of enveloping
$\mO$-monoidal $\infty$-category that is closed.
 
\begin{notation}\label{envors}
For a small $\mO$-operad $\mC \to \mO$ we define the enveloping closed $\mO$-monoidal $\infty$-category $\mP\Env_\mO(\mC) \to \mO$ of $\mC \to \mO$ applying Proposition \ref{presday} to $\Env_\mO(\mC) \to \mO $, where we omit $\mO$ if $\mO=\Ass.$
\end{notation}

The enveloping closed $\mO$-monoidal $\infty$-category comes with an embedding of $\mO$-operads $\mC \subset \Env_\mO(\mC) \subset \mP\Env_\mO(\mC) $ such that for any $\mO$-monoidal $\infty$-category $\mD \to \mO$ compatible with small colimits the functor $$\Fun_\mO^{\ot,\L}(\mP\Env_\mO(\mC),\mD)\to \Alg_{ \mC/ \mO}(\mD)$$ is an equivalence.

\vspace{1mm}

In particular, we consider the following cases, where $\mO \in \{\Ass,\LM,\RM,\BM\}:$
\begin{itemize}
\item For $\mV^\ot \to \Ass$ an $\infty$-operad we call $\mP\Env(\mV)^\ot \to \Ass $ the enveloping closed monoidal $\infty$-category
of $\mV^\ot \to \Ass.$

\vspace{1mm}

\item For a $\LM$-operad $\mC \to \LM$ classifying a left tensored $\infty$-category $\mM^\circledast \to \mV^\ot$ the $\LM$-monoidal $\infty$-category $\mP\Env_\LM(\mC) \to \LM $ classifies a left tensored $\infty$-category  $\mP\L\Env(\mM)^\circledast \to \mP\Env(\mV)^\ot$, which we call the enveloping $\infty$-category with closed left action.

\vspace{1mm}

\item For a $\RM$-operad $\mC \to \RM$ classifying a right tensored $\infty$-category $\mM^\circledast \to \mW^\ot$ the $\RM$-monoidal $\infty$-category $\mP\Env_\RM(\mC) \to \RM $ classifies a right tensored $\infty$-category  $\mP\R\Env(\mM)^\circledast \to \mP\Env(\mW)^\ot$, which we call the enveloping $\infty$-category with closed right action.

\vspace{1mm}

\item For a $\BM$-operad $\mC \to \BM$ classifying a bitensored $\infty$-category $\mM^\circledast \to \mV^\ot \times \mW^\ot$ the $\BM$-monoidal $\infty$-category $\mP\Env_\BM(\mC) \to \BM $ classifies a bitensored $\infty$-category $\mP\B\Env(\mM)^\circledast \to \mP\Env(\mV)^\ot \times \mP\Env(\mW)^\ot$, which we call the enveloping $\infty$-category with closed biaction.

\end{itemize}

\subsection{Enriched and pseudo-enriched $\infty$-categories}
\label{Enpse}
In this subsection we define enriched and pseudo-enriched $\infty$-categories in the sense of Lurie as certain weakly left tensored $\infty$-categories.

To define enrichment we define morphism objects (Definition \ref{moor}).

\begin{warning}
The following definition of morphism object in case where $\mV^\ot$ is a monoidal $\infty$-category is not the definition \cite[Definition 4.2.1.28.]{lurie.higheralgebra} given by Lurie. See Lemma \ref{Rem} for the relationship between both definitions.

\end{warning}

\begin{definition}\label{moor}	
Let $\mM^\circledast\to \mV^\ot $ be an $\infty$-category weakly left tensored over an $\infty$-operad.

A morphism object of $\X, \Y \in \mM$ is an object 
$\mathrm{Mor}_{\mM}(\X, \Y) \in \mV $ together with a multi-morphism
$\alpha \in \Mul_{\mM}(\mathrm{Mor}_{\mM}(\X, \Y), \X; \Y) $ that induces for all objects $\Z_1, ..., \Z_\n \in \mV$ an equivalence
$$\Mul_{\mV}( \Z_1, ..., \Z_\n;  \mathrm{Mor}_{\mM}(\X, \Y)  ) \simeq \Mul_{\mM}( \Z_1, ..., \Z_\n, \X; \Y).$$

\end{definition}

\begin{definition}\label{Enr}
We say that an $\infty$-category $\mM^\circledast\to \mV^\ot$ weakly left tensored over an $\infty$-operad exhibits $\mM$ as enriched in $\mV$ 
if for every $\X, \Y \in \mM$ there is a morphism object $\mathrm{Mor}_{\mM}(\X, \Y) \in \mV$.
\end{definition}

For any weakly left tensored $\infty$-category $\mM^\circledast\to \mV^\ot$ that exhibits $\mM$ as enriched in $\mV$ the functor
$$\Mul_{\mM}(-,-;-): \mV^\op \times \mM^\op \times \mM \to \mS$$
is adjoint to a functor
$\mathrm{Mor}_{\mM}(-, -) : \mM^\op \times \mM \to \mV \subset \Fun(\mV^\op,\mS),$ which we call the graph of $\mM.$

\vspace{2mm}
If $\mV^\ot$ is a monoidal $\infty$-category, the definition of morphism object gets easier:
\begin{lemma}\label{Rem}

Let $\mV^\ot$ be a monoidal $\infty$-category and $\mM^\circledast\to \mV^\ot$ an $\infty$-category weakly left tensored over $\mV$.
Let $\X, \Y \in \mM, \T \in \mV$ and $\alpha \in \Mul_{\mM}(\T, \X; \Y) $ a multi-morphism.
Then $\alpha$ exhibits $\T$ as the morphism object of $\X$ and $\Y$,
i.e. for every $\Z_1, ..., \Z_\n \in \mV$ the multi-morphism $\alpha$ induces an equivalence 
$$\psi: \mV(\otimes(\Z_1, ..., \Z_\n), \T) \simeq \Mul_{\mM}(\Z_1, ..., \Z_\n, \X; \Y)$$
if and only if the following two conditions are satisfied:
\begin{enumerate}
\item For every $\Z_1, ..., \Z_\n \in \mV$ the canonical map
$$\gamma: \Mul_{\mM}(\ot(\Z_1, ..., \Z_\n), \X; \Y) \to \Mul_{\mM}(\Z_1, ..., \Z_\n, \X; \Y)$$ is an equivalence.
\item For every $\Z \in \mV$ the canonical map $\rho_\Z: \mV(\Z, \T) \to \Mul_{\mM}( \Z, \X; \Y)$ is an equivalence.
\end{enumerate}

\end{lemma}

\begin{proof}
Condition (2) is equivalent to the condition that for every $\Z_1, ..., \Z_\n \in \mV$ the canonical map 
$$\rho_{\ot(\Z_1, ..., \Z_\n)}: \mV(\ot(\Z_1, ..., \Z_\n), \T) \to \Mul_{\mM}(\ot(\Z_1, ..., \Z_\n), \X; \Y)$$ is an equivalence.
Moreover we have $ \psi = \gamma \circ \rho_{\ot(\Z_1, ..., \Z_\n)}$. 
Consequently if (2) holds, then condition (1) holds if and only if $\alpha$ exhibits $\T$ as the morphism object of $\X$ and $\Y$.
We conclude by observing that if $\alpha$ exhibits $\T$ as the morphism object of $\X$ and $\Y$, then (2) holds taking $\n=1.$

\end{proof}

\begin{corollary}\label{zvqq}
An $\infty$-category $\mM^\circledast\to \mV^\ot$ weakly left tensored over a monoidal $\infty$-category exhibits $\mM$ as enriched in $\mV$ if and only if condition (1) and (2) of Lemma \ref{Rem} hold for any $\X,\Y \in \mM.$
	
\end{corollary}

\begin{remark}
Lurie \cite[Definition 4.2.1.25.]{lurie.higheralgebra} says that a $\LM$-operad $\mO \to \LM$ exhibits $\mO_{[0] *[0]}$ as enriched in
$\mV^\ot:= \Ass \times_\LM \mO$ if $\mV^\ot \to \Ass$ is a monoidal $\infty$-category
and conditions (1) and (2) of Lemma \ref{Rem} hold for any $\X,\Y \in \mM$
in accordance with Corollary \ref{zvqq}.

\end{remark}
Moreover Lurie relaxes the notion of enrichment (\cite[Definition 4.2.1.28.]{lurie.higheralgebra}):
\begin{definition}\label{Lu}
We say that an $\infty$-category $\mM^\circledast\to \mV^\ot$ weakly left tensored over a monoidal $\infty$-category exhibits $\mM$ as pseudo-enriched in $\mV$ if condition (1) of Lemma \ref{Rem} holds for any $\X,\Y \in \mM.$

\end{definition}

\begin{example}\label{pset}
Any left tensored $\infty$-category $\mM^\circledast\to \mV^\ot$ exhibits $\mM$ as pseudo-enriched in $\mV$ since for any $\X,\Y \in \mM$ and $\Z_1, ..., \Z_\n \in \mV$ for $\n \geq 0$ the map
$\Mul_{\mM}(\ot(\Z_1, ..., \Z_\n), \X; \Y) \to \Mul_{\mM}(\Z_1, ..., \Z_\n, \X; \Y)$ factors as
$\Mul_{\mM}(\ot(\Z_1, ..., \Z_\n), \X; \Y) \simeq \Mul_{\mM}(\ot(\Z_1, ..., \Z_\n) \ot \X, \Y) \simeq \Mul_{\mM}(\Z_1, ..., \Z_\n, \X; \Y).$

\end{example}

Corollary \ref{zvqq} and Example \ref{pset} imply the following one:

\begin{example}\label{cloq}
	
A left tensored $\infty$-category $\mM^\circledast\to \mV^\ot$ exhibits $\mM$ as enriched in $\mV$ if and only if for every $\X, \Y \in \mM$
there is an object $ \T \in \mV$ and a morphism $\alpha: \T \ot \X \to \Y $ in $\mM$ such that for every $\Z \in \mV$ the canonical map $\mV(\Z, \T) \to \mM(\Z \ot \X, \Y)$ is an equivalence, in other words if and only if
for every $\X \in \mM$ the functor $(-)\ot \X:\mV \to \mM$ admits a right adjoint.

\noindent
So a left tensored $\infty$-category $\mM^\circledast\to \mV^\ot$ exhibits $\mM$ as $\mV$-enriched if and only the left action is closed.

\end{example}

\vspace{1mm}

By Example \ref{cloq} every $\infty$-category $\mM^\circledast\to \mV^\ot$
with closed left $\mV$-action exhibits $\mM$ as $\mV$-enriched.
Thus for any $\infty$-category $\mM^\circledast\to \mV^\ot$ weakly left tensored 
over an $\infty$-operad the enveloping $\infty$-category $\mP\L\Env(\mM)^\circledast \to \mP\Env(\mV)^\ot$ with closed left $\mP\Env(\mV)$-action exhibits $\mP\L\Env(\mM)$ as enriched in $\mP\Env(\mV)$. Consequently, every weakly left tensored
$\infty$-category embeds into an enriched $\infty$-category.

This motivates the following notation:
\begin{notation}\label{Exta}
Let $\mM^\circledast \to \mV^\ot$ be an $\infty$-category weakly left tensored over an $\infty$-operad.
Let $\bar{\mM}^\circledast \subset \mP\L\Env(\mM)^\circledast \to \mP\Env(\mV)^\ot$ be the full weakly left tensored subcategory spanned by $\mM$.
\end{notation}
Then $\bar{\mM}^\circledast \to \mP\Env(\mV)^\ot$ is a weakly left tensored $\infty$-category whose pullback to $\mV^\ot$ is $\mM^\circledast.$
In particular, $\bar{\mM} \simeq \mM.$
We call $\bar{\mM}^\circledast \to \mP\Env(\mV)^\ot $ the canonical extension of $\mM^\circledast \to \mV^\ot$ to $\mP\Env(\mV)^\ot.$

\vspace{1mm}	
Since $\mP\L\Env(\mM)^\circledast \to \mP\Env(\mV)^\ot$ exhibits $\mP\L\Env(\mM)$ as enriched in $\mP\Env(\mV)$, the weakly left tensored $\infty$-category $\bar{\mM}^\circledast \to \mP\Env(\mV)^\ot$ exhibits $\mM$ as enriched in $\mP\Env(\mV)$.

\vspace{1mm}	
By the next lemma the canonical embedding $\mM^\circledast\subset \mP\L\Env(\mM)^\circledast$ preserves morphism objects
so that the canonical embedding $\mM^\circledast\subset \bar{\mM}^\circledast$ preserves morphism objects, too:

\begin{lemma}\label{morpre}

Let $\mM^\circledast \to \mV^\ot$ be an $\infty$-category weakly left tensored over an $\infty$-operad $\mV^\ot.$

\vspace{2mm}

The embedding $\mM^\circledast \subset \mP\L\Env(\mM)^\circledast$ of weakly left tensored $\infty$-categories preserves morphism objects.

\end{lemma}

\begin{proof}

The $\infty$-category $\mP\Env(\mV)$ is generated by $\mV$ under small colimits and tensor products.

For every $\X, \Y \in \mM$ and $\V_1,..., \V_\n \in \mV$ for $\n \geq 0$ there is a canonical equivalence 
$$\mP\Env(\mV)(\V_1 \ot ... \ot \V_\n;  \mathrm{Mor}_{\mM}(\X, \Y)) \simeq \Mul_{\mP\Env(\mV)}( \V_1, ..., \V_\n;  \mathrm{Mor}_{\mM}(\X, \Y)) \simeq $$$$ \Mul_{\mV}( \V_1, ..., \V_\n;  \mathrm{Mor}_{\mM}(\X, \Y)) \simeq \Mul_{\mM}(\V_1, ..., \V_\n, \X; \Y) \simeq \Mul_{ \mP\L\Env(\mM)}(\V_1, ..., \V_\n, \X; \Y) \simeq $$$$ \Mul_{\mP\Env(\mV)}( \V_1, ..., \V_\n;  \mathrm{Mor}_{ \mP\L\Env(\mM)}(\X, \Y)) \simeq \mP\Env(\mV)(\V_1 \ot ... \ot \V_\n;  \mathrm{Mor}_{ \mP\L\Env(\mM)}(\X, \Y)).$$

\end{proof} 

Motivated by Lemma \ref{morpre} we make the following definition:

\begin{definition}
Let $\mM^\circledast \to \mV^\ot$ be an $\infty$-category weakly left tensored over an $\infty$-operad $\mV^\ot.$ The graph of $\mM$ is the graph of $\bar{\mM}.$
\end{definition}

\vspace{1mm}
For any $\V_1,...,\V_\n \in \mV,\X,\Y \in \mM$ for $\n \geq 0$ we have
$\mathrm{Mor}_{\bar{\mM}}(\X,\Y)(\V_1 \ot ... \ot \V_{\n}) \simeq$$$ \mP\Env(\mV)(\V_1 \ot ... \ot \V_{\n},\mathrm{Mor}_{\bar{\mM}}(\X,\Y)) \simeq \Mul_{\bar{\mM}}(\V_1,...,\V_\n,\X;\Y)\simeq \Mul_\mM(\V_1,...,\V_\n,\X;\Y).$$

\vspace{1mm}

We extend Notation \ref{Exta} to weakly bitensored $\infty$-categories:
\begin{notation}\label{Exta2}
	
Let $\mM^\circledast \to \mV^\ot \times \mW^\ot$ be an $\infty$-category weakly bitensored over $\infty$-operads.
Let $\bar{\mM}^\circledast \subset \mP\B\Env(\mM)^\circledast \to \mP\Env(\mV)^\ot \times \mP\Env(\mW)^\ot$ the full weakly bitensored subcategory spanned by $\mM$.
\end{notation}
The pullback of $\bar{\mM}^\circledast \to \mP\Env(\mV)^\ot \times \mP\Env(\mW)^\ot $ to $\mV^\ot \times \mW^\ot$ is $\mM^\circledast$ so that $\bar{\mM} \simeq \mM.$	
We call $\bar{\mM}^\circledast \to \mP\Env(\mV)^\ot \times \mP\Env(\mW)^\ot $ the canonical extension of $\mM^\circledast \to \mV^\ot \times \mW^\ot$
to $\mP\Env(\mV)^\ot \times \mP\Env(\mW)^\ot $.

We also consider the pullbacks of $\bar{\mM}^\circledast \to \mP\Env(\mV)^\ot \times \mP\Env(\mW)^\ot $ to $\mV^\ot \times \mP\Env(\mW)^\ot$, $ \mP\Env(\mV)^\ot \times \mW^\ot$, which we call the canonical extensions of $\mM^\circledast \to \mV^\ot \times \mW^\ot$ to $\mV^\ot \times \mP\Env(\mW)^\ot$ and $ \mP\Env(\mV)^\ot \times \mW^\ot$. 

\vspace{1mm}
Next we define tensors:

\begin{definition}
Let $\mM^\circledast \to \mV^\ot$ be an $\infty$-category weakly left tensored over an $\infty$-operad $\mV^\ot$ and $\V \in \mV,\X \in \mM.$
The tensor of $\V$ and $\X $ is an object $\V \ot \X \in \mM$ 
equipped with an object $\alpha \in \Mul_\mM(\V,\X,\V \ot \X)$
such that for any object $\V_1,...,\V_\n \in \mV$ for $\n \geq 0$ 
the following map is an equivalence:
$$\Mul_\mM(\V_1,...,\V_\n, \V \ot \X;\Y) \to \Mul_\mM(\V_1,...,\V_\n, \V,\X;\Y).$$ 

We say that $\mM^\circledast \to \mV^\ot$ exhibits $\mM$ as locally left tensored over $\mV$ if any $\V \in \mV,\X \in \mM$ admit a tensor.

\end{definition}

\begin{remark}
Let $\mM^\circledast \to \mV^\ot$ be an enriched $\infty$-category
such that $\mV^\circledast \to \mV^\ot$ ($\mV$ seen as weakly left tensored over itself) is an enriched $\infty$-category and $\V \in \mV,\X \in \mM.$
An object $\alpha \in \Mul_\mM(\V,\X,\X') \simeq \mV(\V, \Mor_\mM(\X,\X'))$ exhibits an object $\X'\in \mM$ as the tensor of
$\V$ and $\X$ if and only if for any $\Y \in \mM$ the induced morphism 
$\Mor_\mM(\X',\Y) \to \Mor_\mV(\V,\Mor_\mM(\X,\Y))$ is an equivalence.

\end{remark}

\begin{remark}\label{assp}
Let $\mM^\circledast \to \mV^\ot$ be a pseudo-enriched $\infty$-category
and $\V \in \mV,\X \in \mM$. An object $\alpha \in \Mul_\mM(\V,\X,\X')$ exhibits an object $\X' \in \mM$ as the tensor of $\V$ and $\X$ if and only if for any $\W \in \mW, \Y \in \mM$ the map 
$\Mul_\mM(\W, \X';\Y) \to \Mul_\mM(\W, \V,\X;\Y) \simeq \mV(\W \ot \V,\X;\Y)$ is an equivalence.

In particular, if the functors $\Mul_\mM(\W, \X';-), \mV(\W \ot \V,\X;-):\mM \to \mS$ are corepresentable, say by objects $\W \ot \X', (\W \ot \V) \ot \X$,
then $\alpha$ exhibits $\X'$ as the tensor of
$\V$ and $\X$ if and only if the canonical morphism
$(\W \ot \V) \ot \X \to \W \ot \X'$ in $\mM$ is an equivalence.
\end{remark}

Remark \ref{assp} 
implies the next proposition.
See \cite[Proposition 4.2.1.26.]{lurie.higheralgebra} for a proof.

\begin{proposition}\label{lur}
A weakly left tensored $\infty$-category is a left tensored $\infty$-category if and only if it is a pseudo-enriched $\infty$-category and locally left tensored $\infty$-category.
\end{proposition}

\section{Turning a weak action into a weak enrichment}
\label{extr}

In this section we define weakly enriched $\infty$-precategories and construct a functor $$\chi: \omega\LMod \to \{\text{weakly enriched $\infty$-precategories} \}$$ that associates to any weakly left tensored $\infty$-category $\mM^\circledast \to \mV^\ot$ an $\infty$-precategory weakly enriched in $\mV$ with space of objects $\mM^\simeq$.
 
\subsection{Weakly enriched $\infty$-categories}\label{eenr}

In the following we give the definition of enriched $\infty$-categories in the sense of \cite{GEPNER2015575} and a variation of it (Definition \ref{weenr}) that we call weakly enriched $\infty$-categories.
We start with defining a many object version of $\Ass$.

\vspace{1mm}

\begin{notation}\label{priory}
	
For any $\infty$-category $\X$ the functor $\nu: \Delta \to \Set, \ [\n] \mapsto \{0,..., \n \} $
gives rise to a functor \begin{equation}\label{rkal}
\Ass \xrightarrow{(-)^\op}\Ass \xrightarrow{\nu^\op} \Set^\op \xrightarrow{\Fun(-,\X)} \Cat_\infty\end{equation} 
classifying a cocartesian fibration 
$ \Ass_{\X} \to \Ass $ that is a generalized monoidal $\infty$-category.

\end{notation}

\begin{remark}\label{contr}
If $\X$ is contractible, the functor $ \Ass_{\X} \to \Ass $ is an equivalence.
\end{remark}

\begin{remark}\label{uniqp}
	
The functor (\ref{rkal}) is the right kan extension of
the functor $* \to \Cat_\infty$ taking $\X$ along the functor
$* \to \Ass$ taking $[0].$
So for any cocartesian fibration $\mO \to \Ass$ restriction $\Fun^\cocart_\Ass(\mO,\Ass_\X) \to\Fun(\mO_{[0]},\X)$
is an equivalence.
	
\end{remark}

Now we can define enriched $\infty$-precategories in the sense
of Gepner-Haugseng \cite[Definition 2.4.4.]{GEPNER2015575}:

\begin{definition}\label{GH} (Gepner-Haugseng)

Let $\X$ be a space and $\mV^\ot \to \Ass$ a generalized $\infty$-operad. A $\mV$-enriched $\infty$-precategory (or $\mV$-precategory) with space of objects $\X$ is a map $\Ass_{\X} \to \mV^\ot $ of generalized $\infty$-operads.
\end{definition}

In this work we are mainly interested in $\infty$-precategories enriched in an $\infty$-operad since this is the structure we compare
to Lurie's definition of enriched $\infty$-categories.
But we also consider the more general case of enrichment in a generalized $\infty$-operad as in many constructions there is no difference.

\begin{remark}\label{rem}
Definition \ref{GH} is a slight variant from the original definition
of Gepner and Haugseng:
Gepner and Haugseng define homotopy-coherent enrichment in a generalized $\infty$-operad $\mV^\ot$ under the name categorical algebra in $\mV$ (or $\mV$-enriched $\infty$-category) with space of objects $\X$ as a map $\Ass_{\X} \to (\mV^\ot)^\rev$ of generalized $\infty$-operads, where $(\mV^\ot)^\rev$ is the reversed generalized $\infty$-operad.

In case where $\mV^\ot \to \Ass$ is a monoidal $\infty$-category, for any objects $\A, \B, \C \in \X$ the composition maps of a $\mV$-precategory
$\mC$ with space of objects $\X$ are morphisms
$ \Mor_\mC(\B, \C) \ot \Mor_\mC(\A, \B) \to  \Mor_\mC(\A, \C)$ in our definition but maps
$ \Mor_\mC(\A, \B) \ot \Mor_\mC(\B, \C) \to  \Mor_\mC(\A, \C)$ in 
the definition of Gepner-Haugseng.

\end{remark}

\begin{example}
Remark \ref{contr} implies that $\mV$-precategories whose space of objects is contractible, are associative algebras in $\mV.$

\end{example}

\begin{notation}\label{uzpp}
For any $\mV$-precategory $\mC: \Ass_{\X} \to \mV^\ot $ with space of objects $\X$ there is an induced functor
$\mC_{[1]}: (\Ass_{\X})_{[1]} \simeq \X \times \X \to \mV^\ot_{[1]}$,
which we denote by $\mC(-,-),$ and call the graph of $\mC.$
\end{notation}

\begin{remark}\label{funcp}
A map of spaces $\X \to\Y$ and map of generalized $\infty$-operads $\theta:\mV^\ot \to \mW^\ot$ induce a functor $\Alg_{\Ass_\Y}(\mV) \to \Alg_{\Ass_\X}(\mW)$ precomposing with $\Ass_\X \to \Ass_\Y$ and postcomposing with $\theta.$
	
\end{remark}

\begin{lemma}\label{under}
For any $\infty$-operad $\mV^\ot \to \Ass$ whose multi-morphism spaces
$\Mul_\mV(\emptyset,\Z) $ are small for $\Z \in \mV$, the functor $\Mul_\mV(\emptyset,-):\mV \to \mS$ refines to a map of $\infty$-operads $\mV^\ot \to \mS^\times.$
If $\mV^\ot \to \Ass$ is a monoidal $\infty$-category compatible with small colimits, $\mV^\ot \to \mS^\times$ admits a left adjoint relative to $\Ass.$
\end{lemma}

\begin{proof}
For every large monoidal $\infty$-category $\mW^\ot \to \Ass$ let $\widehat{\mP}(\mW)^\ot$ be the very large monoidal $\infty$-category of presheaves of large spaces.	
By Proposition \ref{presday} there is a unique adjunction
$\widehat{\mS}^\times \rightleftarrows \widehat{\mP}\Env(\mV)^\ot:\R$
relative to $\Ass$, where $\R$ induces on underlying $\infty$-categories the functor
$\widehat{\mP}\Env(\mV)(\tu,-) $.
Thus $\mV^\ot \subset \widehat{\mP}\Env(\mV)^\ot \xrightarrow{\R} \widehat{\mS}^\times$
induces on underlying $\infty$-categories the functor
$\Mul_\mV(\emptyset,-): \mV \to \mS \subset \widehat{\mS}$.

Let $\mV^\ot \to \Ass$ be a monoidal $\infty$-category compatible with small colimits and $\widehat{\Ind}(\mV)^\ot \subset \widehat{\mP}(\mV)^\ot$
the full suboperad of presheaves preserving small limits. By Remark \ref{nos} the full suboperad $\widehat{\Ind}(\mV)^\ot \subset \widehat{\mP}(\mV)^\ot$
is a localization relative to $\Ass.$ Lemma \ref{looocx} implies that $\widehat{\mP}(\mV)^\ot \subset \widehat{\mP}\Env(\mV)^\ot$ is a localization relative to $\Ass.$
The resulting adjunction
$ \widehat{\mS}^\times \leftrightarrows \widehat{\mP}\Env(\mV)^\ot \leftrightarrows \widehat{\mP}(\mV)^\ot\leftrightarrows \widehat{\Ind}(\mV)^\ot $ relative to $\Ass$
restricts to an adjunction
$ \mS^\times \leftrightarrows \mV^\ot $ relative to $\Ass$
since $\mV$ is closed in $\widehat{\Ind}(\mV)$ under small colimits.

\end{proof}

By Lemma \ref{under} for any locally small $\infty$-operad $\mV^\ot \to \Ass$ 
the functor $\Mul_\mV(\emptyset,-):\mV \to \mS$ refines to a map
of $\infty$-operads sending a $\mV$-precategory to its underlying $\mS$-precategory (Remark \ref{funcp}).

\vspace{1mm}
Next we define weakly enriched $\infty$-categories
via the enveloping closed monoidal $\infty$-category:

\begin{definition}\label{weenr}

For any space $\X$ and small $\infty$-operad $\mV^\ot \to \Ass$ an $\infty$-precategory weakly enriched in $\mV$ with space of objects $\X$ is an $\infty$-precategory enriched in $\mP\Env(\mV)$ with space of objects $\X$.

\end{definition}

The embedding $\mV^\ot \subset \mP\Env(\mV)^\ot$ of $\infty$-operads yields an embedding $\Alg_{\Ass_\X}(\mV) \subset \Alg_{\Ass_\X}(\mP\Env(\mV))$
of $\infty$-categories enriched in $\mV$ into $\infty$-categories weakly enriched in $\mV.$

\begin{example}
Let $\emptyset^\ot \to \Ass$ be the initial $\infty$-operad (see Notation \ref{empp}).	
As $\Env(\emptyset)^\ot \simeq \Ass,$ we have $\mP\Env(\emptyset)^\ot \simeq \mS^\times.$ 
Thus $\infty$-precategories weakly enriched in $\emptyset$ are $\infty$-precategories enriched in $\mS$.

\end{example}

\subsection{Enriched quivers and their actions}

To construct the functor \begin{equation}\label{Funq}
\chi: \omega\LMod \to \{ \text{weakly enriched $\infty$-precategories}\}\end{equation}
we will describe $\infty$-precategories weakly enriched in an $\infty$-operad $\mV^\ot \to \Ass$ as associative algebras in a monoidal $\infty$-category.
Following Hinich \cite{HINICH2020107129} we will construct for any small space $\X$ and monoidal $\infty$-category $\mV^\ot \to \Ass$ compatible with small colimits a monoidal structure on the $\infty$-category $\Fun(\X \times \X,\mV)$ with the following key properties:

\begin{enumerate}\label{enum}
\item Associative algebras in $\Fun(\X \times \X,\mV) $ are $\mV$-enriched $\infty$-precategories with space of objects $\X$. 

\item For every $\mV$-enriched $\infty$-category $\mM^\circledast \to \mV^\ot $ the $\infty$-category $\Fun(\X,\mM)$ is enriched in $\Fun(\X \times \X,\mV)$, where the morphism object of two functors
$\F,\G: \X \to \mM$ is the functor $\X \xrightarrow{(\F^\op,\G)} \mM^\op \times \mM\xrightarrow{\Mor_\mM(-,-)} \mV$ (Proposition \ref{xre}). 

\end{enumerate}

We will use (1) and (2) to turn small $\mV$-enriched $\infty$-categories 
$\mM^\circledast \to \mV^\ot $ to $\mV$-enriched $\infty$-precategories:
by (2) for every map of spaces $\F: \X \to \mM^\simeq$ the functor $\X \xrightarrow{(\F^\op,\F)} \mM^\op \times \mM\xrightarrow{\Mor_\mM(-,-)} \mV$
is the endomorphism object with respect to an enrichment of
$\Fun(\X,\mM)$ in $\Fun(\X \times \X,\mV)$ and so by Proposition \ref{urr} carries a unique structure of an associative algebra in $\Fun(\X \times \X,\mV)$, which by (1) is identified with a $\mV$-enriched $\infty$-precategory.
For $\F$ the inclusion $\mM^\simeq \subset \mM$ we produce this way a 
canonical $\mV$-enriched $\infty$-precategory.
For an arbitrary weakly left tensored $\infty$-category $\mM^\circledast \to \mV^\ot $ we extend $\mM^\circledast \to \mV^\ot $ to a 
$\mP\Env(\mV)$-enriched $\infty$-category $\bar{\mM}^\circledast \to \mP\Env(\mV)^\ot $ (see Notation \ref{Exta}) to produce an $\infty$-precategory weakly enriched in $\mV$.
We will organize this assignment to functor (\ref{Funq}).

\vspace{1mm}

In this subsection we will construct the monoidal structure on $\Fun(\X \times \X, \mV)$, in the next subsection we will construct $\chi.$ 
Since we need to consider the monoidal structure on $\Fun(\X \times \X, \mV)$ as functorial in $\X$, we construct a fibered version via families of $\infty$-operads (Definition \ref{famm}).

\begin{notation}\label{noriol}
For every functor $\X \to \rS$ let $(\Ass^{\triangleright})^\rS_{\X}\to \rS\times \Ass^{\triangleright}$ be
the pullback of the map $ (\Ass_{\X})^{\triangleright} \to (\Ass_{\rS})^{\triangleright}
$ of cocartesian fibrations over $\Ass^{\triangleright} $ along the unique map $\rS \times \Ass^{\triangleright} \to (\Ass_{\rS})^{\triangleright}$ of cocartesian fibrations over $\Ass^{\triangleright} $ inducing the identity on the fiber over $[0]\in \Ass^{\triangleright}$ (Remark \ref{uniqp}).
Let $\Ass^\rS_{\X}\to \rS\times \Ass$ be the pullback of $(\Ass^{\triangleright})^\rS_{\X}\to \Ass^{\triangleright}$ to $\Ass.$ If $\rS$ is contractible, we omit $\rS$.

\end{notation}	

\begin{remark}\label{Reqppp}
For any functor $\rS' \to \rS$ there is a canonical equivalence
$\rS' \times_\rS (\Ass^{\triangleright})^\rS_{\X} \simeq  (\Ass^{\triangleright})^{\rS'}_{\rS' \times_\rS \X}$.
 
\end{remark}

\begin{remark}
The functor $\Ass^\rS_{\X}\to \rS\times \Ass$ is a $\rS$-family of generalized monoidal $\infty$-categories since it is the pullback of 
generalized monoidal $\infty$-categories.
If $\X \to \rS$ is a cocartesian fibration, $(\Ass_\X)^{\triangleright} \to (\Ass_\rS)^{\triangleright}$ and so $(\Ass^{\triangleright})^\rS_{\X}\to \rS\times \Ass^{\triangleright}$ are cocartesian fibrations since it induces fiberwise cocartesian fibrations and the fiber transports over maps of $\Ass^{\triangleright} $ preserve cocartesian morphisms (\cite[Lemma A.1.8.]{haugseng_melani_safronov_2020}). 
\end{remark}

\begin{notation}\label{noris}
For every functors $\X \to \rS, \Y \to \rS$ we set $$(\BM^{\triangleright})^\rS_{\X,\Y}:=(\Ass^{\triangleright})^\rS_{\X} \times_{\rS} (\Ass^{\triangleright})^\rS_{\Y} \to (\Ass^{\triangleright} \times \rS) \times_{\rS}(\Ass^{\triangleright} \times \rS) \simeq \Ass^{\triangleright} \times \Ass^{\triangleright} \times \rS \simeq \BM^{\triangleright} \times \rS.$$
Let $ \BM^\rS_{\X,\Y}$ be the pullback of $(\BM^{\triangleright})^\rS_{\X,\Y} \to \BM^{\triangleright} \times \rS$ to $\BM \subset \BM^{\triangleright}.$ 
If $\rS$ is contractible, we omit $\rS$.

\end{notation}

If $\X \to \rS, \Y \to \rS$ are cocartesian fibrations, $(\BM^{\triangleright})^\rS_{\X,\Y}\to \BM^{\triangleright} \times \rS$ is a cocartesian fibration since $(\Ass_\X)^{\triangleright} \to (\Ass_\rS)^{\triangleright}$, $(\Ass_\Y)^{\triangleright} \to (\Ass_\rS)^{\triangleright}$ are cocartesian fibrations.
Remark \ref{Reqppp} implies that for any functor $\rS' \to \rS$ there is a canonical equivalence
$\rS' \times_\rS (\BM^{\triangleright})^\rS_{\X,\Y} \simeq  (\BM^{\triangleright})^{\rS'}_{\rS' \times_\rS \X, \rS' \times_\rS \Y}$
over $\rS' \times \BM^{\triangleright}$.
	
\begin{remark}\label{clars}
The functor $\BM^\rS_{\X,\Y} \to \BM \times \rS$ is a $\rS$-family of generalized $\BM$-monoidal $\infty$-categories and classifies the $\rS$-family of weakly bitensored $\infty$-categories $\id: \Ass^\rS_\X \times_\rS \Ass^\rS_\Y \to \Ass^\rS_\X \times_\rS \Ass^\rS_\Y$.
\end{remark}

Now we combine the operad families of Notation \ref{noriol} and \ref{noris} with a theory of generalized Day-convolution developed in section \ref{genDay}:

\begin{notation}\label{uzbx}
Let $\X \to \rS, \Y \to \rS$ be cocartesian fibrations.
Let \begin{equation*}\label{wlkhn}
\Quiv_{\X}^{\rS}(-)^\ot: \Op^{\Ass \times \rS,\gen}_{\infty} \to \Op^{\Ass \times \rS,\gen}_{\infty}, \B\Quiv_{\X,\Y}^{\rS}: \Op^{\BM \times \rS,\gen}_{\infty} \to \Op^{\BM \times \rS,\gen}_{\infty}\end{equation*}
be the right adjoints of the functors $$(-) \times_{(\Ass \times \rS)} \Ass^\rS_\X : \Op^{\Ass \times \rS,\gen}_{\infty} \to \Op^{\Ass \times \rS,\gen}_{\infty}, \ (-) \times_{(\BM \times \rS)} \BM^\rS_{\X,\Y} : \Op^{\BM \times \rS,\gen}_{\infty} \to \Op^{\BM \times \rS,\gen}_{\infty}, $$
respectively, that exist by Theorem \ref{proop}. If $\rS$ is contractible, we omit $\rS$.
\end{notation}

By Proposition \ref{proooo} the right adjoints of Notation \ref{uzbx} send $\rS$-families of $\mO$-operads to $\rS$-families of $\mO$-operads for $\mO=\Ass,\BM,$ respectively.
Lemma \ref{pulll} implies the following remark:
\begin{remark}\label{fibbr}
For any functor $\rS' \to \rS$ and $\rS$-family of generalized $\infty$-operads $\mV^\ot \to \Ass \times \rS$
there is a canonical equivalence
$$\rS' \times_\rS \Quiv_{\X}^{\rS}(\mV)^\ot \simeq \Quiv_{\rS' \times_\rS \X}^{\rS'}(\rS' \times_\rS\mV)^\ot$$
of $\rS'$-families of generalized $\infty$-operads and similarly for $\B\Quiv_{\X,\Y}^{\rS}.$
\end{remark}

Lemma \ref{forms} implies the following remark:
\begin{remark}
	
For every $\rS$-family of $\infty$-operads $\mV^\ot \to \Ass \times \rS$
there is a canonical equivalence
$$\Quiv^\rS_\X(\mV) \simeq \Fun^\rS(\X \times_\rS \X,\mV)$$ since $(\Ass^\rS_\X)_{[1]} \simeq \X \times_\rS \X.$

\end{remark}

\begin{remark}\label{sim}
In \cite[3.2.11.]{HINICH2020107129} Hinich constructs similar functors as in Notation \ref{uzbx} without using generalized $\infty$-operads by constructing elaborate operadic models of $\Ass_\X, \BM_{\X,*}$.
Precisely, Hinich constructs for $\mO =\Ass, \BM$ an
$\mO$-operad $\mO'_\X$ and proves that the functor $(-)\times_\mO \mO'_\X: \Op^\mO_\infty \to \Op^\mO_\infty$ admits a right adjoint $\Quiv^\mO_\X(-)$ \cite[3.2.11.]{HINICH2020107129}.
For $\mO=\Ass$ MacPherson \cite{MR4185309} identifies $\Ass'_\X$ as the initial $\infty$-operad under $\Ass_\X$. 
It follows from Hinich's description of $\Ass_\X'$ that the map $\Ass_\X \to \Ass'_\X$ is a strong approximation in the sense of \cite[Definition 2.3.3.6.]{lurie.higheralgebra}.
This implies that for any $\infty$-operad $\mV^\ot \to \Ass$ the pullback $ \mV^\ot \times_\Ass \Ass_\X \to \mV^\ot \times_\Ass \Ass'_\X $ exhibits $\mV^\ot \times_\Ass \Ass'_\X$ as the initial $\infty$-operad under $\mV^\ot \times_\Ass \Ass_\X$
so that Hinich's $\Quiv_\X^\Ass(-)$ is equivalent to our $\Quiv_\X(-)^\ot.$
\end{remark}

\begin{notation}\label{exxxx} Let $\X \to \rS, \Y \to \rS$ be cocartesian fibrations and $\mM^\circledast \to \mV^\ot \times_\rS \mW^\ot$ a $\rS$-family of weakly bitensored $\infty$-categories classified by a $\rS$-family of
generalized $\BM$-operads $\mD \to \rS \times \BM$.

The $\rS$-family of generalized $\BM$-operads $\B\Quiv^\rS_{\X,\Y}(\mD)\to \rS \times \BM $ exhibits $\Fun^\rS(\X \times_\rS \Y, \mM) \to \rS$ as weakly bitensored over $\Quiv^\rS_\X(\mV) \to \rS, \Quiv^\rS_\Y(\mW) \to \rS$ and so classifies
a weakly bitensored $\infty$-category $$\Fun^\rS(\X \times_\rS \Y, \mM)^\circledast \to \Quiv^\rS_\X(\mV)^\ot \times_\rS \Quiv^\rS_\Y(\mW)^\ot.$$

If $\rS$ is contractible, we omit $\rS$.

\end{notation}

\vspace{1mm}

We often apply Notation \ref{exxxx} when $\mW^\ot = \rS \times \emptyset^\ot$. Then $\Quiv^\rS_\Y(\mW)^\ot \simeq \rS \times \emptyset^\ot$
and $\mM^\circledast \to \mV^\ot \times_\rS \mW^\ot \simeq \mV^\ot,\ \Fun^\rS(\X \times_\rS \Y, \mM)^\circledast \to \Quiv^\rS_\X(\mV)^\ot$
are $\rS$-families of weakly left tensored $\infty$-categories.

\begin{remark}\label{swits}
Let $\nu: \Delta \to \Set$ be the functor forgetting the order. There is an equivalence
$ \nu \circ (-)^\op \simeq \nu$ of functors $\Delta \to \Set$,
whose component at $[\n]\in \Delta$ is the bijection
$\{1,...,\n\} \to \{1,...,\n\}, \bi \mapsto \n-\bi$.
This equivalence yields for every $\infty$-category $\X$ an equivalence 
$\Ass_{\X}^\rev\simeq \Ass_{\X}$ over $\Ass.$
 
For every functor $\X \to \rS$ we obtain an equivalence
$((\Ass^{\triangleright})^\rS_{\X})^\rev\simeq (\rS \times \Ass^{\triangleright})\times_{(\Ass^\rev_\rS)^{\triangleright}} (\Ass^\rev_\X)^{\triangleright} \simeq  (\Ass^{\triangleright})^\rS_{\X}$ 
over $\Ass^{\triangleright} \times \rS$
that restricts to an equivalence $(\Ass^\rS_{\X})^\rev\simeq \Ass^\rS_{\X} $ over $\Ass \times \rS$. We obtain an equivalence
\begin{equation}\label{swirl2}
((\BM^{\triangleright})^\rS_{\X,\Y})^\rev \simeq ((\Ass^{\triangleright})^\rS_{\Y})^\rev \times_{\rS}  ((\Ass^{\triangleright})^\rS_{\X})^\rev \simeq (\Ass^{\triangleright})^\rS_{\Y} \times_{\rS}(\Ass^{\triangleright})^\rS_{\X}= (\BM^{\triangleright})^\rS_{\Y,\X}
\end{equation}
over $\Ass^{\triangleright} \times \Ass^{\triangleright} \times\rS$,
where the left equivalence holds since 
the involution on $ \BM^{\triangleright} $
induced via $\BM=(\Delta_{/[1]})^\op $ identifies under $ \Ass^{\triangleright} \times \Ass^{\triangleright} \simeq \BM^{\triangleright} $ 
with switching the factors and applying the involution on $ \Ass^{\triangleright}$ in each factor.
(\ref{swirl2}) restricts to an equivalence $(\BM^\rS_{\X,\Y})^\rev \simeq \BM^\rS_{\Y,\X}$ over $\BM \times \rS$.

\end{remark}

\begin{remark}\label{swi}

By Remark \ref{swits} for any $\rS$-family of weakly bitensored $\infty$-categories $\mM^\circledast \to \mV^\ot \times_\rS \mW^\ot$
there are canonical equivalences
$$ (\Quiv_{\X}^{\rS}(\mV)^\ot)^\rev \simeq \Quiv_{\X}^{\rS}(\mV^\rev)^\ot, \ (\Fun^\rS(\X \times_\rS \Y, \mM)^\circledast)^\rev \simeq \Fun^\rS(\Y \times_\rS \X, \mM^\rev)^\circledast$$
over $\rS \times \Ass$ and $(\Quiv^\rS_\Y(\mW)^\ot)^\rev \times_\rS (\Quiv^\rS_\X(\mV)^\ot)^\rev \simeq \Quiv^\rS_\Y(\mW^\rev)^\ot \times_\rS \Quiv^\rS_\X(\mV^\rev)^\ot,$ respectively.
\end{remark}

\begin{notation}

Let $\X,\Y$ be $\infty$-categories and $\mM^\circledast \to \mV^\ot \times \mW^\ot$
a weakly bitensored $\infty$-category.
We set $\mV_\X^\ot:= \Ass_\X \times_\Ass \mV^\ot$
and $\mM^\circledast_{\X,\Y}:= {(\Ass_\X \times \Ass_\Y)} \times_{(\Ass \times \Ass)}\mM^\circledast \to \mV^\ot_\X \times \mW^\ot_\Y$.
	
\end{notation}

Note that there is a canonical map of generalized $\infty$-operads
$\Quiv_\X(\mV)_\X^\ot \to \mV^\ot.$

\begin{remark}\label{paraph}Let $\mM^\circledast \to \mV^\ot \times \mW^\ot$
be a weakly bitensored $\infty$-category and set 
$\widetilde{\mV}^\ot:= \Quiv_\X(\mV)^\ot, \widetilde{\mW}^\ot:= \Quiv_\Y(\mW)^\ot.$ As a consequence of Remark \ref{clars} the weakly bitensored $\infty$-category
$\Fun(\X \times \Y, \mM)^\circledast \to \widetilde{\mV}^\ot \times \widetilde{\mW}^\ot $ has the following corresponding universal property:
for any weakly bitensored $\infty$-category $\mN^\circledast \to \widetilde{\mV}^\ot \times \widetilde{\mW}^\ot $ there is a canonical equivalence
$$\LaxLinFun_{\widetilde{\mV},\widetilde{\mW}}(\mN,\Fun(\X \times \Y, \mM))\simeq \LaxLinFun_{\widetilde{\mV}_\X, \widetilde{\mW}_\Y}(\mN_{\X,\Y},\widetilde{\mV}_\X \times_{\mV} \mM\times_{\mW}\widetilde{\mW}_\Y).$$

\end{remark}

\begin{lemma}\label{diagg} Let $\X$ be an $\infty$-category and $\mM^\circledast \to \mV^\ot \times \mW^\ot$ a weakly bitensored $\infty$-category.
The weakly right tensored $\infty$-category underlying $\Fun(\X, \mM)^\circledast \to \Quiv_\X(\mV)^\ot \times \mW^\ot $ 
is $(\mM^\circledast)^\X \to \mW^\ot$ (\ref{diagor}).
\end{lemma}
\begin{proof}We apply Remark \ref{paraph} for $\mV^\ot=\emptyset^\ot$ and $\Y$ contractible.
For every weakly right tensored $\infty$-category $\mN^\circledast \to \mW^\ot$
we have $\mN_{\X,\Y}^\circledast \simeq \mN^\circledast \times \X$. Hence there is an equivalence
$ \LaxLinFun_{\mW}(\mN,\Fun(\X,\mM)) \simeq \LaxLinFun_{\mW}(\mN \times \X, \mM)\simeq \LaxLinFun_{\mW}(\mN, \mM^\X). $
\end{proof}

\begin{lemma}\label{comppp}

\begin{enumerate}
\item Let $\X \to \rS$ be a cocartesian fibration, $\mV^\ot \to \rS \times \Ass$ a
$\rS$-family of generalized $\infty$-operads and $\mO:= \BM \times_\Ass \mV^\ot \to \BM$ the associated $\rS$-family of generalized $\BM$-operads.
There is a canonical equivalence $$ \B\Quiv^\rS_{\X,\X}(\mO) \simeq \BM \times_\Ass \Quiv^\rS_\X(\mV)^\ot$$ of $\rS$-families of generalized $\BM$-operads.
In particular, there is a canonical equivalence $$\Ass \times_\BM \B\Quiv^\rS_{\X,\X}(\mO) \simeq \Quiv^\rS_\X(\mV)^\ot$$ of $\rS$-families of generalized $\infty$-operads, where $\Ass \subset \BM$ is the left or right embedding.

\vspace{1mm}
\item Let $\X\to \rS,\X'\to \rS,\Y\to \rS,\Y'\to \rS$ be cocartesian fibrations and $\mO \to \rS \times \BM$ a $\rS$-family of generalized $\BM$-operads.
There is a canonical equivalence $$ \B\Quiv^\rS_{\X',\Y'}(\B\Quiv^\rS_{\X,\Y}(\mO)) \simeq \B\Quiv^\rS_{\X \times_\rS \X',\Y \times_\rS \Y'}(\mO)$$ of $\rS$-families of generalized $\BM$-operads.

\end{enumerate}

\end{lemma}

\begin{proof}
(1): For any $\infty$-category $\Y$ there is an equivalence
$(\BM^{\triangleright})_{\Y,\Y}= (\Ass^{\triangleright})_\Y \times (\Ass^{\triangleright})_\Y \to \BM^{\triangleright} \times_{\Ass^{\triangleright}} (\Ass^{\triangleright})_\Y $
over $\Ass^{\triangleright} \times \Ass^{\triangleright} \simeq \BM^{\triangleright}$
classifying the equivalence 
$$([\n], [\m]) \mapsto \Fun(\{0,...,\n\},\Y) \times \Fun(\{0,...,\m\},\Y) \simeq \Fun(\{0,...,\n+\m+1\},\Y)$$
of functors $\Ass^{\triangleright} \times \Ass^{\triangleright}\to \Cat_\infty.$
The induced equivalence
$(\BM^{\triangleright})^\rS_{\X,\X}= (\Ass^{\triangleright})^\rS_\X \times_\rS (\Ass^{\triangleright})^\rS_\X \simeq $$$ (\rS \times \Ass^{\triangleright} \times \Ass^{\triangleright}) \times_{((\Ass^{\triangleright})_\rS \times (\Ass^{\triangleright})_\rS)} (\Ass^{\triangleright})_\X \times (\Ass^{\triangleright})_\X \simeq (\rS \times \BM^{\triangleright}) \times_{(\BM^{\triangleright})_{\rS,\rS}} (\BM^{\triangleright})_{\X,\X}
\simeq $$$$(\rS \times \BM^{\triangleright}) \times_{(\BM^{\triangleright} \times_{\Ass^{\triangleright}} (\Ass^{\triangleright})_\rS)} (\BM^{\triangleright} \times_{\Ass^{\triangleright}} (\Ass^{\triangleright})_\X)
\simeq \BM^{\triangleright} \times_{\Ass^{\triangleright}}  (\Ass^{\triangleright})^\rS_\X $$
over $\Ass^{\triangleright} \times \Ass^{\triangleright} \simeq \BM^{\triangleright}$
restricts to an equivalence $ \BM^\rS_{\X,\X} \to \BM \times_\Ass \Ass^\rS_\X  $ of $\rS$-families of generalized $\BM$-operads that represents the equivalence of (1).

\vspace{1mm}

The equivalence of (2) is represented by an equivalence
$\BM^\rS_{\X,\Y} \times_{(\BM \times \rS)} \BM^\rS_{\X',\Y'} \to \BM^\rS_{\X \times_\rS\X',\Y \times_\rS \Y'} $ of $\rS$-families of generalized $\BM$-operads
that is the restriction of the following equivalence over $\rS \times\BM^{\triangleright}$:
$$(\BM^{\triangleright})^\rS_{\X,\Y} \times_{(\BM^{\triangleright} \times \rS)} (\BM^{\triangleright})^\rS_{\X',\Y'}=((\Ass^{\triangleright})^\rS_\X \times_{\rS} (\Ass^{\triangleright})^\rS_{\Y}) \times_{(\rS \times \Ass^{\triangleright} \times_\rS \rS \times \Ass^{\triangleright})} ((\Ass^{\triangleright})^\rS_{\X'} \times_{\rS} (\Ass^{\triangleright})^\rS_{\Y'})$$
$$\simeq (\Ass^{\triangleright})^\rS_\X \times_{(\rS \times \Ass^{\triangleright})} (\Ass^{\triangleright})^\rS_{\X'} \times_{\rS} (\Ass^{\triangleright})^\rS_\Y \times_{(\rS \times \Ass^{\triangleright})} (\Ass^{\triangleright})^\rS_{\Y'} \simeq
(\Ass^{\triangleright})^\rS_{\X \times_\rS \X'} \times_{\rS} (\Ass^{\triangleright})^\rS_{\Y \times_\rS \Y'}.$$ 
\end{proof}

Proposition \ref{proooo} and Remark \ref{spezi} imply the following remark:

\begin{remark}\label{continok}
Let $\X,\Y$ be small spaces.

\begin{enumerate}
\item For every corepresentable $\infty$-operad $\mV^\ot \to \Ass$ such that $\mV$
has small colimits the $\infty$-operad $\Quiv_{\X}(\mV)^\ot \to \Ass$ is corepresentable.		

For any functors $\F_1,..., \F_\n: \X \times \X \to \mV$ for $\n \geq2$ and 
$\A, \A' \in \X$ we have 
$$\hspace{5mm} (\F_1 \ot.... \ot \F_\n)(\A, \A') \simeq \colim_{\B_1,...,\B_{\n-1} \in \X}\F_1(\B_1,\A') \ot\F_2(\B_2, \B_1)  \ot ... \ot \F_\n(\A,\B_{\n-1}).
$$

The tensor unit is the functor $\X \times \X \to \mV$ sending $(\A, \B)$ to $\X(\A,\B)\otimes \tu$.

\vspace{1mm}

\item For every corepresentable $\BM$-operad $\mO \to \BM$ classifying a weakly bitensored $\infty$-category $\mM^\circledast \to \mV^\ot \times \mW^\ot$
such that $\mM, \mV,\mW$ have small colimits the $\BM$-operad $\B\Quiv_{\X,\Y}(\mO) \to \BM$ is corepresentable.

For any functors $\F_1,...\F_\n: \X \times \X \to \mV, \G: \Y \times \X \to \mM, \rH_1,...,\rH_\m: \Y \times \Y \to \mW$ for $\n,\m \geq 2$ and $\A \in \Y,\A' \in \X$ we have $$(\F_1 \ot ... \ot \F_\n \ot \G \ot \rH_1 \ot ... \ot \rH_\m)(\A,\A') \simeq \colim_{\B_1,...,\B_{\n} \in \X,\C_1,...,\C_\m \in \Y} $$$$\hspace{10mm} \F_1(\B_1,\A') \ot\F_2(\B_2, \B_1)  \ot ... \ot \F_\n(\B_\n,\B_{\n-1}) \ot \G(\C_1,\B_\n) \ot \rH_1(\C_2,\C_1) \ot ... \ot \rH_\m(\A,\C_{\m}). $$

\end{enumerate}

\end{remark}

\begin{corollary}\label{contino}
Let $\X,\Y$ be small spaces.

\begin{enumerate}
\item For every monoidal $\infty$-category $\mV^\ot \to \Ass$ compatible with small colimits $\Quiv_{\X}(\mV)^\ot \to \Ass$ is a monoidal $\infty$-category compatible with small colimits.

\vspace{1mm}

\item For every bitensored $\infty$-category $\mM^\circledast \to \mV^\ot \times \mW^\ot$ compatible with small colimits $$\Fun(\X \times \Y, \mM)^\circledast \to \Quiv_\X(\mV)^\ot \times \Quiv_\Y(\mW)^\ot $$ is a bitensored $\infty$-category compatible with small colimits.

\end{enumerate}

\end{corollary}

\vspace{0,1mm}

\begin{proof}
By Remark \ref{continok} the $\infty$-operad $\Quiv_{\X}(\mV)^\ot \to \Ass$ is corepresentable and so is a monoidal $\infty$-category (Lemma \ref{repmon}) since for any $\F_1,\F_2,\F_3: \X \times \X \to \mV$ the following morphisms are equivalences: $$ (\F_1 \ot \F_2) \ot \F_3 \leftarrow \F_1 \ot \F_2 \ot \F_3 \to \F_1 \ot (\F_2 \ot \F_3), \ \F_1 \to \F_1 \ot \tu.$$

Let $\mO \to \BM$ be the $\BM$-monoidal $\infty$-category commpatible with small colimits classifying $\mM^\circledast \to \mV^\ot \times \mW^\ot$.
By Remark \ref{continok} the $\BM$-operad $\B\Quiv_{\X}(\mO) \to \BM$ is corepresentable and so is a $\BM$-monoidal $\infty$-category (Lemma \ref{repmon}) since for any $ \G: \X \times \Y \to \mM, \rH_1,\rH_2: \Y \times \Y \to \mW$ the
bitensor $\F_1 \ot \F_2 \ot \G \ot \rH_1 \ot \rH_2$ is canonically equivalent to the following objects:
$$ ((\F_1 \ot \F_2) \ot \G) \ot (\rH_1 \ot \rH_2), \ (\F_1 \ot (\F_2 \ot \G)) \ot (\rH_1 \ot \rH_2), \ (((\F_1 \ot \F_2) \ot \G) \ot \rH_1) \ot \rH_2, $$$$ ((\F_1 \ot (\F_2 \ot \G)) \ot \rH_1) \ot \rH_2, \ (\F_1 \ot \F_2) \ot (\G \ot (\rH_1 \ot \rH_2)), \ (\F_1 \ot \F_2) \ot ((\G \ot \rH_1) \ot \rH_2), $$$$ \F_1 \ot (\F_2 \ot (\G \ot (\rH_1 \ot \rH_2))), \ \F_1 \ot (\F_2 \ot ((\G \ot \rH_1) \ot \rH_2)). $$

\end{proof}

The next proposition is due to Hinich \cite[Proposition 6.3.1.]{HINICH2020107129} 
in case of endomorphism objects and describes the morphism objects 
of the weakly left tensored $\infty$-category 
$\Fun(\X,\mM)^\circledast \to \Quiv_\X(\mV)^\ot$ associated to a space $\X$ and $\infty$-category $\mM^\circledast \to \mV^\otimes$ with weak left action of an $\infty$-operad. 

\begin{proposition}\label{xre}
Let $\mM^\circledast \to \mV^\otimes$ be an $\infty$-category weakly left tensored over an $\infty$-operad.
Let $\X$ be a space and $\F,\G: \X \to \mM$ be functors.
\begin{enumerate}
\item Let $\T : \X \times \X \to \mV$ be a functor. A multi-morphism $\theta \in \Mul_{\Fun(\X,\mM)}(\T, \F; \G)$ exhibits $\T$ as the morphism object of
$\F$ and $\G$ if for any $\A,\B \in \X$ the induced multi-morphism in $\Mul_{\mV}(\T(\A,\B), \F(\A); \G(\B))$ exhibits $\T(\A,\B) $ as the morphism object of $\F(\A), \G(\B).$ 

\vspace{1mm}	
\item If for every $\A, \B \in \X$ there is a morphism object $\Mor_\mM(\F(\A),\G(\B)) \in \mV$, there is a functor $\T : \X \times \X \to \mV$ and $\theta \in \Mul_{\Fun(\X,\mM)}(\T, \F; \G)$ such that for any $\A,\B \in \X$ the induced multi-morphism in $\Mul_{\mV}(\T(\A,\B), \F(\A); \G(\B))$ exhibits $\T(\A,\B) $ as the morphism object of $\F(\A), \G(\B).$ 

\end{enumerate}

\end{proposition}

\begin{proof}
The weakly left tensored $\infty$-category $\mM^\circledast \to \mV^\ot$ embeds into $\mP\L\Env(\mM)^\circledast \to \mP\Env(\mV)^\ot$. By Lemma \ref{morpre} the embedding $\mM^\circledast \subset \mP\L\Env(\mM)^\circledast$ preserves morphism objects and yields an embedding $\Fun(\X,\mM)^\circledast \subset \Fun(\X,\mP\L\Env(\mM))^\circledast $ of
weakly left tensored $\infty$-categories. So we can reduce to the case where $\mM^\circledast \to \mV^\ot$ is an $\infty$-category presentably left tensored over $\mV$.
We start with proving (2). 

Let $\E$ be the tensor unit of $\Fun(\X \times \X,\mV) \simeq\Fun(\X,\Fun(\X,\mV))$.
The functor $\E: \X \to \Fun(\X,\mV)$ factors as $\X \subset \Fun(\X, \mS) \to \Fun(\X,\mV),$ where the first embedding is the Yoneda-embedding and the second functor is induced by the small colimits preserving functor $\mS \to \mV$ selecting the tensor unit. By Lemma \ref{dfgj} the functor $\F:\X \to \mM$
factors as $ \X \xrightarrow{\E} \Fun(\X,\mV) \xrightarrow{\bar{\F}} \mM $
for a unique $\mV$-linear functor $\bar{\F}: \Fun(\X,\mV) \to \mM$ right adjoint to the functor $\R: \mM \to \Fun(\X,\mV)$
corresponding to the functor $\X \times \mM \xrightarrow{\F^\op \times \mM} \mM^\op \times \mM \xrightarrow{\Mor_\mM(-,-)} \mV.$
We set $\T:= \R \circ \G \in \Fun(\X,\Fun(\X,\mV)) \simeq \Fun(\X \times \X,\mV)$ 
such that $\T$ corresponds to the functor $\X \times \X \xrightarrow{\F^\op \times \G} \mM^\op \times \mM \xrightarrow{\Mor_\mM(-,-)} \mV$.
Next we construct $\theta: \T \otimes \F \to \G$.
The unit $\id \to \R \circ \bar{\F}$ gives rise to a map $\E \to \R \circ \bar{\F} \circ \E \simeq \R \circ \F.$
The functor $\bar{\F}: \Fun(\X,\mV) \to \mM$ yields a $\Fun(\X \times \X,\mV)$-linear functor $\bar{\F}_\ast: \Fun(\X,\Fun(\X,\mV)) \simeq \Fun(\X \times \X,\mV) \to \Fun(\X,\mM).$ By linearity for any $\rH\in \Fun(\X \times \X,\mV) $ we obtain an equivalence 
$\rH \ot \F \simeq \rH \ot \bar{\F}_\ast(\E) \simeq \bar{\F}_\ast(\rH \otimes \E) \simeq \bar{\F} \circ \rH $
in $\Fun(\X,\mV) $.
We define $\theta: \T \otimes \F \to \G$ in $\Fun(\X,\mM)$ as the counit
$ \T \otimes \F \simeq \bar{\F} \circ \T = (\bar{\F} \circ \R) \circ \G  \to \G.$
Then for any $\A,\B \in \X$ the induced multi-morphism in $\Mul_{\mV}(\T(\A,\B), \F(\A); \G(\B))$ exhibits $\T(\A,\B) $ as the morphism object of $\F(\A), \G(\B)$.
It remains to see that the morphism $\theta: \T \ot \F \to \G$ yields an  equivalence
$$\rho: \Fun(\X, \Fun(\X,\mV))(\rH, \T) \to \Fun(\X,\mM)(\rH \otimes \F, \T \otimes \F) \to \Fun(\X,\mM)(\rH \otimes \F, \G).$$ 
The map $\rho$ factors as the following composition
$$ \Fun(\X, \Fun(\X,\mV))(\rH, \T) \to \Fun(\X,\mM)(\bar{\F} \circ \rH, \bar{\F} \circ \T) \xrightarrow{} \Fun(\X,\mM)(\bar{\F} \circ \rH, \G)  \simeq \Fun(\X,\mM)(\rH \otimes \F, \G)$$
inverse to the map $\Fun(\X,\mM)(\bar{\F} \circ \rH, \G) \to \Fun(\X, \Fun(\X,\mV))(\R \circ \bar{\F} \circ \rH, \T) \to \Fun(\X, \Fun(\X,\mV))(\rH, \T)$
induced by the unit $\id \to \R \circ \bar{\F}$ using the triangle identities.

We complete the proof by verifying (1):
Let $\T : \X \times \X \to \mV$ be a functor and $\theta: \T \ot \F \to \G$ a morphism. The morphism $\theta: \T \ot \F \simeq \bar{\F} \circ \T \to \G$
is adjoint to a morphism $\kappa: \T \to \R \circ \G,$
whose component at $\A,\B \in \X$ is the morphism
$\T(\A,\B) \to \Mor_\mM(\F(\A),\G(\B)) $ adjoint to the induced morphism
$\T(\A,\B)\ot \F(\A)\to \G(\B)$ in $\mM$. So $\kappa$ is an equivalence if for any $\A,\B \in \X$ the induced morphism $\T(\A,\B)\ot \F(\A)\to \G(\B)$ exhibits $\T(\A,\B) $ as the morphism object of $\F(\A), \G(\B)$.
We conclude by observing that by definition of $\kappa$ the morphism $\theta$ factors as $ \bar{\F} \circ \T \xrightarrow{\bar{\F} \circ \kappa} \bar{\F} \circ \R \circ \G \to \G$ and so exhibits $\T$ as the morphism object of
$\F$ and $\G$ by the proof of (2) if $\kappa$ is an equivalence.

\end{proof}

For Proposition \ref{xre} we used the following lemma:

\begin{lemma}\label{dfgj}

Let $\mB $ be a small $\infty$-category and $\mM^\circledast \to \mV^\ot $ a presentably left tensored $\infty$-category.
\begin{enumerate}
\item Composition with the functor
$\E: \mB \subset \Fun(\mB^\op,\mS) \xrightarrow{(-)\ot\tu} \Fun(\mB^\op,\mV)$ defines an equivalence	
$$  \mathrm{Lin}\Fun^{\L}_\mV(\Fun^{}(\mB^\op, \mV),\mM) \to \Fun^{}(\mB, \mM ). $$

\item For any functor $\F:\mB \to \mM$ the corresponding functor
$\bar{\F}: \Fun^{}(\mB^\op, \mV) \to \mM$ is left adjoint to the functor
$\mM \to \Fun^{}(\mB^\op, \mV)$ corresponding to the functor 
$\mB^\op \times \mM \xrightarrow{\F^\op \times \mM} \mM^\op \times \mM \xrightarrow{\Mor_\mM(-,-)} \mV.$
\end{enumerate}
\end{lemma}

\begin{proof}
(1): The functor $(-)\ot\tu: \Fun(\mB^\op,\mS) \to \Fun(\mB^\op,\mV)$ exhibits $(\mV^\circledast)^{\mB^\op} \to \mV^\ot $ as the free $\infty$-category left tensored over $\mV$ compatible with small colimits on $\Fun(\mB^\op,\mS)$ as the composition $\mV \ot \Fun(\mB^\op,\mS) \to \mV \ot \Fun(\mB^\op,\mV) \to \Fun(\mB^\op,\mV)$ is an equivalence \cite[Proposition 4.8.1.17.]{lurie.higheralgebra}.

(2): For any functor $\F:\mB \to \mM$ the corresponding $\mV$-linear functor
$\bar{\F}: \Fun^{}(\mB^\op, \mV) \to \mM$ is left adjoint to a functor
$\R: \mM \to \Fun^{}(\mB^\op, \mV)$.
For any $\Z \in \mB, \Y \in \mM$ there is a canonical equivalence
$$\R(\Y)(\Z) \simeq \Mor_{\Fun(\mB^\op, \mV)}(\E(\Z),\R(\Y)) \simeq
\Mor_\mM(\bar{\F}(\E(\Z)),\Y) \simeq \Mor_\mM(\F(\Z),\Y).$$

\end{proof}

Recall that every $\infty$-category $\mM^\circledast \to \mV^\ot \times \mW^\ot$ weakly bitensored over $\infty$-operads canonically extends to a weakly bitensored $\infty$-category $\bar{\mM}^\circledast \to \mP\Env(\mV)^\ot \times \mP\Env(\mW)^\ot$ (Notation \ref{Exta2}), whose pullbacks to $ \mP\Env(\mV)^\ot \times \mW^\ot$ and $\mV^\ot \times \mP\Env(\mW)^\ot$
we abusively also denote by $\mM^\circledast$:

\begin{notation}\label{enrfun}
Let $\mM^\circledast \to \mV^\ot \times \mW^\ot$ be an $\infty$-category
weakly bitensored over $\infty$-operads.

\begin{itemize}
\item For any $\infty$-precategory $\mC$ weakly enriched in $\mV$ with space of objects $\X$ let $$\Fun^\mV(\mC,\mM)^\circledast:=\LMod_\mC(\Fun(\X,\mM))^\circledast \to \mW^\ot$$
be the weakly right tensored $\infty$-category of left $\mC$-modules in $$\Fun(\X, \mM)^\circledast \to \Quiv_\X(\mP\Env(\mV))^\ot \times \mW^\ot.$$
We call $\Fun^\mV(\mC,\mM)$ the $\infty$-category of $\mV$-enriched functors $\mC \to \chi(\mM).$

This terminology is justified by Proposition \ref{proooq}.

\vspace{2mm}

\item For any $\infty$-precategory $\mC$ weakly enriched in $\mW$ with space of objects $\X$ let $$\hspace{12mm} \mP_\mW^\mM(\mC)^\circledast := \RMod_\mC(\Fun(\X,\mM))^\circledast  \simeq (\LMod_{\mC^\op}(\Fun(\X,\mM))^\circledast)^\rev =(\Fun^{\mW^\rev}(\mC^\op,\mM)^\circledast)^\rev \to \mV^\ot $$
be the weakly left tensored $\infty$-category of right $\mC$-modules in $$\Fun(\X, \mM)^\circledast \to \mV^\ot \times \Quiv_\X(\mP\Env(\mW))^\ot,$$
where the middle equivalence follows from Remark \ref{swi}.
\vspace{1mm}
We call $\mP_\mV^\mM(\mC)$ the $\infty$-category of $\mW$-enriched presheaves $\mC \to \chi(\mM)$.
For $\mM=\mV$ we omit $\mM$ from the notation.

\end{itemize}	
\end{notation}

We also sometimes use the following variant of the last notation:

\begin{notation}\label{enrfun2}
Let $\mM^\circledast \to \mV^\ot \times \mW^\ot$ be a weakly bitensored $\infty$-category.
\begin{itemize}
\item For any $\mV$-enriched $\infty$-precategory $\mC$ with space of objects $\X$ let $$\Fun^\mV(\mC,\mM)^\circledast:=\LMod_\mC(\Fun(\X,\mM))^\circledast \to \mW^\ot$$
be the weakly right tensored $\infty$-category of left $\mC$-modules in $$\Fun(\X, \mM)^\circledast \to \Quiv_\X(\mV)^\ot \times \mW^\ot.$$
		
\item For any $\mW$-enriched $\infty$-precategory $\mC$ with space of objects $\X$ let $$\hspace{12mm} \mP_\mW^\mM(\mC)^\circledast := \RMod_\mC(\Fun(\X,\mM))^\circledast  \simeq (\LMod_{\mC^\op}(\Fun(\X,\mM))^\circledast)^\rev =(\Fun^{\mW^\rev}(\mC^\op,\mM)^\circledast)^\rev \to \mV^\ot $$
be the weakly left tensored $\infty$-category of right $\mC$-modules in $$\Fun(\X, \mM)^\circledast \to \mV^\ot \times \Quiv_\X(\mW)^\ot.$$
For $\mM=\mV$ we omit $\mM$ from the notation.
\end{itemize}	
\end{notation}

\begin{lemma}\label{prestor}

Let $\mM^\circledast \to \mV^\ot \times \mW^\ot$ be a bitensored $\infty$-category compatible with small colimits and $\mC$ a $\mW$-precategory with small space of objects $\X$. Then $\Fun^\mV(\mC,\mM)^\circledast \to \mW^\ot$ is a right tensored $\infty$-category compatible with small colimits.
Dually, $\mP_\mW^\mM(\mC)^\circledast \to \mV^\ot$ is a left tensored $\infty$-category compatible with small colimits.
If $\mM^\circledast \to \mV^\ot \times \mW^\ot$ is presentably bitensored, $\Fun^\mV(\mC,\mM)^\circledast \to \mW^\ot$ is presentably right tensored
and $\mP_\mW^\mM(\mC)^\circledast \to \mV^\ot$ is presentably left tensored.

\end{lemma}
\begin{proof}
By Proposition \ref{contino} the weakly bitensored $\infty$-category $\Fun(\X,\mM)^\circledast \to \mV^\ot \times \Quiv_\X(\mW)^\ot$ is a bitensored
$\infty$-category compatible with small colimits.
So the claim follows from Notation \ref{rightact} and Corollary \ref{comppres}.
\end{proof}

\begin{remark}
Let $\mM^\circledast \to \mV^\ot$ be a weakly left tensored $\infty$-category.
We apply Notation \ref{enrfun} to $\mM^\circledast \to \mV^\ot \times \emptyset^\ot$
to obtain an $\infty$-category $\Fun^\mV(\mC,\mM).$
Let $\mM^\circledast \to \mW^\ot$ be a weakly right tensored $\infty$-category.
We apply Notation \ref{enrfun} to $\mM^\circledast \to \emptyset^\ot \times \mW^\ot$
to obtain an $\infty$-category $\mP_\mW^\mM(\mC)$.
\end{remark}

\vspace{2mm}

\begin{notation}\label{envfa}
	
\begin{itemize}
\item
For any cocartesian $\rS$-family of small $\infty$-operads $\mV^\ot \to \rS \times \Ass$ classifying a functor $\alpha: \rS \to \Op_{\infty}$ let $\mP\Env^\rS(\mV)^\ot \to \rS \times \Ass$ be the bicartesian $ \rS$-family of presentably monoidal $\infty$-categories classifying the functor $\rS \xrightarrow{\alpha}\Op_\infty \xrightarrow{ \mP\Env(-)} \Alg(\Pr^\L)$.

\vspace{1mm}
\item 
For any cocartesian $\rS$-family of weakly left tensored $\infty$-categories $\mM^\circledast \to \mV^\ot$ classifying a functor $\beta: \rS \to \omega\LMod$ let $\mP\L\Env^\rS(\mM)^\circledast \to \mP\Env^\rS(\mV)^\ot$ be the bicartesian $\rS$-family of presentably left tensored $\infty$-categories classifying the functor $\rS \xrightarrow{\beta} \omega\LMod \xrightarrow{\mP\L\Env(-)} \LMod(\Pr^\L)$.

\vspace{1mm}

\item Let $\bar{\mM}^\circledast \subset \mP\L\Env^\rS(\mM)^\circledast $ be the full weakly left tensored subcategory spanned by $\mM$ so that $\bar{\mM}^\circledast\to \mP\Env^\rS(\mV)^\ot$ is a cocartesian $\rS$-family of weakly left tensored $\infty$-categories whose pullback to $\mV^\ot$ is $\mM^\circledast.$ 
\end{itemize}
\end{notation}

\begin{notation}

\begin{itemize}
\item For $\mV^\ot \to \rS \times \Ass$ the cocartesian $\rS$-family of $\infty$-operads classifying the projection $\Op_{\infty} \times \mS \to \Op_\infty$ and $\X \to \rS$ the left fibration $\mS_* \times \Op_\infty \to \mS\times \Op_\infty$ we write $\Quiv^\ot \to \mS \times \Op_\infty \times \Ass$ for $ \Quiv_{\X}^{\rS}(\mV)^\ot$	
and $\omega\Quiv^\ot \to \mS \times \Op_\infty \times \Ass$ for $ \Quiv_{\X}^{\rS}(\mP\Env^\rS(\mV))^\ot.$	

\vspace{1mm}
\item
For	$\mM^\circledast \to \mV^\ot$ the cocartesian $\rS$-family of
weakly left tensored $\infty$-categories classifying the projection $\omega\LMod \times \mS \to \omega\LMod$ and $\X \to \rS$ the left fibration $\mS_* \times \omega\LMod \to \mS\times \omega\LMod$ we write $\Sigma^\circledast \to \omega\LMod \times_{\Op_\infty} \omega\Quiv^\ot $ for $\Fun^\rS(\X, \bar{\mM})^\circledast \to \Quiv_{\X}^{\rS}(\mP\Env^\rS(\mV))^\ot.$ 
\end{itemize}
\end{notation}

\begin{remark}\label{reusoxy}
	
By Notation \ref{exxxx} (using Lemma \ref{pulll}) there is a canonical equivalence $$\Sigma \simeq \Fun^{\mS \times \omega\LMod}(\mS_* \times \omega\LMod , \mS \times \mU \times_{\Cat_\infty} \omega\LMod) \simeq \Fun^{\mS \times \Cat_\infty}(\mS_* \times\Cat_\infty, \mS \times \mU) \times_{\Cat_\infty} \omega\LMod$$ over $\mS \times \omega\LMod,$ where $\mU \to \Cat_\infty$ is the cocartesian fibration classifying the identity.
\end{remark}

\begin{remark}\label{hhhh}
By Corollary \ref{remaq} (4) the functor $\omega\Quiv^\ot \to \Op_\infty \times \mS \times \Ass$ is a map of bicartesian fibrations over $\Op_\infty$ and cartesian fibrations over $\mS$ and the functor $\Sigma^\circledast \to \omega\LMod \times_{\Op_\infty} \omega\Quiv^\ot $ is a map of cocartesian fibrations over $\omega\LMod$ and cartesian fibrations over $\mS$.
\end{remark}

\begin{notation}
Let $$\PreCat_\infty:= \Alg^{\mS \times \Op_\infty}(\Quiv), \hspace{3mm} \omega\PreCat_\infty:=\Alg^{\mS \times \Op_\infty}(\omega\Quiv).$$

We call $\PreCat_\infty, \omega\PreCat_\infty$ the $\infty$-categories of
small enriched, weakly enriched $\infty$-precategories.

\end{notation}

\begin{remark}\label{abc}By Remark \ref{hhhh} and Lemma \ref{fibb} (3) the functor $\PreCat_\infty \to \mS \times \Op_\infty $ is a map of cocartesian fibrations over $\Op_\infty$ and cartesian fibrations over $\mS$
and the functor $ \omega\PreCat_\infty \to \mS \times \Op_\infty $ is a map of bicartesian fibrations over $\Op_\infty$ and cartesian fibrations over $\mS$.	
\end{remark}

\begin{notation}
For any (not necessarily small) $\infty$-operad $\mV^\ot \to \Ass $ let $\PreCat_\infty^\mV, \omega\PreCat_\infty^\mV $ be the full subcategories of the fibers of the functors $\widehat{\PreCat}_\infty \to \widehat{\Op}_\infty, \omega\widehat{\PreCat}_\infty \to \widehat{\Op}_\infty $ over $\mV^\ot \to \Ass $
spanned by the $\infty$-precategories (weakly) enriched in $\mV$ whose space of objects is small.
\end{notation}

By Lemma \ref{fibb} (1) for any space $\X$ and small $\infty$-operad $\mV^\ot \to \Ass $ there are canonical equivalences$$ (\PreCat_\infty^\mV)_\X \simeq \Alg_{\Ass_\X}(\mV), \ (\omega\PreCat_\infty^\mV)_\X \simeq \Alg_{\Ass_\X}(\mP\Env(\mV)).$$ 

By Remark \ref{hhhh} and Lemma \ref{fibb} there are forgetful functors
\begin{equation}\label{for}
\LMod^{\omega\LMod \times \mS}(\Sigma) \to \Alg^{\omega\LMod \times \mS}(\omega\LMod \times_{\Op_\infty} \omega\Quiv) \simeq \omega\LMod \times_{\Op_\infty} \omega\PreCat_\infty, \end{equation}
\begin{equation}\label{for2}
\LMod^{\omega\LMod \times \mS}(\Sigma) \to \Sigma \end{equation}
over $\omega\LMod \times\mS$ that are maps of cocartesian fibrations over $ \omega\LMod$ and cartesian fibrations over $ \mS$.

For any (weakly) enriched $\infty$-precategory there is an opposite one:

\begin{remark}\label{oppos}
By Remark \ref{swi} there are canonical equivalences
$\Quiv \simeq \Quiv^\rev, \omega\Quiv \simeq \omega\Quiv^\rev$ over $\mS \times \Op_{\infty}$ that give rise to equivalences over $\mS \times \Op_\infty:$
$$\PreCat_\infty= \Alg^{\mS \times \Op_\infty}(\Quiv) \simeq\Alg^{\mS \times \Op_\infty}(\Quiv^\rev)\simeq \rev^*(\PreCat_\infty)$$
$$\omega\PreCat_\infty= \Alg^{\mS \times \Op_\infty}(\omega\Quiv) \simeq\Alg^{\mS \times \Op_\infty}(\omega\Quiv^\rev)\simeq \rev^*(\omega\PreCat_\infty).$$

\end{remark}

\vspace{1mm}

\subsection{The comparison functor}
\label{compq}
Now we are ready to construct the functor $$\chi: \omega\LMod \to \omega \PreCat_\infty $$ that endows the graph $ \mM^\simeq \times \mM^\simeq \to \mP\Env(\mV) $ of a weakly left tensored $\infty$-category $\mM^\circledast \to \mV^\ot$ with the structure of a $\mP\Env(\mV)$-enriched $\infty$-precategory.
To construct $\chi$ we take the following steps:

\begin{itemize}
\item First we construct a functor $$\rho: \omega\LMod \to \mP^{\Op_\infty}(\omega\PreCat_\infty)$$ over $\mS \times \Op_\infty,$ 
where $\mP^{\Op_\infty}(\omega\PreCat_\infty) \to \Op_\infty$
is a relative version of presheaf $\infty$-category, whose fiber over a small
$\infty$-operad $\mV^\ot \to \Ass $ is $\mP(\omega\PreCat_\infty^\mV)$.
The functor $\rho$ sends a weakly left tensored $\infty$-category 
$\mM^\circledast \to \mV^\ot$ to the presheaf on $\omega\PreCat^\mV_\infty$
mapping a $\mV$-precategory $\mC$ with space of objects $\X$ to $\LMod_\mC(\Fun(\X,\mM))^\simeq.$

\item In a second step we factor $\rho$ through the desired functor $\chi: \omega\LMod \to \omega\PreCat_\infty$ (Theorem \ref{pqp}) using a relative version of
Yoneda-embedding (Notation \ref{yone}): $$\omega\PreCat_\infty \hookrightarrow \mP^{\Op_\infty}(\omega\PreCat_\infty). $$  

\item Finally, we prove that $\chi$ is a map of cartesian fibrations over $\Op_\infty$ (Lemma \ref{leeem}).

\end{itemize}
\begin{notation}
Let $ \Lambda:=\Fun([1],\mS) \times_{\Fun(\{1\},\mS)} \omega\LMod$
be the pullback of evaluation at the target along the forgetful functor $\omega\LMod \to \mS$ sending $\mM^\circledast \to \mV^\ot$ to its maximal subspace $\mM^\simeq$. 
\end{notation}
The diagonal embedding $\mS \subset \Fun([1],\mS)$ yields an embedding
$ \omega\LMod \subset \Lambda.$
\vspace{1mm}

\begin{remark}
The functor $\Fun([1],\mS) \to \Fun(\{0\},\mS) \times \Fun(\{1\},\mS)$ is a
map of cocartesian fibrations over $ \Fun(\{1\},\mS)$ and cartesian fibrations over $ \Fun(\{0\},\mS)$.
Thus the functor $\Lambda \to \mS \times \omega\LMod$ is a map of cocartesian fibrations over $\omega\LMod $ and cartesian fibrations over $\Fun(\{0\},\mS).$
\end{remark}

\begin{lemma}
There is an inclusion $ \kappa: \Fun([1],\mS) \times_{\Fun(\{1\},\mS)} \Cat_\infty \subset \Fun^{\mS \times \Cat_\infty}(\mS_* \times\Cat_\infty, \mS \times \mU)$ over $\Fun(\{0\},\mS) \times \Cat_\infty$ that induces on the fiber over $(\X, \Y) \in \mS \times \Cat_\infty$ the inclusion $\Fun(\X,\Y)^\simeq \subset \Fun(\X,\Y),$ where we use $(-)^\simeq: \Cat_\infty \to \mS$ for the left hand pullback.
The functor $\kappa$ is a map of cocartesian fibrations over $ \Cat_\infty$ and cartesian fibrations over $ \mS$.
	
\end{lemma}

\begin{proof}
We set $\Theta_1:=\Fun([1],\mS) \times_{\Fun(\{1\},\mS)} \Cat_\infty$,
$\Theta_2:=  \Fun^{\mS \times \Cat_\infty}(\mS_* \times\Cat_\infty, \mS \times \mU).$
Functors $\rS \to \Theta_1$ over $\mS \times \Cat_\infty $ are classified by maps $\X \to \Y$ of cocartesian fibrations over $\rS$, where $\X \to \rS$ is a left fibration.
Functors $\rS \to \Theta_2$ over $\mS \times \Cat_\infty$ are
classified by functors $\X \to \Y$ over $\rS$ from a left fibration to a cocartesian fibration.
So we obtain a functor $\kappa:\Theta_1 \to \Theta_2$ over $\Fun(\{0\},\mS) \times \Cat_\infty$, which induces on the fiber over $(\X, \Y) \in \mS \times \Cat_\infty$ the inclusion $\Fun(\X,\Y)^\simeq \subset \Fun(\X,\Y)$ and so is an inclusion if we
have shown that $\kappa$ is a map of cocartesian fibrations over $ \Cat_\infty$ and cartesian fibrations over $ \mS$.

Let $\alpha$ be a morphism of $\Theta_1$ classified by a map $\X \to \Y$ of cocartesian fibrations over $[1]$ starting at a left fibration.
Then $\alpha$ is cocartesian over $\Cat_\infty$ (cartesian over $\Fun(\{0\},\mS)$) if and only if its image in $\Fun(\{0\},\mS)$ (in $\Cat_\infty$) is an equivalence.
This is equivalent to say that the left fibration $\X \to [1]$ is a right fibration (the cocartesian fibration $\Y \to [1]$ is a cartesian fibration whose cocartesian and cartesian morphisms coincide).
Let $\beta$ be a morphism of $\Theta_2$ classified by a functor
$\sigma: \X \to \Y$ over $[1]$ from a left fibration to a cocartesian fibration.
Then $\beta$ is cocartesian over $\Cat_\infty$ (cartesian over $\Fun(\{0\},\mS)$) if and only if its image in $\Fun(\{0\},\mS)$ (in $\Cat_\infty$) is an equivalence and $\sigma$ is a map of cocartesian fibrations over $[1]$. 
This follows from Lemma \ref{flafla}, where we use that a cocartesian fibration over $[1]$ classifies an equivalence if and only if it is a cartesian fibration whose cocartesian and cartesian morphisms coincide.

\end{proof} 
 
By Remark \ref{reusoxy} we obtain the following corollary:

\begin{corollary}\label{inc}

There is a canonical inclusion $\Lambda \subset \Sigma$,
which is a map of cocartesian fibrations over $ \omega\LMod$ and cartesian fibrations over $ \mS$.

\end{corollary}

\begin{notation}
Let $\theta$ be the composition $$\Lambda \times_{\Sigma}  \LMod^{\mS \times \omega\LMod}(\Sigma) \subset  \LMod^{\mS \times \omega\LMod}(\Sigma) \xrightarrow{} \omega\LMod \times_{\Op_\infty} \omega\PreCat_\infty$$
of the canonical inclusion and map (\ref{for}).
\end{notation}
\begin{remark}
So $\theta$ is a map of cocartesian fibrations over $ \omega\LMod$ and cartesian fibrations over $ \mS$.
\end{remark}
We construct $\rho$ via $\theta$ using Remark \ref{cla}, Lemma \ref{rifi}
and the following construction of relative presheaf $\infty$-category:

\begin{notation}
Let $\mR \subset \Fun([1],\Cat_\infty)$ be the full subcategory of right fibrations.
\end{notation}

\begin{notation}\label{clas}
 
For every cocartesian fibration $\mX \to \rS$
classifying a functor $\alpha: \rS \to \Cat_\infty$ we define the $\infty$-category of presheaves relative to $\rS$ as the pullback
$\mP^\rS(\mX):=\rS \times_{\Cat_\infty} \mR$
taken along $\alpha$ and evaluation at the target.
	
\end{notation}

\begin{remark}\label{cla}
Let $\mY\to \rS$ be a functor. A functor $\mY \to \mP^\rS(\mX) \subset \rS \times_{\Fun(\{1\},\Cat_\infty)}\Fun([1],\Cat_\infty) $ over $\rS$
corresponds to a natural transformation of functors $\mY \to \Cat_\infty$
whose target factors as $\mY \to \rS \xrightarrow{\alpha}\Cat_\infty$
and whose components are right fibrations, which is precisely classified by a map $\mZ \to \mX\times_\rS \mY$ of cocartesian fibrations over $\mY$ inducing on the fiber over any object of $\mY$ a right fibration.

\end{remark}

Next we will construct a relative Yoneda-embedding $\mX \hookrightarrow \mP^\rS(\mX)$
for any cocartesian fibration $\mX \to \rS$.

\begin{remark}\label{cocartf}
For every cocartesian fibration $\mX \to \rS$ the canonical map
$\mX^{[1]} \to \mX^{\{0\}} \times_\rS \mX^{\{1\}}$ of cocartesian fibrations over
$\rS$ is a map of cocartesian fibrations over $\mX^{\{1\}}$
since for every $\s \in \rS$ the induced functor
$\Fun([1],\mX_\s) \to \Fun(\{0\},\mX_\s) \times \Fun(\{1\},\mX_\s)$
is a map of cocartesian fibrations over $ \Fun(\{1\},\mX_\s)$
that reflects cocartesian morphisms.	
\end{remark}

\begin{notation}\label{yone}
For any cocartesian fibration $\mX \to \rS$ the canonical functor
$\mX^{[1]} \to \mX^{\{0\}} \times_\rS \mX^{\{1\}}$ is a map of cocartesian fibrations over $\mX^{\{1\}}$ (Remark \ref{cocartf}) and so is classified by a functor $\y: \mX \to \mP^\rS(\mX)$ over $\rS$ (Remark \ref{cla}) that is fully faithful and a map of cocartesian fibrations over $\rS$ (\cite
[Lemma 2.7.]{Heine2022}.

\end{notation}

\begin{lemma}\label{rifi}
The map $\theta$ of cocartesian fibrations over $\omega\LMod$ induces on the fiber over every weakly left tensored $\infty$-category $\mM^\circledast \to \mV^\ot$ a right fibration.  

\end{lemma}
\begin{proof}
As $\theta$ is a map of cartesian fibrations over $\mS$, also $\theta_\mM$ is.
So by \cite[Lemma A.1.8.]{haugseng_melani_safronov_2020} it is enough to check that $\theta_\mM$ induces on the fiber over every small space $\X$ a right fibration. The functor $\theta_\mM$ induces on the fiber over $\X$
the following functor induced by the left action of $\omega\Quiv_\X^{\mV}$ on $\Fun(\X,\mP\L\Env(\mM)):$
\begin{equation}\label{tgvk}
\Fun(\X,\mM)^\simeq \times_{\Fun(\X,\mM)}\LMod(\Fun(\X,\mM)) \to \Alg(\omega\Quiv_\X^{\mV}).
\end{equation}
(\ref{tgvk}) is a right fibration as the functor $\LMod(\Fun(\X,\mM)) \to \Alg(\omega\Quiv_\X^{\mV})$
is a cartesian fibration whose cartesian morphisms are those whose image in $\Fun(\X,\mM) $ is an equivalence.

\end{proof}

\begin{notation}
By Remark \ref{cla} and Lemma \ref{rifi} the functor $\theta$ classifies a functor $\rho: \omega\LMod \to \mP^{\Op_\infty}(\omega\PreCat_\infty)$ over $\Op_\infty.$
\end{notation}

Next we make $\mP^{\Op_\infty}(\omega\PreCat_\infty)$ to an $\infty$-category over $\mS$ and $\rho$ to a functor over $\mS$ (Lemma \ref{faac}).

\begin{remark}
Any cartesian fibration $\phi: \mY \to \rS$ such that for every $\s \in \rS$ the fiber $\mY_\s$ has an initial object admits a unique section sending each object $\s \in \rS$ to the initial object of the fiber $\mY_\s:$

Let $\mY' \subset \mY$ be the full subcategory of initial objects of each fiber.
The restriction $\phi': \mY' \subset \mY \to \rS$ is a cartesian fibration, where a $\phi'$-cartesian lift of a morphism $\s' \to \s$ in $\rS$
factors as the unique morphism $\A \to \B$ followed by a $\phi$-cartesian morphism $\B \to \C$.
As $\phi'$ has contractible fibers, $\phi'$ is an equivalence.

\end{remark}

By the last remark the cartesian fibration $\{[0]\}\times_\mS \omega\PreCat_\infty 
\to \Op_{\infty}$ (and so also the cartesian fibration $\omega\PreCat_\infty 
\to \Op_{\infty}$) admits a unique section $\kappa$ sending a small $\infty$-operad $\mV^\ot \to \Ass$ to the initial object of $\Alg(\mP\Env(\mV)) \simeq \{[0]\}\times_\mS \omega\PreCat^\mV_\infty$ lying over the tensor unit of $\mP\Env(\mV)^\ot \to \Ass.$ 

This section gives rise to a functor $\kappa^*: \mP^{\Op_\infty}(\omega\PreCat_\infty) \to \mP^{\Op_\infty}(\Op_\infty) \simeq \mS \times \Op_\infty $ over $\Op_\infty.$

\begin{lemma}\label{faac}\emph{ }
	
\begin{enumerate}
	
\item The restriction $\omega\PreCat_\infty \subset \mP^{\Op_\infty}(\omega\PreCat_\infty) \xrightarrow{\kappa^*} \mS \times \Op_\infty $ is the canonical functor.

\vspace{1mm}
\item The composition $$\omega\LMod \xrightarrow{\rho} \mP^{\Op_\infty}(\omega\PreCat_\infty) \xrightarrow{\kappa^*} \mS \times \Op_\infty $$ is the canonical functor sending $\mM^\circledast \to \mV^\ot$ to $(\mM^\simeq,\mV^\ot).$
\end{enumerate}	

\end{lemma}
\begin{proof}
	
1.: An equivalence over $\Op_{\infty}$ between the composition $\omega\PreCat_\infty \subset \mP^{\Op_\infty}(\omega\PreCat_\infty) \xrightarrow{\kappa^*} \mS \times \Op_\infty $ and the canonical functor is classified by an equivalence
over $\omega\PreCat_\infty $ between the pullback $\mW$ of the functor $\omega\PreCat_\infty^{[1]}\to \omega\PreCat_\infty^{\{0\}} $
along the functor $\Op_\infty \xrightarrow{\kappa} \omega\PreCat_\infty$,
and the projection $\mS_\ast \times_{\mS} \omega\PreCat_\infty \to \omega\PreCat_\infty.$
The map $\omega\PreCat_\infty \to \mS \times \Op_{\infty} $
of cocartesian fibrations over $\Op_{\infty}$ gives rise to a map
$\omega\PreCat_\infty^{[1]}\to \Fun([1],\mS) \times \Op_{\infty} $ 
of cocartesian fibrations over $\Op_{\infty}$ that induces a map 
$\omega\PreCat_\infty^{[1]}\to \Fun([1],\mS) \times_{\Fun(\{1\},\mS) } \omega\PreCat_\infty^{\{1\}}$ 
of cocartesian fibrations over $\Op_{\infty}$.
This functor restricts to a map
$\gamma: \mW \to \mS_* \times_{\mS } \omega\PreCat_\infty $
of cocartesian fibrations over $\Op_{\infty}$ that induces on the fiber over any small
$\infty$-operad $\mV^\ot \to \Ass$ the canonical functor
$(\omega\PreCat_\infty^\mV)_{ \tu_{\mP\Env(\mV)} / }  \to \mS_* \times_{\mS } \omega\PreCat_\infty^\mV$
between left fibrations over $\omega\PreCat_\infty^\mV$.
This functor induces on the fiber over any $\mV$-precategory $\mC$ with space of objects $\X$ the canonical map
$\omega\PreCat_\infty^\mV(\tu_{\mP\Env(\mV)}, \mC) \to \mS([0],\X) \simeq \X, $
whose fiber over $\T\in \X$
is the contractible space $\Alg(\mP\Env(\mV))(\tu_{\mP\Env(\mV)}, \mC(\T,\T)).$
Hence $\gamma$ is an equivalence.

\vspace{1mm}

2.: An equivalence over $\Op_{\infty}$ between the composition $\omega\LMod \xrightarrow{\rho} \mP^{\Op_\infty}(\omega\PreCat_\infty) \xrightarrow{\kappa^*} \mS \times \Op_\infty $ and the canonical functor is classified by an equivalence
over $\omega\LMod $ between the pullback $\mQ$ of the functor $ \Lambda \times_{\Sigma}  \LMod^{\mS \times \omega\LMod}(\Sigma) \xrightarrow{\theta} \omega\LMod \times_{\Op_\infty} \omega\PreCat_\infty$
along the functor $ \omega\LMod \simeq \omega\LMod \times_{\Op_\infty} \Op_\infty \xrightarrow{\omega\LMod \times_{\Op_\infty}\kappa}\omega\LMod \times_{\Op_\infty} \omega\PreCat_\infty$, and the projection $\mS_\ast \times_{\mS} \omega\LMod \to \omega\LMod.$ The composition
$\mQ \to  \Lambda \times_{\Sigma}  \LMod^{\mS \times \omega\LMod}(\Sigma) \to \Lambda$ factors as a map of cocartesian fibrations 
$\mQ \to \mS_\ast \times_{\mS} \omega\LMod \subset \Lambda $ over $\omega\LMod$ 
that induces on the fiber over any weakly left tensored $\infty$-category $\mM^\circledast \to \mV^\ot$ the canonical equivalence  
$\LMod_{\tu_{\mP\Env(\mV)}}(\mM)^\simeq \simeq \mM^\simeq$.
\end{proof}

Next we prove the main theorem of this subsection:
 
\begin{theorem}\label{pqp}
	
The functor $\rho: \omega\LMod \xrightarrow{} \mP^{\Op_\infty}( \omega\PreCat_\infty)$ over $\mS \times \Op_\infty$ factors as 
$$ \omega\LMod \xrightarrow{\chi} \omega\PreCat_\infty \subset \mP^{\Op_\infty}( \omega\PreCat_\infty)$$ for a unique functor 
$\chi: \omega\LMod \to \omega\PreCat_\infty$ over $\mS \times \Op_{\infty}$.
\end{theorem} 

Theorem \ref{pqp} follows immediately from the next proposition.
For that we fix the following notation:

\begin{notation}
Let $\mM^\circledast \to \mV^\ot$ be a weakly left tensored $\infty$-category and $\X \in \mM.$ 
The functor $\sigma:[1] \to \Ass$ corresponding to the morphism $\{0\} \subset [1]$ in $\Delta$ gives rise to an $\infty$-category $\Fun_{\Ass}([1],\mM^\circledast)$.
Let $$\mV[\X]:= \{(\X,\X)\} \times_{(\mM \times \mM)} \Fun_{\Ass}([1],\mM^\circledast) \to \mV$$ be the pullback of the canonical functor $\Fun_{\Ass}([1],\mM^\circledast) \to \mM^\circledast_{[1]} \times \mM^\circledast_{[0]} \simeq \mV \times \mM \times \mM$
along the functor $\mV \times \{(\X,\X)\} \to \mV \times \mM \times \mM.$
\end{notation}
\begin{remark}\label{remil}
The functor $\mV[\X]\to \mV$ is a right fibration classifying the presheaf $\Mul_{\mM}(-,\X;\X): \mV^\op \to \mS. $ So we may identify objects of $\mV[\X]$ with pairs $(\A, \alpha)$ consisting of an object $\A \in \mV$ and a multi-morphism $\alpha \in \Mul_{\mM}(\A,\X;\X)$. 
Moreover a pair $(\A, \alpha)$ is a final object in $\mV[\X]$ if and only if
$\alpha$ exhibits $\A$ as the endomorphism object of $\X.$
\end{remark}
Composition with $\sigma:[1] \to \Ass$ defines a functor $\LMod(\mM) \subset \Fun_{\Ass}(\Ass,\mM^\circledast) \to \Fun_{\Ass}([1],\mM^\circledast)$ that gives rise to a functor $\{\X\}\times_\mM \LMod(\mM) \to \Alg(\mV) \times_\mV \mV[\X]$ over $\Alg(\mV).$

\begin{proposition}\label{endfi}\label{urr}
Let $\mM^\circledast \to \mV^\ot$ be an $\infty$-category weakly left tensored
over an $\infty$-operad and $\X \in \mM$ an object that admits an endomorphism object (corresponding to a final object of $\mV[\X]$
by Remark \ref{remil}). The $\infty$-category
$\{\X\}\times_\mM \LMod(\mM)$ has a final object lying over the final object of $\mV[\X]$.
\end{proposition}

\begin{proof}
If $\mM^\circledast \to \mV^\ot$ is a weakly left tensored
$\infty$-category, this is \cite[Theorem 4.7.1.34.]{lurie.higheralgebra}. 
In the general case the embedding $\mM^\circledast \subset \mP\L\Env(\mM)^\circledast$ preserves endomorphism objects (Lemma \ref{morpre})
and induces an embedding
$\{\X\}\times_\mM \LMod(\mM) \subset \{\X\}\times_{\mP\L\Env(\mM)} \LMod(\mP\L\Env(\mM))$. So the result follows.

\end{proof}
\vspace{1mm}

For the proof of the next proposition recall that a right fibration
$\mC \to \mD$ is called representable if it classifies a representable presheaf
$\mD^\op \to \mS$, which is equivalent to say that $\mC$ has a final object. 

\begin{remark}\label{fina}
For any cartesian fibration
$\phi: \mC \to \mD$, where $\mD$ and the fibers of $\phi$ have a final object preserved by the fiber transports, also $\mC$ has a final object that is preserved by $\phi.$
This holds as for any $\X,\Y \in \mC$ and $\alpha: \phi(\X)\to \phi(\Y)$ 
there is an equivalence $\{\phi\} \times_{\mD(\phi(\X),\phi(\Y)) } \mC(\X,\Y) \simeq \mC_{\phi(\X)}(\X, \alpha^*(\Y)).$
\end{remark}
\begin{proposition}\label{proooq}
Let $\mM^\circledast \to \mV^\ot$ be an $\infty$-category weakly left tensored over an $\infty$-operad. 

The presheaf $$\rho(\mM): (\omega\PreCat_\infty^\mV)^\op \to \mS, \  \mC \mapsto \Fun_\mV(\mC,\mM)^\simeq=\LMod_\mC(\Fun(\X,\mM))^\simeq $$
is representable.
 
\end{proposition}

\begin{proof}
	
The presheaf $\rho(\mM): (\omega\PreCat_\infty^\mV)^\op \to \mS$
is classified by the right fibration
$$\theta_\mM: \mS_{/\mM^\simeq} \times_{\Fun^{\mS}(\mS_*, \mS \times \mM)} \LMod^\mS(\Fun^{\mS}(\mS_*, \mS \times \mM)) \to \omega\PreCat_\infty^\mV.$$
 
We will prove that $\theta_\mM$ classifies a representable presheaf, which is equivalent to say that $\theta_\mM$ admits a final object.
By Lemma \ref{flafla} the functor $\LMod^\mS(\Fun^{\mS}(\mS_*, \mS \times \mM)) \to \Fun^{\mS}(\mS_*, \mS \times \mM)$
is a cartesian fibration so that the pullback $\gamma: \mS_{/\mM^\simeq} \times_{\Fun^{\mS}(\mS_*, \mS \times \mM)} \LMod^\mS(\Fun^{\mS}(\mS_*, \mS \times \mM)) \to \mS_{/\mM^\simeq}$ is a cartesian fibration.
Since $\mS_{/\mM^\simeq}$ has a final object, by Remark \ref{fina} it is enough to see that the fibers of $\gamma$ admit a final object that is preserved by the fiber transports of $\gamma.$

The fiber of $\gamma$ over a map of small spaces $\alpha : \X \to \mM^\simeq$ 
is $\{\alpha\}\times_{\Fun(\X,\mM)} \LMod_{}(\Fun(\X,\mM))$.
For any map of small spaces $\sigma: \Y \to \X$ over $\mM^\simeq$
the induced fiber transport of $\gamma$ is the canonical functor
\begin{equation*} 
\sigma^*: \{\alpha\}\times_{\Fun(\X,\mM)} \LMod_{}(\Fun(\X,\mM)) \to \{\alpha \circ \sigma \}\times_{\Fun(\Y,\mM)} \LMod_{}(\Fun(\Y,\mM)). \end{equation*}

By Proposition \ref{xre} the map $\alpha: \X \to \mM$ 
admits an endomorphism object $$\Mor_{\Fun(\X,\mM)}(\alpha,\alpha) \in \Fun(\X \times \X,\mP\Env(\mV)).$$

So by Remark \ref{remil} the $\infty$-category 
$\Fun(\X \times \X,\mP\Env(\mV))[\alpha]$ has a final object, which by
Proposition \ref{endfi} lifts to a final object $\zeta(\X)$ of $\{\alpha\}\times_{\Fun(\X,\mM)} \LMod_{}(\Fun(\X,\mM))$.

The unique morphism
$ \sigma^*(\zeta(\X)) \to \zeta(\Y)$ in the target of $\sigma^*$
lies over the unique morphism in $\Fun(\Y \times \Y,\mP\Env(\mV))[\alpha \circ \sigma]$.
Using Remark \ref{remil} the latter map lies over the canonical morphism
$$\Mor_{\Fun(\X,\mM)}(\alpha,\alpha) \circ (\sigma \times \sigma) \to \Mor_{\Fun(\Y,\mM)}(\alpha \circ \sigma,\alpha \circ \sigma)$$
in $ \Fun(\Y \times \Y,\mP\Env(\mV))$, which is an equivalence by Proposition \ref{xre}.

\end{proof}

\vspace{1mm}
 
By Proposition \ref{bica} the canonical functor $\omega\LMod \to \Op_\infty $ is a cartesian fibration and by Remark \ref{abc} the functor $\omega \PreCat_\infty \to\Op_\infty$ is a cartesian fibration.
We complete this section by proving that the functor $\chi: \omega\LMod \to \omega\PreCat_\infty$ is a map of cartesian fibrations over $\Op_\infty$:

\begin{lemma}\label{leeem}
The functor  $\rho: \omega\LMod \to \mP^{\Op_\infty}( \omega\PreCat_\infty)$ and so the functor $\chi: \omega\LMod \to \omega\PreCat_\infty $ preserve cartesian morphisms over $\Op_\infty$.

\end{lemma}

\begin{proof}
For any weakly left tensored $\infty$-category $\mM^\circledast \to \mV^\ot$ and map of $\infty$-operads $\varphi: \mW^\ot \to \mV^\ot$ we want to see that the canonical morphism $\alpha: \rho(\varphi^\ast(\mM)) \to \varphi^\ast(\rho(\mM))$ in $ \Fun((\omega\PreCat_\infty^\mW)^\op, \mS)$ lying over the identity of $\X:= \mM^\simeq$ is an equivalence. 
The induced commutative square
\begin{equation*} 
\begin{xy}
\xymatrix{
\LMod(\Fun(\X,\varphi^\ast(\mM)))  \ar[d]^{} \ar[r]^{ }
&\LMod(\Fun(\X, \mM)) \ar[d]^{}
\\ \Alg(\Fun(\X \times \X, \mW))
\ar[r]^{}  & \Alg(\Fun(\X \times \X, \mV))}
\end{xy} 
\end{equation*}
yields on the fiber over any $\mW$-precategory $\mC$ with space of objects $\X=\mM^\simeq$ a functor $$\LMod_\mC(\Fun(\X,\varphi^\ast(\mM))) \to \LMod_{\varphi_\ast(\mC)}(\Fun(\X, \mM)),$$ which induces on maximal subspaces the component map $\alpha_\mC.$
So it is enough to check that the last square is a pullback square. This follows from the fact that the following square is a pullback square:
\begin{equation*} 
\begin{xy}
\xymatrix{
\Fun(\X,\varphi^\ast(\mM))^\circledast  \ar[d]^{} \ar[r]^{ }
&\Fun(\X, \mM)^\circledast \ar[d]^{}
\\ \Fun(\X \times \X, \mW)^\ot
\ar[r]^{}  & \Fun(\X \times \X, \mV)^\ot.}
\end{xy} 
\end{equation*}

\end{proof}

\section{A universal property of enriched presheaves}
\label{univpr}

In the last section we defined the $\infty$-category of enriched presheaves.
In this section we prove that it satisfies a universal property:

\begin{theorem}\label{tgre}
Let $\mV^\ot \to \Ass$ be a monoidal $\infty$-category compatible with small colimits, 
$\mW^\ot \to \Ass$ a monoidal $\infty$-category, $\mC$ a $\mV$-precategory with small space of objects $\X$ and $\mM^\circledast \to \mV^\ot \times \mW^\ot$ a bitensored $\infty$-category whose underlying left tensored 
$\infty$-category is compatible with small colimits.
There is a canonical equivalence of $\infty$-categories right tensored over $\mW$:
\begin{equation}\label{gdcy}
\Psi: \LinFun^\L_{\mV}(\mP_\mV(\mC),\mM)^\circledast \simeq \Fun^\mV(\mC,\mM)^\circledast
\end{equation}

\end{theorem}

\begin{remark}
Note that for $\mW^\ot =\Ass$ the forgetful functor $\BMod_{\mV,\mW} \to \LMod_\mV$
is an equivalence. So in this case Theorem \ref{tgre} applies to 
$\infty$-categories with left $\mV$-action compatible with small colimits.
\end{remark}
We explain now how we construct the functor $\Psi$ appearing in Theorem \ref{tgre}.
As the first step we construct for any generalized $\infty$-operad,
$\mV^\ot \to \Ass$ and $\mV$-precategory $\mC$ a $\mV$-enriched Yoneda-embedding,
which is an object of $\Fun^\mV(\mC, \mP_\mV(\mC))$ lying over a map of spaces $\iota$, and write $\L(\mC) \subset \mP_\mV(\mC)$ for the full subcategory spanned by the essential image of $\iota.$
In a second step we give for any weakly bitensored $\infty$-category $\mM^\circledast \to \mV^\ot \times \mW^\ot$ a map $$\alpha: \LaxLinFun_{\mV}(\L(\mC),\mM)^\circledast \to
\Fun(\Fun^\mV(\mC,\L(\mC)), \Fun^\mV(\mC,\mM))^\circledast$$
of $\infty$-categories weakly right tensored over $\mW.$
We define a map $\Phi$ of $\infty$-categories weakly right tensored over $\mW$
as the composition of $\alpha$ with the map $$ \Fun(\Fun^\mV(\mC, \L(\mC)),\Fun^\mV(\mC,\mM))^\circledast \to  \Fun^\mV(\mC,\mM)^\circledast$$ evaluating at the $\mV$-enriched Yoneda-embedding $\mC \to \chi(\L(\mC)).$

If the assumptions of Theorem \ref{tgre} hold, we define $\Psi$ as the composition
$$ \LinFun^\L_{\mV}(\mP_\mV(\mC),\mM)^\circledast \to \LaxLinFun_{\mV}(\L(\mC), \mM)^\circledast\xrightarrow{\Phi} \Fun^\mV(\mC,\mM)^\circledast,$$ where the first map is restriction.

\vspace{1mm}
In the following we explain how we construct the $\mV$-enriched Yoneda-embedding:
By Remark \ref{exxxx} for all spaces $\X,\Y$ and weakly bitensored $\infty$-category $\mM^\circledast \to \mV^\ot \times \mW^\ot $ there is a weakly bitensored $\infty$-category $\Fun(\X \times \Y,\mM)^\circledast \to \Quiv_\X(\mV) \times \Quiv_\Y(\mW)$.
We construct (Proposition \ref{imporo}) for any $\mV$-precategory $\mC$ with space of objects $\X$ and $\mW$-precategory $\mB$ with space of objects $\Y$ an equivalence $$\Fun^\mV(\mC, \mP^\mM_\mV(\mB)) \simeq \BMod_{\mC,\mB}(\Fun(\X \times\Y,\mM).$$
In the special case where we view $\mV$ as weakly bitensored over itself, we obtain a
canonical equivalence
\begin{equation}\label{rtgn}
\Fun^\mV(\mC, \mP_\mV(\mC)) \simeq \BMod_{\mC,\mC}(\Quiv_\X(\mV))
\end{equation}
and make the following definition:

\begin{definition}\label{yoneet}
Let $\mV^\ot \to \Ass$ be a generalized $\infty$-operad and $\mC$ a $\mV$-enriched $\infty$-precategory with space of objects $\X$.
The $\mV$-enriched Yoneda-embedding of $\mC$ is the object of $$\Fun^\mV(\mC, \mP_\mV(\mC)) \simeq \BMod_{\mC,\mC}(\Quiv_\X(\mV))$$ corresponding to
$\mC$ seen as bimodule over itself.
\end{definition}

\begin{notation}\label{younder}
Let $\iota: \X \to \mP_\mV(\mC)= \RMod_\mC(\Fun(\X,\mV))$ be the image of
the $\mV$-enriched Yoneda-embedding under the forgetful functor
$$\Fun^\mV(\mC, \mP_\mV(\mC))= \LMod_\mC(\Fun(\X,\mP_\mV(\mC))) \to \Fun(\X,\mP_\mV(\mC)).$$
\end{notation}

\begin{remark}
Let $\mV^\ot \to \Ass$ be a generalized $\infty$-operad and $\mC$ a $\mV$-enriched $\infty$-precategory with space of objects $\X$. 
By Lemma \ref{impor} there is a canonical equivalence $\Fun(\X, \mP_\mV(\mC)) \simeq  \RMod_\mC(\Fun(\X \times \X,\mV))$
that induces equivalence (\ref{rtgn}) and so sends $\iota$ to $\mC$ seen as right module over itself.
\end{remark}

\begin{remark}\label{tau}
Let $\mV^\ot \to \Ass$ be a monoidal $\infty$-category compatible with small colimits
and $\mC$ a $\mV$-precategory with small space of objects $\X$.
Then $\Fun(\X,\mV)^\circledast \to \Quiv_\X(\mV)^\ot$ is
a left tensored $\infty$-category so that the forgetful functor $ \mP_{\mV}(\mC) = \RMod_\mC(\Fun(\X,\mV)) \to \Fun(\X,\mV)$ admits a left adjoint.

The tensor unit of $\Quiv_\X(\mV)$ corresponds to the composition
$\tau: \X \to \Fun(\X,\mS) \xrightarrow{((-)\ot \tu)_*} \Fun(\X,\mV)$ of the Yoneda-embedding and the functor induced by the unique left adjoint monoidal functor $(-)\ot \tu: \mS \to \mV$ (Proposition \ref{contino}).
Since $\mC \in \RMod_\mC(\Fun(\X \times \X,\mV)) \simeq \Fun(\X, \mP_\mV(\mC)) $
viewed as right module over itself is free on the tensor unit of $\Quiv_\X(\mV),$
the functor $\iota: \X \to \mP_{\mV}(\mC)$ corresponding to $\mC \in \RMod_\mC(\Fun(\X \times \X,\mV)) $
factors as $\tau$ followed by the $\mV$-linear left adjoint $\mF: \Fun(\X,\mV) \to \mP_{\mV}(\mC)$ of the $\mV$-linear forgetful functor $\nu:  \mP_{\mV}(\mC) \to \Fun(\X,\mV).$

By the Yoneda lemma for any $\Z \in \X$ there is a $\mV$-linear adjunction
$ (-) \otimes \tau(\Z): \mV \rightleftarrows \Fun(\X,\mV):\ev_\Z, $
where the right adjoint evaluates at $\Z.$
Forming the composed $\mV$-linear adjunction
$$ \iota(\Z)\simeq \mF \circ ((-) \otimes \tau(\Z)): \mV \rightleftarrows \Fun(\X,\mV) \rightleftarrows \mP_{\mV}(\mC): \ev_\Z \circ \nu$$
gives an equivalence $\mathrm{Mor}_{\mP_{\mV}(\mC)}(\iota(\Z),-) \simeq \ev_\Z \circ \nu$ of lax $\mV$-linear functors $\mP_{\mV}(\mC) \to \mV.$
In particular, $\mathrm{Mor}_{\mP_{\mV}(\mC)}(\iota(\Z),-)$ is $\mV$-linear and preserves small colimits and for any $\rH \in \mP_{\mV}(\mC), \V \in \mV$ and $\Z \in \X$ the following map is an equivalence giving an enriched version of Yoneda-lemma: $${\mP_{\mV}(\mC)}(\V \ot \iota(\Z),\rH) \to \mV(\V \ot \iota(\Z)(\Z),\rH(\Z)) \to \mV(\V, \rH(\Z)).$$

\end{remark}

\vspace{1mm}

\begin{notation}\label{younder}
Let $\L(\mC)^\circledast \subset \mP_\mV(\mC)^\circledast$ be the full subcategory with weak left $\mV$-action spanned by the essential image of $\iota: \X \to \mP_\mV(\mC)$.

\end{notation}
\begin{remark}\label{yott}
The embedding $\L(\mC)^\circledast \subset \mP_\mV(\mC)^\circledast $ induces a pullback square
\begin{equation*}
\begin{xy}
\xymatrix{
\Fun^\mV(\mC, \L(\mC)) \ar[d]^{} \ar[rrr]
&&&\Fun^\mV(\mC, \mP_\mV(\mC)) \ar[d]^{}
\\ \Fun(\X,\L(\mC))
\ar[rrr]^{} &&& \Fun(\X,\mP_\mV(\mC)).
}
\end{xy} 
\end{equation*} 
Hence the $\mV$-enriched Yoneda-embedding belongs to the full subcategory
$\Fun^\mV(\mC, \L(\mC)).$ 
\end{remark}

For later applications we extend the last notation to the case of weakly enriched
$\infty$-precategories:

\begin{notation}
Let $\mV^\ot \to \Ass$ be an $\infty$-operad and $\mC$ an $\infty$-precategory weakly enriched in $\mV$.
We write $\L(\mC)^\circledast \subset \mP_{\mP\Env(\mV)}(\mC)^\circledast $ for the full subcategory with weak left $\mP\Env(\mV)$-action spanned by the essential image of $\iota: \X \to \mP_{\mP\Env(\mV)}(\mC)$.
	
\end{notation}

We start with constructing the $\mV$-enriched Yoneda-embedding by proving the following lemma:

\begin{lemma}\label{impor}

Let $\mM^\circledast \to \mV^\ot \times \mW^\ot$ be an $\infty$-category with weak biaction and $\mC$ a $\mW$-precategory with space of objects $\Y.$
For any space $\X$ there is a canonical equivalence
$$\Fun(\X,\RMod_\mC(\Fun(\Y,\mM)))^\circledast \simeq \RMod_\mC(\Fun(\X \times \Y,\mM))^\circledast $$
of $\infty$-categories weakly left tensored over $\Quiv_\X(\mV)$ compatible with the forgetful functors. 
\end{lemma}

\begin{proof}
For any weakly left tensored $\infty$-category $\mD^\circledast \to \mV^\ot$ 
and associative algebra $\A $ in $\mV$ by Remark \ref{coala} there is a canonical equivalence
$$ \LaxLinFun_{\mV^\ot}(\mD^\circledast,\RMod_\A(\mM)^\circledast)\simeq \LaxLinFun_{\mV^\ot,\Ass}(\mD^\circledast \times \Ass, \Ass \times_{\mW} \mM^\circledast). $$
So for any weakly left tensored $\infty$-category $\mE^\circledast \to \Quiv_\X(\mV)$ we obtain canonical equivalences
$$\LaxLinFun_{\Quiv_\X(\mV)}(\mE^\circledast,\RMod_\mC(\Fun(\X \times \Y,\mM))^\circledast)\simeq$$$$ \LaxLinFun_{\Quiv_\X(\mV),\Ass}(\mE^\circledast \times \Ass, \Ass \times_{\Quiv_{\Y}(\mW)} \Fun(\X \times\Y,\mM)^\circledast)\simeq $$
$$ \LaxLinFun_{\Quiv_\X(\mV)_\X^\ot, \Ass_\Y}(\mE^\circledast \times_{\Ass} \Ass_\X \times \Ass_\Y, \Quiv_\X(\mV)_\X^\ot \times_{\mV^\ot} \mM^\circledast \times_{\mW^\ot} \Ass_\Y)$$
and
$$\LaxLinFun_{\Quiv_\X(\mV)}(\mE^\circledast,\Fun(\X,\RMod_\mC(\Fun(\Y,\mM)))^\circledast) \simeq $$$$ \LaxLinFun_{\Quiv_\X(\mV)^\ot_\X}(\mE^\circledast \times_{\Ass} \Ass_\X, \Quiv_\X(\mV)_\X^\ot \times_{\mV^\ot} \RMod_\mC(\Fun(\Y,\mM))^\circledast) \simeq $$ 
$$\LaxLinFun_{\Quiv_\X(\mV)_\X^\ot}(\mE^\circledast \times_{\Ass} \Ass_\X, \RMod_\mC(\Quiv_\X(\mV)_\X^\ot \times_{\mV^\ot} \Fun(\Y,\mM)^\circledast))\simeq $$ 
$$\LaxLinFun_{\Quiv_\X(\mV)_\X^\ot,\Ass}(\mE^\circledast \times_{\Ass} \Ass_\X \times \Ass, \Quiv_\X(\mV)_\X^\ot \times_{\mV^\ot} \Fun(\Y,\mM)^\circledast \times_{\Quiv_{\Y}(\mW)^\ot} \Ass) $$ 
$$\simeq \LaxLinFun_{\Quiv_\X(\mV)_\X^\ot, \Ass_\Y}(\mE^\circledast \times_{\Ass} \Ass_\X \times \Ass_\Y, \Quiv_\X(\mV)_\X^\ot \times_{\mV^\ot} \mM^\circledast \times_{\mW^\ot} \Ass_\Y). $$	
	
\end{proof}

From Lemma \ref{impor} we obtain the following proposition:

\begin{proposition}\label{imporo}

Let $\mM^\circledast \to \mV^\ot \times \mW^\ot$ be a weakly bitensored $\infty$-category, $\mB$ a $\mV$-precategory with space of objects $\X$	
and $\mC$ a $\mW$-precategory with space of objects $\Y.$
There is a canonical equivalence
$$\Fun^\mV(\mB,\mP^\mM_\mW(\mC)) \simeq \BMod_{\mB,\mC}(\Fun(\X \times \Y,\mM)) $$
compatible with the forgetful functors.

\end{proposition}

\begin{proof}
The desired equivalence is the composition $$\Fun^\mV(\mB,\mP^\mM_\mW(\mC))
=\LMod_\mB(\Fun(\X,\RMod_\mC(\Fun(\Y,\mM)))) \simeq $$$$ \LMod_\mB(\RMod_\mC(\Fun(\X \times \Y,\mM))) \simeq \BMod_{\mB,\mC}(\Fun(\X \times \Y,\mM)), $$
where the first equivalence is provided by Lemma \ref{impor} and the second one is due to Proposition \ref{compright}.

\end{proof}

Now we are ready to construct the functor $\Psi.$ 
For that we prove the following proposition:

\begin{proposition}\label{qqqiv}
Let $\X$ be a space, $\mM^\circledast \to \mV^\ot$ a weakly left tensored $\infty$-category and $\mN^\circledast \to \mV^\ot \times \mW^\ot$ a weakly bitensored $\infty$-category. There is a canonical map
$$\theta: \LaxLinFun_\mV(\mM,\mN)^\circledast \to
\LaxLinFun_{\Quiv_\X(\mV)}(\Fun(\X,\mM),\Fun(\X,\mN))^\circledast$$
of $\infty$-categories weakly right tensored over $\mW$.
\end{proposition}

\begin{proof}

Let $\mO^\circledast \to \mW^\ot$ be a weakly right tensored $\infty$-category.
The canonical map
$$ (\Fun(\X,\mM)^\circledast \times \mO^\circledast) \times_{(\Ass \times \Ass)} (\Ass_{\X} \times \Ass) \simeq (\Fun(\X,\mM)^\circledast \times_\Ass \Ass_\X) \times \mO^\circledast \to \mM^\circledast \times \mO^\circledast $$
of weakly bitensored $\infty$-categories lying over the maps 
$\Quiv_\X(\mV)^\ot \times_\Ass \Ass_\X \to \mV^\ot, \id : \mW^\ot \to \mW^\ot$
of generalized $\infty$-operads is adjoint to a map 
$$ \beta: \Fun(\X,\mM)^\circledast \times \mO^\circledast \to \Fun(\X,\mM \times \mO)^\circledast$$
of $\infty$-categories weakly bitensored over $\Quiv_\X(\mV),\mW$.

The map $\beta$ and the lax $\mV,\mW$-linear evaluation 
$ \mM^\circledast \times \LaxLinFun_\mV(\mM,\mN)^\circledast \to \mN^\circledast$
give rise to a map $$\Fun(\X,\mM)^\circledast \times \LaxLinFun_\mV(\mM,\mN)^\circledast \xrightarrow{\beta} \Fun(\X,\mM \times \LaxLinFun_\mV(\mM,\mN))^\circledast \to\Fun(\X,\mN)^\circledast $$
of $\infty$-categories weakly bitensored over $\Quiv_\X(\mV),\mW$
adjoint to the desired map $\theta$ of $\infty$-categories weakly right tensored over $\mW.$

\end{proof}

\begin{notation}
Let $\mN^\circledast \to \mV^\ot \times \mW^\ot$ be a weakly bitensored $\infty$-category, $\mM^\circledast \to \mV^\ot$ a weakly left tensored $\infty$-category and $\mC$ a $\mV$-precategory with space of objects $\X$.
Let $\alpha$ be the composition 
\begin{equation*}
\LaxLinFun_{\mV}(\mM,\mN)^\circledast \xrightarrow{\theta}\LaxLinFun_{\Quiv_\X(\mV)}(\Fun(\X,\mM),\Fun(\X,\mN))^\circledast \to$$$$
\Fun(\Fun^\mV(\mC,\mM), \Fun^\mV(\mC,\mN))^\circledast
\end{equation*}
of maps of $\infty$-categories weakly right tensored over $\mW$,
where the first map is given by Proposition \ref{qqqiv} and the second map
by Remark \ref{jbfigh} using that $\Fun^\mV(\mC,\mM)^\circledast = \LMod_\mC(\Fun(\X,\mM))^\circledast.$
	
\end{notation}

\begin{notation}\label{fuuu}
Let $\mM^\circledast \to \mV^\ot \times \mW^\ot$ be a weakly bitensored $\infty$-category and $\mC$ a $\mV$-precategory with space of objects $\X$.
Let $\Phi$ be the composition 
$$ \LaxLinFun_{\mV}(\L(\mC),\mM)^\circledast \xrightarrow{\alpha}
\Fun(\Fun^\mV(\mC,\L(\mC)), \Fun^\mV(\mC,\mM))^\circledast\to  \Fun^\mV(\mC,\mM)^\circledast$$
of maps of $\infty$-categories weakly right tensored over $\mW$, where the second map evaluates at the $\mV$-enriched Yoneda-embedding of $\mC$ (Definition \ref{yoneet}, Remark \ref{yott}).

\end{notation}

\begin{remark}
There is a commutative square of $\infty$-categories weakly right tensored over $\mW$
\begin{equation}\label{sqq}
\begin{xy}
\xymatrix{
\LaxLinFun_{\mV}(\L(\mC),\mM)^\circledast \ar[d]^{} \ar[rrr]^{\Phi}
&&&\Fun^\mV(\mC,\mM)^\circledast \ar[d]^{}
\\ \Fun(\L(\mC),\mM)^\circledast
\ar[rrr]^{\iota^\ast} &&& \Fun(\X,\mM)^\circledast,
}
\end{xy} 
\end{equation}  
where the vertical functors are the forgetful functors and $\iota: \X \to \L(\mC)\subset \mP_\mV(\mC)$ is defined in \ref{younder}.	
\end{remark}

\begin{remark}
Let $\mM^\circledast \to \mV^\ot \times \mW^\ot$ be a bitensored $\infty$-category compatible with small colimits and $\mN^\circledast \to \mV^\ot$ a weakly left tensored $\infty$-category.
Since the right $\mW$-action on $\LaxLinFun_{\mV}(\mN,\mM)$ is object-wise, it
restricts to $\LinFun^\L_{\mV}(\mN,\mM)$.
\end{remark}

\begin{notation}
 
Let $\mM^\circledast \to \mV^\ot \times \mW^\ot$ a bitensored $\infty$-category 
whose underlying $\infty$-category with left $\mV$-action is compatible with small colimits and $\mC$ a $\mV$-precategory with small space of objects $\X$.
We write $\Psi$ for the composition
$$ \LinFun^\L_{\mV}(\mP_\mV(\mC),\mM)^\circledast \to \LaxLinFun_{\mV}(\L(\mC), \mM)^\circledast\xrightarrow{\Phi} \Fun^\mV(\mC,\mM)^\circledast$$
of maps of $\infty$-categories weakly right tensored over $\mW$, where the first functor is restriction.

\end{notation}

\begin{remark}
There is a commutative square of $\infty$-categories right tensored over $\mW$
\begin{equation}\label{sqqhal}
\begin{xy}
\xymatrix{
\LinFun^\L_{\mV}(\mP_\mV(\mC),\mM)^\circledast \ar[d]^{} \ar[rrr]^{\Psi}
&&&\Fun^\mV(\mC,\mM)^\circledast \ar[d]^{}
\\ \LinFun^\L_{\mV}(\Fun(\X,\mV),\mM)^\circledast
\ar[rrr]^{\simeq} &&& \Fun(\X,\mM)^\circledast,
}
\end{xy} 
\end{equation}  
where the left vertical functor precomposes with the unique small colimits preserving $\mV$-linear functor $\Fun(\X,\mV) \to \mP_\mV(\mC)$ extending $\iota: \X \to \mP_\mV(\mC)$ and the right vertical functor forgets. Thus $\Psi$ is $\mW$-linear.
\end{remark}

Next we will show that $\Psi$ is an equivalence proving Theorem \ref{tgre}:

\begin{proof}[Theorem \ref{tgre}]

Let $\phi: \Fun(\X,\mV) \rightleftarrows \RMod_\mC(\Fun(\X,\mV)): \gamma $	
be the free-forgetful adjunction.
The right adjoint $\gamma$ is $\mV$-linear and makes the left adjoint $\phi$
canonically $\mV$-linear since the composition $\gamma \circ \phi \simeq (-) \otimes \mC$ is $\mV$-linear.
So this adjunction yields an adjunction
$$\gamma^\ast: \Fun(\X,\mM) \simeq \LinFun^\L_\mV(\Fun(\X,\mV),\mM) \rightleftarrows \LinFun^\L_\mV(\RMod_\mC(\Fun(\X,\mV)),\mM): \phi^\ast. $$

By the commutativity of square (\ref{sqqhal}) it will be enough to prove the following two conditions:
\begin{enumerate}
\item The left vertical functor $\phi^\ast$ in the square is monadic.
\item For any functor $\rH:\X \to \mM$ the canonical map
$$\theta: \mC \otimes \rH \to \mC \otimes \phi^\ast(\gamma^\ast(\rH)) \simeq \mC \otimes \nu(\Psi(\gamma^\ast(\rH))) \to \nu(\Psi(\gamma^\ast(\rH))) $$ in
$\Fun(\X,\mM) $ is an equivalence.
\end{enumerate}

We start with the first condition: the right adjoint $\phi^\ast$ is conservative since $\RMod_\mC(\Fun(\X,\mV))$ is generated by the free right $\mC$-modules under geometric realizations. 
Moreover $\phi^\ast$ preserves small colimits by Lemma \ref{fghhnml}. 

To check (2) it is enough to prove the case that $\mM= \Fun(\X,\mV)$
and $\rH:\X \to \mM = \Fun(\X,\mV)$ adjoint to the tensor unit $\E: \X \times \X \to \mV$ of $\Quiv_\X(\mV)$ since both vertical adjunctions in square (\ref{sqq}) are natural in $\mM$
with respect to $\mV$-linear functors preserving small colimits.
For these choices of $\mM$ and $\rH$ the left $\Quiv_\X(\mV)$-action on $\Fun(\X,\mM) \simeq \Fun(\X\times\X,\mV)$ is the left action of
the monoidal $\infty$-category $\Quiv_\X(\mV)$ on itself (Lemma \ref{comppp}).
The unit $\E \to \phi^\ast(\gamma^\ast(\E))$ of the adjunction is the morphism 
$$\E \simeq ((-) \otimes \E) \circ \E \to ((-) \otimes \mC) \circ \E $$
induced by the unit $\E \to \mC$ of the associative algebra, where $(-) \otimes \mC, \ (-) \otimes \E \simeq \id: \Fun(\X,\mV) \to  \Fun(\X,\mV)$ are formed using the right $\Quiv_\X(\mV)$-action on $\Fun(\X,\mV).$
By Lemma \ref{diagg} the right $\Quiv_\X(\mV)$-action on $\Fun(\X,\mV)$ is compatible with the diagonal left $\mV$-action, which yields a left $\Quiv_\X(\mV)$-action on $\Fun(\X,\Fun(\X,\mV)) \simeq \Fun(\X \times \X,\mV)$ compatible with the diagonal right $\Quiv_\X(\mV)$-action. By Lemma \ref{comppp} this biaction is the biaction of the monoidal $\infty$-category $\Quiv_\X(\mV)$ on itself.
Consequently for any $\mB \in \Fun(\X \times \X,\mV)$
the induced functor $$\Fun(\X, (-) \otimes \mB): \Fun(\X,\Fun(\X,\mV)) \to \Fun(\X,\Fun(\X,\mV))$$ is canonically equivalent to the functor
$(-) \otimes \mB: \Fun(\X \times\X,\mV) \to \Fun(\X \times\X,\mV)$
induced by the monoidal structure on $\Fun(\X \times\X,\mV)$.
Hence the unit $\E \to \phi^\ast(\gamma^\ast(\E))$ of the adjunction is the unit $\E \to \mC$ of the associative algebra.
By definition of $\Psi$ the left $\mC$-module structure on $\nu(\Psi(\gamma^\ast(\rH))):\X \to \mM=\Fun(\X,\mV)$ is the left $\mC$-module structure on $\mC$ coming from the associative algebra structure in $\Quiv_\X(\mV) $.
Therefore the map $\mC \otimes \nu(\Psi(\gamma^\ast(\rH))) \to \nu(\Psi(\gamma^\ast(\rH))) $ is the multiplication map $\mC \otimes \mC \to \mC.$ 

\end{proof}

By the next proposition the $\mV$-enriched Yoneda-embedding $\mC \to \chi(\mP_\mV(\mC))$ is fully faithful in the appropriate enriched sense:

\begin{proposition}\label{yooon}
Let $\mV^\ot \to \Ass$ be an $\infty$-operad and $\mC$ a $\mV$-precategory.
The $\mV$-enriched Yoneda-embedding $\mC \to \chi(\mP_\mV(\mC))$
induces an equivalence on underlying graphs. 
\end{proposition}

\begin{proof}

A $\mV$-enriched functor $\mC \to \chi(\mP_\mV(\mC))$ corresponds to a left
$\mC$-module with respect to the left action of $\Quiv_\X(\mV)$
on $\Fun(\X,\mP_\mV(\mC))$, where the left $\mC$-module exhibits $\mC$ as endomorphism algebra if and only if the corresponding $\mV$-enriched functor induces an equivalence on underlying graphs.
The associative algebra $\mC$ in $\Quiv_\X(\mV) $ gives rise to a right $\mC$-module $\mC'$ in $\Quiv_\X(\mV) $ and to a $(\mC,\mC)$-bimodule in $\Quiv_\X(\mV) $ that corresponds to a left $\mC$-module structure on $\mC'$ with respect to the left action of $\Quiv_\X(\mV)$ on $ \RMod_\mC(\Quiv_\X(\mV)).$
By Corollary \ref{impor} there is a $\Quiv_\X(\mV) $-linear equivalence 
$\RMod_\mC(\Quiv_\X(\mV)) \simeq \Fun(\X, \mP_\mV(\mC))
$ under which the $\mV$-enriched Yoneda-embedding
corresponds to the left $\mC$-module structure on $\mC'$,
which exhibits $\mC$ as endomorphism algebra by \cite[Theorem 4.8.5.5.]{lurie.higheralgebra}.
\end{proof}

We add the following Proposition \ref{nice} that is a corollary of Proposition \ref{imporo} and Lemma \ref{hacon}.

\begin{notation}\label{switchh}
For any generalized $\infty$-operads $\mV^\ot \to \Ass,\mW^\ot \to \Ass $ let  $$\theta_{\mV,\mW}: \mV^\ot \times_\Ass (\mW^\ot)^\rev \to \mV^\ot \times (\mW^\ot)^\rev \simeq \mV^\ot \times \mW^\ot $$
be the composition of projection and the functor induced by the 
equivalence $(\mW^\ot)^\rev \simeq \mW^\ot$.
Taking pullback along $\theta_{\mV,\mW}$ induces a functor
$(\theta_{\mV,\mW})^*: \omega\BMod_{\mV,\mW} \to \omega\LMod_{\mV \times \mW^\rev}.$ 
\end{notation}

\begin{proposition}\label{nice}

For any weakly bitensored $\infty$-category $\mM^\circledast \to \mV^\ot \times (\mW^\ot)^\rev,$ any $ \mV$-precategory $\mB$ with space of objects $\X$ and
$\mW$-precategory $\mC$ with space of objects $\Y$ there is a canonical equivalence
$$\Fun^\mV(\mB,\Fun^\mW(\mC,\mM)) \simeq \Fun^{\mV \times \mW}(\mB \times \mC, \theta_{\mV,\mW^\rev}^*(\mM))$$
compatible with the forgetful functors.
\end{proposition}

\begin{proof}
There is a canonical equivalence
$$\Fun^\mV(\mB,\Fun^\mW(\mC,\mM)) \simeq \BMod_{\mB,\mC^\op}(\Fun(\X \times \Y,\mM)) \simeq $$$$ \LMod_{(\mB, \mC)}(\theta_{\Quiv_\X(\mV),\Quiv_\Y(\mW)^\rev}^*(\Fun(\X, \times \Y,\mM))) \simeq \LMod_{\mB \times \mC}(\Fun(\X \times \Y,\theta_{\mV,\mW^\rev}^*(\mM))) =$$$$ \Fun^{\mV \times \mW}(\mB \times \mC, \theta_{\mV,\mW^\rev}^*(\mM))$$
compatible with the forgetful functors,
where the first equivalence is by Proposition \ref{imporo},
the second equivalence is by Proposition \ref{eqaybbn} and the third one is by Proposition \ref{hacon}.	
\end{proof}

\begin{notation}For any generalized $\infty$-operads $\mV^\ot \to \Ass,\mW^\ot \to \Ass $ and spaces $\X,\Y$ let $\rho$ be the canonical map of generalized $\infty$-operads $$\Quiv_\X(\mV)^\ot \times_\Ass (\Quiv_\Y(\mW)^\ot)^\rev \simeq \Quiv_\X(\mV)^\ot \times_\Ass \Quiv_\Y(\mW^\rev)^\ot \to \Quiv_{\X \times \Y}(\mV \times \mW^\rev)^\ot.$$

\end{notation}

\begin{lemma}\label{hacon}
For any weakly bitensored $\infty$-category $\mM^\circledast \to \mV^\ot \times \mW^\ot$ and spaces $\X,\Y$ there is a canonical equivalence
\begin{equation}\label{alf}
\theta_{\Quiv_\X(\mV),\Quiv_\Y(\mW)}^*(\Fun(\X \times \Y,\mM)^\circledast) \simeq \rho^*(\Fun(\X \times \Y,\theta_{\mV,\mW}^*(\mM))^\circledast)\end{equation}
of $\infty$-categories weakly left tensored over $\Quiv_\X(\mV) \times \Quiv_\Y(\mW)^\rev$ that induces on underlying $\infty$-categories the 
identity.
\end{lemma}
 
\begin{proof}
It will be enough to construct equivalence (\ref{alf}) for $\mP\B\Env(\mM)$ since (\ref{alf}) for $\mP\B\Env(\mM)$ restricts to $\mM$ 
by the description on underlying $\infty$-categories.
So we can reduce to the case that $\mM^\circledast \to \mV^\ot \times \mW^\ot$ is a presentably bitensored $\infty$-category, where we set $\widetilde{\mV}^\ot:= \Quiv_\X(\mV)^\ot, \widetilde{\mW}^\ot:= \Quiv_\Y(\mW)^\ot$.

Let $ \mN^\circledast \to \widetilde{\mV}^\ot \times \widetilde{\mW}^\ot$ be a weakly bitensored $\infty$-category
and $\beta: \mN^\circledast \to \Fun(\X \times \Y, \mM)^\circledast$
a lax $\widetilde{\mV},\widetilde{\mW}$-linear functor.
By the universal property (Remark \ref{paraph}) of 
$ \Fun(\X \times \Y, \mM)^\circledast \to \widetilde{\mV}^\ot \times \widetilde{\mW}^\ot$
the functor $\beta$ corresponds to a lax 
$\widetilde{\mV}_\X^\ot,\widetilde{\mW}^\ot_\Y$-linear functor
$\gamma: \mN^\circledast_{\X,\Y} \to \widetilde{\mV}^\ot_\X \times_{\mV^\ot} \mM^\circledast \times_{\mW^\ot} \widetilde{\mW}^\ot_\Y$.

The induced lax $ \widetilde{\mV}^\ot_\X \times_\Ass (\widetilde{\mW}^\ot_\Y)^\rev \simeq (\widetilde{\mV}^\ot \times_\Ass (\widetilde{\mW}^\ot)^\rev)_{\X \times\Y} $-linear functor
$$\theta^*_{\widetilde{\mV},\widetilde{\mW}}(\mN^\circledast)_{\X \times \Y} \simeq \theta^*_{\widetilde{\mV}_\X,\widetilde{\mW}_\Y}(\mN^\circledast_{\X,\Y}) \xrightarrow{\theta^*_{\widetilde{\mV}_\X,\widetilde{\mW}_\Y} (\gamma)} \theta^*_{\widetilde{\mV}_\X,\widetilde{\mW}_\Y} (\widetilde{\mV}^\ot_\X \times_{\mV^\ot} \mM^\circledast \times_{\mW^\ot} \widetilde{\mW}^\ot_\Y)
\simeq$$$$ {(\widetilde{\mV}^\ot \times_\Ass (\widetilde{\mW}^\ot)^\rev)}_{\X \times \Y} \times_{\mV^\ot \times_\Ass (\mW^\ot)^\rev} \theta_{\mV,\mW}^*(\mM^\circledast) \simeq $$$$
\rho_{\X \times \Y}^*({\Quiv_{\X \times \Y}(\mV \times \mW^\rev)}_{\X \times \Y} \times_{\mV^\ot \times_\Ass (\mW^\ot)^\rev} \theta_{\mV,\mW}^*(\mM^\circledast)) $$
is adjoint to a lax $ \Quiv_{\X \times \Y} (\mV \times \mW^\rev)_{\X \times \Y}^\ot $-linear functor
$$\rho_*(\theta^*_{\widetilde{\mV},\widetilde{\mW}}(\mN^\circledast))_{\X \times \Y} \simeq (\rho_{\X \times \Y})_*(\theta^*_{\widetilde{\mV},\widetilde{\mW}}(\mN^\circledast)_{\X \times \Y}) \to {\Quiv_{\X \times \Y}(\mV \times \mW^\rev)}_{\X \times \Y} \times_{\mV^\ot \times_\Ass (\mW^\ot)^\rev} \theta_{\mV,\mW}^*(\mM^\circledast) $$
corresponding to a lax 
$\Quiv_{\X \times \Y}(\mV \times \mW^\rev)^\ot$-linear functor
$\rho_*(\theta^*_{\widetilde{\mV},\widetilde{\mW}}(\mN^\circledast)) \to \Fun(\X \times \Y,\theta_{\mV,\mW}^*(\mM))^\circledast $
by the universal property of the right hand side.
The latter map is adjoint to a lax $\widetilde{\mV}^\ot \times_\Ass (\widetilde{\mW}^\ot)^\rev $-linear functor
$\beta': \theta^*_{\widetilde{\mV},\widetilde{\mW}}(\mN^\circledast) \to \rho^*(\Fun(\X \times \Y,\theta_{\mV,\mW}^*(\mM))^\circledast).$

We define (\ref{alf}) to be $\beta'$ for $\beta$ the identity.
Then (\ref{alf}) induces on underlying $\infty$-categories the identity and is 
$\widetilde{\mV} \times_\Ass \widetilde{\mW}^\rev $-linear by the description of biactions of Proposition \ref{contino} and so is an equivalence.

\end{proof}

\section{The equivalence}
\label{eqqiv}

In section \ref{extr} we constructed a functor 
$\chi: \omega\LMod \to \omega\PreCat_\infty.$ 
Now we come to our main goal: we prove that $\chi$ is fully faithful with essential image the weakly enriched $\infty$-categories defined next:

Let $\ast$ be the final $\mS$-precategory, whose space of objects is contractible, (corresponding to the final associative algebra in $\mS$), and $\mJ \to \ast$ a cartesian lift in $\PreCat_\infty^\mS$ of the map of spaces $ \{0,1\} \to [0]$.
So $\mJ$ is the $\infty$-precategory with contractible mapping spaces and space of objects a two-element set.

\begin{definition}\label{comp}
We call an $\mS$-precategory $\mD$ with small space of objects an $\mS$-category
if the map
$$\theta: \PreCat^\mS_\infty(\ast, \mD) \to \PreCat^\mS_\infty(\mJ, \mD)$$ is an equivalence.

\end{definition}

\begin{remark}
Note that for any $\mS$-precategory $\mD$ with small space of objects $\X$ the canonical map $ \PreCat^\mS_\infty(\ast, \mD) \to \mS([0],\X) \simeq \X$ is an equivalence since the fiber over any $\Z \in \X$ is the contractible space $ \Alg(\mS)(\ast, \mD(\Z,\Z)).$  	
On the other hand we think of $\PreCat^\mS_\infty(\mJ, \mD) $ as the space of equivalences of $\mD$. So $\mD$ is an $\mS$-category if the space of objects is equivalent to the space of equivalences.
\end{remark}

\begin{remark}\label{compl}
A very important property of $\mS$-categories is that an $\mS$-enriched functor between $\mS$-categories is an equivalence if it induces a surjective map on
path components of spaces of objects and an equivalence on underlying graphs \cite[Corollary 5.2.8.]{GEPNER2015575}.	
\end{remark}

By Lemma 4.8 any $\infty$-precategory enriched in an $\infty$-operad $\mV^\ot \to \Ass $ has an underlying $\mS$-precategory. In particular, any $\infty$-precategory $\mC$ weakly-enriched in $\mV$ has an underlying $\mS$-precategory that coincides with the former if $\mC$ is enriched since there is an embedding $\mV^\ot \subset \mP\Env(\mV)^\ot.$

\begin{definition}\label{comp}
Let $\mC$ be an $\infty$-precategory (weakly) enriched in an $\infty$-operad $\mV$. 
We call $\mC$ an $\infty$-category (weakly) enriched in $\mV$ if 
the underlying $\mS$-precategory of $\mC$ is an $\mS$-category.
\end{definition}

\begin{notation}
	
Let $\omega\Cat_\infty \subset \omega\PreCat_\infty$ be the full subcategory of weakly enriched $\infty$-categories.
	
\end{notation}

\begin{remark}\label{comple}
Let $\mV^\ot \to \Ass,\mW^\ot \to \Ass$ be $\infty$-operads and $\F: \mV^\ot \rightleftarrows \mW^\ot:\G$ an adjunction relative to $\Ass$.
For any $\mW$-category $\mC: \Ass_\X \to \mW^\ot$ the composition
$ \Ass_\X \xrightarrow{\mC} \mW^\ot \xrightarrow{ \G} \mV^\ot$ is a $\mV$-category.
This follows from the fact that $\G$ commutes with forming the underlying 
$\mS$-precategory.
	
\end{remark}

We will prove the following theorem:
\begin{theorem}\label{trfd}
The functor $$\chi: \omega\LMod \to \omega\PreCat_\infty $$
is fully faithful, admits a left adjoint, and the essential image precisely consists of the weakly enriched $\infty$-categories.
A $\mV$-enriched functor from an $\infty$-precategory weakly enriched in an $\infty$-operad $\mV$ to an $\infty$-category weakly enriched in $\mV$ is a local equivalence if and only if it induces an essentially surjective map on spaces of objects and an equivalence on underlying graphs.

\end{theorem}
\begin{corollary}\label{eqvvj}
The functor $\chi$ induces an equivalence 
$\omega\LMod \simeq \omega\Cat_\infty. $
\end{corollary}

Before we prove Theorem \ref{trfd}, we draw some conclusions.

\begin{definition}\label{tensor}
Let $\mV^\ot \to \Ass$ be an $\infty$-operad and $\mC$ an $\infty$-precategory weakly enriched in $\mV$ with space of objects $\X.$ 
The tensor of an object $\V \in \mV$ and $ \Y \in \X$
is an object $\V \otimes \Y \in \X$ equipped with a morphism
$\V \to \mC(\Y,\V \otimes \Y)$ in $\mP\Env(\mV)$ such that for any $\Z \in \X$ 
the canonical composition
\begin{equation}\label{eqfff}
\mC(\V \otimes \Y,\Z)  \to \Mor_{\mP\Env(\mV)}(\mC(\Y,\V \otimes \Y),\mC(\Y,\Z)) \to \Mor_{\mP\Env(\mV)}(\V,\mC(\Y,\Z))
\end{equation}
is an equivalence.
\vspace{1mm}
We call $\mC$ tensored over $\mV$ if every object $\V \in \mV$ and $ \Y \in \X$ admit a tensor.
\end{definition}

We apply Definition \ref{tensor} in particular to the case, where $\mC$ is a  $\mV$-enriched $\infty$-precategory seen as $\infty$-precategory weakly enriched in $\mV$ to talk about $\mV$-enriched $\infty$-precategories tensored over $\mV.$

\begin{proposition}\label{pseu}
Under the equivalence of Corollary \ref{eqvvj} the following objects correspond:
\begin{enumerate}
\item $\infty$-categories enriched in an $\infty$-operad $\mV$ in the sense of Lurie and $\mV$-enriched $\infty$-categories in the sense of Gepner-Haugseng.
\item $\infty$-categories pseudo-enriched in a monoidal $\infty$-category $\mV$ and $\infty$-categories enriched in $\mP(\mV)$.
\item $\infty$-categories locally left tensored over an $\infty$-operad $\mV$ and $\infty$-categories weakly enriched in $\mV$ that are tensored over $\mV.$ 
\item $\infty$-categories left tensored over a monoidal $\infty$-category $\mV$ and $\mP(\mV)$-enriched $\infty$-categories that are tensored over $\mV$.
\item $\infty$-categories with closed left action of a monoidal $\infty$-category $\mV$ and $\mV$-enriched $\infty$-categories that are tensored over $\mV.$

\end{enumerate}
\end{proposition}
\begin{proof}
Let $\mM^\circledast \to \mV^\ot$ be a weakly left tensored $\infty$-category and
$\chi(\mM)$ the underlying $\infty$-category enriched in $\mP\Env(\mV).$
The graph of $\chi(\mM)$ is the functor $\mM^\simeq \times \mM^\simeq \to \mP\Env(\mV)$ sending $\X,\Y$ to the presheaf whose value at the family $\V_1,...,\V_\n \in \mV$ for $\n \geq 0$ is the space $\Mul_{\mM}(\V_1,...,\V_\n,\X;\Y).$

\noindent So $\mM^\circledast \to \mV^\ot$ exhibits $\mM$ as enriched in $\mV$
if and only if the graph of $\chi(\mM)$ lands in $\mV.$
This proves (1).

For (2) we prove that the functor $\mM^\circledast \to \mV^\ot$ exhibits $\mM$ as pseudo-enriched in $\mV$ if and only if the graph of $\chi(\mM)$ lands in $\mP(\mV).$
By Lemma \ref{looocx} for every monoidal $\infty$-category $\mV^\ot \to \Ass$ the embedding $\iota: \mV \subset \Env(\mV)$ admits a left adjoint $\L$. Hence the embedding $\iota_!: \mP(\mV) \subset \mP\Env(\mV)$ coincides with $\L^\ast.$
Now observe that $\mM^\circledast \to \mV^\ot$ exhibits $\mM$ as pseudo-enriched in $\mV$ if and only if the graph of $\mM$ factors through the essential image of $\L^\ast: \mP(\mV) \subset \mP\Env(\mV)$.

(3): Let $\V \in \mV$, $ \Y,\Y',\Z \in \mM$ and $\alpha \in \Mul_\mM(\V,\Y;\Y').$
The morphism (\ref{eqfff}) for $\mC=\chi(\mM)$ evaluated at $\V_1 \ot ... \ot \V_\n$
for $\n \geq 0$ and $\V_1,...,\V_\n \in \mV$ is the canonical map
$\Mul_\mM(\V_1,...,\V_\n, \Y';\Z) \to \Mul_\mM(\V_1,...,\V_\n, \V,\Y;\Z).$
This shows (3).

(4) follows from (2) and (3) since a weakly left tensored $\infty$-category is a left tensored $\infty$-category if and only if it is a pseudo-enriched $\infty$-category and locally left tensored $\infty$-category 
(Proposition \ref{lur}).
(5) follows from (1) and (4) because a weakly left tensored $\infty$-category $\mM^\circledast \to \mV^\ot$ endows $\mM$ with a closed left action if and only if it is a left tensored $\infty$-category and exhibits $\mM$ as $\mV$-enriched.

\end{proof}

\begin{notation}
	
Let $\P\LMod \subset \omega\LMod$ be the full subcategory of pseudo-enriched $\infty$-categories.	
Let $\Enr^\Lur \subset \omega\LMod$ be the full subcategory of enriched $\infty$-categories.
	
%For any small $\infty$-operad $\mV^\ot \to \Ass$ let $\Cat_\infty^{\mV, \Lur} \subset \omega\widehat{\LMod}_\mV$ be the full subcategory of $\mV$-enriched $\infty$-categories $\mM^\circledast \to \mV^\ot$ such that $\mM$ is small.
	
\end{notation}	

\begin{notation}
	
Let $\P\PreCat_\infty$ be the pullback of $\widehat{\PreCat}_\infty \to  \widehat{\Op}_\infty \times \widehat{\mS} $ along
the functors $\mP: \Mon \to \widehat{\Mon} \subset \widehat{\Op}_\infty$ and $\mS \subset \widehat{\mS} $.
	
\end{notation}

%We obtain the following corollary:
%\begin{corollary}\label{cortos}Let $\mV^\ot \to \Ass$ be a small monoidal $\infty$-category.There is a canonical equivalence $$ \P\LMod_\mV \simeq \Cat_\infty^{\mP(\mV)}.$$\end{corollary}
	
\begin{corollary}\label{eqvvjj}
The functor $\chi: \omega\LMod \to \omega\PreCat_\infty $
restricts to a fully faithful functor $$\Enr^\Lur \to \PreCat_\infty $$
that admits a left adjoint.
The local objects are the enriched $\infty$-categories in the sense of Gepner-Haugseng.
A $\mV$-enriched functor from an $\infty$-precategory enriched in an $\infty$-operad $\mV$ to a $\mV$-enriched $\infty$-category is a local equivalence if and only if it induces an essentially surjective map on spaces of objects and an equivalence on underlying graphs.

\end{corollary}

\begin{corollary}
The functor $\chi: \omega\LMod \to \omega\PreCat_\infty $
restricts to a fully faithful functor $\P\LMod\to \P\PreCat_\infty $
that admits a left adjoint.
The local objects are the pairs $(\mV^\ot \to \Ass, \mC)$, where $\mV^\ot \to \Ass$ is a monoidal $\infty$-category and $\mC$ is a $\mP(\mV)$-enriched $\infty$-category in the sense of Gepner-Haugseng.	
	
\end{corollary}

Recall from Notation \ref{coli} the monoidal $\infty$-category $(\Cat_\infty^{\rc\rc})^\ot \to \Ass$ and the full monoidal subcategory $ (\Pr^\L)^\ot \subset (\Cat_\infty^{\rc\rc})^\ot$ of presentable $\infty$-categories.
For any presentably monoidal $\infty$-category $\mV^\ot \to \Ass$ there is a canonical inclusion $ \LMod_\mV(\Cat_\infty^{\rc\rc}) \subset  \LMod_\mV(\widehat{\Cat}_\infty) \simeq \widehat{\LMod}_\mV \subset \omega\widehat{\LMod}_\mV$ whose image is the subcategory of $\infty$-categories left tensored over $\mV$ compatible with small colimits,
which are $\mV$-enriched by presentability of $\mV$, and $\mV$-linear functors preserving small colimits. So we obtain the following corollary:
\begin{corollary}\label{cortos}
Let $\mV^\ot \to \Ass$ be a presentably monoidal $\infty$-category.
There is a canonical inclusion $ \LMod_\mV(\Cat_\infty^{\rc\rc}) \hookrightarrow \widehat{\Cat}_\infty^{\mV}$ and so a canonical inclusion $ \LMod_\mV(\Pr^\L) \hookrightarrow \widehat{\Cat}_\infty^{\mV}$.
\end{corollary}

\vspace{1mm}

\begin{notation}

For any (not necessarily small) $\infty$-operad $\mV^\ot \to \Ass$ let $\Cat_\infty^{\mV, \Lur} \subset \omega\widehat{\LMod}_\mV$ be the full subcategory of $\mV$-enriched $\infty$-categories $\mM^\circledast \to \mV^\ot$ such that $\mM$ is small.

\end{notation}		

By Corollary \ref{eqvvjj} the functor $\chi: \Cat_\infty^{\mS, \Lur} \to \Cat_\infty^\mS$ is an equivalence. So by the next Lemma we obtain the following corollary:

\begin{corollary}
There is a canonical equivalence
$\Cat_\infty \simeq \Cat_\infty^\mS$ over $\mS$,
where $\Cat_\infty$ lies over $\mS$ via the right adjoint of the embedding $\mS \subset \Cat_\infty.$
	
\end{corollary}

\begin{lemma}\label{Caa}
The forgetful functor $	\Cat_\infty^{\mS, \Lur} \to \Cat_\infty$ is an equivalence
	
\end{lemma}

\begin{proof}
	
The forgetful functor $ \LMod_\mS(\Cat_\infty^{\rc\rc}) \to \Cat_\infty^{\rc\rc}$ is an equivalence since $\mS$ is the tensor unit.
By Remark \ref{modules} there is a canonical equivalence $ \LMod_\mS(\Cat_\infty^{\rc\rc}) \simeq \LMod_\mS^{\rc\rc}$ over $\Cat_\infty^{\rc\rc}$. So the forgetful functor $ \LMod_\mS^{\rc\rc} \to \Cat_\infty^{\rc\rc}$ is an equivalence.
Hence for any small $\infty$-category $\mC$ the $\infty$-category $\mP(\mC) $ is canonically left tensored over $\mS$ compatible with small colimts and the full subcategory $\bar{\mC} \subset \mP(\mC)$ spanned by the representable presheaves is a $\mS$-enriched $\infty$-category.
Since the forgetful functor $\Cat_\infty^{\mS, \Lur} \to \Cat_\infty$
is conservative, it will be enough to see that it admits a fully faithful left adjoint. 
We prove that for any small $\mS$-enriched $\infty$-category
$\mD^\circledast \to \mS^\times$ the functor $ \LaxLinFun_\mS(\bar{\mC}, \mD) \to \Fun(\mC, \mD)$	is an equivalence. The latter functor is the pullback of the functor $ \LaxLinFun_\mS(\bar{\mC}, \mP_\mS(\chi(\mD))) \to \Fun(\mC, \mP_\mS(\chi(\mD)))$.
Consider the commutative square:
\begin{equation*}
\begin{xy}
\xymatrix{
\LinFun_\mS^\L(\mP_\mS(\chi(\mC)), \mP_\mS(\chi(\mD)))  \ar[d]^{} \ar[rr]
&& \LaxLinFun_\mS(\bar{\mC}, \mP_\mS(\chi(\mD))) \ar[d]^{}
\\ \Fun^\L(\mP(\mC), \mP_\mS(\chi(\mD)))
\ar[rr]^{} && \Fun(\mC, \mP_\mS(\chi(\mD))).
}
\end{xy} 
\end{equation*}
The left vertical functor is induced by the equivalence $ \LMod_\mS^{\rc\rc} \to \Cat_\infty^{\rc\rc}$ and the small colimits preserving functor
$ \beta: \mP(\mC) \to \mP_\mS(\chi(\mC))$ extending the functor $\mC \to \mP_\mS(\chi(\mC))$
underlying the $\mS$-enriched Yoneda-embedding. 
The functor $\beta$ is an equivalence because $\mP_\mS(\chi(\mC))$ is generated by
the representable $\mS$-enriched presheaves under small colimits and every representable $\mS$-enriched presheaf corepresents a small colimits preserving functor $\mP_\mS(\chi(\mC)) \to \mS$ (Remark \ref{tau}).
By Corollary \ref{korjl} the top horizontal functor is an equivalence.
By \cite[Theorem 5.1.5.6.]{lurie.HTT} the bottom horizontal functor is an equivalence.
	
\end{proof}

\vspace{1mm}
Now we prepare the proof of Theorem \ref{trfd}: 
Since the functor $\chi$ is a map of cartesian fibrations over $\Op_{\infty}$ (Lemma \ref{leeem}), to see that $\chi$ is fully faithful and admits a left adjoint, it is enough to show that for any $\infty$-operad $\mV^\ot \to \Ass $ the induced functor $\chi_\mV: \omega\LMod_\mV \to \omega\PreCat^\mV_\infty$
is fully faithful and admits a left adjoint.
This follows from Theorems \ref{cooo} and \ref{werf}:

\begin{lemma}\label{leop}Let $\mC$ be an $\infty$-precategory weakly enriched in an $\infty$-operad $\mV$ with small space of objects $\X$.
Let $\mM^\circledast\to \mP\Env(\mV)^\ot,\mN^\circledast\to \mP\Env(\mV)^\ot$ be left tensored $\infty$-categories compatible with small colimits, $\tau: \Y \to \mM^\simeq $ a map and $\theta: \mM^\circledast \to \mN^\circledast$ a left adjoint $\mP\Env(\mV)$-linear functor that induces equivalences on morphism objects between objects in the essential image of $\tau.$	
	
\begin{enumerate}
\item The functor 
$$\kappa: \Fun(\X,\Y) \times_{\Fun(\X,\mM)} \LinFun^\L_{\mP\Env(\mV)}(\mP_{\mP\Env(\mV)}(\mC),\mM) \to $$$$\Fun(\X,\Y) \times_{\Fun(\X,\mN)} \LinFun^\L_{\mP\Env(\mV)}(\mP_{\mP\Env(\mV)}(\mC),\mN) $$
is an equivalence.

\item The functor 
$$\kappa': \Fun(\X,\Y) \times_{\Fun(\X,\mM)} \LinFun^\L_{\mP\Env(\mV)}(\mP\L\Env(\L(\mC)), \mM) \to$$$$  \Fun(\X,\Y) \times_{\Fun(\X,\mN)} \LinFun^\L_{\mP\Env(\mV)}(\mP\L\Env(\L(\mC)),\mN) $$
is an equivalence
\end{enumerate}	
\end{lemma}

\begin{proof}
	
(1): By Theorem \ref{tgre} the functor $\kappa $ is equivalent to the functor 
$$\Fun(\X,\Y) \times_{\Fun(\X,\mM)} \Fun^{\mP\Env(\mV)}(\mC,\mM) \to \Fun(\X,\Y) \times_{\Fun(\X,\mN)} \Fun^{\mP\Env(\mV)}(\mC,\mN),$$
which is equivalent to the functor 
$\Fun^{\mP\Env(\mV)}(\mC,\tau^\ast(\mM)) \to \Fun^{\mP\Env(\mV)}(\mC,\tau^\ast(\theta^\ast(\mN))).$
Thus $\kappa$ is an equivalence if the $\mP\Env(\mV)$-enriched functor
$\tau^\ast(\mM) \to \tau^\ast(\theta^\ast(\mN))$ is an equivalence.

The latter $\mP\Env(\mV)$-enriched functor induces an equivalence on spaces of objects and so is an equivalence if it induces an equivalence on underlying graphs. This is equivalent to ask that the $\mP\Env(\mV)$-linear functor
$\theta: \mM^\circledast \to \mN^\circledast$ induces equivalences on morphism objects between objects in the essential image of $\tau.$	

\vspace{1mm}
(2): Let $\mM^\circledast_{\mid\tau} \subset \mM^\circledast, \mN^\circledast_{\mid\theta \circ \tau} \subset \mN^\circledast$ be the full subcategories with weak left $\mV$-action spanned by the essential images of $\tau, \theta \circ \tau,$ respectively.
The $\mP\Env(\mV)$-linear functor $\theta: \mM^\circledast \to \mN^\circledast$ 
restricts to a lax $\mV$-linear functor $\mM_{\mid\tau}^\circledast \to \mN^\circledast_{\mid\theta \circ \tau}$ that is essentially surjective.
The functor $\kappa' $ is equivalent to the functor
$$\Fun(\X,\Y) \times_{\Fun(\X,\mM_{\mid\tau})} \LaxLinFun_{\mV}(\L(\mC), \mM_{\mid\tau}) \to \Fun(\X,\Y) \times_{\Fun(\X,\mN_{\mid\theta \circ \tau})} \LaxLinFun_{\mV}(\L(\mC),\mN_{\mid\theta \circ \tau}). $$
Thus $\kappa'$ is an equivalence if $\theta: \mM^\circledast \to \mN^\circledast$ induces equivalences on morphism objects between objects in the essential image of $\tau.$
	
\end{proof}

\begin{theorem}\label{cooo}
Let $\mC$ be an $\infty$-precategory weakly enriched in an $\infty$-operad $\mV$ with small space of objects $\X$.
The canonical embedding $\L(\mC)^\circledast \subset \mP_{\mP\Env(\mV)}(\mC)^\circledast$ 
induces a $\mP\Env(\mV)$-linear equivalence $$\zeta: \mP\L\Env(\L(\mC))^\circledast \to \mP_{\mP\Env(\mV)}(\mC)^\circledast. $$

\end{theorem}

\begin{proof}
We prove that $\zeta$ admits a section, which itself has a section. 
To produce this section we apply Lemma \ref{leop} (1), where we take $\theta$ to be $\zeta: \mP\L\Env(\L(\mC))^\circledast \to \mP_{\mP\Env(\mV)}(\mC)^\circledast $ and $\tau$ to be the composition $ \X \xrightarrow{\iota} \L(\mC)^\simeq \subset \mP\L\Env(\L(\mC))^\simeq $. 
Since the canonical $\mP\Env(\mV)$-enriched functor
$$\tau^\ast(\mP\L\Env(\L(\mC))) \simeq \mC \to \tau^\ast(\zeta^\ast(\mP_{\mP\Env(\mV)}(\mC))) \simeq \mC $$ is the identity, there is a $\mP\Env(\mV)$-linear left adjoint functor 
$\alpha: \mP_{\mP\Env(\mV)}(\mC)^\circledast \to \mP\L\Env(\L(\mC))^\circledast  $ that is a section of $\zeta$ and a commutative triangle
\begin{equation*} 
\begin{xy}
\xymatrix{
&\X \ar[rd] \ar[ld]
\\ \mP_{\mP\Env(\mV)}(\mC)^\simeq\ar[rr]^{\alpha^\simeq} && \mP\L\Env(\L(\mC))^\simeq.
}
\end{xy} 
\end{equation*}

We complete the proof by showing that the section $\alpha$ of $\zeta$ has itself a section.

By the existence of the last triangle $\alpha: \mP_{\mP\Env(\mV)}(\mC)^\circledast \xrightarrow{\alpha} \mP\L\Env(\L(\mC))^\circledast$ restricts to a lax $\mV$-linear functor
$\alpha': \L(\mC)^\circledast \to \L(\mC)^\circledast.$
Composing with $\zeta$ we obtain a commutative triangle
\begin{equation*} 
\begin{xy}
\xymatrix{
\L(\mC)^\circledast \ar[rd] \ar[rr]^{\alpha'} && \L(\mC)^\circledast \ar[ld] 
\\ & \mP_{\mP\Env(\mV)}(\mC)^\circledast,
}
\end{xy} 
\end{equation*}
where the vertical functors are the canonical embeddings.
Thus $\alpha'$ is an embedding so that $\alpha$ induces equivalences on morphism objects between objects in the essential image of the map $\X \to \mP_{\mP\Env(\mV)}(\mC)^\simeq.$ 

Taking $\theta$ to be $\alpha: \mP_{\mP\Env(\mV)}(\mC)^\circledast \to \mP\L\Env(\L(\mC))^\circledast$ and $\tau$ to be the map $\X \to \mP_{\mP\Env(\mV)}(\mC)^\simeq $ by Lemma \ref{leop} (2) we obtain a $\mP\Env(\mV)$-linear left adjoint functor
$\beta: \mP\L\Env(\L(\mC))^\circledast \to \mP_{\mP\Env(\mV)}(\mC)^\circledast $
that is a section of $\alpha.$
Hence $\zeta$ is an equivalence.

\end{proof}

Theorem \ref{cooo} and Remark \ref{tau} give the following corollary:

\begin{corollary}\label{kory}
Let $\mM^\circledast \to \mV^\ot$ be a weakly bitensored $\infty$-category and
$\X \in \mM$. The lax $\mP\Env(\mV)$-linear functor $\Mor_{\mP\L\Env(\mM)}(\X,-): \mP\L\Env(\mM) \to \mP\Env(\mV)$ right adjoint to the $\mP\Env(\mV)$-linear functor $(-)\ot \X: \mP\Env(\mV) \to \mP\L\Env(\mM)$
preserves small colimits and is $\mP\Env(\mV)$-linear.	
	
\end{corollary}

Theorems \ref{tgre} and \ref{cooo} give the following theorem:
\begin{theorem}\label{werf}
Let $\mV^\ot \to \Ass$ be an $\infty$-operad, $\mC$ an $\infty$-precategory weakly enriched in $\mV$ with small space of objects $\X$ and $\mM^\circledast \to \mV^\ot \times \mW^\ot$ an $\infty$-category with weak biaction.

\begin{enumerate}
\item There is a lax $\mW$-linear equivalence 
$$
\LaxLinFun_\mV(\L(\mC),\mM)^\circledast \simeq \Fun^\mV(\mC,\mM)^\circledast
$$
that for $\mM=\L(\mC)$ sends the identity to the $\mV$-enriched Yoneda-embedding $\mC \to \chi(\L(\mC))$.

\vspace{1mm}

\item If $\mC$ is a $\mV$-precategory, the equivalence of (1) is equivalent to the lax $\mW$-linear functor $$\Phi_\mM: \LaxLinFun_\mV(\L(\mC),\mM)^\circledast \to \Fun^\mV(\mC,\mM)^\circledast.$$
\end{enumerate}	

\end{theorem}

\begin{proof}
(1): By Theorem \ref{tgre} the lax $\mW$-linear functor 
$ \Psi: \LinFun^\L_{\mP\Env(\mV)}(\mP_{\mP\Env(\mV)}(\mC),\mP\L\Env(\mM))^\circledast \to $

$ \Fun^{\mP\Env(\mV)}(\mC,\mP\L\Env(\mM))^\circledast$
is an equivalence. By Theorem \ref{cooo} the lax $\mW$-linear functor
$$ \LinFun^\L_{\mP\Env(\mV)}(\mP_{\mP\Env(\mV)}(\mC),\mP\L\Env(\mM))^\circledast \to \LaxLinFun_\mV(\L(\mC),\mP\L\Env(\mM))^\circledast, $$ 
which we denoty by $\alpha$, is an equivalence.
The resulting lax $\mW$-linear equivalence
$$\Psi \circ \alpha^{-1} : \LaxLinFun_\mV(\L(\mC),\mP\L\Env(\mM))^\circledast \simeq \Fun^{\mP\Env(\mV)}(\mC,\mP\L\Env(\mM))^\circledast$$ restricts to an equivalence $\Theta: \LaxLinFun_\mV(\L(\mC),\mM)^\circledast \to \Fun^{\mV}(\mC,\mM)^\circledast$
that for $\mM=\L(\mC)$ sends the identity to the $\mV$-enriched Yoneda-embedding $\mC \to \chi(\L(\mC))$.
	
\vspace{1mm}

(2): We have constructed a lax $\mW$-linear functor $\Phi_\mM: \LaxLinFun_{\mV}(\L(\mC),\mM)^\circledast \to \Fun^\mV(\mC,\mM)^\circledast$
(Notation \ref{fuuu}).
Let $\kappa: \mV^\ot \subset \mP\Env(\mV)^\ot$ the canonical embedding.
There is a commutative diagram
\begin{equation*}\label{sqq}
\begin{xy}
\xymatrix{
\LaxLinFun_{\mV}(\L(\mC),\mM)^\circledast \ar[d]^{} \ar[rrr]^{\Phi_\mM,\Theta}
&&&\Fun^\mV(\mC,\mM)^\circledast \ar[d]^{}
\\ 
\LinFun^\L_{{\mP\Env(\mV)}}(\mP_{\mP\Env(\mV)}(\mC),\mP\L\Env(\mM))^\circledast \ar[d]^{} \ar[rrr]^{\Psi} &&&\Fun^{\mP\Env(\mV)}(\mC,\mP\L\Env(\mM))^\circledast\ar[d]^{=}
\\ 
\LaxLinFun_{{\mP\Env(\mV)}}(\L(\mC),\mP\L\Env(\mM))^\circledast \ar[d]^{} \ar[rrr]^{\Phi_{\mP\L\Env(\mM)}} &&&\Fun^{\mP\Env(\mV)}(\mC,\mP\L\Env(\mM))^\circledast\ar[d]^{\simeq}
\\ 
\LaxLinFun_{\mV}(\L(\mC),\mP\L\Env(\mM))^\circledast \ar[rrr]^{\Phi_{\kappa^*(\mP\L\Env(\mM))}} &&&\Fun^{\mV}(\mC,\kappa^*(\mP\L\Env(\mM)))^\circledast
}
\end{xy} 
\end{equation*} 
of $\infty$-categories with weak right $\mW$-action.
The middle square commutes by definition of $\Psi$, the bottom square by construction of
$\Phi$, the top square for $\Theta$ commutes by construction of $\Theta$ and
the top square for $\Phi_\mM$ commutes since the outer square commutes for $\Phi_\mM.$
Hence $\Theta=\Phi_\mM.$
 
\end{proof}

Now we are ready to prove Theorem \ref{trfd}.

\begin{proof}[Theorem \ref{trfd}]
By Lemma \ref{leeem} the functor  $\chi : \omega\LMod \to \omega\PreCat_\infty$ is a map of cartesian fibrations over $\Op_{\infty}$, which induces on the fiber over any $\infty$-operad $\mV^\ot \to \Ass $ a right adjoint functor $\chi_\mV: \omega\LMod_\mV \to \omega\PreCat^\mV_\infty$ (Theorem $\ref{werf}$).
Thus $\chi$ admits a left adjoint relative to $\Op_\infty$ (see \ref{rell}).

We will prove that $\chi$ lands in the full subcategory of enriched $\infty$-categories.
This will imply that a weakly enriched $\infty$-precategory $\mC$ is a weakly enriched $\infty$-category if and only if the unit $\mC \to \chi(\L(\mC))$ is an equivalence.
The non-trivial direction follows from Remark \ref{compl} and the fact that the unit induces an essentially surjective map on spaces of objects and an equivalence on underlying graphs. Especially for any weakly left tensored $\infty$-category $\mM^\circledast \to \mV^\ot$ the unit $\chi(\mM) \to \chi(\L(\chi(\mM)))$ at $\chi(\mM)$
is an equivalence. As $\chi$ is conservative, by the triangle identities
the counit $\L(\chi(\mM)) \to \mM$ is an equivalence, too. So $\chi$ is fully faithful.

So we prove that $\chi$ lands in enriched $\infty$-categories.
Since $\chi$ preserves cartesian morphisms over $\Op_\infty$, any map of $\infty$-operads $\alpha: \emptyset^\ot \to \mV^\ot$ gives rise to a commutative square:
\begin{equation*} 
\begin{xy}
\xymatrix{
\omega\LMod_\mV \ar[d]^{\alpha^*} \ar[rr]^{\chi_\mV}
&&\omega\PreCat^\mV \ar[d]^{\alpha^*} 
\\  \omega\LMod_\emptyset \simeq \Cat_\infty \ar[rr]^{\chi_\emptyset}  && \omega\PreCat^\emptyset = \PreCat_\infty^\mS.
}
\end{xy} 
\end{equation*}

So we can reduce to $\mV^\ot=\emptyset^\ot$. In this case we need to see that for any $\infty$-category $\mM$ the image $\chi(\mM)$ is an $\mS$-enriched $\infty$-category.
By adjointness the map 
$ \PreCat_\infty^\mS(\ast, \chi(\mM)) \to \PreCat_\infty^\mS(\mJ, \chi(\mM))$ is equivalent to the map
$ \Cat_\infty(\L(\ast), \mM) \to \Cat_\infty(\L(\mJ), \mM).$
So we need to see that the functor $\L(\mJ) \to \L(\ast)$ is an equivalence.
As a cartesian lift of an essentially surjective map of spaces $\mJ \to \ast$ induces an essentially surjective map on spaces of objects and an equivalence on underlying graphs like the units at $\mJ$ and $*$ do.
So also $\chi(\L(\mJ)) \to \chi(\L(*))$ does. As $\chi$ preserves the space of objects and underlying graph, the functor $\L(\mJ) \to \L(*)$ is essentially surjective and fully faithful and so an equivalence.

The characterization of local equivalences follows from the fact that
an enriched functor whose target is local is inverted by $\L$ if and only if
its image under $\chi \circ \L$ induces an essentially surjective map on spaces of objects and an equivalence on underlying graphs (since $\chi$ preserves both).

\end{proof}

\begin{corollary}\label{korjl}
Let $\mV^\ot \to \Ass$ be a monoidal $\infty$-category compatible with small colimits, 
$\mW^\ot \to \Ass$ a monoidal $\infty$-category, $\mM^\circledast \to \mV^\ot$ a small
$\mV$-enriched $\infty$-category and $\mN^\circledast \to \mV^\ot \times \mW^\ot$ a bitensored $\infty$-category whose underlying left tensored
$\infty$-category is compatible with small colimits.
	
The canonical lax $\mV$-linear functor
$\mM^\circledast \simeq \L(\chi(\mM))^\circledast \subset \mP_\mV(\chi(\mM))^\circledast $
induces a $\mW$-linear equivalence
\begin{equation*}
\LinFun^\L_{\mV}(\mP_\mV(\chi(\mM)),\mN) \to \LaxLinFun_\mV(\mM,\mN).
\end{equation*}
	
\end{corollary}

\begin{proof}
	
By Theorem \ref{tgre} the $\mW$-linear functor
\begin{equation*}
\Psi: \LinFun^\L_{\mV}(\mP_\mV(\chi(\mM)),\mN) \to \LaxLinFun_\mV(\L(\chi(\mM)),\mN) \xrightarrow{\Phi} \Fun^\mV(\chi(\mM),\mN)
\end{equation*}
is an equivalence.
By Theorem \ref{werf} $\Phi$ is an equivalence.
By Theorem \ref{trfd} there is a canonical equivalence $\L(\chi(\mM))^\circledast \simeq \mM^\circledast$ of $\infty$-categories with weak left $\mV$-action.
	
\end{proof}

For later reference we also add the following proposition:

\begin{proposition}\label{EmB}

Let $\mN^\circledast \to \mV^\ot$ be a presentably left tensored $\infty$-category,
$\mM^\circledast \to \mV^\ot$ an enriched $\infty$-category such that $\mM$ is small and $\theta: \mP_\mV(\chi(\mM))^\circledast \to \mN^\circledast$ a $\mV$-linear small colimits preserving functor.

\begin{enumerate}
\item The functor $\theta$ is an embedding if the lax $\mV$-linear functor $\mM^\circledast \subset \mP_\mV(\chi(\mM))^\circledast \xrightarrow{\theta} \mN^\circledast$ is an embedding and for every $\X\in \mM$ the lax $\mV$-linear functor $\Mor_\mN(\theta(\X),-): \mN \to \mV$ preserves small colimits and is $\mV$-linear.

\item The functor $\theta$ is an equivalence if and only if for every $\X\in \mM$ the lax $\mV$-linear functor $\Mor_\mN(\theta(\X),-): \mN \to \mV$ preserves small colimits and is $\mV$-linear and the essential image of the functor $\mM \to \mP_\mV(\chi(\mM)) \xrightarrow{\theta} \mN$ generates $\mN$ under small colimits and the left $\mV$-action.
\end{enumerate}

\end{proposition}

\begin{proof}

(1) We first prove that for any $\A, \B \in \mP_\mV(\chi(\mM))$ the induced morphism
$$\rho: \Mor_{\mP_\mV(\chi(\mM))}(\A,\B) \to \Mor_\mN(\theta(\A), \theta(\B))$$ in $\mV$ is an equivalence. The $\infty$-category $\mP_\mV(\chi(\mM))$ is generated
by $\mM$ under small colimits and the left $\mV$-action.
Since $\mP_\mV(\chi(\mM)), \mN$ are cotensored over $\mV$ and admit small limits
and $\theta$ is $\mV$-linear and preserves small colimits, we can assume that $\A$ is representable.
By Remark \ref{tau} for every $\X \in \mM$ the lax $\mV$-linear functor $\Mor_{\mP_\mV(\chi(\mM))}(\X,-): \mP_\mV(\chi(\mM)) \to \mV$ right adjoint to the $\mV$-linear functor $(-)\ot \X: \mV \to \mP_\mV(\chi(\mM))$ preserves small colimits and is $\mV$-linear.
The same holds by assumption for the essential image of $\mM$ under $\theta$.
Since $\mP_\mV(\chi(\mM))$ is generated by $\mM$ under small colimits and the left $\mV$-action, and $\theta$ is $\mV$-linear and preserves small colimits, we can assume that $\A, \B \in \mM.$ In this case $\rho$ is an equivalence by assumption. Thus $\theta$ is an embedding.

(2) follows from (1), Remark \ref{tau} and the fact that the essential image of $\theta$ is closed in $\mN$ under small colimits and the left $\mV$-action.

\end{proof}

\begin{corollary}\label{kojj}

Let $\mN^\circledast \to \mV^\ot$ be a presentably left tensored $\infty$-category,
$\A$ an associative algebra in $\mV$ and $\theta: \RMod_\A(\mV)^\circledast \to \mN^\circledast$ a $\mV$-linear small colimits preserving functor.
The functor $\theta$ is an equivalence if and only if the lax $\mV$-linear functor $\Mor_\mN(\theta(\A),-): \mN \to \mV$ preserves small colimits and is $\mV$-linear and $\theta(\A) $ generates $\mN$ under small colimits and the left $\mV$-action.

\end{corollary}

\begin{remark}

In the latter corollary one can replace small colimits by geometric realizations everywhere because $ \RMod_\A(\mV)$ is generated by $\A$ under geometric realizations and the left $\mV$-action.
\end{remark}

\vspace{2mm}

Ayala-Francis \cite[Definition 0.12.]{MR3869643},\cite[Definition 0.2]{MR4074276} introduce a model for $\mS$-precategories called 
flagged $\infty$-categories. 
A flagged $\infty$-category consists of an $\infty$-category $\mC$ equipped with an essentially surjective map of spaces $\X \to \mC^\simeq $.
They prove \cite[Theorem 0.4]{MR4074276} that the $\infty$-category of small flagged $\infty$-categories
defined by the pullback $$ \Fun([1],\mS)^\mathrm{surj} \times_{\Fun(\{1\},\mS)} \Cat_\infty, $$ where $\Fun([1],\mS)^\mathrm{surj} \subset \Fun([1],\mS)$ is the full subcategory of essentially surjective maps,
is equivalent to the $\infty$-category of Segal spaces, which by \cite[Theorem 4.4.6.]{GEPNER2015575} is equivalent to $\PreCat_\infty^\mS$.

In analogy we make the following definition:
\begin{definition}
A flagged $\infty$-category with weak left action 
is an $\infty$-category with weak left action $\mM^\circledast \to \mV^\ot$
equipped with an essentially surjective map of spaces $\X \to \mM^\simeq$.
\end{definition}
The $\infty$-category of small flagged $\infty$-categories with weak left action is the pullback $$\omega\LMod^\fl := \Fun([1],\mS)^\mathrm{surj} \times_{\Fun(\{1\},\mS)} \omega\LMod. $$ There are forgetful functors $\omega\LMod^\fl \to \omega\LMod \to \Op_\infty, \omega\LMod^\fl \to \Fun([1],\mS)^\mathrm{surj} \to \Fun(\{0\},\mS).$
\begin{corollary}

There is a canonical equivalence
$ \omega\PreCat_\infty \simeq \omega\LMod^\fl$ over $\mS \times \Op_\infty.$

\end{corollary}

\begin{proof}

In view of the canonical equivalence
$\omega\LMod \simeq \omega\Cat_\infty$ over $\mS \times \Op_{\infty}$ of Theorem \ref{trfd}, it will be enough to
construct a canonical equivalence $$\gamma: \omega\PreCat_\infty \simeq \Fun([1],\mS)^\mathrm{surj} \times_{\Fun(\{1\},\mS)} \omega\Cat_\infty.$$
over $\Fun(\{0\},\mS) \times \Op_\infty.$
As $\phi: \omega\PreCat_\infty \to \mS$ is a cartesian fibration, by \ref{cocartt} evaluation at the target
$$\rho: \Fun([1], \omega\PreCat_\infty)^\cart \to \Fun([1],\mS) \times_{\Fun(\{1\},\mS)} \Fun(\{1\},\omega\PreCat_\infty) $$
is an equivalence, where $ \Fun([1], \omega\PreCat_\infty)^\cart \subset \Fun([1], \omega\PreCat_\infty)$ is the full subcategory of $\phi$-cartesian morphisms.
The equivalence $\rho$ restricts to an equivalence 
$$\Fun([1], \omega\PreCat_\infty)' \simeq \Fun([1],\mS)^\mathrm{surj} \times_{\Fun(\{1\},\mS)} \Fun(\{1\},\omega\Cat_\infty), $$
where the left hand side is the full subcategory of $\phi$-cartesian morphisms $\mC \to \mD$ that induce an essentially surjective map on spaces of objects and such that $\mD \in \omega\Cat_\infty,$
which by Theorem \ref{trfd} is precisely the full subcategory of local equivalences with local target.
Thus by Corollary \ref{cocccq} evaluation at the source restricts to an 
equivalence $\Fun([1],\omega\PreCat_\infty)' \simeq \omega\PreCat_\infty$ so that we obtain $\gamma$. The functor $\gamma$ is an equivalence over $\Op_\infty$ since the functors $\Fun([1], \omega\PreCat_\infty)^\cart \to \Fun(\{\bi\}, \omega\PreCat_\infty) \to \Op_\infty$ for $\bi=0,1$ are both equivalent to the functor
$\Fun([1], \omega\PreCat_\infty)^\cart \to \Fun([1], \Op_\infty)' \simeq 
\Op_\infty$, where $\Fun([1], \Op_\infty)' $ is the full subcategory of equivalences. 

\end{proof}

\begin{lemma}\label{cocartt}
For any cocartesian fibration $\phi: \mC \to \mD$ the induced functor
$$\rho_\phi: \Fun([1],\mC)^\cocart \to \Fun(\{0\},\mC) \times_{ \Fun(\{0\},\mD) } \Fun([1],\mD)$$ is an equivalence, where the left hand side is the full subcategory of $\phi$-cocartesian morphisms. 
\end{lemma}

\begin{proof}
For any $\infty$-category $\K$ the functor $\Fun(\K,\phi): \Fun(\K,\mC) \to \Fun(\K,\mD)$ is a cocartesian fibration
whose cocartesian morphisms are object-wise $\phi$-cocartesian. So $\Fun(\K,\rho_\phi)$ is equivalent to $\rho_{\Fun(\K,\phi)}$.
Thus it is enough to see that $\rho_\phi$ induces a bijection on equivalence classes.
Since $\phi$ is a cocartesian fibration, $\rho_\phi$ is essentially surjective.
By the defining property of $\phi$-cocartesian morphism for any
$\A, \B \in  \Fun([1],\mC)^\cocart$ and $\alpha: \rho_\phi(\A) \to \rho_\phi(\B)$ there is a morphism $\A \to \B$ lying over $\alpha$ that is an equivalence if
$\alpha$ is. So $\rho_\phi$ induces an injection on equivalence classes.
\end{proof}

\begin{corollary}\label{cocccq}

Let $\phi: \mM \to [1]$ be a bicartesian fibration classifying an adjunction $\mC \rightleftarrows \mD$ with fully faithful right adjoint.
The functor
$\rho_\phi: \Fun([1],\mM)^\cocart \to \Fun(\{0\},\mM) \times_{ \Fun(\{0\},[1]) } \Fun([1],[1])$ induces on the fiber over the identity of $[1]$
the equivalence 
$$\mD \times_{\Fun(\{1\},\mC)} \Fun([1],\mC)' \simeq \Fun_{[1]}([1],\mM)^\cocart \to \mC $$
evaluating at the source, where $\Fun([1],\mC)'\subset \Fun([1],\mC)$ is the full subcategory of local equivalences.	

\end{corollary}

\vspace{2mm}

\section{A monoidal equivalence}
\label{moneq}
By Corollary \ref{cortos} for any presentably monoidal $\infty$-category $\mV^\ot \to \Ass$ there is a canonical inclusion 
\begin{equation}\label{Eqt}
\LMod_\mV(\Pr^\L) \hookrightarrow \widehat{\Cat}_\infty^{\mV}.
\end{equation} 
In this section we prove that the inclusion (\ref{Eqt}) is lax $\bE_{\n}$-monoidal with respect to the relative tensor product
of $\mV$-modules in $\Pr^\L $, the $\infty$-category of presentable $\infty$-categories, and the tensor product of $\mV$-categories
constructed by Gepner-Haugseng \cite[Proposition 4.3.10.]{GEPNER2015575}
if $\mV$ is an $\bE_{\n+1}$-monoidal $\infty$-category for $1 \leq \n\leq \infty.$
To define $\bE_{\n}$-monoidal $\infty$-categories
we use the notion of symmetric $\infty$-operad \cite[Definition 2.1.1.10.]{lurie.higheralgebra}, the symmetric counterpart of $\infty$-operad in the sense of Definition \ref{opap}.
More generally, we consider $\mO$-monoidal $\infty$-categories for any symmetric
$\infty$-operad $\mO$ and prove that the inclusion \ref{Eqt}
is lax $\mO$-monoidal when $\mV$ is an $\mO$-algebra in $\Alg(\Pr^\L)$, the $\infty$-category of presentably monoidal $\infty$-categories (Theorem \ref{corrr}).
 
We start with the definition of symmetric $\infty$-operads:

\begin{notation}
Let $\Fin_\ast$ be the category of pointed finite sets. We denote pointed finite sets as $\langle \n \rangle:= \{ \ast, 1, ..., \n\}$ for $\n \geq 0$, where $\ast$ is the base point.
We call a map $\theta$ of pointed finite sets $\langle \n \rangle \to \langle \m \rangle$ 
\begin{itemize}
\item inert if for every $1 \leq \bi \leq \m$ the fiber of $\theta$ over $\bi$ consists precisely of one element.
\item standard inert if $\m=1$.
\end{itemize}
For every $\n \geq 0$ there are $\n$-many standard inert morphisms $\langle \n \rangle \to \langle 1 \rangle$, where the $\bi$-th standard inert morphism $ \langle \n \rangle \to \langle 1 \rangle$ for $1 \leq \bi \leq \n$ sends $\bi$ to 1.
\end{notation}
\begin{definition}
A symmetric $\infty$-operad is a cocartesian fibration $\phi: \mO^\ot \to \Fin_*$ relative to the collection of inert morphisms such that the following conditions are satisfied:
\begin{enumerate}
\item for any $\langle \n \rangle \in \Fin_*$ 
the standard inert morphisms $\langle \n \rangle \to \langle 1 \rangle $ for $1 \leq \bi \leq \n$ induce an equivalence $\mO^\ot_{\langle \n \rangle} \to \mO^{\times \n}$, where $\mO:= \mO^\ot_{\langle 1 \rangle} $, and $\mO^\ot_{\langle 0 \rangle}$ is contractible.
\vspace{1mm}
\item for any $\Y \in \mO^\ot$ lying over $\langle\n\rangle$ and $\Z \in \mO^\ot$ lying over $\langle\m\rangle$ the $\phi$-cocartesian lifts $\Y \to \Y_{\bi}$ of the standard inert morphisms induce an equivalence
$$ \mO^\ot(\Z,\Y)\to \Fin_*(\langle \m \rangle,\langle \n \rangle) \times_{\prod_{\bi=1}^\n \Fin_*(\langle \m \rangle,\langle 1 \rangle)} \prod_{\bi=1}^\n \mO^\ot(\Z,\Y_{\bi}).$$
\end{enumerate}

\end{definition}

For any symmetric $\infty$-operad $\mO^\ot \to \Fin_*$ we call
$\mO:= \mO^\ot_{\langle 1 \rangle}$ the underlying $\infty$-category.

We also consider maps of symmetric $\infty$-operads:
\begin{definition}
Let $\mO^\ot \to \Fin_*, \mO'^\ot \to \Fin_*$ be symmetric $\infty$-operads.

\begin{itemize}
\item A map of symmetric $\infty$-operads $\phi: \mO'^\ot \to \mO^\ot$ or $\mO'$-algebra in $\mO$
is a map of cocartesian fibrations $\mO'^\ot \to \mO^\ot$ relative to the collection of inert morphisms of $\Fin_*$.	

\item An $\mO$-monoidal $\infty$-category is a map of symmetric $\infty$-operads $\phi: \mO'^\ot \to \mO^\ot$ that is a cocartesian fibration. For $\mO^\ot = \Fin_*$ we call  $\mO$-monoidal $\infty$-categories symmetric monoidal $\infty$-categories. 
\end{itemize}
Let $ \mO'^\ot \to \mO^\ot, \mO''^\ot \to \mO^\ot$ be $\mO$-monoidal $\infty$-categories.

\begin{itemize}

\item A lax $\mO$-monoidal functor $ \mO'^\ot \to \mO''^\ot$ is a map of symmetric $\infty$-operads $ \mO'^\ot \to \mO''^\ot$ over $\mO^\ot.$

\item An $\mO$-monoidal functor $ \mO'^\ot \to \mO''^\ot$ is a map of cocartesian fibrations over $\mO^\ot.$
\end{itemize}

\end{definition}

\begin{notation}
Let $ \mO'^\ot \to \mO^\ot, \mO''^\ot \to \mO^\ot$ be maps of symmetric $\infty$-operads.
Let $$\Alg_{\mO'/\mO}(\mO'') \subset \Fun_{\mO^\ot}(\mO'^\ot, \mO''^\ot)$$
be the full subcategory of maps of symmetric $\infty$-operads over $\mO^\ot \to \Fin_*.$
For $\mO^\ot =\Fin_*$ we set $$\Alg_{\mO'}(\mO''):=\Alg_{\mO'/\mO}(\mO'').$$
\end{notation}

\begin{remark} By \cite[Definition 5.1.0.4., Theorem 5.1.2.2]{lurie.higheralgebra} for every $\n \geq 0$ there are symmetric $\infty$-operads $\bE_\n \to \Fin_*$
and canonical maps of symmetric $\infty$-operads
$\bE_{\n} \to \bE_{\n+1}$ for $\n \geq 0$ such that the canonical map
$\bE_\infty:=\colim_{\n\geq 0} (\bE_{\n}) \to \Fin_*$ is an equivalence.
For any $\n, \m \geq 0$ and symmetric monoidal $\infty$-category
$\mC^\ot \to \Fin_*$ there is a canonical equivalence $$\Alg_{\bE_{\n+\m}}(\mC) \simeq \Alg_{\bE_{\n}}(\Alg_{\bE_\m}(\mC)).$$

Note that $\bE_1$-algebras are by definition associative algebras and we set $\Alg(\mC):= \Alg_{\bE_1}(\mC).$

\end{remark}

\begin{definition}
We call a symmetric monoidal $\infty$-category $\mO^\ot \to \Fin_*$ cartesian 
if the tensor unit $\tu$ is a final object of $\mO$ and for any
$\X,\Y \in \mO$ the canonical maps $\X \ot \Y \to \X \ot \tu \simeq \X, \X \ot \Y \to \tu \ot \Y \simeq \Y$ exhibit $\X \ot \Y$ as the product of $\X$ and $\Y$ in $\mO.$
\end{definition}

\begin{definition} Let $\mD$ be an $\infty$-category with finite products and $\phi: \mO^\ot \to \Fin_*$ a symmetric $\infty$-operad. We call a functor $\alpha: \mO^\ot \to \mD$ an $\mO$-monoid in $\mD$ if for every $\X \in \mO^\ot$ lying over $\langle \n \rangle \in \Fin_*$ the $\phi$-cocartesian lifts $\X \to \X_{\bi} $ of the standard inert morphisms induce an equivalence $\alpha(\X) \to \prod_{\bi=1}^\n \alpha(\X_{\bi}).$ 
\end{definition}\begin{example}
Note that $\mO$-monoids in $\Cat_\infty$ are classified by $\mO$-monoidal $\infty$-categories.
\end{example}
\begin{notation}Let $\Mon_{\mO}(\mD) \subset \Fun(\mO^\ot,\mD) $ be the full subcategory of $\mO$-monoids and set $\Cmon(\mD):= \Mon_{\Fin_*}(\mD).$ \end{notation}
By \cite[Proposition 2.4.1.7.]{lurie.higheralgebra} for any $\infty$-category $\mD$ with finite products there is a cartesian symmetric monoidal $\infty$-category $\mD^\times \to \Fin_*$, where  $\mD^\times_{\langle 1 \rangle} \simeq \mD,$
and a $\mD^\times$-monoid in $ \mD$ such that for any symmetric $\infty$-operad $\mO^\ot \to \Fin_*$ the functor $ \Alg_\mO(\mD) \to \Mon_{\mO}(\mD)$ is an equivalence.

\vspace{2mm}
Next we recall the definition of tensor product of enriched $\infty$-precategories.
By \cite[Proposition 3.6.17.]{GEPNER2015575} the $\infty$-category $\PreCat_\infty $ carries a symmetric monoidal structure that provides the outer tensor product of enriched $\infty$-precategories.

The outer tensor product of two $\infty$-precategories $\mC: \Ass_\X \to \mV^\ot, \mD:\Ass_\Y \to \mW^\ot$ is the $\infty$-precategory
$\mC \boxtimes \mD: \Ass_{\X \times \Y} \simeq \Ass_\X \times_\Ass \Ass_\Y \to \mV^\ot \times_\Ass \mW^\ot$
enriched in $\mV \times \mW$ with space of objects $\X \times \Y$.
The tensor unit is the $\infty$-precategory $\id: \Ass \to \Ass$ enriched in the final
$\infty$-operad with contractible space of objects, which is the final object of $\PreCat_\infty.$ We make the following observation:
\begin{lemma}\label{obsa}
Let $\mC,\mD$ be $\infty$-precategories enriched in $\infty$-operads $\mV^\ot \to \Ass,\mW^\ot \to \Ass$, respectively, with space of objects $\X,\Y$, respectively.
The canonical maps $\mC \boxtimes \mD \to \mC \boxtimes \tu \simeq \mC, \ \mC \boxtimes \mD \to \tu \boxtimes \mD \simeq \mD$ in $\PreCat_\infty$ exhibit $\mC \boxtimes \mD$ as the product of $\mC, \mD$ in $\PreCat_\infty$.
\end{lemma}
\begin{proof}
For any $\infty$-precategory $\mB: \Ass_\Z \to \mU^\ot$ the map $\PreCat_\infty(\mB, \mC \boxtimes \mD) \to \PreCat_\infty(\mB, \mC) \times \PreCat_\infty(\mB, \mD)$ induces on the fiber over any map of spaces $\alpha= (\alpha_1,\alpha_2): \Z \to \X \times \Y$ and
map of $\infty$-operads $\varphi=(\varphi_1,\varphi_2): \mU^\ot \to \mV^\ot \times_\Ass \mW^\ot$ the equivalence $$\Alg_{\Ass_\Z}(\mV \times \mW)(\varphi_!(\mB), \alpha^*(\mC \boxtimes \mD)) \to \Alg_{\Ass_\Z}(\mV)((\varphi_1)_!(\mB), \alpha_1^*(\mC)) \times \Alg_{\Ass_\Z}(\mW)((\varphi_2)_!(\mB), \alpha_2^*(\mD)).$$
\end{proof}

By Lemma \ref{obsa} the $\infty$-category $\PreCat_\infty$ admits finite products,
which are preserved by the forgetful functor $\PreCat_\infty \to \mS \times \Op_\infty$,
and the cartesian structure $\PreCat_\infty^\times \to \Fin_*$ provides the outer tensor product of enriched $\infty$-precategories.

\begin{remark}\label{Reyolk}
The induced symmetric monoidal functor $\PreCat_\infty^\times \to \Op_\infty^\times$ on cartesian structures is a cocartesian fibration.
This follows by \cite[Lemma A.1.8.]{haugseng_melani_safronov_2020} since the functor $\kappa: \PreCat_\infty \to \Op_\infty$ is a cocartesian fibration (Remark \ref{abc}) and the $\kappa$-cocartesian morphisms are closed under finite products:
a morphism $\mC \to \mD$ in $\PreCat_\infty$ lying over a map of 
$\infty$-operads $\alpha: \mV^\ot \to \mW^\ot$ and a map of spaces $\tau: \X \to \Y$ is $\kappa$-cocartesian if and only if $\tau$ is an equivalence and for any 
$\A,\B \in \X$ the induced morphism $\alpha(\mC(\A,\B)) \to \mD(\tau(\A),\tau(\B))$
in $\mW$ is an equivalence. 

Similarly, the induced symmetric monoidal functor $\P\PreCat_\infty^\times \to \Mon^\times$ on cartesian structures is a cocartesian fibration since the functor $\rho: \P\PreCat_\infty \to \Mon$ is a cocartesian fibration (Remark \ref{abc}) and the $\rho$-cocartesian morphisms are closed under finite products:
a morphism $\mC \to \mD$ in $\P\PreCat_\infty$ lying over a map of 
monoidal $\infty$-categories $\alpha: \mV^\ot \to \mW^\ot$ and a map of spaces $\tau: \X \to \Y$ is $\rho$-cocartesian if and only if $\tau$ is an equivalence and for any $\A,\B \in \X$ the induced morphism $\alpha_!(\mC(\A,\B)) \to \mD(\tau(\A),\tau(\B))$ in $\mP(\mW)$ is an equivalence. 
\end{remark}
 
\begin{notation}\label{tenss}
	
For any symmetric $\infty$-operad $\mO^\ot \to \Fin_*$ and $\mO$-algebra $\mV$ in $\Op_\infty^\times $
we write $$ (\PreCat^\mV_\infty)^\ot:= \mO^\ot \times_{\Op_\infty^\times} \PreCat_\infty^\times \to \mO^\ot$$ for the pullback along $\mV. $
We write $ (\Cat_\infty^\mV)^\ot \subset (\PreCat^\mV_\infty)^\ot$ for the full suboperad spanned by the objects that belong to $ \Cat_\infty^{\mV(\X)} \subset \PreCat_\infty^{\mV(\X)}$ for some $\X \in \mO.$
\end{notation}

Since the functor $\PreCat_\infty^\times \to \Op_\infty^\times$ is a cocartesian fibration, $(\PreCat^\mV_\infty)^\ot \to \mO^\ot$ is a cocartesian fibration that provides the tensor product of $\mV$-precategories.
Theorem \ref{trfd} implies that the embedding $(\Cat^\mV_\infty)^\ot \subset (\PreCat^\mV_\infty)^\ot$ admits a left adjoint relative to $\mO^\ot$
using the description of local equivalences with local target.
By \cite[Lemma 2.2.4.11.]{lurie.higheralgebra} this implies that the restriction $ (\Cat_\infty^\mV)^\ot \to \mO^\ot$ is a cocartesian fibration providing the tensor product of $\mV$-categories.

\begin{notation}

Let $\P\PreCat_\infty^\ot$ be the pullback of $\widehat{\PreCat}_\infty^\times \to  \widehat{\Op}_\infty^\times \times_{\Fin_*} \widehat{\mS}^\times $ along
the functors $\mP: \Mon^\times \to \widehat{\Mon}^\times \subset \widehat{\Op}_\infty^\times$ and $\mS^\times \subset \widehat{\mS}^\times.$

\end{notation}

\begin{lemma}\label{obsay}
The functor $\P\PreCat_\infty^\ot \to \Fin_*$ exhibits $\P\PreCat_\infty$
as the cartesian symmetric monoidal structure.

\end{lemma}
\begin{proof}
The tensor unit of $\P\PreCat_\infty^\ot \to \Fin_*$
lies over the final $\infty$-operad and is the tensor unit of the fiber
$\Fin_* \times_{\Mon^\times} \P\PreCat_\infty^\ot \simeq (\PreCat_\infty^\mS)^\ot \simeq (\PreCat_\infty^\mS)^\times,$ which is the final $\infty$-precategory enriched in spaces.	

Let $\mV^\ot \to \Ass,\mW^\ot \to \Ass$ be small monoidal $\infty$-categories
and $\mC,\mD$ be $\infty$-precategories enriched in $\mP(\mV)^\ot \to \Ass, \mP(\mW)^\ot \to \Ass$, respectively, with space of objects $\X,\Y$, respectively.
Let $\mC*\mD$ be the image of $\mC, \mD$ under the
functor $\PreCat_\infty^{\mP(\mV)} \times \PreCat_\infty^{\mP(\mW)}
\to \PreCat_\infty^{\mP(\mV) \times \mP(\mW)} \to \PreCat_\infty^{\mP(\mV  \times \mW)}.$
The latter functor is induced by the monoidal functor
$\mP(\mV) \times \mP(\mW) \to \mP(\mV  \times \mW)$ that sends 
$\A, \B$ to $\L, \T \mapsto \A(\L) \times \B(\T)$, and is left adjoint to the functor $\mP(\mV  \times \mW) \to \mP(\mV) \times \mP(\mW)$ induced by the functors $\mV \times \mW \to \mV, \mV \times \mW \to \mW.$

We will prove that the canonical maps $\mC * \mD \to \mC * \tu \simeq \mC, \ \mC * \mD \to \tu * \mD \simeq \mD$ in $\PreCat_\infty$ exhibit $\mC * \mD$ as the product of $\mC, \mD$ in $\P\PreCat_\infty$.
For any monoidal $\infty$-category $\mU^\ot\to \Ass$ and $\infty$-precategory $\mB: \Ass_\Z \to \mP(\mU)^\ot$ the map $\P\PreCat_\infty(\mB, \mC * \mD) \to \P\PreCat_\infty(\mB, \mC) \times \P\PreCat_\infty(\mB, \mD)$ induces on the fiber over any map of spaces $\alpha= (\alpha_1,\alpha_2): \Z \to \X \times \Y$ and
map of monoidal $\infty$-categories $\varphi=(\varphi_1,\varphi_2): \mU^\ot \to \mV^\ot \times_\Ass \mW^\ot$ the equivalence $$\Alg_{\Ass_\Z}(\mP(\mV \times \mW))(\varphi_!(\mB), \alpha^*(\mC * \mD))  \simeq \Alg_{\Ass_\Z}(\mP(\mV) \times \mP(\mW))(((\varphi_1)_!(\mB), (\varphi_2)_!(\mB)), \alpha^*(\mC \boxtimes \mD)) \to$$$$ \Alg_{\Ass_\Z}(\mP(\mV))((\varphi_1)_!(\mB), \alpha_1^*(\mC)) \times \Alg_{\Ass_\Z}(\mP(\mW))((\varphi_2)_!(\mB), \alpha_2^*(\mD)).$$

\end{proof}

\vspace{2mm}
Next we explain how to construct the relative tensor product of left tensored $\infty$-categories.
By \cite[Proposition 2.4.3.16., 3.2.4.7.]{lurie.higheralgebra}
the forgetful functor $\Cmon(\Op_{\infty}^\mon) \to \Cmon(\Cat_\infty)$ is an equivalence.
So any symmetric monoidal $\infty$-category $\mV^\ot \to \Fin_*$ classifying a commutative monoid in $\Cat_\infty$ uniquely lifts to a commutative monoid
in $\Op_{\infty}^\mon$ classified by a cocartesian $\Fin_*$-family 
of monoidal $\infty$-categories $\bar{\mV}^\ot \to \Fin_* \times \Ass$,
whose fiber over $[1] \in \Ass$ is $\mV^\ot \to \Fin_*.$

\begin{notation}\label{commm}
For any symmetric monoidal $\infty$-category $\mV^\ot \to \Fin_*$ we write (see Notation \ref{modol}) $$\LMod(\mV)^\ot:= \LMod^{\Fin_*}(\bar{\mV}) \to \Fin_*, \ \Alg(\mV)^\ot:= \Alg^{\Fin_*}(\bar{\mV}) \to \Fin_*.$$
\end{notation}
The cocartesian fibrations $\LMod(\mV)^\ot \to \Fin_*, \Alg(\mV)^\ot \to \Fin_*$
are symmetric monoidal $\infty$-categories and by Example \ref{exemp} there are canonical symmetric monoidal functors $$\gamma:\LMod(\mV)^\ot \to \Alg(\mV)^\ot,\ \LMod(\mV)^\ot \to \mV^\ot.$$

\begin{remark}
If $\mV^\ot \to \Fin_*$ endows $\mV$ with the cartesian structure,
$\gamma$ is a symmetric monoidal functor between cartesian structures.
\end{remark}

\begin{lemma}\label{relprod}
Let $\mV^\ot \to \Fin_*$ be a symmetric monoidal $\infty$-category compatible with geometric realizations (which are $\Ass$-indexed colimits).
Then $\gamma: \LMod(\mV)^\ot \to \Alg(\mV)^\ot$ is a cocartesian fibration.
\end{lemma}

\begin{proof}
By \cite[Lemma 4.5.3.6.]{lurie.higheralgebra} the functor $\rho: \LMod(\mV) \to \Alg(\mV)$ is a cocartesian fibration. So by \cite[Lemma A.1.8.]{haugseng_melani_safronov_2020} it is enough to check that for any $\rho$-cocartesian morphisms $\alpha: \X \to \Y,\alpha': \X' \to \Y' $ the morphism $\alpha \ot \alpha'$ is $\rho$-cocartesian.
A morphism $\alpha: \X \to \Y$ in $\LMod(\mV)$ lying over a morphism
$\A \to \B $ in $\Alg(\mV)$ is $\rho$-cocartesian if and only if
the canonical compatible morphisms $ \X \ot \A^{\ot \n} \ot \B \to \Y$ for $[\n] \in \Ass$
induce an equivalence $\colim_{[\n] \in \Ass} \X \ot \A^{\ot \n} \ot \B \to \Y.$
This implies the claim using that $\mV^\ot \to \Fin_*$ is compatible with $\Ass$-indexed colimits and that the diagonal $\Ass \to \Ass \times \Ass$ is cofinal. 
\end{proof}

\begin{notation}\label{prsym}\emph{ } 
Let $(\Cat_\infty^{\rc\rc})^\ot \subset \widehat{\Cat}_\infty^\times $
be the symmetric suboperad whose colors are the $\infty$-categories having small colimits and and whose multi-morphisms $\mC_1,..., \mC_\n \to \mD$
correspond to a functor $\mC_1 \times ...\times \mC_\n \to \mD $ preserving small colimits component-wise.
Let $(\Pr^\L)^\ot \subset (\Cat_\infty^{\rc\rc})^\ot $ be the full symmetric suboperad of presentable $\infty$-categories.
\end{notation}
By \cite[Corollary 4.8.1.4.]{lurie.higheralgebra} the symmetric $\infty$-operad
$(\Cat_\infty^{\rc\rc})^\ot \to \Fin_*$ is a symmetric monoidal $\infty$-category compatible with small colimits
that by \cite[Proposition  4.8.1.15.]{lurie.higheralgebra} restricts to a symmetric monoidal $\infty$-category $(\Pr^\L)^\ot \to \Fin_*$ compatible with small colimits.
By Notation \ref{commm} there are symmetric monoidal $\infty$-categories
$\Alg(\Cat_\infty^{\rc\rc})^\ot \to \Fin_*, \LMod(\Cat_\infty^{\rc\rc})^\ot \to \Fin_*$ and a symmetric monoidal functor $\LMod(\Cat_\infty^{\rc\rc})^\ot \to \Alg(\Cat_\infty^{\rc\rc})^\ot$ that is a cocartesian fibration (Lemma \ref{relprod}), similarly for $\Pr^\L.$ 

\begin{notation}
For every symmetric $\infty$-operad $\mO^\ot \to \Fin_*$
and $\mO$-algebra $\mV$ in $\Alg(\Pr^\L)^\ot$ we write 
$$\LMod_\mV(\Cat_\infty^{\rc\rc})^\ot := \mO^\ot \times_{\Alg(\Cat_\infty^{\rc\rc})^\ot}  \LMod(\Cat_\infty^{\rc\rc})^\ot \to \mO^\ot, $$ 
$$\LMod_\mV(\Pr^\L)^\ot := \mO^\ot \times_{\Alg(\Pr^\L)^\ot}  \LMod(\Pr^\L)^\ot \to \mO^\ot.$$ 
\end{notation}

The embedding $(\Pr^\L)^\ot \subset (\Cat_\infty^{\rc\rc})^\ot $ induces for any symmetric $\infty$-operad $\mO^\ot \to \Fin_*$ and $\mO$-algebra $\mV$ in $\Alg(\Pr^\L)^\ot$ a lax $\mO$-monoidal embedding
$ \LMod_\mV(\Pr^\L)^\ot \subset \LMod_\mV(\Cat_\infty^{\rc\rc})^\ot $.
Note that this embedding is $\mO$-monoidal is a consequence of the proof of Lemma \ref{relprod} because the embedding $\Pr^\L \subset \Cat_\infty^{\rc\rc}$ preserves geometric realizations by \cite[Proposition 5.1.4.]{stefanich2020presentable}.

\begin{notation}For every symmetric $\infty$-operad $\mO^\ot \to \Fin_*$
and $\mO$-algebra $\mV$ in $\Mon^\times$ we write 
$$\LMod_\mV^\ot= \mO^\otimes \times_{\Mon^\times}\LMod^\times.$$	

\end{notation}

The symmetric monoidal functor $\LMod^\times \to \Mon^\times$ is a cocartesian fibration by Lemma \ref{relprod} and Remark \ref{modules} so that $\LMod_\mV^\ot \to \mO^\ot$ is an $\mO$-monoidal $\infty$-category.

\begin{theorem}\label{corrr}
Let $\mO^\ot \to \Fin_*$ be a symmetric $\infty$-operad.
\begin{enumerate}
\item Let $\mV$ be an $\mO$-monoid in $\Mon$.
There is a canonical lax $\mO$-monoidal inclusion 
$$ \LMod_\mV^\ot \hookrightarrow (\Cat_\infty^{\mP(\mV)})^\otimes.$$

\vspace{1mm}	
	
\item Let $\mV$ be an $\mO$-algebra in $\Alg(\Pr^\L)^\ot$.
There is a canonical lax $\mO$-monoidal inclusion 
$$\LMod_\mV(\Cat_\infty^{\rc\rc})^\ot \hookrightarrow (\widehat{\Cat}_\infty^\mV)^\ot$$ and so a lax $\mO$-monoidal inclusion 
$ \LMod_\mV(\Pr^\L)^\ot \hookrightarrow (\widehat{\Cat}_\infty^\mV)^\ot.$
\end{enumerate}	

\end{theorem}

\begin{proof}
(1): Corollary \ref{eqvvjj} gives a canonical embedding
$\P\LMod \hookrightarrow \P\PreCat_\infty$ over $\mS \times \Mon$.
Thus Lemma \ref{obsay} gives a canonical embedding of symmetric monoidal
$\infty$-categories $\P\LMod^\times \hookrightarrow \P\PreCat_\infty^\times \simeq \P\PreCat_\infty^\otimes $ over $\mS^\times \times_{\Fin_*} \Mon^\times$
whose pullback along $\mV: \mO^\ot \to \Mon^\times$ 
gives an embedding $$ \P\LMod_\mV^\ot:= \mO^\otimes \times_{\Mon^\times}\P\LMod^\times \hookrightarrow (\PreCat_\infty^{\mP(\mV)})^\otimes $$ of $\mO$-monoidal $\infty$-categories
that induces an equivalence $ \P\LMod_\mV^\ot\simeq (\Cat_\infty^{\mP(\mV)})^\otimes$ of $\mO$-monoidal $\infty$-categories.
The pullback along $\mV: \mO^\ot \to \Mon^\times$ of the lax $\Mon^\times$-monoidal inclusion $\LMod^\times \subset \P\LMod^\times $ is the desired lax
$\mO$-monoidal inclusion $$ \LMod_\mV^\ot= \mO^\otimes \times_{\Mon^\times}\LMod^\times \hookrightarrow \P\LMod_\mV^\ot \simeq (\Cat_\infty^{\mP(\mV)})^\otimes.$$

(2): Corollary \ref{eqvvjj} gives a canonical equivalence
$ (\widehat{\Cat}_\infty^\mV)^\otimes \simeq \mO^\otimes \times_{\widehat{\Op}_\infty^\times}\widehat{\Enr}_\Lur^\times$
of symmetric $\infty$-operads over $\mO^\ot.$
By definition we have $\LMod_\mV(\Cat_\infty^{\rc\rc})^\ot= \mO^\otimes \times_{\Alg(\Cat_\infty^{\rc\rc})^\ot}\LMod(\Cat_\infty^{\rc\rc})^\ot.$

The lax symmetric monoidal inclusion $(\Cat_\infty^{\rc\rc})^\ot \subset \widehat{\Cat}_\infty^\times $ induces a lax symmetric monoidal inclusion 
$$\LMod(\Cat_\infty^{\rc\rc})^\ot \subset \LMod(\widehat{\Cat}_\infty)^\times \simeq \widehat{\LMod}^\times \subset \omega\widehat{\LMod}^\times.$$ 
The latter restricts to a lax symmetric monoidal inclusion $\LMod(\Cat_\infty^{\rc\rc})^\ot \subset \widehat{\Enr}_\Lur^\times $ over $\Alg(\Cat_\infty^{\rc\rc})^\ot \subset\widehat{\Op}_\infty^\times,$
whose pullback along $\mV:\mO^\ot \to \Alg(\Cat_\infty^{\rc\rc})^\ot $ is the desired lax $\mO$-monoidal inclusion.
\end{proof}

\begin{remark}\label{futur}
	
One can show that the inclusion $ \LMod_\mV(\Cat_\infty^{\rc\rc})^\ot \hookrightarrow (\widehat{\Cat}_\infty^\mV)^\ot$
has an $\mO$-monoidal left adjoint sending a small
$\mV$-enriched $\infty$-category $\mC$ to $\mP_\mV(\mC)^\circledast \to \mV^\ot.$
This will be topic of future work.
	
\end{remark}

\begin{corollary}
Let $1 \leq \n \leq \infty$ and $\mV^\ot \to \bE_{\n+1}$ be an $\bE_{\n+1}$-monoidal $\infty$-category viewed as an $\bE_{\n}$-monoid in $\Mon.$
Then $\mV$ is an $\bE_\n$-monoidal $\mP(\mV)$-enriched $\infty$-category.

\end{corollary}

\begin{proof}
	
By Theorem \ref{corrr} there is a canonical lax $\bE_\n$-monoidal inclusion 
$ \LMod_\mV^\ot \hookrightarrow (\Cat_\infty^{\mP(\mV)})^\otimes$
that induces an inclusion
$\Alg_{\bE_{\n}}(\LMod_\mV) \hookrightarrow \Alg_{\bE_{\n}}(\Cat_\infty^{\mP(\mV)})$
that sends the initial object lying over the tensor unit of $\LMod_\mV$
and lying over $\mV \in \Alg_{\bE_{\n}}(\Cat_\infty)$ to
an $\bE_\n$-monoidal $\mP(\mV)$-enriched $\infty$-category.
		
\end{proof}

\begin{corollary}
Let $1 \leq \n \leq \infty$ and $\mV^\ot \to \bE_{\n+1}$ be a closed $\bE_{\n+1}$-monoidal $\infty$-category. Then $\mV$ is an $\bE_\n$-monoidal $\mV$-enriched $\infty$-category.

\end{corollary}

\section{An $(\infty,2)$-categorical refinement of the equivalence}\label{2Fu}

\vspace{1mm}

Let $\mV^\ot \to \Ass$ be a small $\infty$-operad.
In this section we will prove that $\chi: \omega\LMod_\mV \to \Cat_\infty^{\mP\Env(\mV)}$ is a $\Cat_\infty$-linear equivalence, an equivalence of $(\infty,2)$-categories,
with respect to natural left $\Cat_\infty$-actions on source and target of $\chi$,
which we explain in the following.

By Lemma \ref{2enr} the $\infty$-category $\omega\LMod_\mV $ carries a canonical left $\Cat_\infty$-action.

\begin{notation}
	
Let $(\Op_\infty^{\mon, \rc\rc})^\ot\subset (\widehat{\Op}_\infty^\mon)^\times$ be the $\infty$-operad whose colors are the monoidal $\infty$-categories having small colimits and and whose multi-morphisms $\mC_1,..., \mC_\n \to \mD$
correspond to a monoidal functor $\mC_1 \times ...\times \mC_\n \to \mD $ preserving small colimits component-wise.
\end{notation}

By \cite[Corollary 4.8.1.4.]{lurie.higheralgebra} the $\infty$-operad
$(\Op_\infty^{\mon, \rc\rc})^\ot \to \Ass$ is a monoidal $\infty$-category compatible with small colimits and the monoidal forgetful functor
$(\widehat{\Op}_\infty^\mon)^\times \to (\widehat{\Cat}_\infty)^\times$
on cartesian structures restricts to a monoidal functor
$(\Op_\infty^{\mon, \rc\rc})^\ot \to (\Cat_\infty^{\rc\rc})^\ot.$
Moreover $\mS^\times \to \Ass$ is the tensor unit of $\Op_\infty^{\mon, \rc\rc}$.

\begin{construction}\label{CoN}

Because $\mS^\times \to \Ass$ is the tensor unit of $\Op_\infty^{\mon, \rc\rc}$, the forgetful functor $$\LMod_\mS(\Op_\infty^{\mon, \rc\rc}) \to \Op_\infty^{\mon, \rc\rc}$$ is an equivalence.
So for any monoidal $\infty$-category $\mW^\ot \to \Ass$ compatible with small colimits there is a $\LM$-algebra $\alpha_\mW$ in $\Op_\infty^{\mon, \rc\rc}$ encoding a left action of $\mS^\times \to \Ass$ on $\mW^\ot \to \Ass$.
 
By Remark \ref{Reyolk} the monoidal functor $\widehat{\PreCat}_\infty^\times \to \widehat{\Op}_\infty^\times$ on cartesian structures is a cocartesian fibration.
The pullback of $\widehat{\PreCat}_\infty^\times \to \widehat{\Op}_\infty^\times$
along $\alpha_\mW: \LM \to (\Op_\infty^{\mon, \rc\rc})^\ot \subset (\widehat{\Op}_\infty^\mon)^\times \subset (\widehat{\Op}_\infty)^\times $
provides a left action of $\PreCat_\infty^\mS $ on $\PreCat_\infty^{\mW}$
that restricts to left action of $\Cat_\infty^\mS $ on $\Cat_\infty^{\mW}$.
  
By Corollary \ref{eqvvjj} the functor $\chi: \Cat_\infty^{\mS, \Lur} \to \Cat_\infty^\mS$ is an equivalence. By Lemma \ref{Caa} the forgetful functor
$ \Cat_\infty^{\mS, \Lur} \to \Cat_\infty$ is an equivalence.
The resulting equivalence
$\Cat_\infty \simeq \Cat_\infty^\mS$ induces a monoidal equivalence on cartesian structures. So we obtain a left action of $\Cat_\infty$ on $\Cat_\infty^{\mW}$.
\end{construction}

By Construction \ref{CoN} for any small $\infty$-operad $\mV^\ot \to \Ass$ there is a left $\Cat_\infty $-action on $\Cat_\infty^{\mP\Env(\mV)}$.

We prove the following theorem:

\begin{theorem}\label{PP}
	
Let $\mV^\ot \to \Ass$ be a small $\infty$-operad.
The functor $\chi: \omega\LMod_\mV \to \Cat_\infty^{\mP\Env(\mV)}$ is $\Cat_\infty$-linear.	
	
\end{theorem}

We prepare the proof of Theorem \ref{PP}:

\begin{lemma}\label{LyN}
Let $\mV^\ot \to \Ass$ be a small $\infty$-operad.	
The functor $\LMod_{\B\Env(\mV)} \to \omega\LMod_\mV$ restricting along the embedding $\mV \subset \B\Env(\mV)$ is $\Cat_\infty$-linear and admits a $\Cat_\infty$-linear left adjoint, where source and target carry the left $\Cat_\infty$-actions of Lemma \ref{2enr}.
	
\end{lemma}

\begin{proof}
	
The functor $\LMod_{\B\Env(\mV)} \to \omega\LMod_\mV$ restricting along the embedding $\mV \subset \B\Env(\mV)$ is $\Cat_\infty$-linear by definitions of the actions of Lemma \ref{2enr}. This functor admits a left adjoint 
$\B\L\Env(-): \omega\LMod_\mV \to \LMod_{\B\Env(\mV)}$ that is $\Cat_\infty$-linear since for any small $\infty$-category $\K$ and weakly left tensored $\infty$-category $\mM^\circledast \to \mV^\ot$ the canonical functor
$\K \times \L\Env(\mM)^\circledast \to \L\Env(\K \times \mM)^\circledast$
is an equivalence.	
	
\end{proof}

Corollary \ref{kojj} gives the following remark:

\begin{remark}\label{Reumi}
Let $\mN^\circledast \to \mV^\ot $ be a presentably left tensored $\infty$-category, $\A$ an associative algebra in $\mV$, $\X \in \mN$ and $\RMod_\A(\mV) \to \mN$ an equivalence sending
$\A$ to $\X$ satisfying the following conditions:
\begin{enumerate}
\item The functor $\mV \xrightarrow{(-)\ot \A} \RMod_\A(\mV) \simeq \mN $
is the functor $(-)\ot \X.$ 

\vspace{1mm}
\item The induced equivalence $\Ho(\RMod_\A(\mV)) \to \Ho(\mN)$ is $\Ho(\mV)$-linear.
\end{enumerate}

\end{remark}

\begin{proof}

By Theorem \ref{tgre} there is a unique small colimits preserving $\mV$-linear functor
$\RMod_\A(\mV)^\circledast \to \mN^\circledast$ preserving $\A.$
The $\infty$-category $\RMod_\A(\mV)$ is generated by $\A$ under small colimits and the left $\mV$-action. So by (2) also $\mN$ is generated by $\X$ under small colimits and the left $\mV$-action. By (1) the functor $ \RMod_\A(\mV) \simeq \mN \xrightarrow{\Mor_{\mN}(\X,-)} \mV$ is the forgetful functor. Thus the functor $\Mor_{\mN}(\X,-): \mN \to \mV$ preserves small colimits. By (2) there is a factorization 
$$ \Ho(\Mor_{\RMod_\A(\mV)}(\A,-)): \Ho(\RMod_\A(\mV)) \simeq \Ho(\mN) \xrightarrow{\Ho(\Mor_{\mN}(\X,-))}\Ho(\mV)$$ of lax $\Ho(\mV)$-linear functors.
Thus $\Ho(\Mor_{\mN}(\X,-)): \Ho(\mN) \to \Ho(\mV)$ is $\Ho(\mV)$-linear
so that $\Mor_{\mN}(\X,-): \mN \to \mV$ is $\mV$-linear.
So $\theta$ is an equivalence by Corollary \ref{kojj}.
	
\end{proof}

The forgetful functor $\LMod_{[0]}(\Op^\mon_\infty)\to \Op^\mon_\infty$ is an equivalence because $\Ass$ is the final monoidal $\infty$-category. For any monoidal $\infty$-category $\mV^\ot \to \Ass$ let $\beta_\mV$ be the $\LM$-algebra in $\Op^\mon_\infty$ corresponding to $\mV^\ot \to \Ass$.

\begin{lemma}\label{Clav}
Let $\mV^\ot \to \Ass$ be a small monoidal $\infty$-category.	
The $\LM$-monoidal $\infty$-category $$\beta_\mV^*(\LMod^\times) \to \LM$$
classifies the left $\Cat_\infty$-action on $\LMod_{\mV}$ of Lemma \ref{2enr}.  

\end{lemma}

\begin{proof}
The statement follows from Remark \ref{Reumi} because for any $\K \in \Cat_\infty$ and $\mM^\circledast \to \mV^\ot$ the left $\Cat_\infty$-action on $\LMod_{\mV}$ of Lemma \ref{2enr} 
gives $\K \times \mM^\circledast \to \mV^\ot$, where the left $\Cat_\infty \simeq \LMod_*$-action classified by $\beta_\mV^*(\LMod^\times) \to \LM$ gives $(\K \times \Ass) \times_\Ass \mM^\circledast \to \Ass \times_\Ass \mV^\ot$.	

\end{proof}

\begin{definition}Let $\mC$ be an $\infty$-category.
We call a morphism $\X \to \Y$ of $\mC$ a monomorphism
if for every $\Y \in \mC$ the induced map
$\mC(\Z,\X)\to \mC(\Z,\Y)$ is an embedding of spaces.	
	
\end{definition}

\begin{example}\label{EX}
A map of spaces is a monomorphism in $\mS$ if and only if it is an embedding
because for every space $\X$ the functor $\mS(\X,-): \mS \to \mS$ preserves embeddings.
	
\end{example}

\begin{remark}\label{Monol}
	
Let $\mC$ be an $\infty$-category. A morphism $\f: \X \to \Y$ of $\mC$ is a monomorphism if and only if the commutative square
\begin{equation*} 
\begin{xy}
\xymatrix{
\X \ar[d]^{\id} \ar[rr]^\id
&&\X \ar[d]^\f
\\  \X \ar[rr]^\f  && \Y
}
\end{xy} 
\end{equation*}
is a pullback square.

\end{remark}

\begin{proof}
	
Since monomorphisms and pullbacks are detected by mapping out from any object of $\mC$ this follows from the case $\mC=\mS.$
By Example \ref{EX} a map of spaces is a monomorphism if and only if it is an embedding. A map of spaces $\f:\X \to \Y$ is an embedding if and only if all its fibers are empty or contractible. This is equivalent to say that the latter square is a pullback square.
\end{proof}

\begin{remark}\label{LLo}
	
A morphism $\f: \X \to \Y$ of the fiber $\mC_\s$ for some $\s \in \rS$
is a monomorphism if its image in $\mC$ is a monomorphism.
The converse holds if $\mC \to \rS$ is a locally cocartesian fibration.
\end{remark}
\begin{proof}
	
Let $\f$ be a monomorphism in $\mC$ that lies over the identity of $\s$ and $\Z \in \mC_\s$.
The map $\mC_\s(\Z,\X) \to \mC_\s(\Z,\Y)$
is induced by the map $\mC(\Z,\X) \to \mC(\Z,\Y)$ over $\rS(\s,\s)$ on the fiber over the identity.

Conversely, let $\f$ be a monomorphism in $\mC_\s$ and let $\Z \in \mC$ lying over $\rt \in \rS$. The map $\mC(\Z,\X) \to \mC(\Z,\Y)$ over $\rS(\rt,\s)$ induces on the fiber over any $\g:\rt \to \s$ the map $\mC_\s(\g_*\Z,\X) \to \mC_\s(\g_*\Z,\Y)$
and so is an embedding if the first map is an embedding for any $\g: \rt \to \s.$
\end{proof}

\begin{lemma}

Let $\mC \to \rS, \mD \to \rS$ be functors. The functor $\Fun_\rS(\mC,\mD) \to \prod_{\s \in \rS}\Fun(\mC_\s, \mD_\s)$ is a cartesian fibration relative to the collection of morphisms $(\alpha^\s)_{\s \in \rS}$ such that for any $\s \in \rS$
and $\X \in \mC_\s$ the morphism $\alpha^\s_\X \in \mD_\s$ is a monomorphism
in $\mD.$	
	
\end{lemma}

\begin{proof}
	
Let $\mC, \mD$ be $\infty$-categories. We prove first that the functor $\theta: \Fun(\mC, \mD) \to \prod_{\X \in \mC} \mD $ evaluating at all $\X \in \mC$ is a cartesian fibration relative to the collection $\Theta$ of morphisms $(\alpha^\X)_{\X \in \mC}$ such that for any $\X \in \mC$ the morphism $\alpha^\X $ in $\mD$ is a monomorphism.

Assume first that $\mD=\mS.$ 
Let $\mL_\mC \subset \Cat_{\infty / \mC}$ be the full subcategory of left fibrations.
The functor $\theta$ factors as $\Fun(\mC, \mS) \simeq \mL_\mC \to \prod_{\X \in \mC} \mS $ taking the fiber over all $\X \in \mC.$
The functor $\mL_\mC \to \prod_{\X \in \mC} \mS $ taking the fiber at all $\X \in \mC$ is a cartesian fibration relative to $\Theta$: indeed, let $\mD \to \mC$ be a left fibration and $\alpha^\X: \mB^\X \to \mD_\X$ an embedding and let $\mB \subset \mD$ be the full subcategory spanned by $\mB^\X$
for all $\X \in \mC.$ The embedding $\mB \subset \mD$ induces on the fiber over $\X \in \mC$ the embedding $\alpha^\X: \mB^\X \to \mD_\X$. For any left fibration $\mE \to \mC$ the embedding
$$ \Fun_\mC(\mE, \mB) \to \prod_{\X \in \mC} \mS(\mE_\X, \mB_\X) \times_{\prod_{\X \in \mC} \mS(\mE_\X, \mD_\X)} \Fun_\mC(\mE, \mD) $$
is an equivalence.

Now let $\mD$ be arbitrary. For any $\Z \in \mD$ the functor $\mD(\Z,-): \mD \to \mS$ preserves limits
so that the embedding $\mD \subset \mP(\mD)$ sends monomorphisms to object-wise monomorphisms. 
Thus it is enough to see that the functor $\theta: \Fun(\mC, \mP(\mD)) \to \prod_{\X \in \mC} \mP(\mD) $ is a cartesian fibration relative to $\Theta.$

By the first part of the proof the functor $\theta: \mP(\mD) \to  \prod_{\Y \in \mD^\op} \mS $ is a cartesian fibration relative to $\Theta.$
So the functor $ \prod_{\X \in \mC}\mP(\mD) \to \prod_{\X \in \mC} \prod_{\Y \in \mD^\op} \mS $ is a cartesian fibration relative to the collection $\Phi$ of morphisms $(\alpha^{\X, \Y})_{\X \in \mC,\Y \in \mD}$ such that for any $\X \in \mC, \Y \in \mD$ the morphism $\alpha^{\X, \Y} $ is an embedding.
So the functor $\theta: \Fun(\mC, \mP(\mD)) \to \prod_{\X \in \mC} \mP(\mD) $ is a cartesian fibration relative to $\Theta$ if the composition
$$ \Fun(\mC, \mP(\mD))\to \prod_{\X \in \mC} \mP(\mD) \to \prod_{\X \in \mC} \prod_{\Y \in \mD^\op} \mS $$ is a cartesian fibration relative to $\Phi.$
The latter functor identifies with the functor $\theta: \Fun(\mC \times \mD^\op, \mS) \to \prod_{\X \in \mC} \prod_{\Y \in \mD^\op} \mS $. So the result follows from the case $\mD=\mS.$

Let now $\mC \to \rS, \mD \to \rS$ be functors and $\s \in \rS.$
So we know that the functor $\theta: \prod_{\s \in \rS} \Fun(\mC_\s, \mD_\s) \to \prod_{\s \in \rS}\prod_{\X \in \mC_\s} \mD_\s $ is a cartesian fibration relative to the collection of morphisms $(\alpha^{\s,\X})_{\s \in \rS, \X \in \mC_\s}$ such that for any $\s \in \rS$ and $\X \in \mC_\s$ the morphism $\alpha^{\s, \X} \in \mD_\s$ is a monomorphism in $\mD_\s$ (and so in particular in $\mD$).	

Consequently, the functor $\Fun_\rS(\mC,\mD) \to \prod_{\s \in \rS}\Fun(\mC_\s, \mD_\s)$ is a cartesian fibration relative to the collection of morphisms $(\alpha^\s)_{\s \in \rS}$ such that for any $\s \in \rS$
and $\X \in \mC_\s$ the morphism $\alpha^\s_\X \in \mD_\s$ is a monomorphism
in $\mD$ if the composition
$$\zeta: \Fun_\rS(\mC,\mD) \to \prod_{\s \in \rS}\Fun(\mC_\s, \mD_\s) \to \prod_{\s \in \rS}\prod_{\X \in \mC_\s} \mD_\s $$ is a cartesian fibration relative to the collection of morphisms $(\alpha^{\s,\X})_{\s \in \rS, \X \in \mC_\s}$ such that for any $\s \in \rS$ and $\X \in \mC_\s$ the morphism $\alpha^{\s,\X} \in \mD_\s$ is a monomorphism in $\mD$.
The functor $\zeta$ identifies with the functor
$$ * \times_{\Fun(\mC,\rS)} \Fun(\mC,\mD) \to * \times_{(\prod_{\X \in \mC} \rS)} \prod_{\X \in \mC} \mD.$$

\end{proof}

\begin{corollary}
	
Let $\mC \to \rS$ be a functor and $ \mD \to \rS$ a locally cocartesian fibration. The functor $\Fun_\rS(\mC,\mD) \to \prod_{\s \in \rS}\Fun(\mC_\s, \mD_\s)$ is a cartesian fibration relative to the collection of morphisms $(\alpha^\s)_{\s \in \rS}$ such that for any $\s \in \rS$
and $\X \in \mC_\s$ the morphism $\alpha^\s_\X \in \mD_\s$ is a monomorphism.
	
\end{corollary}

\begin{corollary}\label{proqjz}
Let $\mO \to \Ass$ be a cocartesian fibration relative to the collection of inert morphisms, $\mC \to \mO$ a generalized $\mO$-operad and $\mD \to \mO$ a generalized $\mO$-monoidal $\infty$-category.
Let $\mO^\el \subset \mO$ be the full subcategory of objects lying over $[0],[1]\in\Ass.$ 
The functor $$\Alg_{\mC/\mO}(\mD) \to \prod_{\X \in \mO^\el} \Fun(\mC_\X, \mD_\X)$$ is a cartesian fibration relative to the collection of morphisms $(\alpha^\X)_{\X \in \mO}$ such that for any $\X \in \mO$
and $\Z \in \mC_\X$ the morphism $\alpha^\X_\Z \in \mD_\X$ is a monomorphism.
	
\end{corollary}

\begin{notation}
Let $\mM^\circledast \to \mV^\ot, \mN^\circledast \to \mW^\ot$ be weakly left tensored $\infty$-categories.
Let $$\LaxLinFun(\mM,\mN) \subset \Fun_\Ass(\mV^\ot, \mW^\ot) \times_{\Fun_\Ass(\mM^\circledast, \mW^\ot) } \Fun_\Ass(\mM^\circledast, \mN^\circledast)$$
be the full subcategory of maps of weakly left tensored $\infty$-categories.
There are induced functors $$\LaxLinFun(\mM,\mN) \to \Alg_\mV(\mW), \LaxLinFun(\mM,\mN)  \to \Fun(\mM, \mN).$$
\end{notation}

Let $\mP^\ot \to \LM, \mQ^\ot \to \LM$ be the generalized $\LM$-operads classified by $\mM^\circledast \to \mV^\ot, \mN^\circledast \to \mW^\ot$. 
By Proposition \ref{proqjz} there is a canonical equivalence $\LaxLinFun(\mM,\mN) \simeq \Alg_{\mP/\LM}(\mQ)$ over $ \Alg_\mV(\mW) \times \Fun(\mM, \mN).$
Corollary \ref{proqjz} gives the following proposition:

\begin{proposition}\label{proqj}
	
Let $\mM^\circledast \to \mV^\ot$ a weakly left tensored $\infty$-category and $ \mN^\circledast \to \mW^\ot$ a left tensored $\infty$-category.
The functor $$\LaxLinFun(\mM,\mN) \to \Fun(\mV_{[0]}, \mW_{[0]}) \times \Fun(\mV, \mW) \times \Fun(\mM, \mN)$$ is a cartesian fibration relative to the collection of monomorphisms.	
	
\end{proposition}

\begin{proof}[Proof of Theorem \ref{PP}]

For any small weakly left tensored $\infty$-category $\mM^\circledast \to \mV^\ot$ let $$\xi(\mM)^\circledast \subset \mP\L\Env(\mM)^\circledast \to \mP\Env(\mV)^\ot$$ be the full weakly left tensored subcategory spanned by $\mM.$
So $\xi(\mM)^\circledast\to \mP\Env(\mV)^\ot$ is an enriched $\infty$-category
whose pullback to $\mV^\ot$ is $\mM^\circledast \to \mV^\ot$.
We obtain a functor $\xi: \omega\LMod_\mV \to \Cat_\infty^{\mP\Env(\mV), \mathrm{Lur}}$.
The functor $\chi$ restricts to an equivalence $\Cat_\infty^{\mV, \Lur} \to \Cat_\infty^{\mV}$.
Applied to a larger universe we obtain an equivalence $\chi': \Cat_\infty^{\mP\Env(\mV), \Lur} \to \Cat_\infty^{\mP\Env(\mV)}$.
The functor $\chi$ factors as 
$$ \omega\LMod_\mV \xrightarrow{\xi}\Cat_\infty^{\mP\Env(\mV), \Lur} \xrightarrow{\chi'} \Cat_\infty^{\mP\Env(\mV)}.$$

The pullback $\alpha_{\mP\Env(\mV)}^*((\widehat{\omega\LMod^{\mathrm{enr}}})^\times) \to \LM$ of the cocartesian fibration $(\widehat{\omega\LMod^{\mathrm{enr}}})^\times \to \widehat{\Op}_\infty^\times$ along $\alpha_{\mP\Env(\mV)}: \LM \to \widehat{\Op}_\infty^\times$ exhibits 
$\Cat_\infty^{\mP\Env(\mV), \mathrm{Lur}} $ as left tensored over $\Cat_\infty \simeq \Cat_\infty^{\mS, \Lur}.$

Let $\Enr \subset \PreCat_\infty$ be the full subcategory of enriched $\infty$-categories. 
By Corollary \ref{eqvvjj} there is a canonical equivalence
$\alpha_{\mP\Env(\mV)}^*((\omega\widehat{\LMod}^{\mathrm{enr}})^\times) \simeq \alpha_{\mP\Env(\mV)}^*((\widehat{\Enr})^\times) $
over $ \LM$ whose fibers over the colors of $\LM$ is $\chi':\Cat_\infty^{\mP\Env(\mV), \mathrm{Lur}} \to \Cat_\infty^{\mP\Env(\mV)} $ and $\chi':\Cat_\infty \simeq \Cat_\infty^{\mS, \Lur} \simeq \Cat^\mS_\infty $.
This proves that $\chi':\Cat_\infty^{\mP\Env(\mV), \mathrm{Lur}} \to \Cat_\infty^{\mP\Env(\mV)} $ is $\Cat_\infty$-linear.
It remains to see that $\xi: \omega\LMod_\mV \to \Cat_\infty^{\mP\Env(\mV), \mathrm{Lur}} $ is $\Cat_\infty$-linear.

By Lemmas \ref{LyN} and \ref{Clav} the functor $\B\L\Env(-): \omega\LMod_\mV \to \LMod_{\B\Env(\mV)}$ is $\Cat_\infty$-linear, where the source carries the action of Lemma \ref{2enr} and the target carries the action classified by  
$\beta_{\B\Env(\mV)}^*(\LMod^\times) \to \LM.$
By \cite[Proposition 4.8.1.3.]{lurie.higheralgebra} there is a monoidal functor $\mP: \Cat_\infty^\times \to (\Cat_\infty^{\rc\rc})^\ot$ that gives rise to a $\LM$-monoidal functor
$\beta_{\B\Env(\mV)}^*(\LMod^\times) \to \alpha_{\mP\Env(\mV)}^*((\LMod^{\rc\rc})^\otimes) \subset \alpha_{\mP\Env(\mV)}^*((\omega\widehat{\LMod}^{\mathrm{enr}})^\times)$.
Thus by Lemma \ref{Clav} the functor 
$\mP\B\Env(-): \omega\LMod_\mV \to \widehat{\Cat}^{\mP\Env(\mV), \mathrm{\Lur}}_\infty $ is $\Cat_\infty$-linear, where the left action on the source is by Example \ref{2enr}
and the left action on the target is the pullback of the left action classified by $\alpha_{\mP\Env(\mV)}^*((\omega\widehat{\LMod}^\mathrm{enr})^\times) \to \LM$
along the monoidal functor $\Cat_\infty \xrightarrow{\mP} \Cat_\infty^{\rc\rc} \subset \widehat{\Cat}^\mS_\infty \simeq \widehat{\Cat}_\infty.$
By Proposition \ref{proqj} this implies that $\xi: \omega\LMod_\mV \to \Cat_\infty^{\mP\Env(\mV), \Lur}$ is $\Cat_\infty$-linear.

\end{proof}

We consider corollaries of Theorem \ref{PP}.

\begin{lemma}\label{zzz}
Let $\mV^\ot \to \Ass$ be a monoidal $\infty$-category compatible with small colimits.
The embedding $\Cat_\infty^\mV \to \widehat{\Cat}_\infty^{\widehat{\mP}\Env(\mV)}$
induced by the embedding of $\infty$-operads $\mV^\ot \subset \widehat{\mP}\Env(\mV)^\ot$ is an embedding of $\infty$-categories weakly left tensored over $\Cat_\infty$, where $\Cat_\infty^\mV, \widehat{\Cat}_\infty^{\widehat{\mP}\Env(\mV)}$ carry the left $\Cat_\infty$-actions of Construction \ref{CoN}.
	
\end{lemma}

\begin{proof}

The monoidal adjunction $\L: \widehat{\mP}\Env(\mV) \rightleftarrows \widehat{\mP}(\mV) \rightleftarrows \widehat{\Ind}(\mV): \bj$
gives rise to an adjunction 
$\L_!: \widehat{\Cat}_\infty^{\widehat{\mP}\Env(\mV)} \rightleftarrows \widehat{\Cat}_\infty^{\widehat{\mP}(\mV)} \rightleftarrows \widehat{\Cat}_\infty^{\widehat{\Ind}(\mV)}: \bj_!.$
Since $\L$ preserves small colimits, $\L$ is $\mS$-linear. Thus $\L_!$ is
$\Cat_\infty$-linear. Hence $\bj_!$ is lax  $\Cat_\infty$-linear and therefore a lax $\Cat_\infty$-linear embedding as $\bj_!$ is an embedding.

The monoidal functor $\tau: \mV \to \widehat{\Ind}(\mV)$ preserves small colimits and so is $\mS$-linear. Hence $\tau_!$ is $\Cat_\infty$-linear and therefore a $\Cat_\infty$-linear embedding as $\tau_!$ is an embedding.
	
\end{proof}

Theorem \ref{PP} implies that for any small $\infty$-operad $\mV^\ot \to\Ass$ the left action of $\Cat_\infty^{\mP\Env(\mV)}$ is closed.
Lemma \ref{zzz} implies that the left $\Cat_\infty$-action on  $\Cat_\infty^{\mV}$ is closed.

Corollary \ref{eqvvjj} and Lemma \ref{zzz} give the following corollary:

\begin{corollary}\label{QQ}
Let $\mV^\ot \to \Ass$ be a monoidal $\infty$-category compatible with small colimits.
The $\Cat_\infty$-linear functor $\chi: \omega\widehat{\LMod}_\mV \to \widehat{\Cat}_\infty^{\widehat{\mP}\Env(\mV)}$ restricts to a $\Cat_\infty$-enriched equivalence
$$\chi: \Cat_\infty^{\mV, \Lur} \to \Cat_\infty^{\mV}.$$
	
\end{corollary}

Corollary \ref{QQ} implies that for any monoidal $\infty$-category $\mV^\ot \to\Ass$ compatible with small colimits the restricted $\Cat_\infty$-enrichment on $\Cat_\infty^{\mV, \Lur} $
makes $\Cat_\infty^{\mV, \Lur} $ left tensored over $\Cat_\infty$.

\begin{corollary}
	
Let $\mV^\ot \to \Ass$ be a small $\infty$-operad and $\mM^\circledast \to \mV^\ot, \mN^\circledast \to \mV^\ot$ small weakly left tensored $\infty$-categories. There is a canonical equivalence of $\infty$-categories
$$ \LaxLinFun_\mV(\mM, \mN) \simeq \Mor_{\Cat_\infty^{\mP\Env(\mV)}}(\chi(\mM), \chi(\mN)).$$

\end{corollary}

Theorems \ref{trfd} and \ref{werf} imply the following corollary:

\begin{corollary}\label{uuij}
	
Let $\mV^\ot \to \Ass$ be a small $\infty$-operad and $\mC, \mD$ be
$\mP\Env(\mV)$-enriched $\infty$-categories. There is a canonical equivalence of $\infty$-categories
$$ \Mor_{\Cat_\infty^{\mP\Env(\mV)}}(\mC, \mD) \simeq \Fun^\mV(\mC, \L(\mD)).$$

\end{corollary}

\begin{corollary}\label{uui}
	
Let $\mV^\ot \to \Ass$ be a monoidal $\infty$-category compatible with small colimits and $\mC, \mD$ be $\mV$-enriched $\infty$-categories. There is a canonical equivalence of $\infty$-categories
$$ \Mor_{\Cat_\infty^\mV}(\mC, \mD) \simeq \Fun^\mV(\mC, \L(\mD)).$$
	
\end{corollary}

\section{An embedding of stable into spectral $\infty$-categories}
\label{stabspec}

\vspace{2mm}

In this subsection we give an application of our theory by proving the following theorem:
\begin{theorem}\label{spectrall}

There is a canonical lax symmetric monoidal embedding
$$ \Cat_\infty^\mathrm{st} \hookrightarrow \Cat_\infty^\Sp$$
of the $\infty$-category of small stable $\infty$-categories into the
$\infty$-category of small spectral $\infty$-categories.

\end{theorem}
On the right hand side of the embedding of Theorem \ref{spectrall} appears the symmetric monoidal structure on $\Sp$-enriched $\infty$-categories constructed by Gepner-Haugseng \cite[Proposition 4.3.10.]{GEPNER2015575}.
On the left hand side is a symmetric monoidal structure on $ \Cat_\infty^\mathrm{st}$, which is localized from a symmetric monoidal structure on the $\infty$-category $\Cat_{\infty}^{\mathrm{rex}}$ of small $\infty$-categories having finite colimits \cite[Corollary 4.8.1.4.]{lurie.higheralgebra}.

\vspace{1mm}

In the following we explain how we prove Theorem \ref{spectrall},
where we fix a few definitions: 
recall that an $\infty$-category $\mC$ is called stable if it has a zero object, fibers and cofibers, and any commutative square in $\mC$ is a pullback square if and only if it is a pushout square.
We call a functor of stable $\infty$-categories exact if it preserves the zero object and fibers.
By \cite[Proposition 1.1.4.1.]{lurie.higheralgebra} a functor between stable $\infty$-categories is exact if and only if it preserves finite colimits or equivalently finite limits.

We abstract the notion of stable $\infty$-category to spectral $\infty$-categories:
\begin{definition}
We call a spectral $\infty$-category $\mM^\circledast \to \Sp^\ot$ stable if $\mM$ is stable and for every $\Y \in \mM$
the functors $\Mor_\mM(\Y,-): \mM \to \Sp$ and $\Mor_\mM(-,\Y):= \Mor_{\mM^\op}(\Y,-): \mM^\op \to \Sp$ are exact.
\end{definition}

%\begin{example}\label{examp}Every $\infty$-category $\mM^\circledast \to \Sp^\ot $ presentably left tensored over $\Sp$ is a stable spectral $\infty$-category: By \cite[Proposition 4.8.2.18.]{lurie.higheralgebra} the $\infty$-category $\mM$ is stable.For every $\Y \in \mM$ the functors $\Mor_\mM(\Y,-): \mM \to \Sp, \Mor_\mM(-,\Y): \mM^\op \to \Sp$ are right adjoints and so exact.\end{example}

Next we consider examples of stable spectral $\infty$-categories.
For the next Lemma \ref{Zero} we use that the $\infty$-category $\mS_*$ of pointed spaces carries a presentably (symmetric) monoidal structure given by the smash product \cite[Remark 4.8.2.14.]{lurie.higheralgebra}, which is the initial presentably (symmetric) monoidal $\infty$-category that admits a zero object.

%To prove Lemma \ref{Specu} we use the following lemma:

\begin{lemma}\label{Zero}
Let $\mM^\circledast \to \mS_*^\ot$ be an $\infty$-category left tensored over $\mS_*$
such that $\mM$ admits an initial object and the left action preserves initial objects in both components. Then $\mM$ admits a zero object.
\end{lemma}

\begin{proof}
For any $\X \in \mM$ the unique map of pointed spaces $\tu_{\mS_*} \to 0$ gives rise to a morphism $\X \simeq \tu_{\mS_*} \ot \X \to 0 \ot \X \simeq \emptyset$
in $\mM.$ If $\mM$ has a final object $*$, there is a map $* \to \emptyset$
so that $\emptyset \simeq *$ and $\mM$ has a zero object.

We prove that $\mM$ has a final object. 
By Proposition \ref{presday} there is a left tensored $\infty$-category $\widehat{\mP}(\mM)^\circledast\to \widehat{\mP}(\mS_*)^\ot$ compatible with large colimits and an embedding $\mM^\circledast \to \widehat{\mP}(\mM)^\circledast$ of left tensored $\infty$-categories lying over a monoidal embedding $\mS_*^\ot \subset \widehat{\mP}(\mS_*)^\ot.$

For any $\infty$-category $\mB$ let $\widehat{\mP}(\mB)' \subset \widehat{\mP}(\mB)$ be the full subcategory of presheaves on $\mB$ preserving the final object.
The Yoneda-embedding $\y$ of $\mB$ lands in $\widehat{\mP}(\mB)'$ and preserves initial objects. If $\mB$ has an initial object, $\widehat{\mP}(\mB)' \subset \mP(\mB)$ is a localization with respect to the single map $\emptyset \to \y(\emptyset)$ and so $\widehat{\mP}(\mB)'$ admits a final object.
Moreover the embeddings $\widehat{\mP}(\mS_*)'^\ot \subset \widehat{\mP}(\mS_*)^\ot, \widehat{\mP}(\mM)'^\circledast \subset \widehat{\mP}(\mM)^\circledast$ are localizations relative to $\Ass$ since the tensor product and left action of the map $\emptyset \to \y(\emptyset)$ with itself gives this map. Hence $ \widehat{\mP}(\mM)'^\circledast \to \widehat{\mP}(\mS_*)'^\ot$ is a left tensored $\infty$-category compatible with large colimits so that the pullback $ \mS_*^\ot \times_{\widehat{\mP}(\mS_*)^\ot} \widehat{\mP}(\mM)'^\circledast \to \mS_*^\ot$ is a left tensored $\infty$-category compatible with the initial object.
Since $\widehat{\mP}(\mM)'$ has a final object and by what we have seen,
it admits a zero object.
%But $\widehat{\mP}(\mM)'$ is a localization of $\widehat{\mP}(\mM)$ and so admits a final object.
This zero object belongs to $\mM$ since $\mM$ has an initial object that is preserved by the embedding $\mM \subset \widehat{\mP}(\mM)'.$

\end{proof}

\begin{corollary}\label{Specu}
Every $\infty$-category $\mM^\circledast \to \Sp^\ot$ left tensored over $\Sp$ compatible with small colimits is a stable spectral $\infty$-category.

\end{corollary}

\begin{proof}
For every $\Y \in \mM$ the functors $\Mor_\mM(\Y,-): \mM \to \Sp, \Mor_\mM(-,\Y): \mM^\op \to \Sp$ preserve limits using the left action compatible with small colimits. We prove that $\mM$ is stable.

There is a unique left adjoint (symmetric) monoidal functor $\mS_*^\ot \to \Sp^\ot$. The pullback $\mS_*^\ot \times_{\Sp^\ot} \mM^\circledast \to \mS_*^\ot$
exhibits $\mM$ as left tensored over $\mS_*$ compatible with small colimits.
So $\mM$ admits a zero object by Lemma \ref{Zero}.
For every $\X \in \mM$ there is a canonical equivalence $\Sigma(\X) \simeq \Sigma(\tu_\Sp \ot \X) \simeq \Sigma(\tu_\Sp) \ot \X$ so that the functor
$\Sigma: \mM \to \mM$ identifies with the functor $\Sigma(\tu_\Sp) \ot (-): \mM \to \mM$.
The latter functor is inverse to the functor $\Omega(\tu_\Sp) \ot (-): \mM \to \mM$, which proves stability.

\end{proof}

\begin{corollary}
For any small spectral $\infty$-category $\mM^\circledast \to \Sp^\ot $ the spectral $\infty$-category $\mP_\Sp(\mM)^\circledast \to \Sp^\ot$ of spectral presheaves
is presentably left tensored (Lemma \ref{prestor}) and so a stable spectral $\infty$-category.

\end{corollary}

\begin{example}
For every small spectral $\infty$-category $\mM^\circledast \to \Sp^\ot $ let $\mU(\mC) \subset \mP_\Sp(\mC)$ the smallest full stable subcategory containing the essential image of the spectral Yoneda-embedding $\mC \to  \mP_\Sp(\mC)$.
Then $\mU(\mC)^\circledast \to \Sp^\ot$ is a stable spectral $\infty$-category, which is small by construction.	We call $\mU(\mC)$ the stable closure or stable envelope of $\mC.$
	
\end{example}

\begin{remark}\label{invv}
Let $\mM^\circledast \to \Sp^\ot $ be a stable spectral $\infty$-category.
Since for every $\Y \in \mM$ the functors $\Mor_\mM(\Y,-): \mM \to \Sp$ and $\Mor_\mM(-,\Y): \mM^\op \to \Sp$ are exact,
for every $\Y,\Z \in \mD$ and $\n \in \bZ$ there is a canonical equivalence
of spectra $$ \Mor_\mM(\Sigma^\n(\Y),\Z) \simeq \Omega^\n(\Mor_\mM(\Y,\Z)) \simeq \Mor_\mM(\Y,\Omega^\n(\Z))$$
and so a canonical equivalence of spaces
$$ \mM(\Sigma^\n(\Y),\Z) \simeq \Omega^\infty(\Omega^\n(\Mor_\mM(\Y,\Z))) \simeq \mM(\Y,\Omega^\n(\Z)),$$
where $\Sigma: \mM \to \mM, \Y \mapsto 0 \coprod_\Y 0 $ is the suspension functor, $\Omega: \mM \to \mM, \Y \mapsto 0 \prod_\Y 0  $ is the loops functor and $\Omega^\infty: \Sp \to \mS$ is the infinite loops functor.
\end{remark}

\begin{example}\label{eqst}The forgetful functor $\LMod_\Sp(\Pr^\L) \to \Pr^\L$
induces an equivalence $\LMod_\Sp(\Pr^\L) \simeq \Pr^{\L}_\mathrm{st}$
\cite[Proposition 4.8.2.18.]{lurie.higheralgebra}.
So for every presentable stable $\infty$-category $\mB$ there is a canonical stable spectral $\infty$-category $\mB^\circledast \to \Sp^\ot $ whose underlying $\infty$-category is $\mB.$
\end{example}

\begin{notation}\label{ind}
Let $\mC$ be a small $\infty$-category having finite colimits and $\Ind(\mC) \subset \Fun(\mC^\op,\mS)$ the full subcategory of functors preserving finite limits. 
	
\end{notation}

\begin{example}\label{indig}
Let $\mC$ be a small $\infty$-category having finite colimits.
The $\infty$-category $\Ind(\mC)$ is an accessible localization of $\mP(\mC)$ with respect to the set of morphisms $$\{ \colim(\y \circ \rH) \to \y(\colim(\rH)) \mid \rH:\K \to \mC, \ \K \ \text{small} \},$$
where $\y: \mC \to \mP(\mC)$ is the Yoneda-embedding. Hence $\Ind(\mC)$ is presentable.

By \cite[Proposition 1.1.3.6.]{lurie.higheralgebra}
$\Ind(\mC)$ is stable if $\mC$ is stable. So there is a canonical stable spectral $\infty$-category $\Ind(\mC)^\circledast \to \Sp^\ot $
whose underlying $\infty$-category is $\Ind(\mC)$. 
\end{example}

\begin{notation}\label{liftsp}
For every small stable $\infty$-category $\mC$ let $\mC^\circledast \subset \Ind(\mC)^\circledast$ be the full spectral subcategory spanned by the essential image of the Yoneda-embedding $\mC \to \Ind(\mC)$.
\end{notation}

\begin{example}Let $\mC$ be a small stable $\infty$-category.
Since the Yoneda-embedding $\mC \to \Ind(\mC)$ is exact, $\mC^\circledast \to \Sp^\ot $ is a stable spectral $\infty$-category.
\end{example}
\begin{notation}
Let $(\Cat_\infty^{\Sp,\mathrm{st}})^\ot \subset (\Cat_\infty^{\Sp})^\ot$
be the full symmetric suboperad whose colors are the stable spectral $\infty$-categories.
\end{notation}

\begin{notation}\label{Rexx}

Let $(\Cat_\infty^\mathrm{rex})^\otimes \subset \Cat_\infty^\times$
be the symmetric suboperad whose colors are the $\infty$-categories that admit finite colimits and whose multi-morphisms $\mC_1  ... \mC_\n \to \mD$ correspond to functors $\mC_1 \times ... \times \mC_\n \to \mD$ that preserve finite colimits in each component.

Let $(\Cat_\infty^\mathrm{st})^\otimes \subset (\Cat_\infty^\mathrm{rex})^\otimes$ be the full symmetric suboperad whose colors are the stable $\infty$-categories.
\end{notation}

\begin{remark}
By \cite[Remark 4.8.1.6.]{lurie.higheralgebra} the symmetric $\infty$-operad $(\Cat_\infty^\mathrm{rex})^\otimes \to \Fin_*$ is a symmetric monoidal $\infty$-category, which is closed, where the internal hom of two small $\infty$-categories $\mC,\mD$ having finite colimits is the $\infty$-category $\Fun^{\mathrm{rex}}(\mC,\mD) \subset \Fun(\mC,\mD)$ of finite colimits preserving functors $\mC \to \mD$.
\end{remark}

\begin{lemma}\label{lmo}
The full subcategory $ \Cat_\infty^{\mathrm{st}} \subset \Cat_\infty^{\mathrm{rex}}$ is a localization compatible with the symmetric monoidal structure. 	
\end{lemma}

\begin{proof}[Proof of Lemma \ref{lmo}]
By \cite[Proposition 4.8.2.18.]{lurie.higheralgebra} the full subcategory $\Pr_\mathrm{st}^{\L} \subset \Pr^\L$
admits a left adjoint that sends a presentable $\infty$-category $\mB$ to the stable presentable $\infty$-category of spectra objects $\Sp(\mB):= \Sp \otimes \mB$.	
For any small $\infty$-category $\mC$ having finite colimits and stable $\infty$-category $\mD$ there is a canonical equivalence (using Lemma \ref{Indd}):
$$ \Fun^{\mathrm{rex}}(\mC,\Ind(\mD)) \simeq \Fun^\L(\Ind(\mC),\Ind(\mD)) \simeq  \Fun^\L(\Sp(\Ind(\mC)),\Ind(\mD)) \simeq \Fun^{\mathrm{rex}}(\Sp(\Ind(\mC))',\Ind(\mD)) $$
restricting to an equivalence 
$$ \Fun^{\mathrm{rex}}(\mC,\mD) \simeq \Fun^{\mathrm{rex}}(\Sp(\Ind(\mC))',\mD),$$
where $(-)'$ is the full subcategory of compact objects.
The resulting localization $ \Cat_\infty^{\mathrm{st}} \subset \Cat_\infty^{\mathrm{rex}}$ 
is symmetric monoidal since for any small $\infty$-categories $\mC,\mD$ having finite colimits the internal hom $\Fun^{\mathrm{rex}}(\mC,\mD)$ is stable if $\mD$ is stable.
\end{proof}

Lemma \ref{lmo} implies the following corollary:

\begin{corollary}
The symmetric $\infty$-operad $(\Cat_\infty^\mathrm{st})^\otimes \to \Fin_*$ is a closed symmetric monoidal $\infty$-category, where the internal hom of two stable $\infty$-categories $\mC,\mD$ is the full subcategory $\Fun^{\mathrm{ex}}(\mC,\mD) \subset \Fun(\mC,\mD)$ of exact functors $\mC \to \mD$.
\end{corollary}

The finite products preserving functor $\Enr^\Lur \to \Cat_\infty,  (\mM^\circledast \to \mV^\ot) \mapsto \mM$
induces a symmetric monoidal functor $(\Enr^\Lur)^\times \to \Cat_\infty^\times$
and so a lax symmetric monoidal functor
$$\rho: (\Cat_\infty^{\Sp})^\ot = \Fin_* \times_{\Op_\infty^\times} (\Enr^\Lur)^\times \to (\Enr^\Lur)^\times \to \Cat_\infty^\times.$$

In this section we will prove the following theorem: 
\begin{theorem}\label{Monn}\label{eqqqzw}
The lax symmetric monoidal functor $\rho: (\Cat_\infty^{\Sp})^\ot \to  \Cat_\infty^\times$
restrict to an equivalence $$(\Cat_\infty^{\Sp,\mathrm{st}})^\ot \simeq (\Cat_\infty^\mathrm{st})^\otimes. $$ 
\end{theorem}

Theorem \ref{eqqqzw} follows from Lemma \ref{ecute} and Proposition \ref{Monn}:
Lemma \ref{ecute} says that the lax symmetric monoidal functor $\rho: (\Cat_\infty^{\Sp})^\ot \to  \Cat_\infty^\times$
restricts to a conservative functor $\Cat_\infty^{\Sp,\mathrm{st}} \to \Cat_\infty^\mathrm{st}$.
Proposition \ref{Monn} says that the resulting restriction $(\Cat_\infty^{\Sp,\mathrm{st}})^\ot \to (\Cat_\infty^\mathrm{st})^\otimes$
admits a fully faithful left adjoint relative to $\Fin_*$
and so is an equivalence. 

\begin{lemma}\label{ecute}
Let $\F: \mD^\circledast \to \mE^\circledast $ be a lax $\Sp$-linear functor
between stable spectral $\infty$-categories.

\begin{enumerate}
\item The underlying functor $\mD \to \mE$ is exact.

\item The lax $\Sp$-linear functor $\F: \mD^\circledast \to \mE^\circledast $ is an equivalence if the underlying functor $\mD \to \mE$ is. 
\end{enumerate}
	
\end{lemma}

\begin{proof}
We start with (1): By \cite[Corollary 1.4.2.14.]{lurie.higheralgebra} a functor between stable $\infty$-categories is exact if and only if it preserves the zero object and loops. 

For the case of zero objects note that for any $\Y,\Z \in \mD$ the morphism space $\mD(\Y,\Z)$ is pointed as it underlies a morphism spectrum,
where the base point is the zero morphism $\Y \to 0 \to \Z$ in $\mD.$
Similarly, the induced map $\mD(0,0) \to \mE(\F(0),\F(0))$
on morphism spaces underlies a map of morphism spectra and so preserves the base point. Since the base point of the contractible space $\mD(0,0)$ is the identity, the identity of $\F(0)$ is the base point of the space
$\mE(\F(0),\F(0))$ and so is the zero morphism $\F(0) \to 0 \to \F(0)$ in $\mE.$
Thus the map $0 \to \F(0)$ is inverse to the map $\F(0)\to 0.$

We continue with showing that the underlying functor $\mD \to \mE$ preserves loops.
Let $\Y,\Z \in \mD$. By Remark \ref{invv} the induced map
$\Omega^\infty(\Sigma(\Mor_\mD(\Y,\Z))) \to \Omega^\infty(\Sigma(\Mor_\mE(\F(\Y),\F(\Z))))$ identifies with the map
$\mD(\Omega(\Y),\Z) \to \mE(\Omega(\F(\Y)),\F(\Z))$
that sends the identity of $\Z= \Omega(\Y)$ to a morphism $\alpha: \Omega(\F(\Y)) \to \F(\Omega(\Y)).$
It is easy to see that $\alpha$ is inverse to the canonical map $\F(\Omega(\Y)) \to \Omega(\F(\Y)).$

(2) Now we assume that the underlying functor $\mD \to \mE$ is an equivalence.
We like to see that for any objects $\Y, \Z \in \mD$ and $\n \geq 1 $ the
canonical map
$ \psi: \Omega^\infty(\Sigma^{\n}(\Mor_\mD(\Y,\Z))) \to  \Omega^\infty(\Sigma^{\n}(\Mor_\mE(\F(\Y),\F(\Z))))$
is an equivalence. By Remark \ref{invv} the map $\psi$ identifies with the canonical equivalence
$$\mD(\Y,\Sigma^{\n}(\Z)) \to \mE(\F(\Y),\F(\Sigma^{\n}(\Z))) \simeq \mE(\F(\Y),\Sigma^{\n}(\F(\Z))) . $$	

\end{proof}	

\begin{corollary}\label{flexo}
Let $\mC^\circledast \to \Sp$ be a small spectral $\infty$-category and
$\mD^\circledast \to \Sp$ a small stable spectral $\infty$-category.
The induced functor
$$\lambda: \LaxLinFun_\Sp(\mU(\mC),\mD) \to \LaxLinFun_\Sp(\mC,\mD)$$ is an equivalence.
	
\end{corollary}

\begin{proof} By Lemma \ref{ecute} the functor $\lambda$ is conservative.
By Corollary \ref{lan} the functor $\lambda$ admits a fully faithful left adjoint. Thus $\lambda$ is an equivalence.
\end{proof}

\begin{corollary}\label{symrea}
	
The full subcategory $ \Cat_\infty^{\Sp,\mathrm{st}} \subset \Cat_\infty^{\Sp}$	
is a localization.
	
\end{corollary}

\begin{remark}
Corollary \ref{symrea} implies that $(\Cat_\infty^{\Sp,\mathrm{st}})^\ot \to \Fin_*$ is a locally cocartesian fibration. 
\end{remark}

\vspace{1mm}
Let $\mC$ be a small stable $\infty$-category.
By Theorem \ref{tgre} the lax $\Sp$-linear embedding $ \mC^\circledast \to \Ind(\mC)^\circledast $ uniquely extends to a left adjoint $\Sp$-linear functor $\psi: \mP_\Sp(\mC)^\circledast \to \Ind(\mC)^\circledast $.

\begin{proposition}\label{identif}
Let $\mC$ be a small stable $\infty$-category.
The left adjoint $\Sp$-linear functor $\psi: \mP_\Sp(\mC)^\circledast \to \Ind(\mC)^\circledast $ is an equivalence.
\end{proposition}

\begin{proof}

By Lemma \ref{ecute} the spectral Yoneda-embedding 
$ \mC^\circledast \to \mP_\Sp(\mC)^\circledast$ induces an exact functor $\y: \mC \to \mP_\Sp(\mC)$ and so by Lemma \ref{Indd} uniquely extends to a left adjoint functor $\kappa: \Ind(\mC) \to \mP_\Sp(\mC)$ of presentable stable $\infty$-categories, where the right adjoint sends $\F \in \mP_\Sp(\mC)$ to the presheaf $ \map_{\mP_\Sp(\mC)}(-,\F) \circ \y: \mC^\op \to \mS.$

As finite limits commute with filtered colimits,
$\Ind(\mC)$ is closed in $\mP(\mC)$ under small filtered colimits.
Therefore by the Yoneda-lemma every object of $\mC$ is compact in $\Ind(\mC).$
Moreover $\Ind(\mC)$ is generated by $\mC$ under small filtered colimits  \cite[Proposition 5.3.5.3., Corollary 5.3.5.4.]{lurie.HTT}.
Because objects of $\mC$ are also compact in $\mP_\Sp(\mC)$ (Remark \ref{tau}), the functor $\kappa$ is fully faithful.
The right adjoint of $\kappa$ is conservative since for $\n \in \bZ $ there is a canonical equivalence $$ \Omega^\infty \circ \Sigma^{\n} \circ \Mor_{\mP_\Sp(\mC)}(-,\F) \circ \y \simeq \Omega^\infty \circ \Mor_{\mP_\Sp(\mC)}(-,\F) \circ \y \circ \Omega^{\n} \simeq \map_{\mP_\Sp(\mC)}(-,\F) \circ \y \circ \Omega^{\n} $$ of presheaves on $\mC.$
Thus $\kappa$ is an equivalence. By Remark \ref{eqst} the functor $\kappa$ uniquely refines to a $\Sp$-linear equivalence 
$\bar{\kappa}: \Ind(\mC)^\circledast \simeq \mP_\Sp(\mC)^\circledast$.
We will prove that $\bar{\kappa}$ is inverse to $\psi$.
For that it is enough to see that the composition $\psi \circ \bar{\kappa}: \Ind(\mC)^\circledast \to \mP_\Sp(\mC)^\circledast \to \Ind(\mC)^\circledast$
is the identity. To see this, by Remark \ref{eqst} and Lemma \ref{Indd} it is enough to check that the composition $\mC \subset \Ind(\mC) \to \mP_\Sp(\mC) \to \Ind(\mC)$ is the canonical embedding, which follows from the definition.

\end{proof}

\begin{notation}
Let $\mC_1, ..., \mC_\n,\mD$ be $\infty$-categories.
We write $\Fun'(\mC_1 \times ... \times \mC_\n, \mD)$ for the full subcategory of
functors $\mC_1 \times ... \times \mC_\n \to \mD$ that preserve finite colimits in each component and $\Fun''(\mC_1 \times ... \times \mC_\n, \mD)$ for the full subcategory of functors $\mC_1 \times ... \times \mC_\n \to \mD$ that preserve small colimits in each component.	
\end{notation}

\begin{lemma}\label{Indd}
Let $\n \geq 1$ and $\mC_1, ..., \mC_\n$ small $\infty$-categories with finite colimits and $ \mD$ an $\infty$-category with small colimits.
Then $\Fun''(\Ind(\mC_1) \times ... \times \Ind(\mC_\n), \mD) \to \Fun'(\mC_1 \times ... \times \mC_\n, \mD)$ is an equivalence.
\end{lemma}

\begin{proof}
We first show the case $\n=1$. Let $\mC$ be an $\infty$-category having finite colimits.
By \cite[Theorem 5.1.5.6.]{lurie.higheralgebra} the canonical functor $\Fun^\L(\mP(\mC),\mD) \to \Fun(\mC,\mD)$ is an equivalence.
By the description of the generating local equivalences of $\Ind(\mC)$ of Remark \ref{ind} this equivalence restricts to an equivalence 
$\Fun^\L(\Ind(\mC),\mD) \to \Fun^\mathrm{rex}(\mC,\mD),$ where the right hand side is the full subcategory of finite colimits preserving functors $\mC \to \mD.$	
The functor $\beta$ factors as $$ \Fun''(\Ind(\mC_1) \times ... \times \Ind(\mC_\n), \mD) \simeq \Fun''(\Ind(\mC_1) \times ... \times \Ind(\mC_{\n-1}), \Fun^\L(\Ind(\mC_\n), \mD)) $$
$$ \to \Fun'(\mC_1 \times ... \times \mC_{\n-1},\Fun^\mathrm{rex}(\mC_\n, \mD)) \simeq \Fun'(\mC_1 \times ... \times \mC_\n, \mD)$$
and so is an equivalence by induction and the case $\n=1. $ 

\end{proof}

Proposition \ref{Monn} follows from Lemma \ref{Indd} and the following Proposition, which we prove at the end of this section:

\begin{proposition}\label{enrcol}
Let $\mV^\ot \to \Ass$ be a presentably monoidal $\infty$-category, $\mN^\circledast \to \mV^\ot$ a left tensored $\infty$-category compatible with small colimits and $\mM^\circledast \to \mV^\ot$ a small
$\mV$-enriched $\infty$-category. The functor
\begin{equation*}
\rho: \LaxLinFun_{\mV}(\mP_\mV(\chi(\mM)),\mN) \to \LaxLinFun_\mV(\mM,\mN)
\end{equation*}
admits a fully faithful left adjoint, which takes values in the full subcategory
$\LinFun^\L_{\mV}(\mP_\mV(\chi(\mM)),\mN).$

\end{proposition}

\begin{proposition}\label{Monn}
Let $\mC_1, ..., \mC_\n$ be small stable $\infty$-categories for $\n \geq 1$
and $ \mD^\circledast \to \Sp^\ot$ a small stable spectral $\infty$-category.
The functor $$\psi_\mD: \LaxLinFun_{\Sp^{\times \n}}(\mC_1 \times ... \times \mC_\n, \ot^*(\mD)) \to \Fun'(\mC_1 \times ... \times \mC_\n, \mD)$$
is an equivalence. 
\end{proposition}

\begin{proof}
By Lemma \ref{ecute} the functor $\LaxLinFun_{\Sp^{\times \n}}(\mC_1 \times ... \times \mC_\n, \ot^*(\mD)) \to \Fun(\mC_1 \times ... \times \mC_\n, \mD)$
takes values in $\Fun'(\mC_1 \times ... \times \mC_\n, \mD)$.
By Lemma \ref{ecute} the spectral Yoneda-embedding $ \mD^\circledast \to \mP_\Sp(\mD)^\circledast$ induces an exact functor on underlying $\infty$-categories. So $\psi_\mD$ is the pullback of $\psi_{\mP_\Sp(\mD)}.$ So we can reduce to the case that $\mD^\circledast \to \Sp^\ot$ is a presentably left tensored $\infty$-category.
Consider the commutative square:
\begin{equation*} 
\begin{xy}
\xymatrix{
\LinFun''_{\Sp^{\times \n}}(\mP_\Sp(\mC_1) \times ... \times \mP_\Sp(\mC_\n), \ot^*(\mD)) \ar[rr]^{\gamma} \ar[d]^\alpha
&& \Fun''(\mP_\Sp(\mC_1) \times ... \times \mP_\Sp(\mC_\n), \mD) \ar[d]^\beta
\\ \LaxLinFun_{\Sp^{\times \n}}(\mC_1 \times ... \times \mC_\n, \ot^*(\mD)) \ar[rr]^{\psi_\mD} && \Fun'(\mC_1 \times ... \times \mC_\n, \mD).}
\end{xy} 
\end{equation*}

The functor $\beta$ is an equivalence by Lemma \ref{Indd} and Proposition \ref{identif}. So it is enough to see that $\alpha, \gamma$ are equivalences.

Let $(\Pr^\L)^\ot \subset \widehat{\Cat}_\infty^\times$ be the symmetric suboperad of Notation \ref{prsym} and $(\Pr^{\L}_\mathrm{st})^\ot \subset (\Pr^\L)^\ot$ the full symmetric suboperad of stable presentable $\infty$-categories.
By \cite[Proposition 4.8.2.18.]{lurie.higheralgebra} the lax symmetric monoidal forgetful functor $\LMod_\Sp(\Pr^\L)^\ot \to (\Pr^\L)^\ot$
induces an equivalence $\LMod_\Sp(\Pr^\L)^\ot \simeq (\Pr^{\L}_\mathrm{st})^\ot.$
The embedding of $\infty$-operads
$ \LMod_\Sp(\Pr^\L)^\ot \hookrightarrow
(\Pr^\L)^\ot \times_{\widehat{\Cat}_\infty^\times} \LMod_\Sp(\widehat{\Cat}_\infty)^\ot $
induces an equivalence
$$\varrho: \LinFun''_{\Sp^{\times \n}}(\mP_\Sp(\mC_1) \times ... \times \mP_\Sp(\mC_\n), \ot^*(\mD)) \simeq \LinFun^\L_{\Sp^{\otimes \n}}(\mP_\Sp(\mC_1) \otimes ... \otimes \mP_\Sp(\mC_\n), \ot^*(\mD)) $$$$ \simeq \LinFun^\L_{\Sp}(\mP_\Sp(\mC_1) \otimes_\Sp ... \otimes_\Sp \mP_\Sp(\mC_\n), \mD).$$
The equivalence $\LMod_\Sp(\Pr^\L)^\ot \simeq (\Pr^{\L}_\mathrm{st})^\ot$
of symmetric monoidal $\infty$-categories induces an equivalence
$$\zeta: \LinFun^\L_{\Sp}(\mP_\Sp(\mC_1) \otimes_\Sp ... \otimes_\Sp \mP_\Sp(\mC_\n), \mD) \to \Fun''(\Ind(\mC_1) \times ... \times \Ind(\mC_\n), \mD).$$
The functor $\gamma$ factors as $\varrho$ followed by $\zeta.$
So it remains to see that $\alpha$ is an equivalence. Using Lemma \ref{ecute} (1) and Lemma \ref{Indd} there is a commutative square
\begin{equation*} 
\begin{xy}
\xymatrix{
\LinFun''_{\Sp^{\times \n}}(\mP_\Sp(\mC_1) \times ... \times \mP_\Sp(\mC_\n), \ot^*(\mD)) \ar[rr]^{\alpha} \ar[d]
&&\LaxLinFun_{\Sp^{\times \n}}(\mC_1 \times ... \times \mC_\n, \ot^*(\mD)) \ar[d]
\\ \Fun''(\mP_\Sp(\mC_1) \times ... \times \mP_\Sp(\mC_\n), \mD) \ar[rr]^{\simeq} &&\Fun'(\mC_1 \times ... \times \mC_\n, \mD).}
\end{xy} 
\end{equation*}
The bottom functor is an equivalence by Remark \ref{ind}.
So the top functor $\alpha$ is conservative. Thus $\alpha$ is an equivalence if $\alpha$ admits a fully faithful left adjoint.
By Corollary \ref{lan} the functor
$$ \LaxLinFun_{\Sp^{\times \n}}(\mP_\Sp(\mC_1) \times ... \times \mP_\Sp(\mC_\n), \ot^*(\mD)) \to \LaxLinFun_{\Sp^{\times \n}}(\mC_1 \times ... \times \mC_\n, \ot^*(\mD))$$
admits a fully faithful left adjoint $\phi$.
So it is enough to see that $\phi$ takes values in the full subcategory
$\LinFun''_{\Sp^{\times \n}}(\mP_\Sp(\mC_1) \times ... \times \mP_\Sp(\mC_\n), \ot^*(\mD)).$
To see this by Proposition \ref{la} (5) we can reduce to the case $\n=1.$
The case $\n=1$ follows from Proposition \ref{enrcol}.

\end{proof}

To prove Proposition \ref{enrcol} we will use a theory of enriched Kan extensions based on the theory of \cite{Rune} that treats operadic left Kan extensions in a general framework.
We adapt the notion of extendable cartesian pattern \cite[Definition 8.8.]{Rune} to our setting.

For the next definition recall that for any cocartesian fibration $\mO \to \Ass$ relative to the collection of inert morphisms evaluation at the target $\Act(\mO) \subset \Fun([1],\mO) \to \mO$ is a cocartesian fibration.
This follows from \cite[Lemma 5.2.8.19.]{lurie.HTT} since the inert-active factorization system on $\Ass$ lifts to $\mO$. 

\begin{definition}\label{exttt}
We call a cocartesian fibration $\mO \to \Ass$ relative to the collection of inert morphisms extendable if for any $\X \in \mO$ lying over $[\n] \in \Ass$
the inert lifts $\X \to \X_\bi$ of the standard inert morphisms $[\n] \to [1]$ in $\Ass$
induce an equivalence of spaces \begin{equation}\label{extend}
\Act(\mO)_\X^\simeq \simeq \prod_{\bi=1}^\n \Act(\mO)_{\X_\bi}^\simeq.\end{equation} 

\end{definition}

\begin{remark}\label{exxx}
Any generalized $\infty$-operad $\mO \to \Ass$ whose fiber over $[0]$ is a space, is extendable: the canonical functor $\Act(\mO)_\X \to \Act(\mO)_{\X_1} \times_{ \Act(\mO)_{\X_{1,2}} }
... \times_{ \Act(\mO)_{\X_{\n-1,\n}} }\Act(\mO)_{\X_\n}$
is an equivalence and for any $\Z \in \mO_{[0]}$ the $\infty$-category $\Act(\mO)_\Z \simeq (\mO_{[0]})_{/\Z}$ is contractible.

\end{remark}
By Remark \ref{exxx} $\Ass$ and $ \BM$ are extendable.
Equivalence (\ref{extend}) for $\BM$ restricts to 
(\ref{extend}) for $\LM,\RM.$

\begin{proposition}\label{la}

Let $\mO \to \Ass$ be an extendable cocartesian fibration relative to the collection of inert morphisms, $\tau: \mA \to \mO, \phi:\mB \to \mO$ small $\mO$-operads, $\gamma: \mA \to \mB $ a map of $\mO$-operads and $\mD \to \mO$ an $\mO$-monoidal $\infty$-category compatible with small colimits. 

\begin{enumerate}
\item The forgetful functor $$ \Alg_{\mB/\mO}(\mD) \to \Alg_{\mA/\mO}(\mD) $$ admits a left adjoint $\gamma_!$.
For any map $\F: \mA \to \mD$ of $\mO$-operads and $\B \in \mB$ there is an equivalence $$\gamma_!(\F)(\B) \simeq \colim_{\Y \in \mA, \beta: \gamma(\Y) \to \B \in \mB^\act_{/\B}} \phi(\beta)_*(\F(\Y))$$
and for any $\A \in \mA$ the unit $\F \to \gamma_!(\F) \circ \gamma$ evaluated
at $\A$ factors as 
$$\F(\Y) \simeq \colim_{\Y \in \mA, \alpha: \Y \to \A \in \mA^\act_{/\A}} \tau(\alpha)_*(\F(\Y)) 
\to \colim_{\Y \in \mA, \beta: \gamma(\Y) \to \gamma(\A) \in \mB^\act_{/\gamma(\A)}} \phi(\beta)_*(\F(\Y)).$$

\item If $\gamma$ is an embedding of $\mO$-operads, the left adjoint is fully faithful.

\vspace{1mm}
\item If $\gamma$ is an $\mO$-monoidal functor, for any $\X \in \mO$
in the canonical commutative square
\begin{equation*}
\begin{xy}
\xymatrix{
\Alg_{\mB/\mO}(\mD) \ar[d] \ar[rr]^{}
&& \Alg_{\mA/\mO}(\mD) \ar[d]
\\  \Fun(\mB_\X,\mD_\X)
\ar[rr] &&  \Fun(\mA_\X,\mD_\X)
}
\end{xy} 
\end{equation*} 
the left adjoints commute with the forgetful functors.		

\vspace{1mm}
\item Let $\alpha: \mU \to \mO$ be a map of cocartesian fibrations relative to the collection of inert morphisms that is a right fibration relative to the collection of active morphisms. Then $\mU \to \Ass$ is extendable and in the commutative square 
\begin{equation*}
\begin{xy}
\xymatrix{
\Alg_{\mB/\mO}(\mD) \ar[d] \ar[rr]^{}
&& \Alg_{\mA/\mO}(\mD) \ar[d]
\\ \Alg_{\mU \times_\mO \mB/\mU}(\mU \times_\mO\mD)
\ar[rr] && \Alg_{\mU \times_\mO\mA/\mU}(\mU \times_\mO\mD). 
}
\end{xy} 
\end{equation*}  
the forgetful functors commute with the left adjoints.

\item For any small $\mO$-operads $\mC_\bi \to \mO$ for $1 \leq \bi \leq \n$ and $\n \geq 1$ and maps of $\mO$-operads $\mO \to \mC_\bi$ consider the following commutative square:
\begin{equation*}
\begin{xy}
\xymatrix{
\Alg_{\mC_1 \times_\mO ... \times_\mO \mC_\bi \times_\mO \mB \times_\mO \mC_{\bi+1} \times_\mO ... \times_\mO \mC_\n /\mO}(\mD) \ar[d] \ar[rr]^{}
&& \Alg_{\mC_1 \times_\mO ... \times_\mO \mC_\bi \times_\mO \mA \times_\mO \mC_{\bi+1} \times_\mO ... \times_\mO \mC_\n /\mO}(\mD) \ar[d]
\\ \Alg_{\mB/\mO}(\mD) \ar[rr]^{}
&& \Alg_{\mA/\mO}(\mD)
}
\end{xy} 
\end{equation*}  
The vertical functors in the square commute with the left adjoints
of the horizontal functors.

\end{enumerate}

\end{proposition}

\begin{proof}
(1) and (5) follows from \cite[Corollary 8.13.]{Rune} and \cite[Remark 8.15.]{Rune}.
(2) follows from the description of the unit in (1):
if $\gamma: \mA \to \mB$ is an embedding of $\mO$-operads, for any $\A \in \mA$ the induced functor
$\mA^\act_{/\A} \to \mB^\act_{/\gamma(\A)}$ is an equivalence so that the unit
$\id \to \gamma^* \circ \gamma_! $ is an equivalence.

(3): By Proposition \ref{proooo} (5) there is a lax $\mO$-monoidal functor $\Fun^\mO(\mB,\mD) \to \Fun^\mO(\mA,\mD)$
that induces on the fiber over $\X \in \mO$ the functor
$\Fun(\mB_\X,\mD_\X) \to \Fun(\mA_\X,\mD_\X)$ and on $\mO$-algebras the functor $\gamma^*: \Alg_{\mB/\mO}(\mD) \to \Alg_{\mA/\mO}(\mD) $.
By Proposition \ref{proooo} (4) the lax $\mO$-monoidal functor $\Fun^\mO(\mB,\mD) \to \Fun^\mO(\mA,\mD)$ admits a left adjoint $\kappa$ relative to $\mO$. As a consequence of adjointness $\kappa$ induces on the fiber over $\X \in \mO$ the functor
$\Fun(\mA_\X,\mD_\X) \to \Fun(\mB_\X,\mD_\X)$ taking left Kan extension along the functor
$\mA_\X \to \mB_\X$ and induces on $\mO$-algebras the left adjoint $\gamma_!$ of $\gamma^*.$ This implies the claim.

(4): If $\alpha:\mU \to \mO$ is a right fibration relative to the collection of active morphisms, for any $\X \in \mU$ the functor $\alpha$ induces an equivalence 
$\Act(\mU)_\X \simeq \Act(\mO)_{\alpha(\X)}$ so that $\mU \to \Ass$ is extendable.
If $\alpha:\mU \to \mO$ is a right fibration relative to the collection of active morphisms, the induced functor $\mA'_\act \to \mA_\act$ is a right fibration, where $ \mA':= \mU \times_\mO \mA$, and similar for $\mB \to \mO.$
Thus the induced map
$\gamma': \mA':= \mU \times_\mO \mA \to \mB':= \mU \times_\mO \mB$
of $\mU$-operads yields for any $\B' \in \mB'$ lying over $\B \in \mB$
an equivalence $\mA'_\act \times_{\mB'_\act}(\mB'_\act)_{/\B'} \to \mA_\act \times_{\mB_\act}(\mB_\act)_{/\B} $.
This implies the claim using Proposition \ref{la} (1).

\end{proof}

We apply Proposition \ref{la} (4) to the embedding $\Ass \subset \LM$ and the left and right embeddings $\Ass \subset \BM$ that are right fibrations relative to the collection of active morphisms:

\begin{corollary}\label{lan}

Let $\psi: \mM^\circledast \to \mN^\circledast $ be a map of small weakly bitensored $\infty$-categories inducing embeddings of $\infty$-operads $\gamma: \mV^\ot \to \mW^\ot,\gamma': \mV'^\ot \to \mW'^\ot $.
Let $\mD^\circledast \to \mU^\ot \times \mU'^\ot$ an $\infty$-category bitensored compatible with small colimits.
The functor $ \gamma^*: \Alg_{\mW/\Ass}(\mU) \to \Alg_{\mV/\Ass}(\mU) $ admits a fully faithful left adjoint $\gamma_!$ and similar for $\gamma'.$
Let $\alpha:\mV^\ot \to \mU^\ot, \alpha':\mV'^\ot \to \mU'^\ot$ maps of $\infty$-operads.
The functor $$\LaxLinFun_{\mW,\mW'}(\mN,(\gamma_!(\alpha), \gamma'_!(\alpha'))^*(\mD)) \to \LaxLinFun_{\mV,\mV'}(\mM,(\alpha, \alpha')^*(\mD))$$ admits a left adjoint that is fully faithful if $\psi$ is an embedding of weakly bitensored $\infty$-categories.

\end{corollary} 

\begin{corollary}\label{gggbn}
Let $\mO \to \Ass$ be an extendable cocartesian fibration relative to the collection of inert morphisms, $\mC \to \mO$ a small $\mO$-monoidal $\infty$-category and $\mD \to \mO$ an $\mO$-monoidal $\infty$-category compatible with small colimits.
The $\mO$-monoidal embedding $ \mC \to \mP\mC $ of Proposition \ref{presday} 
induces a functor $$\rho: \Alg_{\mP\mC/\mO}(\mD) \to \Alg_{\mC/\mO}(\mD)$$
that admits a fully faithful left adjoint $\phi$.
For any $\X \in \mO$ the functor $\phi$ covers the left adjoint of the functor
$ \Fun(\mP(\mC_\X), \mD_\X) \to \Fun(\mC_\X, \mD_\X) $ taking left Kan extension along the Yoneda-embedding. Thus $\phi$ sends (lax) $\mO$-monoidal functors to (lax) $\mO$-monoidal functors preserving fiberwise small colimits.

\end{corollary}

\begin{proof}
This follows from Proposition \ref{la} (3) and
the fact that the left Kan extension along the Yoneda-embedding preserves small colimits (\cite[Lemma 5.1.5.5.]{lurie.HTT}). A lax $\mO$-monoidal functor $\mP\mC \to \mD$ preserving small colimits is $\mO$-monoidal if and only if its restriction to $\mC$ is $\mO$-monoidal.
	
\end{proof}

\begin{lemma}\label{indos}
Let $\mV^\ot \to \mU^\ot$ be a small colimits preserving monoidal functor between monoidal $\infty$-categories compatible with small colimits, $\mN^\circledast \to \mV^\ot \times \mW^\ot$ a bitensored $\infty$-category compatible with small colimits and $\mC$ a $\mW$-precategory with small space of objects $\X$.
\begin{enumerate}
\item There is a canonical $\mU,\Quiv_\X(\mW)$-linear equivalence $$\mU \otimes_\mV \Fun(\X,\mN) \simeq \Fun(\X, \mU \otimes_\mV \mN).$$

\item There is a canonical $\mU$-linear equivalence $$\mU \otimes_\mV \mP_\mV(\mC;\mN) \simeq \mP_{\mU}(\mC;\mU \otimes_\mV \mN).$$

\item There is a canonical $\mU$-linear equivalence $\mU \otimes_\mV \mP_\mV(\mC) \simeq \mP_{\mU}(\varphi_*(\mC))$.
\end{enumerate}	
\end{lemma}

\begin{proof}
(1): The canonical $\mV,\mW$-linear map $ \mN \to \mU \otimes_\mV \mN $
yields a $\mV, \Quiv_\X(\mW)$-linear map
$\Fun(\X,\mN) \to \Fun(\X, \mU \otimes_\mV \mN) $
that induces a $\mU, \Quiv_\X(\mW)$-linear map
$\theta: \mU \otimes_\mV \Fun(\X,\mN) \to \Fun(\X, \mU \otimes_\mV \mN). $
Forgetting the right actions 
by Remark \ref{diagg} the left actions on source and target become the diagonal actions so that $\theta$ identifies with the $\mU$-linear equivalence
$\mU \otimes_\mV (\mN \ot \Fun(\X,\mS)) \to (\mU \otimes_\mV \mN) \ot \Fun(\X,\mS). $

\vspace{1mm}
(2): Let $\A$ be an associative algebra in $\mW$.
By \cite[Theorem 4.8.4.6.]{lurie.higheralgebra} there is a $\mV$-linear equivalence $$ \mN \ot_\mW \RMod_\A(\mW) \simeq \RMod_\A(\mN)$$
and so a canonical $\mU$-linear equivalence 
\begin{equation}\label{zuppp}
\RMod_\A(\mU \otimes_\mV \mN) \simeq  (\mU \otimes_\mV \mN) \ot_\mW \RMod_\A(\mW) \simeq \mU \ot_\mV \RMod_\A(\mN).
\end{equation}
Thus the $\mU,\Quiv_\X(\mW)$-linear equivalence of (1) induces an $\mU$-linear equivalence 
$$\mU \otimes_\mV \RMod_\mC(\Fun(\X,\mN)) \simeq \RMod_\mC(\mU \otimes_\mV \Fun(\X,\mN)) \to \RMod_\mC(\Fun(\X, \mU \otimes_\mV \mN)).$$
(3): For $\mN:=\mV=\mW$ seen as bitensored over itself, the $\infty$-category $\mU \otimes_\mV \mN$ bitensored over $\mU,\mV$ arises by pulling back the biaction on $\mU$ over itself along the monoidal functor $\varphi$.
\end{proof}

\begin{remark}\label{conses}
Let $\phi: \mV^\ot \to \mW^\ot$ be a lax monoidal functor and $\kappa: \mM^\circledast \to \mN^\circledast$ a lax $\mW$-linear functor of $\mW$-enriched $\infty$-categories.
Assume that $\mW$ is generated under small colimits by the image of $\phi$. Then $\kappa$ is an equivalence if the induced lax $\mV$-linear functor $\phi^*(\kappa): \phi^\ast(\mM)^\circledast \to \phi^\ast(\mM)^\circledast$ is an equivalence.

Indeed, $\phi^*(\kappa)$ and $\kappa$ induce the same functor on underlying $\infty$-categories.
Thus $\kappa$ is an equivalence if for any $\X,\Y \in \mM$ and $\W \in \mW$
the induced map $\alpha: \mW(\W, \Mor_\mM(\X,\Y)) \to \mW(\W, \Mor_\mN(\kappa(\X),\kappa(\Y))$ is an equivalence.
Because $\mW$ is generated by the essential image of $\phi$ under small colimits,
we can assume that $\W$ is in the essential image of $\phi$.
In this case $\alpha$ is an equivalence as $\phi^*(\kappa)$ is an equivalence.

\end{remark}

\begin{proof}[Proof of Proposition \ref{enrcol}]

By \cite[Proposition 7.15.]{Rune} there is a regular cardinal $\kappa$ such that $\mV$ is $\kappa$-accessible and the monoidal structure on $\mV$ restricts to the full subcategory $\mV^\kappa$ of $\kappa$-compact objects.

Let $(\mV^\kappa)^\ot \subset \mV^\ot$ be the full monoidal subcategory
spanned by $\mV^\kappa$ and $\Ind_\kappa(\mV^\kappa)^\ot \subset \mP(\mV^\kappa)^\ot$ the full suboperad of presheaves on $\mV^\kappa$ preserving $\kappa$-small limits.
By Remark \ref{nos} the full suboperad $\Ind_\kappa(\mV^\kappa)^\ot \subset \mP(\mV^\kappa)^\ot$ is a localization relative to $\Ass$ and the monoidal embedding $\iota: (\mV^\kappa)^\ot \subset \mV^\ot$ uniquely extends to a left adjoint monoidal functor $\Ind_\kappa(\mV^\kappa)^\otimes \to \mV^\ot$, which is an equivalence since $\mV$ is presentable.

By Lemma \ref{looocx} there are adjunctions relative to $\Ass$ whose right adjoints are embeddings:
\begin{equation}\label{adjoum}
\theta: \mP\Env(\mV^\kappa)^\ot \rightleftarrows \mP(\mV^\kappa)^\ot \rightleftarrows \Ind_\kappa(\mV^\kappa)^\ot \simeq \mV^\ot:\gamma.
\end{equation}

The composed adjunction yields for any small $\mV$-enriched $\infty$-category $\mO^\circledast \to \mV^\ot$ a localization
\begin{equation}\label{bjoum}
\mP_{\mP\Env(\mV^\kappa)}(\gamma_*(\mO))^\circledast \rightleftarrows \mP_{\mV}(\mO)^\circledast
\end{equation} relative to $\mV^\ot$ lying over adjunction
(\ref{adjoum}).
Since $\mV \simeq \Ind_\kappa(\mV^\kappa)$ is generated by $\mV^\kappa$ under small filtered colimits, by Remark \ref{conses} the canonical lax $\mV$-linear functor $\iota_*(\iota^*(\mM)^\circledast \to \mM^\circledast $ is an equivalence
since its pullback along $\iota$ is an equivalence.

Let $\bj \simeq \gamma \circ \iota$ be the lax monoidal Yoneda-embedding $(\mV^\kappa)^\ot \subset \mP\Env(\mV^\kappa)^\ot$.
Adjunction $(\ref{bjoum})$ for $\mO^\circledast= \iota_*(\iota^*(\mM))^\circledast \simeq \mM^\circledast $ gives an adjunction
\begin{equation}\label{adjouto}
\phi: \mP\L\Env(\iota^*(\mM))^\circledast\simeq \mP_{\mP\Env(\mV^\kappa)}(\bj_*(\iota^*(\mM)))^\circledast \rightleftarrows \mP_{\mV}(\iota_*(\iota^*(\mM))^\circledast \simeq \mP_{\mV}(\mM)^\circledast
\end{equation}
of left tensored $\infty$-categories lying over adjunction (\ref{adjoum}), where the $\mP\Env(\mV^\kappa)$-linear equivalence on the left hand side is by 
Theorem \ref{cooo}. 
By Lemma \ref{indos} the left adjoint $\phi$ induces a $\mV$-linear equivalence
$ \mV\otimes_{\mP\Env(\mV^\kappa)}\mP\L\Env(\iota^*(\mM))  \simeq \mP_\mV(\mM).$
Therefore adjunction (\ref{adjouto}) gives rise to an adjunction
$$\LaxLinFun_{\mP\Env(\mV^\kappa)}(\mP\L\Env(\iota^*(\mM)),\theta^*(\mN)) \rightleftarrows \LaxLinFun_{\mV}(\mP_\mV(\mM),\mN):\phi^* $$
whose right adjoint is fully faithful, which restricts to an equivalence
$$ \LinFun^\L_{\mP\Env(\mV^\kappa)}(\mP\L\Env(\iota^*(\mM)),\theta^*(\mN)) \simeq \LinFun^\L_{\mV}(\mP_\mV(\mM),\mN) .$$

By Corollary \ref{gggbn} and Proposition \ref{unb} (3) the functor $$\beta: \LaxLinFun_{\mP\Env(\mV^\kappa)}(\mP\L\Env(\iota^*(\mM)),\theta^*(\mN)) \to 
\LaxLinFun_{\mV^\kappa}(\iota^*(\mM),\iota^*(\mN)) $$$$ \simeq \LaxLinFun_{\mV}(\iota_*(\iota^*(\mM)),\mN) \simeq \LaxLinFun_\mV(\mM,\mN)$$
admits a fully faithful left adjoint whose essential image is $\LinFun^\L_{\mP\Env(\mV^\kappa)}(\mP\L\Env(\iota^*(\mM)),\theta^*(\mN)).$

Because $\rho$ factors as $$\LaxLinFun_{\mV}(\mP_\mV(\mM),\mN) \xrightarrow{\phi^*} \LaxLinFun_{\mP\Env(\mV^\kappa)}(\mP\L\Env(\iota^*(\mM)),\theta^*(\mN)) 
\xrightarrow{\beta} $$$$\LaxLinFun_{\mV^\kappa}(\iota^*(\mM),\iota^*(\mN)) \simeq\LaxLinFun_{\mV}(\iota_*(\iota^*(\mM)),\mN) \simeq \LaxLinFun_\mV(\mM,\mN),$$
$\rho$ has a left adjoint, which factors as
$$ \LaxLinFun_\mV(\mM,\mN) \simeq\LaxLinFun_{\mV^\kappa}(\iota^*(\mM),\iota^*(\mN)) \simeq \LinFun^\L_{\mP\Env(\mV^\kappa)}(\mP\L\Env(\iota^*(\mM)),\theta^*(\mN)) $$$$ \simeq \LinFun^\L_{\mV}(\mP_\mV(\mM),\mN) \subset \LaxLinFun_{\mV}(\mP_\mV(\mM),\mN).$$
\end{proof}

\section{A Yoneda-embedding for Lurie-enriched $\infty$-categories}\label{EnrYoneda}

\vspace{2mm}
\label{enrYot}

In this section we construct a model for the opposite enriched $\infty$-category associated to a $\mV$-enriched $\infty$-category  $\mM^\circledast \to \mV^\ot $ (Notation \ref{notori}, Corollary \ref{uuu})
and construct a $\mV$-enriched Yoneda-embedding $\mM^\circledast \to \LaxLinFun_{\mV^\rev}(\mM^\op, \mV)^\circledast$
in terms of Lurie-enriched $\infty$-categories (Notation \ref{nnoo}, Theorem \ref{ooo}).

\begin{lemma}\label{lele}
Let $\mM^\circledast \to \mV^\ot $ be a $\mV$-enriched $\infty$-category.
The functor $ \mM^\op \to \Fun(\mM,\mV), \Z \mapsto \Mor_\mM(\Z,-)$
canonically lifts to an embedding  $\xi: \mM^\op \to \LaxLinFun_{\mV}(\mM,\mV)$.
\end{lemma}

\begin{proof}
First assume that $\mM^\circledast \to \mV^\ot $ is a presentably left tensored $\infty$-category.	
Let $\omega\LMod^\L_{\mV}, \omega\LMod^\R_{\mV} \subset \omega\LMod_{\mV} $
be the wide subcategories whose morphisms are lax $\mV$-linear functors having a left (right) adjoint relative to $\mV^\ot$.
There is a canonical equivalence $$(\omega\LMod^\L_{\mV})^\op \simeq \omega\LMod^\R_{\mV}$$ sending relative left adjoints to relative right adjoints:
for any $\infty$-category $\rS$ functors $\rS \to \omega\LMod^\L_{\mV}$ and $\rS^\op \to \omega\LMod^\R_{\mV}$ are both classified by 
cocartesian and cartesian $\rS$-families of $\infty$-categories weakly left tensored over $\mV.$

Let $\LaxLinFun^\R_{\mV}(\mM, \mV) \subset \LaxLinFun_{\mV}(\mM, \mV)$ be
the full subcategory of lax $ \mV$-linear functors having a left adjoint relative to $ \mV^\ot$. 
For any small $\infty$-category $\mB$ by Lemma \ref{dfgj} (1) the functor $\mB^\op \subset \Fun(\mB,\mS) \xrightarrow{ (-)\ot\tu} \Fun(\mB,\mV)$ exhibits the diagonal left $\mV$-action on $\Fun(\mB,\mV)$ as the free left $\mV$-action on $\mB^\op$ compatible with small colimits. Thus there is a canonical equivalence
$\Cat_\infty(\mB, \mM^\op) \simeq $
$$\Cat_\infty(\mB^\op, \mM) \simeq \omega\LMod^\L_{\mV}(\Fun(\mB, \mV),\mM) \simeq \omega\LMod^\R_{\mV}(\mM, \Fun(\mB, \mV)) \simeq \Cat_\infty(\mB, \LaxLinFun^\R_{\mV}(\mM, \mV)),$$
which is natural in $\mB \in \Cat_\infty$ and by Lemma \ref{dfgj} (2)
sends a functor $\F: \mB \to \mM^\op$ to the functor
$$\mB \to \LaxLinFun^\R_{\mV}(\mM, \mV), \Z \mapsto \Mor_\mM(\F(\Z),-).$$

Let us now assume that $\mM^\circledast \to \mV^\ot $ is a small weakly left tensored $\infty$-category.
We embed $\mM^\circledast \to \mV^\ot $ into the presentably left tensored $\infty$-category $\mP\L\Env(\mM)^\circledast \to \mP\Env(\mV)^\ot $ and 
obtain an equivalence
$$\Cat_\infty(\mB, \mP\L\Env(\mM)^\op) \simeq \Cat_\infty(\mB, \LaxLinFun^\R_{\mP\Env(\mV)}(\mP\L\Env(\mM), \mP\Env(\mV)))$$
that restricts to an equivalence
$$\varphi: \Cat_\infty(\mB, \mM^\op)\simeq \Cat_\infty(\mB, \LaxLinFun^\R_{\mP\Env(\mV)}(\mP\L\Env(\mM), \mP\Env(\mV))'),$$
where $$\LaxLinFun^\R_{\mP\Env(\mV)}(\mP\L\Env(\mM), \mP\Env(\mV))'\subset\LaxLinFun^\R_{\mP\Env(\mV)}(\mP\L\Env(\mM), \mP\Env(\mV))$$
is the full subcategory of lax $\mP\Env(\mV)$-linear functors that admit a left adjoint 
relative to $\mP\Env(\mV)^\ot$ sending the tensor unit of $ \mP\Env(\mV)$ to $\mM.$ 
For $\mB=\mM^\op$ the identity of $\mB$ corresponds 
to a functor $$\theta: \mM^\op \to \LaxLinFun^\R_{\mP\Env(\mV)}(\mP\L\Env(\mM),\mP\Env(\mV))'$$
that by Yoneda induces the equivalence $\varphi$, i.e.
$\varphi= \Cat_\infty(\mB,\theta)$ for any small $\infty$-category
$\mB$. This implies that $\theta$ is an equivalence.
Note that $\theta$ is an equivalence if $\Cat_\infty(\mB,\theta)$
is an equivalence for $\mB=[0],[1]$ and so in particular for all small $\infty$-categories $\mB$ although the target of $\theta$ is a priori
not small. 

By definition $\theta$ sends $\Z$ to $\Mor_{\mP\L\Env(\mM)}(\Z,-): \mP\L\Env(\mM)^\circledast \to \mP\Env(\mV)^\ot,$
which by Theorem \ref{cooo} identifies with the left adjoint $\mP\Env(\mV)$-linear
functor (Remark \ref{tau}):
$$\mP\L\Env(\mM)^\circledast \simeq \mP_{\mP\Env(\mV)}(\chi(\mM))^\circledast \xrightarrow{\mathrm{Mor}_{\mP_{\mP\Env(\mV)}(\chi(\mM)) }(\iota(\Z),-) } \mP\Env(\mV)^\ot.$$ So we find that $\LaxLinFun^\R_{\mP\Env(\mV)}(\mP\L\Env(\mM),\mP\Env(\mV))' \subset \LinFun^\L_{\mP\Env(\mV)}(\mP\L\Env(\mM),\mP\Env(\mV)).$
The resulting embedding $$\theta': \mM^\op \to \LinFun^\L_{\mP\Env(\mV)}(\mP\L\Env(\mM),\mP\Env(\mV)) \simeq  \LaxLinFun_{\mV}(\mM,\mP\Env(\mV))$$
sends $\Z $ to $ \Mor_{\mP\L\Env(\mM)}(\Z,-)_{\mid \mM}.$
If $\mM^\circledast \to \mV^\ot$ exhibits $\mM$ as $\mV$-enriched,
there is an equivalence $$\Mor_{\mP\L\Env(\mM)}(\Z,-)_{\mid \mM} \simeq \Mor_{\mM}(\Z,-) : \mM \to \mV \subset \mP\Env(\mV)$$
(Lemma \ref{morpre}). So $\theta'$ restricts to an embedding 
$\mM^\op \to \LaxLinFun_{\mV}(\mM,\mV)$ that sends $\Z $ to $ \Mor_{\mM}(\Z,-).$
\end{proof}

\begin{notation}\label{notori}
Let $\mM^\circledast \to \mV^\ot $ be a $\mV$-enriched $\infty$-category.
Let $$(\mM^\op)^\circledast \subset \LaxLinFun_\mV(\mM,\mV)^\circledast $$ be the full subcategory with weak right $\mV$-action spanned by the essential image of $\xi.$
\end{notation}
By definition the embedding $\xi$ of Lemma \ref{lele} induces an equivalence $\mM^\op \simeq (\mM^\op)^\circledast_{[0]}$.  

\begin{lemma}\label{lemco}
Let $\mV^\ot \to \Ass$ be a monoidal $\infty$-category compatible with small colimits, $\mC$ a $\mV$-precategory with small space of objects $\X$ and $\psi: \X \to \mP_{\mV}(\mC)=\RMod_\mC( \Fun(\X,\mV)) $ a lift of the functor
$\X \to \Fun(\X,\mV), \ \Z \mapsto \mC(-,\Z)$ along the forgetful functor $\mP_{\mV}(\mC) \to \Fun(\X,\mV).$

For any $\Z \in \X$ the right $\mC$-action on $\mC(-,\Z)$ given by $\psi(\Z)$
induces for any $\A, \Y \in \X$ a map
$\mu_\A: \mC(\A,\Z) \ot \mC(\Y,\A) \to \mC(\Y,\Z)$.
The functor $\psi$ is equivalent to the functor 
$\iota:  \X \to \mP_{\mV}(\mC)$ underlying the enriched Yoneda-embedding
if and only if the canonical map $$\{\id_\Z \}  \ot \mC(\Y,\Z) \to \mC(\Z,\Z) \ot \mC(\Y,\Z)$$ is a section of $\mu_\Z. $ 
\end{lemma}

\begin{proof}
Let $\kappa: \iota \to \psi$ be adjoint to the canonical map from the functor $\X \to \Fun(\X,\mV), \ \Z \mapsto \X(-,\Z) \ot \tu_\mV $ to the functor $ \X \xrightarrow{\psi}  \mP_{\mV}(\mC) \xrightarrow{\text{forget}} \Fun(\X,\mV), \ \Z \mapsto \mC(-,\Z).$
For any $\Z \in \X$ the map $\kappa_\Z$ factors as $$\text{Free}(\X(-,\Z)\ot \tu_{\mV})  \xrightarrow{} \text{Free}(\mC(-,\Z)) \to \mC(-,\Z)$$ in $\mP_\mV(\mC)$.
Evaluating $\kappa_\Z$ at $\Y \in \X$ we obtain the morphism
$$\mC(\Y,\Z) \to \colim_{\A \in \X}(\mC(\A,\Z) \ot \mC(\Y,\A)) \xrightarrow{} \mC(\Y,\Z),$$ where the first map is the inclusion for $\A=\Z$ and the identity of $\Z$ and the second map is induced by $\mu_\A.$ Thus by assumption $\kappa_\Z$ is an equivalence.
\end{proof}

\begin{corollary}\label{co}
Let $\mV^\ot \to \Ass$ be an $\infty$-operad and $\mM^\circledast \to \mV^\ot$
a $\mV$-enriched $\infty$-category.
The functor $$\mM^\simeq \subset \mM \xrightarrow{\xi} \LaxLinFun_{\mV}(\mM,\mV), \ \Z \mapsto \Mor_\mM(\Z,-) $$ is equivalent to the functor 
$\iota: \mM^\simeq \to \mP_{\mV^\rev}(\chi(\mM)^\op) \simeq \LaxLinFun_{\mV}(\mM,\mV)$
underlying the enriched Yoneda-embedding.

\end{corollary}

\begin{proof}
Via the enveloping closed monoidal $\infty$-category and Lemma \ref{morpre}
we can reduce to the case that $\mV^\ot \to \Ass$ is a monoidal $\infty$-category compatible with small colimits.	
In this case we apply Lemma \ref{lemco}.
The statement follows from the fact that for any $\Y, \Z \in \mM$ the composition 
$$\Mor_\mM(\Z,\Y) \ot \{\id_\Z \} \to \Mor_\mM(\Z,\Y) \ot \Mor_\mM(\Z,\Z) \xrightarrow{\mu} \Mor_\mM(\Z,\Y)$$
is the identity, where $\mu$ is adjoint to the map $\Mor_\mM(\Z,\Y) \ot \Mor_\mM(\Z,\Z) \ot \Z \to \Mor_\mM(\Z,\Y) \ot \Z \to \Y.$

\end{proof}

\begin{corollary}\label{uuu}
Let $\mM^\circledast \to \mV^\ot $ be a $\mV$-enriched $\infty$-category.
The equivalence
$$\mP_{\mV^\rev}(\chi(\mM)^\op) \simeq \LaxLinFun_\mV(\mM,\mV) $$
of $\infty$-categories weakly right tensored over $\mV$
of Theorem \ref{werf} restricts to an equivalence
$$\L(\chi(\mM)^\op)^\circledast\simeq (\mM^\op)^\circledast.$$
So there is a canonical equivalence
$\chi(\mM)^\op \simeq \chi(\mM^\op)$ of $\mV^\rev$-enriched $\infty$-categories.
\end{corollary}

Corollary \ref{uuu} and Theorem \ref{werf} give the following corollary:

\begin{corollary}\label{uuuzz}
Let $\mM^\circledast \to \mV^\ot $ be a $\mV$-enriched $\infty$-category.
There is a canonical equivalence
$$\mP_{\mV}(\chi(\mM)) \simeq \LaxLinFun_{\mV^\rev}(\mM^\op,\mV)  $$ 
of $\infty$-categories weakly left tensored over $\mV$ compatible with the forgetful functors to $ \Fun(\mM^\simeq,\mV)$.

\end{corollary}

In the following we construct a $\mV$-enriched Yoneda-embedding $\mM^\circledast \to \LaxLinFun_{\mV^\rev}(\mM^\op, \mV)^\circledast$.

\begin{notation}\label{nnoo}

Under the equivalence of Lemma \ref{lehmmm} the restricted evaluation map $$\mM^\circledast \times (\mM^\op)^\circledast \subset \mM^\circledast \times \LaxLinFun_\mV(\mM,\mV)^\circledast \to \mV^\circledast $$
of $\infty$-categories weakly bitensored over $\mV,\mV$ corresponds to a map 
$$\rho: \mM^\circledast \to \LaxLinFun_{\mV^\rev}(\mM^\op,\mV)^\circledast$$ 
of $\infty$-categories weakly left tensored over $\mV.$
\end{notation}

\begin{remark}
For any $\Z \in \mM$ the functor $\rho(\Z): \mM^\op \to \mV$ factors as $$\mM^\op \xrightarrow{\xi} \LaxLinFun_\mV(\mM,\mV) \xrightarrow{\text{forget}} \Fun(\mM,\mV) \to \mV,$$ where the last functor evaluates at $\Z$, and so is the functor $\Mor_\mM(-,\Z)$. 
\end{remark}

\begin{lemma}\label{cortlou}
Let $\mV^\ot \to \Ass$ be an $\infty$-operad and $\mM^\circledast \to \mV^\ot$ a $\mV$-enriched $\infty$-category. The functor $$\mM^\simeq \subset \mM \xrightarrow{\rho}\LaxLinFun_{\mV^\rev}(\mM^\op,\mV)$$ is equivalent to the functor
$\iota: \mM^\simeq \to \mP_{\mV}(\chi(\mM)) \simeq \LaxLinFun_{\mV^\rev}(\mM^\op,\mV) $ underlying the $\mV$-enriched Yoneda-embedding.
\end{lemma}

\begin{proof}
Via the enveloping closed monoidal $\infty$-category and Lemma \ref{morpre}
we can reduce to the case where $\mV^\ot \to \Ass$ is a monoidal $\infty$-category compatible with small colimits.	
In this case we apply Lemma \ref{lemco}.
The statement follows from the fact that for any $\Y, \Z \in \mM$ the composition 
$$ \{\id_\Z \} \ot \Mor_\mM(\Y,\Z) \to \Mor_\mM(\Z,\Z) \ot  \Mor_\mM(\Y,\Z) \xrightarrow{\mu} \Mor_\mM(\Y,\Z)$$
is the identity, where $\mu$ is adjoint to the map $\Mor_\mM(\Z,\Z) \ot \Mor_\mM(\Y,\Z) \ot \Y \to \Mor_\mM(\Z,\Z) \ot \Z \to \Z.$

\end{proof}

Remark \ref{tau}, Corollary \ref{co} and Lemma \ref{cortlou} imply the following proposition used to prove \ref{ooo}:
\begin{proposition}\label{coog}
Let $\mM^\circledast \to \mV^\ot $ be a $\mV$-enriched $\infty$-category and $\Z \in \mM$.

\begin{enumerate}
\item For any $\rH \in \LaxLinFun_{\mV}(\mM,\mV)$ the canonical morphism
$$\Mor_{\LaxLinFun_{\mV}(\mM,\mV)}(\Mor_\mM(\Z,-),\rH) \to \Mor_{\mV}(\Mor_\mM(\Z,\Z), \rH(\Z)) \to \rH(\Z)$$ in $\mV$ is an equivalence.

\item For any $\rH \in \LaxLinFun_{\mV^\rev}(\mM^\op,\mV)$ the canonical morphism
$$\Mor_{\LaxLinFun_{\mV^\rev}(\mM^\op,\mV)}(\rho(\Z),\rH) \to \Mor_{\mV}(\Mor_\mM(\Z,\Z), \rH(\Z)) \to \rH(\Z)$$ in $\mV$ is an equivalence.
\end{enumerate}	
\end{proposition}

\begin{theorem}\label{ooo}
Let $\mM^\circledast \to \mV^\ot $ be a $\mV$-enriched $\infty$-category.
The canonical map $$\rho: \mM^\circledast \to \LaxLinFun_{\mV^\rev}(\mM^\op,\mV)^\circledast $$ induces an equivalence $\mM^\circledast \to ((\mM^\op)^\op)^\circledast$ of $\infty$-categories weakly left tensored over $\mV.$

\end{theorem}

\begin{proof}\label{coo}

We first prove that $\rho$ induces an essentially surjective functor $\mM^\circledast \to ((\mM^\op)^\op)^\circledast$ by checking that for any $\Z \in \mM$ the canonical map
$$\beta:  \Mor_{\mM^\op}(\Mor_\mM(\Z,-),-) \to \Mor_{\mV}(\Mor_\mM(\Z,\Z),-) \circ \rho(\Z) \to \rho(\Z)$$
in $\LaxLinFun_{\mV^\rev}(\mM^\op,\mV) $ is an equivalence.
For $\rH \in \mM^\op \subset \LaxLinFun_{\mV}(\mM,\mV)$ the map $\beta_\rH$ factors as
$$\Mor_{\mM^\op}(\Mor_\mM(\Z,-),\rH) \simeq \Mor_{\LaxLinFun_{\mV}(\mM,\mV)}(\Mor_\mM(\Z,-),\rH) \to \Mor_{\mV}(\Mor_\mM(\Z,\Z), \rH(\Z)) \to \rH(\Z)$$ and thus is an equivalence by Proposition \ref{coog} (1).
The lax $\mV$-linear functor $\rho$ gives rise to a map
$$\psi: \Mor_\mM(\Z,-) \to \Mor_{\LaxLinFun_{\mV^\rev}(\mM^\op,\mV) }(\rho(\Z),-)\circ \rho $$
of lax $\mV$-linear functors $\mM \to \mV.$
We will complete the proof by checking that $\psi$ is an equivalence.

By construction of $\rho$ there is a canonical equivalence
$\ev_\Z \circ \rho \simeq \Mor_\mM(\Z,-)$ of lax $\mV$-linear functors $\mM \to \mV,$ where $\ev_\Z: \LaxLinFun_{\mV^\rev}(\mM^\op,\mV) \xrightarrow{\text{forget}} \Fun(\X,\mV) \to \mV$ is evaluation at $\Z.$

The lax $\mV$-linear functor $\ev_\Z$ gives rise to a map $$\phi:\Mor_{\LaxLinFun_{\mV^\rev}(\mM^\op,\mV) }(\rho(\Z),-)\circ \rho  \to \Mor_{\mV}(\Mor_{\mM}(\Z,\Z),-) \circ \ev_\Z \circ \rho \to \ev_\Z \circ \rho \simeq \Mor_\mM(\Z,-)$$ of lax $\mV$-linear functors $\mM \to \mV$
that is an equivalence by Proposition \ref{coog} (2).
By Proposition \ref{coog} (1) the map $\phi \circ \psi: \Mor_\mM(\Z,-) \to \Mor_\mM(\Z,-)$ is the identity as $ \phi_\Z \circ \psi_\Z$ preserves the identity of $\Z$. 
\end{proof}

\section{Generalized Day-convolution}
\label{genDay}

\vspace{2mm}

In this section we construct a Day-convolution for generalized $\mO$-operads,
where $\mO \to \Ass$ is a cocartesian fibration relative to the collection of inert morphisms.
By the defining universal property our Day-convolution specializes to the Day-convolution for generalized non-symmetric $\infty$-operads constructed by Haugseng
\cite[Proposition 3.5.6.]{haugseng2021inftyoperads} and non-symmetric $\infty$-operads constructed by Lurie \cite[Theorem 2.2.6.22.]{lurie.higheralgebra}.
We show that for any generalized $\mO$-monoidal $\infty$-category 
$\mC \to \mO$ the functor
$$(-) \times_\mO \mC : \Op_\infty^{\mO,\gen} \to \Op_\infty^{\mO, \gen}$$ admits a right adjoint, which by definition assigns the Day-convolution to a generalized $\mO$-operad.

To prove the existence of the right adjoint we take the following steps,
where we write $\Cat^{\mO^\inert}_{\infty/\mO} \subset \Cat_{\infty / \mO} $ for the subcategory of cocartesian fibrations relative to the collection $\mO^\inert $ of inert morphisms:
 
\begin{enumerate}
\item We first show (Corollary of Proposition \ref{ewq}) that for any cocartesian fibration $\mC \to \mO$ relative to the collection of inert morphisms, which is a flat functor (Definition \ref{flat}), the functor
$$(-) \times_\mO \mC : \Cat^{\mO^\inert}_{\infty/\mO} \to \Cat^{\mO^\inert}_{\infty/\mO}$$ admits a right adjoint \begin{equation}\label{jul} \Fun^\mO(\mC,-): \Cat^{\mO^\inert}_{\infty/\mO} \to \Cat^{\mO^\inert}_{\infty/\mO}.\end{equation}

\item In a second step (Theorem \ref{proop}) we prove that the functor 
(\ref{jul}) sends (generalized) $\mO$-operads to (generalized) $\mO$-operads.
\item Finally, (Proposition \ref{proooo}) we show that if $\mC \to \mO$ is a cocartesian fibration with small fibers, the functor (\ref{jul}) preserves $\mO$-monoidal $\infty$-categories compatible with small colimits.
\end{enumerate}
\begin{notation}
Let $\rS$ be an $\infty$-category and $\mE \subset \Fun([1],\rS)$ a full subcategory. Let $\Cat^\mE_{\infty/\rS} \subset \Cat_{\infty/\rS}$ be the subcategory of cocartesian fibrations relative to $\mE$ and functors over $\rS$ preserving cocartesian lifts of morphisms of $\mE.$

\end{notation}

\begin{notation}
	
Let $\mC \to \rS, \mD \to \rS$ be cocartesian fibrations relative to $\mE$.
Let $\Fun_\rS^{\cocart, \mE}(\mC,\mD) \subset \Fun_\rS(\mC,\mD)$ be the full subcategory of functors over $\rS$ preserving cocartesian lifts of morphisms of $\mE.$
\end{notation}

To state our results we use the following terminology:
\begin{definition}
A relative $\infty$-category is a pair $(\rS,\mE)$,
where $\rS$ is an $\infty$-category and $\mE \subset \Fun([1],\rS)$ a full subcategory that contains all equivalences of $\rS$ 
and the composite of any two composable arrows that belong to $\mE.$
A map of relative $\infty$-categories $(\rS,\mE) \to (\T,\delta)$
is a functor $\rS \to \T$ sending $\mE$ to $\delta.$

\end{definition}

\begin{proposition}\label{ewq}

Let $(\rS,\mE),(\T,\delta)$ be relative $\infty$-categories, $\psi:\mC \to \T$ a cocartesian fibration relative to $\delta$ and $\mC \to \rS$ a flat functor (Definition \ref{flat}) that sends $\psi$-cocartesian lifts of morphisms of $\delta$ to $\mE.$
This guarantees that the functor $(-)\times_\rS \mC: \Cat_{\infty/\rS} \to \Cat_{\infty/\mC}\to  \Cat_{\infty/\T}$ restricts to a functor $\Cat^\mE_{\infty/\rS} \to \Cat^\delta_{\infty/\T}.$ 
If evaluation at the source $\rho: \mE \to \rS$ is flat, the functor $(-)\times_\rS \mC: \Cat^\mE_{\infty/\rS} \to \Cat^\delta_{\infty/\T}$ admits a right adjoint, which we denote by $\Fun^{\rS, \mE, \delta}_{\T}(\mC,-).$

\end{proposition}

\begin{proof}
For any cocartesian fibration $\sigma: \mD \to \T $ relative to $\delta$ we will construct a cocartesian fibration $\Fun_\T^{\rS, \mE,\delta}(\mC,\mD) \to \rS$ relative to $\mE$ such that for any cocartesian fibration
$\phi: \mB \to \rS $ relative to $\mE$ there is a canonical equivalence
\begin{equation}\label{chhe}
\Fun_\T^{\cocart, \delta}(\mB \times_\rS \mC, \mD) \simeq \Fun_\rS^{\cocart, \mE}(\mB, \Fun_\T^{\rS, \mE,\delta}(\mC,\mD)).
\end{equation} 

We start with constructing $\Fun_\T^{\rS, \mE,\delta}(\mC,\mD) \to \rS$: 
since the functors $\rho: \mE \to \rS$ and $\mC \to \rS$ are flat, 
the composition $\mC':= \mE \times_{\Fun(\{1\},\rS)} \mC \to \mE \xrightarrow{\rho} \rS$ is flat. So there is a canonical functor 
$\mC'= \mE \times_{\Fun(\{1\}, \rS)} \mC \to \mE \times \mC \xrightarrow{\rho \times \psi} \rS \times \T,$
whose projection to the first factor is flat.
Hence we may form $\Fun^{\rS \times \ast}_{\rS \times \T}(\mC', \rS \times \mD)$.

The functor $\mE_\X \to \rS$ is a cocartesian fibration relative to $\mE$, where all lifts of maps that belong to $\mE$, are cocartesian over $\rS$.
Thus the functor
$\mE_{\X} \times_\rS \mC \to \mC \to \T$ is a cocartesian fibration relative to $\delta,$ where a lift in $\mE_{\X} \times_\rS \mC $ of a morphism of $\delta$ is cocartesian over $\T$ if and only if its image in $\mC$ is cocartesian over $\T.$
We define $\Fun^{\rS, \mE}(\mC,\mD)$ as the full subcategory of $\Fun^{\rS \times \ast}_{\rS \times \T}(\mC', \rS \times \mD)$ spanned by the objects that belong to
$\Fun^{\cocart, \delta}_{\T}(\mE_{\X} \times_\rS \mC, \mD) \subset \Fun_{\T}(\mE_{\X} \times_\rS \mC, \mD)\simeq \Fun^{\rS \times \ast}_{\rS \times \T}(\mC', \rS \times \mD)_\X$.

Evaluation at the source $\Fun([1],\rS) \to \rS$ is a cartesian fibration,
whose cartesian morphisms are inverted by evaluation at the target.
So the restriction $\mE \subset \Fun([1],\rS) \to \rS$ is a cartesian fibration
relative to $\mE$, whose cartesian morphisms lying over $\mE$ are sent to equivalences under evaluation at the target.
Hence the pullback $\mC' \to \mE \to \rS$ is a cartesian fibration relative to $\mE$, whose cartesian lifts of morphisms of $\mE$ lie over equivalences in $\mC$ and so also in $\T.$ By Lemma \ref{flafla} the functor $ \Fun^{\rS \times \ast}_{\rS \times \T}(\mC', \rS \times \mD)$ is a cocartesian fibration relative to $\mE$
that restricts to a cocartesian fibration $\Fun_\T^{\rS, \mE,\delta}(\mC,\mD) \to \rS$ relative to $\mE$ with the same cocartesian morphisms lying over morphism of $\mE$.

\vspace{1mm}

To check (\ref{chhe}) we first reduce to the case where $\mB=\rS$, i.e. that there is a canonical equivalence
\begin{equation}\label{eqhvb}
\Fun_\T^{\cocart, \delta}(\mC, \mD) \simeq \Fun_\rS^{\cocart, \mE}(\rS, \Fun_\T^{\rS,\mE,\delta}(\mC,\mD)):
\end{equation}
Let $\mE_\mB \subset \Fun([1],\mB) $ be the full subcategory of $\phi$-cocartesian morphisms lying over morphisms of $\mE.$
There is a canonical equivalence $ \mE_\mB \simeq \mB \times_\rS \mE$ over $\mB$,
where we view $\mE, \mE_\mB$ over $\rS,\mB$, respectively, via evaluation at the source.
There is a canonical equivalence $(\mB \times_\rS \mC)' \simeq \mB \times_\rS \mC' $ over $\mB$, where we view $\mC', (\mB \times_\rS \mC)'$ over $\rS$, $\mB$, respectively, via evaluation at the source.
This equivalence yields an equivalence
$ \mB \times_\rS \Fun^{\rS \times \ast}_{\rS \times \T}(\mC', \rS \times \mD) \simeq  \Fun^{\mB \times \ast}_{\mB \times \T}((\mB \times_\rS \mC)', \mB \times \mD)$ over $\mB$ that restricts to an equivalence
\begin{equation}\label{eqwa}
\mB \times_\rS \Fun^{\rS, \mE,\delta}(\mC,\mD) \simeq \Fun_\T^{\mB, \mE_\mB,\delta}(\mB \times_\rS \mC,\mD).
\end{equation}

There is a canonical equivalence $\Fun_\rS^{\cocart, \mE}(\mB, \Fun_\T^{\rS, \mE,\delta}(\mC,\mD)) \simeq \Fun_\mB^{\cocart, \mE_\mB}(\mB, \mB \times_\rS \Fun_\T^{\rS, \mE,\delta}(\mC,\mD)).$
Thus via equivalence (\ref{eqwa}) the general case follows from equivalence (\ref{eqhvb}), which we will prove in the following:
the diagonal embedding $\rS \subset \Fun([1], \rS)$ and so the embedding $\rS \subset \mE$ admit a left adjoint
relative to $\rS$, where we now view $\mE$ over $\rS$ via evaluation at the target.
This localization relative to $\rS$ gives rise to a localization
$\mC \subset \mC'$ relative to $\mC$ (and so also $\T$) with left adjoint the projection $\mC' \to \mC$. This localization gives rise to a localization
$\Fun_\T(\mC, \mD) \subset \Fun_\T(\mC', \mD).$
The following three conditions on a functor
$\alpha: \mC' \to \mD$ over $\T$ are equivalent, where $\theta: \mC' \to \mE \to \rS$ is evaluation at the source:
\begin{itemize}
\item $\alpha$ inverts local equivalences, i.e. those morphisms of $\mC'$, whose image in $\mC$ is an equivalence.
\item $\alpha$ inverts morphisms of $\mC'$ that are $\theta$-cartesian 
lying over morphisms of $\mE$ via $\theta$, which are exactly the maps of
$\mC'$, whose image in $\mC$ is invertible and whose image under $\theta$ belongs to $\mE$.
\item $\alpha$ inverts local equivalences with local target (whose image under $\theta$ has to belong to $\mE$).
\end{itemize}

For any functor $\mC' \to \mD$ over $\T$ corresponding to
a functor $\mC' \to \rS \times \mD$ over $\rS \times \T$ corresponding
to a section $\beta$ of $\kappa: \Fun^{\rS \times \ast}_{\rS \times \T}(\mC', \rS \times \mD) \to \rS$ the functor $\mC' \to \mD$ over $\T$ satisfies the conditions above if and only if the section $\beta$ sends morphisms of $\mE$
to $\kappa$-cocartesian morphisms.
In other words the full subcategories $\Fun_\T(\mC, \mD) \subset \Fun_\T(\mC', \mD)$ and 
$$\Fun^{\cocart,\mE}_\rS(\rS, \Fun^{\rS \times \ast}_{\rS \times \T}(\mC', \rS \times \mD)) \subset \Fun_\rS(\rS, \Fun^{\rS \times \ast}_{\rS \times \T}(\mC', \rS \times \mD))\simeq \Fun_{\rS \times \T}(\mC', \rS \times \mD) \simeq \Fun_{\T}(\mC', \mD)$$
coincide. So we obtain an equivalence
$\Fun_\T(\mC, \mD)\simeq \Fun^{\cocart,\mE}_\rS(\rS, \Fun^{\rS \times \ast}_{\rS \times \T}(\mC', \rS \times \mD))$
that restricts to the desired equivalence 
(\ref{eqhvb}): for any functor $\lambda: \mC \to \mD $ over $\T$ corresponding to a section $\beta$ of $\kappa: \Fun^{\rS \times \ast}_{\rS \times \T}(\mC', \rS \times \mD) \to \rS$
the section $\beta$ lands in $ \Fun_\T^{\rS, \mE}(\mC,\mD)$ if and only if
for all $\X \in \rS$ the composition $\mE_{\X} \times_\rS \mC \to \mC \xrightarrow{\lambda} \mD$ sends morphisms whose image in $\mC$ is a cocartesian lift of a map of $\delta$, to cocartesian lifts of maps of $\delta$.
This is equivalent to ask that $\lambda: \mC \to \mD $ preserves cocartesian lifts of maps of $\delta$.
\end{proof}

\begin{corollary}\label{ewq}
Let $\tau: (\T,\delta) \to (\rS,\mE) $ be a map of relative $\infty$-categories, $\psi:\mC \to \T$ a cocartesian fibration relative to $\delta$ such that the composition $\mC \to \T \to \rS$ is flat.
If evaluation at the source $\rho: \mE \to \rS$ is flat, the functor $(-)\times_\rS \mC: \Cat^\mE_{\infty/\rS} \to \Cat^\delta_{\infty/\T}$ admits a right adjoint, which we denote by $\Fun^{\rS, \mE, \delta}_{\T}(\mC,-).$
	
\end{corollary}

For $\tau$ the identity we obtain the following important corollary:

\begin{corollary}\label{eeayd}

Let $(\rS,\mE)$ be a relative $\infty$-category and $\mC \to \rS$ a cocartesian fibration relative to $\mE$.
If the functors $\mC \to \rS$ and evaluation at the source $\rho: \mE \to \rS$ are flat, the functor $(-)\times_\rS \mC : \Cat^\mE_{\infty/\rS} \to \Cat^\mE_{\infty/\rS}$ admits a right adjoint, which we denote by $\Fun^{\rS, \mE}(\mC,-).$

\end{corollary}

For any functor $\mB \to \rS$ and morphism $\alpha: \X \to \Y$ in $\rS$ let $\mB_\alpha \to [1]$
be the pullback of $\mB \to \rS$ along the functor $[1]\to \rS$ taking $\alpha$.

\begin{remark}\label{forms}
Let $(\rS,\mE),(\T,\delta)$ be relative $\infty$-categories, $\psi:\mC \to \T, \mD \to \T$ cocartesian fibrations relative to $\delta$ and $\tau: \mC \to \rS$ a flat functor that sends $\psi$-cocartesian lifts of morphisms of $\delta$ to $\mE.$
Assume that evaluation at the source $\rho: \mE \to \rS$ is flat.

\begin{enumerate}
\item By construction of $\Fun^{\rS, \mE, \delta}_{\T}(\mC,\mD) \to \rS$
(given in the proof of Proposition \ref{ewq}) for every morphism $\alpha: \X \to \Y$ in $\rS$ there are canonical equivalences
$$\Fun^{\rS, \mE, \delta}_{\T}(\mC,\mD)_\X \simeq \Fun^{\cocart, \delta}_{\T}(\mE_{\X} \times_\rS \mC, \mD),$$
$$\Fun_{[1]}([1],\Fun^{\rS, \mE, \delta}_{\T}(\mC,\mD)_\alpha) \simeq \Fun'_{\T}(\mE_{\alpha} \times_\rS \mC, \mD),$$
where $ \Fun'_{\T}(\mE_{\alpha} \times_\rS \mC, \mD) \subset \Fun_{\T}(\mE_{\alpha} \times_\rS \mC, \mD)$ is the full subcategory of functors over $\T$ whose restrictions to $\mE_\X,\mE_\Y \subset \mE_\alpha$ are maps of cocartesian fibrations relative to $\delta.$
Note that $\mE_{\X} \times_\rS \mC \to \T$ is a cocartesian fibration relative to $\delta$ since $\mE_{\X} \to \rS$ is a cocartesian fibration relative to $\mE.$

\item For any map $\beta: \Y \to \Z$ that belongs to $\mE$ the induced functor $\Fun^{\rS, \mE, \delta}_{\T}(\mC,\mD)_\Y \to \Fun^{\rS, \mE, \delta}_{\T}(\mC,\mD)_\Z$
identifies with the functor $\Fun^\cocart_{\T}(\mE_{\Y} \times_\rS \mC, \mD) \to \Fun^\cocart_{\T}(\mE_{\Z} \times_\rS \mC, \mD)$
precomposing with the functor $ \mE_{\Z} \times_\rS \mC \to \mE_{\Y} \times_\rS \mC$ over $\T$ induced by the functor $ \mE_{\Z} \to \mE_{\Y}$ over $\rS$. 

\item For any map $\alpha \to \beta$ in $\Fun([1],\rS)$
whose evaluation at the source and target belongs to $\mE$, the induced functor $\Fun_{[1]}([1],\Fun^{\rS, \mE, \delta}_{\T}(\mC,\mD)_\alpha) \to \Fun_{[1]}([1], \Fun^{\rS, \mE, \delta}_{\T}(\mC,\mD)_\beta)$
identifies with the functor $\Fun^\cocart_{\T}(\mE_{\alpha} \times_\rS \mC, \mD) \to \Fun^\cocart_{\T}(\mE_{\beta} \times_\rS \mC, \mD)$
precomposing with the functor $ \mE_{\beta} \times_\rS \mC \to \mE_{\alpha} \times_\rS \mC$ over $\T$ induced by the functor $ \mE_{\beta} \to \mE_{\alpha}$
over $\rS$.

\end{enumerate}

\end{remark}

Let $(\rS,\mE)$ be a relative $\infty$-category and $\alpha:\rS' \to \rS$ a 
cocartesian fibration relative to $\mE.$
Let $\mE' \subset \Fun([1],\rS')$ be the full subcategory of $\alpha$-cocartesian lifts of morphisms of $\mE$.
The pair $(\rS',\mE')$ is a relative $\infty$-category that we call the
relative $\infty$-category induced by $\alpha.$

\begin{lemma}\label{pulll}
Let $(\rS,\mE),(\T,\delta)$ be relative $\infty$-categories and $\alpha:\rS' \to \rS, \beta: \T' \to \T$ cocartesian fibrations relative to $\mE,\delta$, respectively, inducing relative $\infty$-categories $(\rS',\mE'),(\T',\delta').$

Let $\psi:\mC \to \T$ be a cocartesian fibration relative to $\delta$
and $\psi': \mC':= \T' \times_\T \mC \to \T'$. Let $\mC \to \rS$ be a flat functor sending $\psi$-cocartesian lifts of morphisms of $\delta$ to $\mE.$
Let $\kappa: \mC' \simeq \rS' \times_\rS \mC$ be an equivalence over $\T$.
The assumptions imply that $\mC' \simeq \rS' \times_\rS \mC \to \rS'$ sends $\psi'$-cocartesian lifts of morphisms of $\delta'$ to $\mE'$.

\vspace{1mm}
If $\rho: \mE \to \rS$ and so $ \rho':\mE' \simeq \rS' \times_\rS \mE \to \rS'$ are flat, the functors $(-)\times_\rS \mC: \Cat^\mE_{\infty/\rS} \to \Cat^\delta_{\infty/\T},$
$(-)\times_{\rS'} \mC': \Cat^{\mE'}_{\infty/\rS'} \to \Cat^{\delta'}_{\infty/\T'}$ have right adjoints $\Fun_\T^{\rS, \mE,\delta}(\mC,-), \Fun_{\T'}^{\rS', \mE',\delta'}(\mC',-)$ by Proposition \ref{ewq}.

\vspace{1mm}

For any cocartesian fibration $\mD \to \T$ relative to $\delta$, there is a canonical equivalence over $\rS'$:
$$\Fun_{\T'}^{\rS', \mE',\delta'}(\T' \times_\T \mC,\T' \times_\T \mD) \simeq \rS' \times_\rS \Fun_\T^{\rS, \mE,\delta}(\mC,\mD).$$

\end{lemma}

\begin{proof}
For any cocartesian fibration $\mB \to \rS'$ relative to $\mE'$
there is a canonical equivalence
$$\Fun^{\mE'}_{\rS'}(\mB,\Fun_{\T'}^{\rS', \mE',\delta'}(\T' \times_\T \mC,\T' \times_\T \mD)) \simeq \Fun^{\delta'}_{\T'}(\mB \times_{\rS'} \mC',\T' \times_\T \mD) \simeq \Fun^{\delta}_{\T}(\mB \times_{\rS'} \mC',\mD) $$
$$\simeq \Fun^{\mE}_{\rS}(\mB,\Fun_\T^{\rS, \mE,\delta}(\mC,\mD))\simeq  \Fun^{\mE'}_{\rS'}(\mB,\rS' \times_\rS \Fun_\T^{\rS, \mE,\delta}(\mC,\mD)).$$
	
\end{proof}

\begin{lemma}\label{bbbn}
	
Let $\L,\R $ be $\infty$-categories and $(\rS,\mE),(\T,\delta)$ relative $\infty$-categories such that evaluation at the source $\mE \to \rS$ is flat.

\begin{enumerate}
\item Let $\psi=(\psi_1,\psi_2):\mC \to \L \times \T$ be a cocartesian fibration relative to $\L \times \delta$, $\phi:\mD \to \T$ a cocartesian fibration relative to $\delta$ and $\kappa=(\kappa_1,\kappa_2): \mC \to \L \times \rS$ a flat functor sending $\psi$-cocartesian lifts of morphisms of $\L \times \delta$ to $\L \times \mE$ such that $(\kappa_1,\psi_2):\mC \to \L \times \T$ is a cocartesian fibration relative to $\L\times \delta$.
If $\kappa, (\kappa_1,\psi_2)$ are maps of cocartesian fibrations over $\L$, then
$\Fun_{\L \times \T}^{\L \times \rS, \L \times \mE,\L \times \delta}(\mC,\L \times \mD)\to \L \times \rS$ is a map of cartesian fibrations over $\L.$

\vspace{1mm}
\item Let $\psi:\mC \to \T $ be a cocartesian fibration relative to $\delta$ and $\kappa: \mC \to \rS $ a flat functor sending $\psi$-cocartesian lifts of morphisms of $\delta$ to $\mE.$
For every cocartesian fibration $\phi:\mD \to \T \times \R$ relative to $\delta \times \Fun([1],\R)$ the functor $\Fun_{\T \times \R}^{\rS \times \R, \mE \times \R,\delta \times \R}(\mC\times \R,\mD)\to \rS \times \R$ is a map of cocartesian fibrations over $\R.$

\vspace{1mm}

\item Let $\psi=(\psi_1,\psi_2):\mC \to \L \times \T$ be a cocartesian fibration relative to $\L \times \delta$, $\kappa=(\kappa_1,\kappa_2): \mC \to \L \times \rS$ a flat functor sending $\psi$-cocartesian lifts of morphisms of $\L \times \delta$ to $\L \times \mE$ and $\phi:\mD \to \T \times \R$ a cocartesian fibration relative to $\delta \times \R$ such that $(\kappa_1,\psi_2):\mC \to \L \times \T$ is a cocartesian fibration relative to $\L\times \delta$.
Let $\zeta$ be the canonical functor $$\Fun_{\L \times \T \times \R}^{\L \times \rS \times \R, \L \times \mE \times \R,\L \times \delta \times \R}(\mC\times \R,\L \times \mD)\to \L \times \rS \times \R.$$

If the functor $\phi:\mD \to \T \times \R$ is a map of cocartesian fibrations over $\R$, then $\zeta$ is a map of cocartesian fibrations over $\R.$
If $\kappa,(\kappa_1,\psi_2)$ are maps of cocartesian fibrations over $\L$, 
then $\zeta$ is a map of cartesian fibrations over $\L$.

\end{enumerate}
	
\end{lemma}

\begin{proof}
(1): We apply the notation of the proof of Proposition \ref{ewq}.
Let $$\lambda: \mC':=\mE \times_{\Fun(\{1\},\rS)} \mC \simeq (\L \times\mE \times \T) \times_{(\L \times\Fun(\{1\},\rS)) \times \T} \mC \to \L \times \mE \times \T \to \L \times \Fun(\{0\},\rS) \times \T,$$ 
where we use $\kappa_1,\kappa_2,\psi_2$ for the pullback.
Note that $\lambda$ is a map of cocartesian fibrations over $\L$ as it is the pullback of cocartesian fibrations over $\L$ and maps of such,
where we use that $\kappa_2, \psi_2$ invert $\kappa_1$-cocartesian morphisms.
By construction 
$\Fun_{\L \times \T}^{\L \times \rS, \L \times \mE,\L \times \delta}(\mC,\L \times \mD)$ is the full subcategory of $$\Fun_{\L \times \rS \times \L \times \T}^{\L \times \rS}(\mC',\L \times \rS \times \L \times \mD) \simeq \Fun_{\L \times \rS \times \T}^{\L \times \rS}(\mC',\L \times \rS \times \mD)$$ whose fiber over any $\X \in \L, \Y \in \rS$
consists of the functors $\mC'_{\X,\Y} \simeq \mE_\Y \times_{\Fun(\{1\},\rS)} \mC_\X \to \mD$ over $\T$
sending morphisms whose image in $\mC_\X$ is a cocartesian lift of a morphism of $\delta$, to $\phi$-cocartesian morphisms.	

By Lemma \ref{fibb} the functor 
$\Fun^{\L \times \rS}_{\L \times \rS \times \T}(\mC',\L \times \rS \times \mD) \to \L \times \rS $ is a map of cartesian fibrations over $\L.$
Because $(\kappa_1,\psi_2)$ is a cocartesian fibration relative to $\Fun([1],\L)\times \delta$,
for every morphism $\X \to \X'$ in $\L$ the induced functor
$\mC_\X \to \mC_{\X'}$ over $\rS$ and over $\T$ is a map of cocartesian fibrations relative to $\delta$. So the restriction $\Fun_{\L \times \T}^{\L \times \rS, \L \times \mE,\L \times \delta}(\mC,\L \times \mD)\to \L$
is a cartesian fibration with the same cartesian morphisms.

\vspace{1mm}
(2): By construction $\Fun_{\T \times \R}^{\rS \times \R, \mE \times \R,\delta \times \R}(\mC\times \R,\mD)$ is the full subcategory of $\Fun_{\rS \times \R \times \T \times \R}^{\rS \times \R}(\mC' \times \R,\rS \times \R \times \mD)$ whose fiber over any $\X \in \rS, \Y \in \R$
consists of the functors $\mE_\X \times_{\Fun(\{1\},\rS)} \mC \to \mD_\Y$ over $\T$
sending morphisms whose image in $\mC$ is a cocartesian lift of a morphism of $\delta$, to $\phi$-cocartesian morphisms.	

The canonical functor $\alpha: \Fun_{\rS \times \R \times \T \times \R}^{\rS \times \R}(\mC' \times \R,\rS \times \R \times \mD) \to \rS \times \R$ is the pullback of the functor $\phi_*: \Fun^{\rS}(\mC',\rS \times \mD)\to \Fun^{\rS}(\mC',\rS \times \T \times \R) \simeq \Fun^{\rS}(\mC',\rS \times \T) \times_\rS  \Fun^{\rS}(\mC', \rS \times \R)$ induced by $\phi$
along the functor $\xi:\rS \times \R \to \Fun^{\rS}(\mC',\rS \times \T \times \R) $ over $\rS $ given by $\mC'\times \R \to \T \times \R.$

Because $\phi: \mD \to \T \times \R$ is a cocartesian fibration relative to the collection of morphisms whose first component is an equivalence, the induced functor $\phi_*$ is a cocartesian fibration relative to the collection of morphisms whose first component is an equivalence. Moreover a morphism whose first component is an equivalence, is $\phi_*$-cocartesian if and only if it is object-wise $\phi$-cocartesian.
Since $\xi$ preserves morphisms whose first component is an equivalence,
the pullback $\alpha$ of $\phi_*$ along $\xi$ is a cocartesian fibration
relative to the collection of maps whose first component is an equivalence.
Because $\phi:\mD \to \T \times \R$ is a cocartesian fibration relative to $\delta \times \Fun([1],\R)$, for any $\Y \to \Y'$ in $\R$ the induced functor $\mD_\Y \to \mD_{\Y'}$ over $\T$ is a map of cocartesian fibrations relative to $\delta$. So the restriction $\Fun_{\T \times \R}^{\rS \times \R, \mE \times \R,\delta \times \R}(\mC\times \R,\mD)\to \rS \times \R$
is a map of cocartesian fibrations over $\R.$ (3) follows from (1),(2).
\end{proof}

Recall the definition of a factorization system \cite[Definition 5.2.8.8., Proposition 5.2.8.17.] {lurie.HTT}:

\begin{definition}
Let $\mC$ be an $\infty$-category. A factorization system on $\mC$ is a pair $(\L,\R)$ of full subcategories of $\Fun([1],\mC)$ such that 
\begin{itemize}
\item $(\mC,\L),(\mC,\R)$ are relative $\infty$-categories,
\item Restriction along the embedding $[1]\simeq \{0<2\} \subset [2]$ induces an equivalence $\Fun([2],\mC)' \to \Fun([1],\mC)$,
where $\Fun([2],\mC)' \subset \Fun([2],\mC)$ is the full subcategory of functors $[2] \to \mC$
whose restriction to $0 \to 1$ is in $\L$ and whose restriction to $1 \to 2$ is in $\R.$
\end{itemize}
\end{definition}

The following definition is due to \cite[Definition 2.1.]{Rune}.

\begin{definition}

An algebraic pattern is a quadruple $(\mO,\mO^\el,\mO^\inert,\mO^\act)$,
where $\mO$ is an $\infty$-category, $\mO^\el $ is a full subcategory
of $\mO$ whose objects we call elementary, and $\mO^\inert,\mO^\act $ are full subcategories of $\Fun([1],\mO)$ whose objects we call inert, active morphisms of
$\mO$, respectively, such that $(\mO^\inert,\mO^\act)$ is a factorization system on
$\mO.$

\end{definition}

\begin{example}\label{nooor}

Let $(\mO,\mO^\el,\mO^\inert,\mO^\act)$ be an algebraic pattern and $\psi: \mO' \to \mO $ a cocartesian fibration relative to $\mO^\inert.$
We call a morphism of $\mO'$ inert if it is a $\psi$-cocartesian lift of an inert morphism of $\mO$. We call a morphism of $\mO'$ active if its image in $\mO$ is active. We call an object of $\mO'$ elementary if its image in $\mO$ is elementary.
Let $\mO'^\inert, \mO'^\act \subset \Fun([1],\mO')$ be the full subcategories of inert, active morphisms, respectively,
and let $\mO'^\el \subset \mO'$ be the full subcategory of elementary objects.

Then $(\mO',\mO'^\el,\mO'^\inert,\mO'^\act)$ is an algebraic pattern.

\end{example}

\begin{example} 
Let $\Ass^\inert,\Ass^\act \subset \Fun([1],\Ass)$ be the full subcategories of inert, active morphisms, respectively.
Then $(\Ass, \{[1]\}, \Ass^\inert,\Ass^\act),\hspace{1mm} (\Ass, \{[0],[1]\}, \Ass^\inert,\Ass^\act)$ are algebraic pattern.
So by the last example for any cocartesian fibration $\mO \to \Ass$ relative to $\Ass^\inert$ we obtain algebraic pattern
$(\mO,\{[1]\} \times_\Ass \mO, \mO^\inert,\mO^\act), \ (\mO,\{[0], [1]\} \times_\Ass \mO, \mO^\inert,\mO^\act).$ 
\end{example}

\begin{example}\label{extor}
Let $\mV^\ot \to \Ass$ be a generalized $\infty$-operad.

There is an algebraic pattern $\mathfrak{P}_\L$ on $\mV^\ot$,
where the inert morphisms are the cocartesian lifts of inert maps $[\n] \to [\m]$ of $\Delta$ preserving the minimum,
the active morphisms are the morphisms of $\mV^\ot$
lying over a map $[\n] \to [\m]$ of $\Delta$ preserving the maximum,
and the elementary objects are the objects of $\mV^\ot$ lying over $[0]\in \Ass.$

Similarly, there is an algebraic pattern $\mathfrak{P}_\R$ on $\mV^\ot$,
where the inert morphisms are the cocartesian lifts of inert maps $[\n] \to [\m]$ of $\Delta$ preserving the maximum,
the active morphisms are the morphisms of $\mV^\ot$
lying over a map $[\n] \to [\m]$ of $\Delta$ preserving the minimum,
and the elementary objects are the objects of $\mV^\ot$ lying over $[0]\in \Ass.$

\end{example}

\begin{notation}

Let $(\mO,\mO^\el,\mO^\inert,\mO^\act)$ be an algebraic pattern.
Let $\mO^\inert_\el \subset \mO^\inert$ be the full subcategory of inert morphisms whose target belongs to $\mO^\el$.
For $\X \in \mO$ let $\mO^\inert_{\X /} \subset \mO_{\X/}$ be the full subcategory of inert morphisms $\X \to \Y$ and $\mO^\el_{\X /} \subset \mO^\inert_{\X/}$ be the full subcategory of inert morphisms $\X \to \Y$ such that $\Y \in \mO^\el.$
So $\mO^\el_{\X/}$ is the fiber of evaluation at the source $\mO^\inert_\el \to \mO$.
\end{notation}

Evaluation at the source $\Fun([1],\mO) \to \mO$ restricts to cartesian fibrations $\mO^\inert \to \mO, \mO^\inert_\el \to \mO$ relative to $\mO^\inert$.
For every $\X \in \mO$ the forgetful functor $\mO^\inert_{\X/} \to \mO$
sends any morphism to an inert morphism.
Since $\mO^\inert_{\X/}$ has an initial object, for any cocartesian fibration $\mC \to \mO$ relative to $\mO^\inert$,
we obtain a canonical map $\mO^\inert_{\X/} \times \mC_\X \to \mO^\inert_{\X/} \times_\mO \mC$ of cocartesian fibrations over $\mO^\inert_{\X/}$.
The pullback of the latter functor to $\mO^\el_{\X/}$ corresponds to a functor $\mC_\X \to \Fun^\cocart_{\mO^\el_{\X/}}(\mO^\el_{\X/},\mO^\el_{\X/} \times_\mO \mC)\simeq \lim_{\Z \in \mO^\el_{\X/}} \mC_\Z$ used in the following definition:

\begin{definition}
Let $\mathfrak{P}=(\mO,\mO^\el,\mO^\inert,\mO^\act)$ be an algebraic pattern.
A $\mathfrak{P}$-fibrous object is a cocartesian fibration $\phi: \mC \to \mO$
relative to $\mO^\inert$ such that the following conditions hold:

\begin{enumerate}
\item For every $\X \in\mO$ the induced functor
$$ \mC_\X \to  \lim_{\Z \in \mO^\el_{\X/}} \mC_\Z$$
is an equivalence.

\vspace{1mm}
\item For any $\Y \in \mC$ the forgetful functor $\mC^\el_{\Y/} \to \mC$ is a $\phi$-limit diagram, i.e. for every $\Z \in \mC$ the following commutative square is a pullback square:
\begin{equation*}
\begin{xy}
\xymatrix{
\mC(\Z, \Y) \ar[d]
\ar[r]^{}
&  \lim_{\Y' \in \mC^\el_{\Y/}} \mC(\Z,\Y') \ar[d]^{} 
\\
\mO(\phi(\Z), \phi(\Y)) \ar[r]^{}  & \lim_{\Y' \in \mC^\el_{\Y/}} \mO(\phi(\Z),\phi(\Y')) }
\end{xy}
\end{equation*}

\end{enumerate}	

\end{definition}

\begin{remark}
Condition (2) implies that the functor in (1) is fully faithful
by taking the fiber over the identity of $\phi(\Z) $ in the square of (2).
So if (2) holds, (1) holds if and only if the functor in (1) is essentially surjective.
\end{remark}

\begin{example}Let $\mO \to \Ass$ be a cocartesian fibration relative to $\Ass^\inert$.

\begin{itemize}
\item 
A cocartesian fibration $\mC \to \mO$ relative to $\mO^\inert$
is $(\mO,\{[1]\} \times_\Ass \mO, \mO^\inert,\mO^\act)$-fibrous if and only if it is an $\mO$-operad.

\item 
A cocartesian fibration $\mC \to \mO$ relative to $\mO^\inert$
is $(\mO,\{[0], [1]\} \times_\Ass \mO, \mO^\inert,\mO^\act)$-fibrous if and only if it is a generalized $\mO$-operad.
\end{itemize}		
\end{example} 

For the next example we use the algebraic pattern of Example \ref{extor}.
It follows from Remark \ref{porto}:

\begin{example}\label{exuk}
	
Let $\mV^\ot \to \Ass, \mW^\ot \to \Ass$ be generalized $\infty$-operads.

\begin{itemize}
\item A functor $\mM^\circledast \to \mV^\ot$ is $\mathfrak{P}_\L$ fibrous if and only if it is a weakly left tensored $\infty$-category.
 
\item A functor $\mM^\circledast \to \mV^\ot$ is $\mathfrak{P}_\R$ fibrous if and only if it is a weakly right tensored $\infty$-category.
 
\item A functor $\mM^\circledast \to \mV^\ot \times \mW^\ot $ is $\mathfrak{P}_\L \times \mathfrak{P}_\R$ fibrous if and only if it is a weakly bitensored $\infty$-category.
	
\end{itemize}
	
\end{example}

\begin{lemma}\label{exx}
Let $(\mO, \mO^\el,\mO^\inert,\mO^\act)$ be an algebraic pattern.
Evaluation at the source $\mO^\inert \to \mO$ is a cartesian fibration
(and so in particular flat).
\end{lemma}

\begin{proof}
By definition there is an inert active factorization system on $\mO$.
So by \cite[Lemma 5.2.8.19.]{lurie.HTT} $\mO^\inert \subset \Fun([1],\mO)$ is a colocalization relative to $\mO$ (via evaluation at the source). Since evaluation at the source $\Fun([1],\mO) \to \mO$ is a cartesian fibration, also its colocalization $\mO^\inert \subset \Fun([1],\mO) \to \mO$ relative to $\mO$
is a cartesian fibration by (the dual of) Lemma \cite[2.2.4.11.]{lurie.higheralgebra}.
\end{proof}

\vspace{1mm}

Let $\mathfrak{P}=(\mO, \mO^\el,\mO^\inert,\mO^\act), \mathfrak{P'}=(\mO', \mO'^\el,\mO'^\inert, \mO'^\act)$ be algebraic pattern, $\mC \to \mO',\mD \to \mO'$ cocartesian fibrations relative to $\mO'^\inert$ and $\mC \to \mO$
a flat functor that sends cocartesian lifts of morphisms of $\mO'^\inert$
to $\mO^\inert.$
By Lemma \ref{exx} evaluation at the source $\rho: \mO^\inert \to \mO$ is flat.
Thus by Proposition \ref{ewq} the functor $(-)\times_\mO \mC: \Cat^{\mO^{\inert}}_{\infty/\mO} \to \Cat^{\mO'^\inert}_{\infty/\mO'}$ admits a right adjoint, which we denote by $\Fun^{\mathfrak{P}}_{\mathfrak{P'}}(\mC,-).$

\begin{theorem}\label{opppl}

Let $\mathfrak{P}=(\mO, \mO^\el,\mO^\inert,\mO^\act), \mathfrak{P'}=(\mO', \mO'^\el,\mO'^\inert, \mO'^\act)$ be algebraic pattern, $\mC \to \mO',\mD \to \mO'$ cocartesian fibrations relative to $\mO'^\inert$ and $\mC \to \mO$
a flat functor that sends cocartesian lifts of morphisms of $\mO'^\inert$
to $\mO^\inert.$

Assume that for every $\X \in \mO$ the $\infty$-category
$\mO^\el_{\X/}$ is weakly contractible and the canonical functor $\colim_{\Z \in\mO^\el_{\X/}}\mO^\el_{\Z/} \to\mO^\el_{\X/}$ is an equivalence.
If $\mD \to \mO'$ is $\mathfrak{P'}$-fibrous, then $\Fun^{\mathfrak{P}}_{\mathfrak{P'}}(\mC,\mD) \to \mO$ is $\mathfrak{P}$-fibrous.	

\end{theorem}

We apply Theorem \ref{opppl} to the algebraic pattern underlying a cocartesian fibration $\mO \to \Ass$ relative to the collecction of inert morphisms.
For that we need the following lemma:

\begin{lemma}\label{limo}
Let $\mO \to \Ass$ be a cocartesian fibration relative to the collection of inert morphisms.
For every $\X \in \mO$ the $\infty$-category
$\mO^\el_{\X/}$ is weakly contractible and the functor $\colim_{\Z \in\mO^\el_{\X/}}\mO^\el_{\Z/} \to\mO^\el_{\X/}$ is an equivalence.

\end{lemma}

\begin{proof}
For any $\Z\in \mO_{[0]}$ the $\infty$-category
$\mO^\el_{\Z/} $ is contractible, for any $\Z\in \mO_{[1]}$ we have
$\mO^\el_{\Z/} \simeq \Lambda_0^2$ and for any $\Z\in \mO_{[\n]}$ for $\n \geq 2$ we have $\mO^\el_{\Z/} \simeq \Lambda_0^2 \coprod_{[0]} \Lambda_0^2\coprod_{[0]}... \coprod_{[0]} \Lambda_0^2$ using $\n$ copies. 

The functor $\Lambda_2^2 \coprod_{[0]} \Lambda_2^2\coprod_{[0]}... \coprod_{[0]} \Lambda_2^2 \simeq (\mO^\el_{\Z/})^\op \to \Cat_\infty, \Z \mapsto \mO^\el_{\Z/}$
encodes the diagram
$$\{0\} \subset \Lambda_0^2 \supset \{2=0\} \subset \Lambda_0^2 \supset \{2=0\} ... \{2=0\} \subset \Lambda_0^2 \supset \{2\}$$
whose colimit is $\Lambda_0^2 \coprod_{[0]} \Lambda_0^2\coprod_{[0]}... \coprod_{[0]} \Lambda_0^2 \simeq\mO^\el_{\X/}$ using $\n$ copies of $\Lambda_0^2.$

\end{proof}

Let $\mO \to \Ass, \mO' \to \Ass, \ \mC \to \mO', \ \mD \to \mO'$ be cocartesian fibrations relative to the collection of inert morphisms and $\mC \to \mO$ a flat functor that preserves inert morphisms.
By Proposition \ref{ewq} the functor $(-)\times_\mO \mC: \Cat^{\mO^{\inert}}_{\infty/\mO} \to \Cat^{\mO'^\inert}_{\infty/\mO'}$ admits a right adjoint, which we denote by $\Fun^{\mO}_{\mO'}(\mC,-).$

If $\mO'=\mO$ and the functors $\mC \to \mO'$ and $\mC \to \mO$ agree,
we write $\Fun^{\mO}(\mC,-)$ for $\Fun^{\mO}_{\mO'}(\mC,-).$

\vspace{2mm}

Theorem \ref{opppl} and Lemma \ref{limo} imply the following corollary:

\begin{theorem}\label{proop}\label{opppl}
Let $\mO \to \Ass, \mO' \to \Ass, \ \mC \to \mO', \ \mD \to \mO'$ be cocartesian fibrations relative to the collection of inert morphisms and $\mC \to \mO$ a functor over $\Ass$, where we see $\mC$ over $\Ass$ via $\mC \to \mO' \to \Ass,$
that is flat and preserves inert morphisms.
	
\begin{enumerate}
\item If $\mD \to \mO'$ is a generalized $\mO'$-operad, $\Fun_{\mO'}^{\mO}(\mC,\mD)\to \mO$ is a generalized $\mO$-operad.
		
\item If $\mD \to \mO'$ is an $\mO'$-operad, $\Fun_{\mO'}^{\mO}(\mC,\mD) \to \mO$ is an $\mO$-operad.
		
\end{enumerate}
	
\end{theorem}

\begin{proof}
(2) follows from Theorem \ref{opppl}.
(3) follows from (2) when we have shown that the functor
$\Fun_{\mO'}^{\mO}(\mC,\mD)_{[0]}\to \mO_{[0]}$ is an equivalence.
By Lemma \ref{pulll} there is a canonical equivalence
$$\Fun_{\mO'}^{\mO}(\mC,\mD)_{[0]} \simeq \mO_{[0]} \times_\mO \Fun_{\mO'}^{\mO}(\mC,\mD) \simeq \Fun_{\mO'_{[0]}}^{\mO_{[0]}}(\mC_{[0]},\mD_{[0]})$$
over $\mO_{[0]}.$ Since $\mD_{[0]} \to \mO'_{[0]}$ is an equivalence,
the result follows.

\end{proof}

\begin{lemma}\label{limot}
Let $\mV^\ot \to \Ass$ be a generalized monoidal $\infty$-category.
For every $\X \in \mV$ the $\infty$-category
$\mV^\el_{\X/}$ is weakly contractible and the functor $\colim_{\Z \in\mV^\el_{\X/}}\mV^\el_{\Z/} \to\mV^\el_{\X/}$ is an equivalence.

\end{lemma}

\begin{proof}
For any $\Z\in \mO$ the $\infty$-category $\mV^\el_{\Z/} $ is contractible.
\end{proof}

Theorem \ref{opppl} and Lemma \ref{limot} imply the following example:

\begin{example}

Let $\mV^\ot \to \Ass, \mW^\ot \to \Ass$ be generalized $\infty$-operads, $\mM^\circledast \to \mV^\ot \times \mW^\ot $ a weakly bitensored $\infty$-category and $\X, \Y$ be spaces.
Example \ref{exuk} and Theorem \ref{opppl} imply that the functor 
$$ \Fun(\X \times \Y,\mM)^\circledast:=$$$$\Fun^{\Quiv_\X(\mV)^\ot \times \Quiv_\Y(\mW)^\ot}_{\Quiv_\X(\mV)^\ot_\X \times \Quiv_\Y(\mW)^\ot_\Y}(\Quiv_\X(\mV)^\ot_\X \times \Quiv_\Y(\mW)^\ot_\Y, \Quiv_\X(\mV)^\ot_\X \times \Quiv_\Y(\mW)^\ot_\Y \times_{\mV^\ot \times \mW^\ot} \mM^\circledast) \to $$$$ \Quiv_\X(\mV)^\ot \times \Quiv_\Y(\mW)^\ot $$
is a weakly bitensored $\infty$-category, where we use the projection
$\Quiv_\X(\mV)^\ot_\X \times \Quiv_\Y(\mW)^\ot_\Y \to \Quiv_\X(\mV)^\ot \times \Quiv_\Y(\mW)^\ot.$

The functor $ \Fun(\X \times \Y,\mM)^\circledast\to \Quiv_\X(\mV)^\ot \times \Quiv_\Y(\mW)^\ot $ satisfies the universal property of Remark \ref{paraph}
and so identifies with the functor of Notation \ref{exxxx} with the same name.

\end{example}

Next we prepare the proof of Theorem \ref{opppl}.
To prove Theorem \ref{proop} we use the following lemma:

\begin{lemma}\label{ley}

Let $\mathfrak{P'}=(\mO', \mO'^\el,\mO'^\inert, \mO'^\act)$ be algebraic pattern, $\mC \to \mO'$ a cocartesian fibrations relative to $\mO'^\inert$, $\gamma: \mD \to \mO'$ a $\mathfrak{P'}$-fibered object and $\mC \to \mO$
a flat functor that sends cocartesian lifts of morphisms of $\mO'^\inert$
to $\mO^\inert$ and sends objects lying over elementary objects to elementary objects.
Let $\alpha: \X \to \Y$ be a morphism of $\mO.$
The functors $$\theta: \Fun_{\mO'}(\mO^\inert_{\X/}\times_\mO \mC, \mD) \to \Fun_{\mO'}(\mO^\el_{\X/} \times_\mO \mC, \mD), $$$$\zeta: \Fun_{\mO'}(\mO^\inert_\alpha \times_\mO \mC,\mD)\to \Fun_{\mO'}(\mO^\el_\alpha \times_\mO \mC,\mD)$$
restricting along the respective embeddings $\mO^\el_{\X/}\times_\mO \mC \subset \mO^\inert_{\X/}\times_\mO \mC$, $\mO^\el_\alpha \times_\mO \mC \subset \mO^\inert_\alpha \times_\mO \mC$ induce equivalences
$$\theta': \Fun^{\inert}_{\mO'}(\mO^\inert_{\X/}\times_\mO \mC, \mD) \to \Fun^{\inert}_{\mO'}(\mO^\el_{\X/} \times_\mO \mC, \mD), $$$$\zeta': \Fun'_{\mO'}(\mO^\inert_\alpha \times_\mO \mC,\mD) \to \Fun'_{\mO'}(\mO^\el_\alpha \times_\mO \mC,\mD).$$	
	
\end{lemma}

In the following we prepare the proof of Lemma \ref{ley}.

\begin{definition}
	
We say that a morphism $\f: \X \to \Y$ of an $\infty$-category
$\mC$ is left orthogonal to a morphism $\g: \A \to \B$ in $\mC$ if
the canonical map
$ \mC(\Y,\A) \to \mC(\X,\A) \times_{\mC(\X,\B)} \mC(\Y,\B) $
is an equivalence.

\end{definition}

\begin{lemma}\label{ortho}
	
Let $\phi: \mC \to \rS$ be a functor. Any $\phi$-cocartesian morphism $\X \to \Y$ is left orthogonal to any morphism $\A \to \B$ of $\mC$ that is inverted by $\phi$.
	
\end{lemma}

\begin{proof}
	
Consider the following commutative diagrams
\begin{equation*}
\begin{xy}
\xymatrix{
\mC(\Y,\A) \ar[d] \ar[r]^{}
&  \mC(\X,\A) \ar[d]
\\ \mC(\Y,\B)
\ar[r] \ar[d] &\mC(\X,\B) \ar[d]
\\\rS(\phi(\Y),\phi(\B))
\ar[r] & \rS(\phi(\X),\phi(\B)),
}
\end{xy} 
\begin{xy}
\xymatrix{
\mC(\Y,\A) \ar[d] \ar[r]^{}
&  \mC(\X,\A) \ar[d]
\\ \rS(\phi(\Y),\phi(\A))
\ar[r] \ar[d]^\simeq & \rS(\phi(\X),\phi(\A)) \ar[d]^\simeq
\\\rS(\phi(\Y),\phi(\B))
\ar[r] & \rS(\phi(\X),\phi(\B)).
}
\end{xy} 
\end{equation*}	
The outer square of left hand diagram is equivalent to the outer square of the right hand diagram.
Since the bottom square of the left hand diagram is a pullback square, and
the top and bottom square of the right hand diagram is a pullback square,
by the pasting laws the top square of the left hand diagram is a pullback square.
	
\end{proof}

\begin{remark}\label{colo}
	
Let $\mC$ be an $\infty$-category and $\f \to \h$ a morphism in $ \Fun([1],\mC)$
corresponding to a commutative square
\begin{equation*}
\begin{xy}
\xymatrix{
\X \ar[d]^{\f} \ar[r]^{\simeq}
& \Y \ar[d]^{\h}
\\  \T
\ar[r]^\g & \Z.
}
\end{xy} 
\end{equation*} 
	
For every morphism $\rt: \A \to \B$ in $\mC$ left orthogonal to $\g$
the induced map $$ \Fun([1],\mC)(\rt, \f) \to \Fun([1],\mC)(\rt, \h)$$
over $\mC(\A, \X) \simeq \mC(\A, \Y)$ is an equivalence
since it induces on the fiber over any map $\alpha : \A \to \X$ in $\mC$
the map induced by the map
$ \mC(\B,\T) \to \mC(\A,\T) \times_{\mC(\A,\Z)} \mC(\B,\Z) $ on the fiber over $\f \circ \alpha: \A \to \X \to \T.$
	
\vspace{2mm}
This has the following consequence: let $\mE, \L, \R \subset \Fun([1], \mC)$ be
full subcategories such that every morphism of $\mE$ factors as a morphism of $\L$ followed by a morphism of $\R$ and every $\f \in \L$ is left orthogonal to every $\g \in \R$.	
If $\L \subset \mE,$ the embedding $\L \subset \mE$ admits a right adjoint, where a morphism is a colocal equivalence if and only if its image under evaluation at the target belongs to $\R$ and its image under evaluation at the source is an equivalence.
	
\vspace{1mm}
	
Hence by Lemma \ref{ortho} for any cocartesian fibration $\phi: \mC \to \rS$ relative to some full subcategory $\mE \subset \Fun([1],\rS)$ the full subcategory of $ \mE \times_{\Fun([1],\rS)} \Fun([1],\mC)$ spanned by the $\phi$-cocartesian morphisms is a colocalization relative to $\Fun(\{0\},\mC) \times \mE.$ 	
	
\end{remark}

\begin{remark}\label{cofinal}
	
Let $\theta : \mA \rightleftarrows \mB:\R$ be an adjunction and $\rho: \mB \to \mC$ be a functor. The map
$$\lambda_\mC: \mC \times_{\Fun(\mB,\mC)} \Fun(\mB,\mC)_{/\rho} \to \mC \times_{\Fun(\mA,\mC)} \Fun(\mA,\mC)_{/\rho \circ \theta}$$
of right fibrations over $\mC$ is an equivalence, where the pullbacks are formed via the diagonal functors $\delta_\mB: \mC \to \Fun(\mB,\mC), \delta_\mA: \mC \to \Fun(\mA,\mC).$
	
Indeed, the adjunction $\theta : \mA \rightleftarrows \mB: \R$ gives rise to an adjunction
$\R^*: \Fun(\mA,\mC) \rightleftarrows \Fun(\mB,\mC): \theta^*$.
For any $\X \in \mC$ the canonical equivalence 
$\Fun(\mB,\mC)(\delta_\mA(\X) \circ \R,\rho) \simeq \Fun(\mA,\mC)(\delta_\mA(\X),\rho \circ \theta)$
factors as $$\Fun(\mB,\mC)(\delta_\mA(\X) \circ \R,\rho) \xrightarrow{\beta} \Fun(\mA,\mC)(\delta_\mA(\X) \circ \R\circ \theta,\rho\circ \theta) \xrightarrow{\gamma} \Fun(\mA,\mC)(\delta_\mA(\X),\rho \circ \theta),$$
where $\beta$ precomposes with $\theta$ and $\gamma$ precomposes with the equivalence $\delta_\mA(\X)\to \delta_\mA(\X) \circ \R \circ \theta.$
Hence $\gamma$ and so $\beta$ are equivalences.
The map induced by $\lambda_\mC$ on the fiber over $\X \in \mC$ identifies with $\beta$, where we use the equivalence $\delta_\mA(\X) \circ \R \simeq \delta_\mB(\X).$ This proves that $\lambda $ is an equivalence.
	
\vspace{2mm}

Let $\phi:\mC \to \mD$ be a functor and $\phi'$ the functor
$\mC \times_{\Fun(\mB,\mC)} \Fun(\mB,\mC)_{/\rho} \to \mD \times_{\Fun(\mB,\mD)} \Fun(\mB,\mD)_{/\phi \circ \rho}.$
By definition a $\phi$-final object is an object $\X \in \mC$ such that for any 
$\Y \in \mC$ the induced map $\mC(\Y,\X) \to \mD(\phi(\Y),\phi(\X))$ is an equivalence. A $\phi$-limit of $\rho$ is a $\phi'$-final object.
	
Since $\lambda_\mC, \lambda_\mD$ are equivalences, $\rho$ admits a $\phi$-limit if and only if $\rho \circ \theta$ admits a $\phi$-limit.
In this case the $\phi$-limit of $\rho$ and of $\rho \circ \theta$ correspond under restriction along $\theta.$
	
If $\mB$ has an initial object $\emptyset$, the embedding $\mA \subset \mB$ of the contractible full subcategory $\mA$ spanned by the initial objects admits a right adjoint. So $\rho(\emptyset)$ is a $\phi$-limit of $\rho$.
	
\end{remark}

\begin{proof}[Proof of Lemma \ref{ley}]
	
It is enough to prove that the following conditions hold:
\begin{enumerate}
\item The functor $\theta'$ is conservative.
		
\item The functor $\zeta'$ is conservative.
		
\item The map $\theta$ has a fully faithful right adjoint
taking $\Fun^{\inert}_{\mO'}(\mO^\el_{\X/} \times_\mO \mC, \mD)$ to $\Fun^{\inert}_{\mO'}(\mO^\inert_{\X/}\times_\mO \mC, \mD)$.
		
\item The functor $\zeta$ has a fully faithful right adjoint 
taking $\Fun'_{\mO'}(\mO^\el_\alpha \times_\mO \mC,\mD)$ to $\Fun'_{\mO'}(\mO^\inert_\alpha \times_\mO \mC,\mD)$.
\end{enumerate}
	
We start with proving (1). Let $\tau: \F \to \G$ be a map in $\Fun^\inert_{\mO'}(\mO^\inert_{\X/}\times_\mO \mC, \mD)$ such that for any 
$\K \in \mO^\el_{\X/}\times_\mO \mC$ the component $\tau_\K: \F(\K)\to \G(\K)$ is an equivalence.
We wish to see that for any $\T \in \mO^\inert_{\X/}\times_\mO \mC$ the component $\tau_\T: \F(\T)\to \G(\T)$ is an equivalence.
Let $\W \in \mO'$ be the image of $\T$ in $\mO'.$
	
The component $\tau_\T$ is a morphism in $\mD_\W \simeq \lim_{\W'\in \mO'^\el_{\W/}} \mD_{\W'}$ and so is an equivalence if and only if for any inert morphism $\W \to \W'$ in $\mO'$, where $\W' \in \mO'^\el$, the image of $\tau_\T$ under the induced functor $\mD_\W \to \mD_{\W'}$ is an equivalence.
Let $\T \to \T'$ be an inert lift in $\mO^\inert_{\X/}\times_\mO \mC$ of an inert map $\W \to \W'$ in $\mO'$, where $\W' \in \mO'^\el,$ 
and consider the commutative square:
\begin{equation*}
\begin{xy}
\xymatrix{
\F(\T) \ar[d] \ar[rr]^{\tau_\T}
&&\G(\T)  \ar[d]
\\ \F(\T')
\ar[rr]^{\tau_{\T'}} && \G(\T')
}
\end{xy} 
\end{equation*} 
By assumption the vertical morphisms of the square are inert.
Thus $\tau_{\T'}$ is the image of $\tau_\T$ under the induced functor $\mD_\W \to \mD_{\W'}$.
Since $\T'$ lies over $\W' \in \mO'^\el$, we find that $\T' \in \mO^\el_{\X/}\times_\mO \mC$ so that $\tau_{\T'}$ is an equivalence by assumption.
This proves (1).
(2) follows from (1) because the commutative square
\begin{equation}\label{qco}
\begin{xy}
\xymatrix{
\Fun_{\mO'}(\mO^\inert_\alpha \times_\mO \mC,\mD) \ar[d] \ar[rr]^{\zeta}
&& \Fun_{\mO'}(\mO^\el_\alpha \times_\mO \mC,\mD)  \ar[d]
\\  \Fun_{\mO'}(\mO^\inert_{\X/}\times_\mO \mC, \mD) \times \Fun_{\mO'}(\mO^\inert_{\Y/}\times_\mO \mC, \mD) 
\ar[rr]^{\id \times \theta} && \Fun_{\mO'}(\mO^\inert_{\X/} \times_\mO \mC, \mD) \times \Fun_{\mO'}(\mO^\el_{\Y/} \times_\mO \mC, \mD)
}
\end{xy} 
\end{equation} 
restricts to the following commutative square, where both vertical functors are conservative:
\begin{equation*}
\begin{xy}
\xymatrix{
\Fun'_{\mO'}(\mO^\inert_\alpha \times_\mO \mC,\mD) \ar[d] \ar[rr]^{\zeta'}
&& \Fun'_{\mO'}(\mO^\el_\alpha \times_\mO \mC,\mD)  \ar[d]
\\  \Fun'_{\mO'}(\mO^\inert_{\X/}\times_\mO \mC, \mD) \times \Fun'_{\mO'}(\mO^\inert_{\Y/}\times_\mO \mC, \mD) 
\ar[rr]^{\id \times \theta'} && \Fun'_{\mO'}(\mO^\inert_{\X/} \times_\mO \mC, \mD) \times  \Fun'_{\mO'}(\mO^\el_{\Y/} \times_\mO \mC, \mD).
}
\end{xy} 
\end{equation*} 
	
We continue with (3) and (4).
By \cite[Lemma 4.3.2.13., Proposition 4.3.2.17.]{lurie.HTT} the functor
$\theta$ admits a fully faithful right adjoint $\mF$ if for every
$ \T \in \mO^\inert_{\X/} \times_\mO \mC$ and map $\phi: \mO^\el_{\X/} \times_\mO \mC \to \mD$ of cocartesian fibrations relative to $\mO'^\inert$ the functor
$(\mO^\el_{\X/} \times_\mO \mC) \times_{(\mO^\inert_{\X/} \times_\mO \mC)} (\mO^\inert_{\X/} \times_\mO \mC)_{\T/} \to \mO^\el_{\X/} \times_\mO \mC \xrightarrow{\phi} \mD$ has a $\gamma$-limit. Moreover $\mF(\phi)(\T) $ is this $\gamma$-limit.
	
For the same reason the functor
$\zeta$ admits a fully faithful right adjoint $\mQ$ if for every
$ \T \in \mO^\inert_{\alpha} \times_\mO \mC$ and map $\phi: \mO^\el_{\alpha} \times_\mO \mC \to \mD$ of cocartesian fibrations relative to $\mO'^\inert$
the functor $$(\mO^\el_{\alpha} \times_\mO \mC) \times_{(\mO^\inert_{\alpha} \times_\mO \mC)} (\mO^\inert_{\alpha} \times_\mO \mC)_{\T/} \to \mO^\el_{\alpha} \times_\mO \mC \xrightarrow{\phi} \mD$$ admits a $\gamma$-limit and $\mQ(\phi)(\T) $ is this $\gamma$-limit.
	
The embedding $\mO^\el_{\alpha} \times_\mO \mC \subset  \mO^\inert_{\alpha} \times_\mO \mC$ over $[1]$ induces an equivalence on the fiber over $0 \in [1].$
So if $\T$ belongs to $\mO^\inert_{\X/} \times_\mO \mC \subset  \mO^\inert_{\alpha} \times_\mO \mC$, then $\T$ belongs to $\mO^\inert_{\X/} \times_\mO \mC \subset \mO^\el_{\alpha} \times_\mO \mC$
so that $$(\mO^\el_{\alpha} \times_\mO \mC) \times_{(\mO^\inert_{\alpha} \times_\mO \mC)} (\mO^\inert_{\alpha} \times_\mO \mC)_{\T/} \simeq (\mO^\el_{\alpha} \times_\mO \mC)_{\T/}.$$
Since the latter $\infty$-category has an initial object,
by Remark \ref{cofinal} the $\gamma$-limit $\mQ(\phi)(\T) $ exists and is computed as $\phi(\T).$
	
If $\T$ belongs to $\mO^\inert_{\Y/} \times_\mO \mC \subset \mO^\inert_{\alpha} \times_\mO \mC$, every object of $\mO^\inert_{\alpha} \times_\mO \mC$
under $\T$ also belongs to $\mO^\inert_{\Y/} \times_\mO \mC.$
Hence the induced embedding
$(\mO^\el_{\Y/} \times_\mO \mC) \times_{(\mO^\inert_{\Y/} \times_\mO \mC)} (\mO^\inert_{\Y/} \times_\mO \mC)_{\T/}  \subset (\mO^\el_{\alpha} \times_\mO \mC) \times_{(\mO^\inert_{\alpha} \times_\mO \mC)} (\mO^\inert_{\alpha} \times_\mO \mC)_{\T/}$ is an equivalence. 
Consequently, 
the commutative square (\ref{qco}) induces a commutative square
\begin{equation}\label{qcoo}
\begin{xy}
\xymatrix{
\Fun_{\mO'}(\mO^\el_\alpha \times_\mO \mC,\mD) \ar[d] \ar[rr]^{\mQ}
&& \Fun_{\mO'}(\mO^\inert_\alpha \times_\mO \mC,\mD) \ar[d]
\\  \Fun_{\mO'}(\mO^\inert_{\X/} \times_\mO \mC, \mD) \times \Fun_{\mO'}(\mO^\el_{\Y/} \times_\mO \mC, \mD)
\ar[rr]^{\id \times \mF} && \Fun_{\mO'}(\mO^\inert_{\X/}\times_\mO \mC, \mD) \times \Fun_{\mO'}(\mO^\inert_{\Y/}\times_\mO \mC, \mD)
}
\end{xy} 
\end{equation} 
such that the unit and counit of the adjunction $\mQ \dashv \zeta$ cover the 
unit and counit of the adjunction $\id \times \mF \dashv \id \times \theta.$ 
So (4) follows from (3), where we use that the vertical functors are conservative.
	
It remains to prove (3).
Let $\W \in \mO', \Z \in \mC, \Y \in \mO$ be the images of $\T.$
There is a canonical equivalence $$ (\mO^\el_{\X/} \times_\mO \mC) \times_{(\mO^\inert_{\X/} \times_\mO \mC)} (\mO^\inert_{\X/} \times_\mO \mC)_{\T/} \simeq (\mO^\el_{\X/} \times_\mO \mC) \times_{(\mO^\inert_{\X/} \times_\mO \mC)} (\mO^\inert_{\Y/} \times_{\mO_{\Y/}}\mC_{\Z/}) \simeq \mO^\el_{\Y/} \times_{\mO_{\Y/}}\mC_{\Z/}.$$
	
By Remark \ref{colo} the embedding $\mC^\inert \subset \mO^\inert \times_{\Fun([1],\mO)} \Fun([1],\mC)$ admits a right adjoint relative to $\Fun(\{0\},\mC) \times \mO^\inert.$ 
Pulling back to $\mO^\el$ the embedding $\mC^\el \subset \mO^\el \times_{\Fun([1],\mO)} \Fun([1],\mC)$ admits a right adjoint relative to $\Fun(\{0\},\mC) \times \mO^\el.$
Taking the fiber over $\Z \in \mC$ the embedding $ \mC^\el_{\Z/} \subset \mO^\el_{\Y/} \times_{\mO_{\Y/}}\mC_{\Z/} $ admits a right adjoint relative to $\mO^\el_{\Y/}.$
So by Remark \ref{cofinal} the functor
$\theta$ admits a fully faithful right adjoint if for every
$ \T \in \mO^\inert_{\X/} \times_\mO \mC$ and $\phi: \mO^\el_{\X/} \times_\mO \mC \to \mD$ the functor
$\mC^\el_{\Z/} \to \mO^\el_{\Y/} \times_\mO \mC\to \mO^\el_{\X/} \times_\mO \mC \xrightarrow{\phi} \mD$ admits a $\gamma$-limit and $\mF(\phi)(\T) $ is this $\gamma$-limit.
	
The functor $\mC_{\Z/} \to \mO'_{\W/}$ restricts to an equivalence
$\mC^\el_{\Z/} \to \mO'^\el_{\W/}$ since $\mC \to \mO'$ is a cocartesian fibration relative to $\mO'^\inert.$ The functor $\mC^\el_{\Z/} \to \mO^\el_{\Y/} \times_\mO \mC\to \mO^\el_{\X/} \times_\mO \mC \xrightarrow{\phi} \mD \to \mO'$
factors as $\mC^\el_{\Z/} \simeq \mO'^\el_{\W/} \to \mO'$ and so induces a map
$\Psi: \mC^\el_{\Z/} \to \mO'^\el_{\W/} \times_{\mO'} \mD$ of cocartesian fibrations over $\mO'^\el_{\W/}.$
The composition $\Psi': \mO'^\el_{\W/} \simeq \mC^\el_{\Z/} \xrightarrow{\Psi} \mO'^\el_{\W/} \times_{\mO'} \mD$ is an object of $\Fun^\cocart_{\mO'^\el_{\W/}}(\mO'^\el_{\W/}, \mO'^\el_{\W/} \times_{\mO'} \mD) \simeq \lim_{(\W \to \W') \in \mO'^\el_{\W/}} \mD_{\mW'}.$
Since $\mD \to \mO'$ is $\mathfrak{P}$-fibrous, the canonical functor
$\mD_\mW \to \lim_{(\W \to \W') \in \mO'^\el_{\W/}} \mD_{\mW'} $ is an equivalence so that $\Psi'$ corresponds to an object $\L \in \mD_\mW$. 
The object $\L $ is sent by the functor
$\mD_\mW \to \lim_{(\W \to \W') \in \mO'^\el_{\W/}} \mD_{\mW'} $
to the canonical map $ \mO'^\el_{\W/} \simeq \mD^\el_{\L/} \to \mO'^\el_{\W/} \times_{\mO'} \mD$ of cocartesian fibrations over $\mO'^\el_{\W/}$,
which therefore identifies with $\Psi'.$
Since $\mD \to \mO'$ is $\mathfrak{P}$-fibrous, the object
$\L \in \mD_\mW$ is the $\gamma$-limit of the functor $ \mO'^\el_{\W/} \simeq \mD^\el_{\L/} \to \mD$ and so the $\gamma$-limit of the functor 
$(\mO'^\el_{\W/} \simeq) \ \mC^\el_{\Z/} \to \mO^\el_{\Y/} \times_\mO \mC \to \mO^\el_{\X/} \times_\mO \mC \xrightarrow{\phi} \mD$.

For any inert morphism $\T \to \T'$ in $\mO^\inert_{\X/} \times_\mO \mC $, where $\T' \in \mO^\el_{\X/} \times_\mO \mC$, let $\W \to \W'$ be the image in $\mO'$
and $\Z \to \Z'$ the cocartesian lift in $\mC$ of the latter morphism.
The object $\Z \to \Z'$ of $\mC^\el_{\Z/}$ gives an inert morphism
$\L \to \phi(\T').$ 
The latter morphism is the image under $\mF(\phi)$ of the morphism
$\T \to \T'$ in $\mO^\inert_{\X/} \times_\mO \mC$.
This implies that $\mF(\phi)$ is a map of cocartesian fibrations relative to $\mO^\inert$. This proves (4).
		
\end{proof}
 
\begin{proof} For any morphism $\alpha: \X \to \Y$ in $\mO$ consider the commutative square
\begin{equation*} 
\begin{xy}
\xymatrix{
\Fun_{[1]}([1], \Fun^{\mathfrak{P}}_{\mathfrak{P'}}(\mC,\mD)_\alpha) \ar[d]^{} \ar[r]^{ }
& \lim_{(\Z,\sigma)\in \mO^\el_{\Y/}}\Fun_{[1]}([1], \Fun^{\mathfrak{P}}_{\mathfrak{P'}}(\mC,\mD)_{\sigma \circ \alpha})\ar[d]^{} 
\\ \Fun^{\mathfrak{P}}_{\mathfrak{P'}}(\mC,\mD)_\Y 
\ar[r]^{}  & \lim_{(\Z,\sigma)\in \mO^\el_{\Y/}}\Fun^{\mathfrak{P}}_{\mathfrak{P'}}(\mC,\mD)_{\Z}
}
\end{xy} 
\end{equation*}
over $ \Fun^{\mathfrak{P}}_{\mathfrak{P'}}(\mC,\mD)_\X$, which induces on the fiber over $\X' \in  \Fun^{\mathfrak{P}}_{\mathfrak{P'}}(\mC,\mD)_\X, \Y' \in \Fun^{\mathfrak{P}}_{\mathfrak{P'}}(\mC,\mD)_\Y$
the map $$\{\alpha\} \times_{\mO(\X,\Y)} \Fun^{\mathfrak{P}}_{\mathfrak{P'}}(\mC,\mD)(\X',\Y') \to \lim_{(\Z,\sigma)\in \mO^\el_{\Y/}} \{\sigma \circ \alpha\}\times_{\mO(\X,\Z)}\Fun^{\mathfrak{P}}_{\mathfrak{P'}}(\mC,\mD)(\X',\sigma_*(\Y')).$$
	
By definition of a $\mathfrak{P}$-fibered object it is enough to see that the horizontal functors in the latter square are equivalences.
By Remark \ref{forms} there are canonical equivalences
$$\Fun^{\mathfrak{P}}_{\mathfrak{P'}}(\mC,\mD)_\Y \simeq \Fun^\inert_{\mO'}(\mO^\inert_{\Y/}\times_\mO \mC, \mD),\hspace{5mm}
\Fun_{[1]}([1], \Fun^{\mathfrak{P}}_{\mathfrak{P'}}(\mC,\mD)_\alpha)  \simeq \Fun'_{\mO'}(\mO^\inert_\alpha \times_\mO \mC, \mD).$$ So we want to see that the following functors are equivalences:
\begin{equation}\label{ghvx}
\Fun^{\inert}_{\mO'}(\mO^\inert_{\Y/}\times_\mO \mC, \mD) \to 
\lim_{(\Z,\sigma) \in\mO^\el_{\Y/}}\Fun^{\inert}_{\mO'}(\mO^\inert_{\Z/} \times_\mO \mC,\mD),  \end{equation}
\begin{equation}\label{ghvxa}
\Fun'_{\mO'}(\mO^\inert_\alpha \times_\mO \mC,\mD) \to \lim_{(\Z,\sigma) \in \mO^\el_{\Y/}} \Fun'_{\mO'}(\mO^\inert_{\sigma \circ \alpha} \times_\mO \mC, \mD). \end{equation}
	
Let $\mO^\el_\alpha \subset \mO^\inert_\alpha $ be the full subcategory spanned by $\mO^\inert_{\X/}, \mO^\el_{\Y/}.$
By Lemma \ref{ley} the functors $$ \Fun^{\inert}_{\mO'}(\mO^\inert_{\Y/}\times_\mO \mC, \mD) \to \Fun^{\inert}_{\mO'}(\mO^\el_{\Y/} \times_\mO \mC, \mD), \ \Fun'_{\mO'}(\mO^\inert_\alpha \times_\mO \mC,\mD)\to \Fun'_{\mO'}(\mO^\el_\alpha \times_\mO \mC,\mD)$$
restricting along the respective embeddings $\mO^\el_{\Y/}\times_\mO \mC \subset \mO^\inert_{\Y/}\times_\mO \mC$, $\mO^\el_\alpha \times_\mO \mC \subset \mO^\inert_\alpha \times_\mO \mC$ are equivalences.

Thus the functors (\ref{ghvx}) and (\ref{ghvxa}) are equivalences if and only if the functors
$$\Fun^\inert_{\mO'}(\mO^\el_{\Y/}\times_\mO \mC, \mD) \to 
\lim_{(\Z,\sigma) \in\mO^\el_{\Y/}}\Fun^\inert_{\mO'}(\mO^\el_{\Z/} \times_\mO \mC,\mD), \ \Fun'_{\mO'}(\mO^\el_\alpha \times_\mO \mC,\mD) \to \lim_{(\Z,\sigma) \in \mO^\el_{\Y/}} \Fun'_{\mO'}(\mO^\el_{\sigma \circ \alpha} \times_\mO \mC, \mD)$$ are equivalences.
	  
The latter functors are the pullback of the respective functors
\begin{equation}\label{ejjj}
\Fun_{\mO'}(\mO^\el_{\Y/}\times_\mO \mC, \mD) \to 
\lim_{(\Z,\sigma) \in\mO^\el_{\Y/}}\Fun_{\mO'}(\mO^\el_{\Z/} \times_\mO \mC,\mD),
\end{equation}
\begin{equation}\label{ejjjo}
\Fun_{\mO'}(\mO^\el_\alpha \times_\mO \mC,\mD) \to \lim_{(\Z,\sigma) \in \mO^\el_{\Y/}} \Fun_{\mO'}(\mO^\el_{\sigma \circ \alpha} \times_\mO \mC, \mD)
\end{equation}
because every morphism of $\mO^\el_{\Y/}$ is in the image
of a functor $\mO^\el_{\Z/} \to\mO^\el_{\Y/}$ induced by an inert morphism
$\Y \to \Z, $ where $\Z\in \mO^\el$, and every morphism of $\mO^\el_{\alpha}$ is in the image of a functor $\mO^\el_{\sigma \circ \alpha} \to\mO^\el_{\alpha}$ induced by an inert morphism $\sigma: \Y \to \Z, $ where $\Z\in \mO^\el$.
So it will suffice to see that the functors (\ref{ejjj}) and (\ref{ejjjo}) are equivalences.
Functor $(\ref{ejjj})$ identifies with the functor
$$\Fun_{\mO'}(\mO^\el_{\Y/}\times_\mO \mC, \mD) \to \Fun_{\mO'}(\colim_{(\Z,\sigma) \in\mO^\el_{\Y/}}\mO^\el_{\Z/} \times_\mO \mC,\mD) \simeq \lim_{(\Z,\sigma) \in\mO^\el_{\Y/}}\Fun_{\mO'}(\mO^\el_{\Z/} \times_\mO \mC,\mD) $$
induced by the functor
$\gamma: \colim_{(\Z,\sigma) \in\mO^\el_{\Y/}}(\mO^\el_{\Z/} \times_\mO \mC) \to (\colim_{(\Z,\sigma) \in\mO^\el_{\Y/}}\mO^\el_{\Z/}) \times_\mO \mC \to \mO^\el_{\Y/}\times_\mO \mC$ over $\mO'$.
The first functor in the definition of
$\gamma$ is an equivalence as $\mC \to \mO$ is flat, the second functor in $\gamma$ is induced by the functor $\kappa:\mQ:=\colim_{(\Z,\sigma) \in \mO^\el_{\Y/}} \mO^\el_{\Z/} \to \mO^\el_{\Y/}$, which is an equivalence by Lemma \ref{limo}.
	
Functor $(\ref{ejjjo})$ identifies with the functor
$$\Fun_{\mO'}(\mO^\el_\alpha \times_\mO \mC,\mD) \to \Fun_{\mO'}(\colim_{(\Z,\sigma) \in \mO^\el_{\Y/}}\mO^\el_{\sigma \circ \alpha} \times_\mO \mC, \mD) \simeq \lim_{(\Z,\sigma) \in \mO^\el_{\Y/}} \Fun_{\mO'}(\mO^\el_{\sigma \circ \alpha} \times_\mO \mC, \mD) $$
induced by the functor
$\tau: \colim_{(\Z,\sigma) \in \mO^\el_{\Y/}}(\mO^\el_{\sigma \circ \alpha} \times_\mO \mC) \to (\colim_{(\Z,\sigma) \in \mO^\el_{\Y/}}\mO^\el_{\sigma \circ \alpha}) \times_\mO \mC \to \mO^\el_\alpha \times_\mO \mC$ over $\mO'$.
Since $\mC \to \mO$ is flat, the first functor in the definition of
$\tau$ is an equivalence.
We will complete the proof by showing that the canonical functor
$\lambda: \colim_{(\Z,\sigma) \in \mO^\el_{\Y/}} \mO^\el_{\sigma \circ \alpha} \to \mO^\el_\alpha $ is an equivalence.

For any cartesian fibration $\mM \to [1]$ the canonical functor
$\mM_0 \coprod_{(\mM_1 \times \{1\})} (\mM_1 \times [1]) \to \mM $ is an equivalence.
We apply this to the cartesian fibration $\mO^\el_\beta \to [1]$ (Lemma \ref{exx}) for any morphism
$\beta:\X \to \Z$ in $ \mO$ to obtain an equivalence
$ \mO^\el_{\beta} \simeq\mO^\el_{\X/} \coprod_{(\mO^\el_{\Z/} \times \{1\})} (\mO^\el_{\Z/} \times [1]).$

The functor $ \colim_{(\Z,\sigma) \in \mO^\el_{\Y/}}\mO^\el_{\X/} \to \mO^\el_{\X /} $
is an equivalence as $ \mO^\el_{\Y/}$ is weakly contractible (Lemma \ref{limo}).
So $\lambda$ identifies with the functor induced by the equivalence $\kappa:\mQ:=\colim_{(\Z,\sigma) \in \mO^\el_{\Y/}} \mO^\el_{\Z/} \to \mO^\el_{\Y/}:$
$$ \colim_{(\Z,\sigma) \in \mO^\el_{\Y/}} (\mO^\el_{\X/} \coprod_{(\mO^\el_{\Z/} \times \{1\})} (\mO^\el_{\Z/} \times [1]))\simeq \mO^\el_{\X/} \coprod_{\mQ \times \{1\}} (\mQ \times [1])  \to\mO^\el_{\X/} \coprod_{(\mO^\el_{\Y/} \times \{1\})} (\mO^\el_{\Y/} \times [1]).
$$	 
	 
\end{proof}

For the next proposition we use the following remark
analogous to \cite[Remark 6.2.4.5.]{lurie.higheralgebra}:

\begin{remark}\label{repr}
Let $\mO \to \Ass$ be a cocartesian fibration relative to the collection of inert morphisms.
	
An $\mO$-operad $\phi: \mC \to \mO$ is corepresentable if and only if
for any active morphism $\alpha: \Y \to \X$ in $\mO$ with $\X$ lying over $[1]\in \Ass$ the induced pullback $\phi': [1] \times_\mO \mC \to [1]$ is a cocartesian fibration.
\end{remark}

\begin{proposition}\label{proooo}
Let $\mO \to \Ass$ be a cocartesian fibration relative to the collection of inert morphisms, $\mC \to \mO$ a cocartesian fibration whose fibers are small, and $\mD \to \mO$ a corepresentable $\mO$-operad whose fibers admit small colimits.
	
\begin{enumerate}
\item The $\mO$-operad $\Fun^\mO(\mC, \mD) \to \mO$ is corepresentable.
		
\vspace{1mm}
		
\item If $\mD \to \mO$ is compatible with small colimits, $\Fun^\mO(\mC, \mD) \to \mO$ is compatible with small colimits.
		
\vspace{1mm}
		
\item For any $\mO$-monoidal functor $\gamma: \mD \to \mD'$ of corepresentable $\mO$-operads inducing on every fiber a small colimits preserving functor between $\infty$-categories with small colimits the functor $\gamma_\ast: \Fun^\mO(\mC, \mD) \to \Fun^\mO(\mC, \mD')$ is $\mO$-monoidal.
\vspace{1mm}
		
\item If $\mD \to \mO$ is compatible with small colimits, for any $\mO$-monoidal functor $\theta: \mC \to \mC'$ of $\mO$-monoidal $\infty$-categories whose fibers are small, the induced map $ \Fun^\mO(\mC', \mD) \to \Fun^\mO(\mC, \mD)$ of $\mO$-operads admits a left adjoint relative to $\mO.$
		
\end{enumerate}
\end{proposition}
\begin{proof}
	
(1): By Theorem \ref{proop} $\Fun^\mO(\mC, \mD) \to \mO$ is an $\mO$-operad. So by Remark \ref{repr} it is enough to show that $\Fun^\mO(\mC, \mD) \to \mO$ is a locally cocartesian fibration relative to the collection of active morphisms $\alpha: \X \to \Y$ whose target lies over $[1]\in \Ass$.
By Lemma \ref{propo} for any $\F \in \Fun^\mO(\mC, \mD)_\X $ lying over $\F' \in \Fun(\mC_\X, \mD_\X)$
and $\G \in \Fun^\mO(\mC, \mD)_\Y \simeq \Fun(\mC_\Y, \mD_\Y)$ there is a canonical equivalence $$ \{\alpha\} \times_{\mO(\X,\Y)} \Fun^\mO(\mC, \mD)(\F, \G) \simeq \Fun(\mC_\X, \mD_\Y)(\alpha^\mD_\ast \circ \F',\G \circ \alpha^\mC_\ast) \simeq \Fun(\mC_\Y, \mD_\Y)( \lan_{\alpha^\mC_\ast}(\alpha^\mD_\ast \circ \F'),\G),$$
where $\lan_{\alpha^\mC_\ast}: \Fun(\mC_{\X}, \mD_\Y) \to \Fun(\mC_\Y, \mD_\Y) $ is the left adjoint to pre-composition along $\alpha^\mC_\ast.$
	
\vspace{1mm}
	
Thus a morphism $\F \to \G$ in $\Fun^\mO(\mC, \mD)$ is cocartesian over $\alpha:\X \to \Y $ if and only if the corresponding map $\alpha^\mD_\ast \circ \F' \to \G \circ \alpha^\mC_\ast$ in $\Fun(\mC_\X, \mD_\Y)$ is adjoint to an equivalence
$\lan_{\alpha^\mC_\ast}(\alpha^\mD_\ast \circ \F') \simeq \G$ in $ \Fun(\mC_\Y, \mD_\Y).$
	
So $\alpha : \X \to \Y $ yields the composition 
$ \alpha_\ast:\Fun^\mO(\mC, \mD)_\X  \simeq \prod_{\bi=1}^\m \Fun^\mO(\mC, \mD)_{\X_\bi} \simeq \prod_{\bi=1}^\m \Fun(\mC_{\X_\bi}, \mD_{\X_\bi}) \to \Fun(\prod_{\bi=1}^\m \mC_{\X_\bi}, \prod_{\bi=1}^\m \mD_{\X_\bi}) \to \Fun(\mC_{\X}, \prod_{\bi=1}^\m \mD_{\X_\bi}) \xrightarrow{\Fun(\mC_\X, \alpha^\mD_\ast)} \Fun(\mC_{\X}, \mD_\Y) \xrightarrow{\lan_{\alpha^\mC_\ast}} \Fun(\mC_\Y, \mD_\Y).$
Hence $\alpha_\ast$ preserves small colimits component-wise if $\alpha^\mD_\ast: \prod_{\bi=1}^\m \mD_{\X_\bi} \to \mD_\Y$ does so.
This shows (2).
	
(3): The induced map $\gamma_\ast: \Fun^\mO(\mC, \mD) \to \Fun^\mO(\mC, \mD')$ of $\mO$-operads sends the morphism $\F \to \G$ in $\Fun^\mO(\mC, \mD)$ cocartesian over $\alpha:\X \to \Y $ to a morphism $\gamma_\ast(\F) \to \gamma_\ast(\G)$ in $\Fun^\mO(\mC, \mD')$ lying over $\alpha$ corresponding to the morphism $\alpha^{\mD'}_\ast \circ \gamma_\ast(\F') \to \gamma_\ast(\G) \circ \alpha^\mC_\ast$ in $\Fun(\mC_\X, \mD'_\Y)$ adjoint to the morphism
$$\lan_{\alpha^\mC_\ast}(\alpha^{\mD'}_\ast \circ \gamma_\ast(\F')) \simeq \lan_{\alpha^\mC_\ast}(\gamma_\ast(\alpha^{\mD}_\ast \circ \F')) \simeq
\gamma_\ast(\lan_{\alpha^\mC_\ast}(\alpha^{\mD}_\ast \circ \F')) \to \gamma_\ast(\G) $$ in $ \Fun(\mC_\Y, \mD'_\Y),$
where the first map is an equivalence since $\gamma$ is an $\mO$-monoidal functor
and the second map is an equivalence since $\gamma$ preserves fiberwise small colimits and so left Kan extensions.
	
\vspace{1mm}
	
(4): Let $\X \in \mO$.
First assume that $\X $ lies over $[1] \in \Ass$. Since $\mC_\X, \mC'_\X$ are small and $\mD_\X$ admits small colimits, the canonical functor $\Fun(\mC'_\X, \mD_\X) \to \Fun(\mC_\X, \mD_\X)$ admits a left adjoint taking left Kan extension along $\theta_\X: \mC_\X \to \mC'_\X$.
For general $\X \in \mO$ we choose for any $1 \leq \bi \leq \n$ an inert lift $\X \to \X_\bi$ in $\mO$ of the inert map $[\n] \to [1]$ of $\Ass$ and find that the  induced functor $ \Fun(\mC', \mD)_\X \simeq  \prod_{\bi=1}^\n\Fun(\mC'_{\X_\bi}, \mD_{\X_\bi}) \to \Fun(\mC, \mD)_\X \simeq  \prod_{\bi=1}^\n\Fun(\mC_{\X_\bi}, \mD_{\X_\bi}) $ admits a left adjoint $\tau_\X.$
	
By (1) the functors $ \Fun^\mO(\mC', \mD) \to \mO, \Fun^\mO(\mC, \mD) \to \mO$ are locally cocartesian fibrations. So to show (4) by \cite[Proposition 7.3.2.11.]{lurie.higheralgebra} it is enough to check that
for any morphism $\alpha: \X \to \Y$ in $\mO$ the canonical natural transformation $\lambda_\alpha:\tau_\Y \circ \alpha_* \to \alpha_* \circ \tau_\X $ is an equivalence.
	
By the operad axioms the transformation $\lambda_\alpha$
is an equivalence if for all inert morphisms $\beta:\Y \to \Z$ in $\mO$,
where $\Z \in \mO_{[1]}$, the transformation $\beta_* \circ \lambda_\alpha $ is an equivalence, where $\beta_*:  \Fun^\mO(\mC', \mD)_\Y \to \Fun^\mO(\mC', \mD)_\Z.$
The transformation $\beta_* \circ \lambda_\alpha $ identifies with $\lambda_{\beta\circ\alpha}$
since $\beta_*$ is a projection and so preserves left Kan extensions and $(\beta \circ \alpha)_* \simeq \beta_* \circ \alpha_*$. So it is enough to show that $\lambda_\alpha$ is an equivalence
when $\Y \in \mO_{[1]}$.
Moreover $\alpha:\X \to \Y$ factors as an inert map $\beta: \X \to \Z$ followed by an active map $\gamma: \Z \to \Y$.
The transformation $\lambda_\alpha$ identifies with $\lambda_\gamma \circ \beta_*$
since $\alpha_* \simeq \gamma_* \circ \beta_*$ and $\beta_*$ is a projection and so preserves left Kan extensions.
So it is enough to show that $\lambda_\alpha$ is an equivalence for $\alpha$ active and $\Y\in \mO_{[1]}$.
	
For $1 \leq \bi \leq \n$ let $\F_\bi:\mC_{\X_{\bi}} \to \mD_{\X_{\bi}} $
be a functor. There is a canonical equivalence
$$\alpha_*(\lan_{\theta_{\X_1}}(\F_1),...,\lan_{\theta_{\X_\n}}(\F_\n)) = \lan_{\alpha_*^{\mC'}}(\alpha^\mD_*\circ (\lan_{\theta_{\X_1}}(\F_1)\times...\times \lan_{\theta_{\X_\n}}(\F_\n)))
\simeq $$$$\lan_{\alpha_*^{\mC'}}(\lan_{\theta_\X}(\alpha^\mD_* \circ (\F_1 \times ...\times \F_\n))) 
\simeq \lan_{\theta_\Y}(\lan_{\alpha_*^{\mC}}(\alpha^\mD_*\circ (\F_1 \times ...\times \F_\n)))= \lan_{\theta_\Y}(\alpha_*(\F_1,...,\F_\n)),$$
where the non-trivial equivalence in the second line holds since
$\theta$ is an $\mO$-monoidal functor, and the non-trivial equivalence in the first line is induced by the equivalence
$\alpha^\mD_*\circ (\lan_{\theta_{\X_1}}(\F_1)\times...\times \lan_{\theta_{\X_\n}}(\F_\n))
\simeq \lan_{\theta_\X}(\alpha^\mD_* \circ (\F_1 \times ...\times \F_\n)) $
coming from the fact that $\alpha^\mD_*$ preserves small colimits component-wise.
\end{proof}

To prove Proposition \ref{proooo} we used the next lemma, an adaption of \cite[Proposition 2.2.6.6.]{lurie.higheralgebra}: 

\begin{lemma}\label{propo}
Let $\mO \to \Ass$ be a cocartesian fibration relative to the collection of inert morphisms, $\mC \to \mO$ a flat functor that is a cocartesian fibration relative to the collection of inert morphisms of $\mO$ and $\mD \to \mO$ an $\mO$-operad.

For any active morphism $\alpha: \X \to \Y$ in $\mO$ with $\Y$ lying over $[1]\in \Ass$ and any $\F,\G \in \Fun^\mO(\mC, \mD)$ lying over $\X,\Y \in \mO$ with images $\F' \in \Fun(\mC_\X, \mD_\X), \G' \in \Fun(\mC_\Y, \mD_\Y)$ there is a canonical equivalence $$ \{\alpha\} \times_{\mO(\X,\Y)} \Fun^\mO(\mC, \mD)(\F, \G) \simeq \{\F',\G'\} \times_{(\Fun(\mC_\X, \mD_\X) \times \Fun(\mC_\Y, \mD_\Y))} \Fun_{[1]}(\mC_\alpha, \mD_\alpha). $$
\end{lemma}

\begin{remark}\label{spezi}
If $\mC \to \mO, \mD \to \mO$ are locally cocartesian fibrations, there is a canonical equivalence
$$ \{\F',\G'\} \times_{(\Fun(\mC_\X, \mD_\X) \times \Fun(\mC_\Y, \mD_\Y))} \Fun_{[1]}(\mC_\alpha, \mD_\alpha) \simeq \Fun(\mC_\X, \mD_\Y)( \alpha^\mD_\ast \circ \F',\G' \circ \alpha^\mC_\ast).$$ 
Hence under the assumptions of Lemma \ref{propo} there is a canonical equivalence
$$ \{\alpha\} \times_{\mO(\X,\Y)} \Fun^\mO(\mC, \mD)(\F, \G) \simeq \Fun(\mC_\X, \mD_\Y)( \alpha^\mD_\ast \circ \F',\G' \circ \alpha^\mC_\ast).$$ 

\end{remark}

\begin{proof}

Since forming mapping spaces preserves pullbacks, there is a canonical equivalence 
\begin{equation}\label{hhh}
\{\alpha\} \times_{\mO(\X,\Y)} \Fun^{\mO}(\mC,\mD)(\F, \G) \simeq \{(\F,\G)\}\times_{(\Fun^{\mO}(\mC, \mD)_\X \times  \Fun^{\mO}(\mC, \mD)_\Y)} \Fun_{\mO}([1],  \Fun^{\mO}(\mC,\mD)). 
\end{equation} 
We set $\mC^\X:= \mO^\inert_{\X/} \times_\mO \mC, \mC^\alpha:= \mO^\inert_{\alpha} \times_\mO \mC.$
By Remark \ref{forms} the right hand side of (\ref{hhh}) is equivalent to 
$$ \{(\F,\G)\} \times_{(\Fun^{\inert}_{\mO}(\mC^\X, \mD) \times \Fun^{\inert}_{\mO}(\mC^\Y, \mD))} \Fun'_{\mO}(\mC^\alpha, \mD).$$
The canonical functor $\lambda: \Fun^{\inert}_{\mO}(\mC^\Y, \mD) \simeq \Fun^{\mO}(\mC,\mD)_\Y \to \Fun(\mC_\Y,\mD_\Y)$ is an equivalence
as $\mD \to \mO$ is an $\mO$-operad and $\Y\in \mO_{[1]}.$
So it is enough to see that the following square is a pullback square:
\begin{equation*} 
\begin{xy}
\xymatrix{
\Fun'_{\mO}(\mC^\alpha, \mD) \ar[d]^{} \ar[r]^{ }
& \Fun_{[1]}(\mC_\alpha, \mD_\alpha) \ar[d]^{} 
\\ \Fun^{\inert}_{\mO}(\mC^\X, \mD)
\ar[r]^{}  & \Fun(\mC_\X, \mD_\X)
}
\end{xy} 
\end{equation*}
We want to see that the functor
$\gamma: \Fun_{\mO}(\mC^\alpha, \mD) \to \mQ:= \Fun_{\mO}(\mC^\X, \mD) \times_{\Fun(\mC_\X, \mD_\X) }\Fun_{[1]}(\mC_\alpha, \mD_\alpha)$
restricts to an equivalence
$ \Fun'_{\mO}(\mC^\alpha, \mD) \to \Fun^{\inert}_{\mO}(\mC^\X, \mD) \times_{\Fun(\mC_\X, \mD_\X) }\Fun_{[1]}(\mC_\alpha, \mD_\alpha).$
We first rewrite $\mQ.$
The diagonal embedding gives an embedding $\mC_\alpha \subset \mC^\alpha$  restricting to an embedding 
$\mC_\X \subset \mC^\X.$
There is an equivalence 
$\mQ \simeq\Fun_{\mO}(\mC^\X, \mD) \times_{\Fun_{\mO}(\mC_\X, \mD)} \Fun_{\mO}(\mC_\alpha, \mD) \simeq \Fun_{\mO}(\mC^\X \coprod_{\mC_\X} \mC_\alpha, \mD),$
under which $\gamma$ identifies with the functor
$\psi^\ast: \Fun_{\mO}(\mC^\alpha, \mD) \to \Fun_{\mO}(\mC^\X \coprod_{\mC_\X} \mC_\alpha, \mD) $
induced by the functor $\psi: \mC^\X \coprod_{\mC_\X} \mC_\alpha \to \mC^\alpha$.
So we need to see that $\psi^\ast$ restricts to an equivalence
${\psi^\ast}': \Fun'_{\mO}(\mC^\alpha, \mD) \to \Fun''_{\mO}(\mC^\X \coprod_{\mC_\X} \mC_\alpha, \mD),$
where the right hand side is the full subcategory of functors over $\mO$ whose restriction to $\mC^\X$ 
sends morphisms whose image in $\mC$ is cocartesian over $\mO$, to morphisms whose image in $\mD$ is cocartesian over $\mO$.
The functor ${\psi^\ast}'$ is conservative as $ \lambda$ is an equivalence.
So it is enough to show that ${\psi^\ast}'$ has a fully faithful right adjoint.

As $\psi$ is fully faithful, by \cite[Lemma 4.3.2.13.]{lurie.higheralgebra} this holds if for any $\rH \in \Fun_{\mO}(\mC^\X \coprod_{\mC_\X} \mC_\alpha, \mD)$ and $\T \in \mC^\alpha$ not in the essential image of $\psi$ the $\tau$-limit of the functor $$\zeta: \mC^\X \coprod_{\mC_\X} \mC_\alpha \times_{\mC^\alpha} \mC^\alpha_{\T/} \to \mC^\X \coprod_{\mC_\X} \mC_\alpha \xrightarrow{\rH}\mD$$ exists, where $\tau$ is the functor $\mD \to \mO.$
In this case the right adjoint of ${\psi^\ast}$ sends $\rH$ to its right Kan extension $\bar{\rH} $ along $\psi$ relative to $\mO$, which sends $\T$ to the $\tau$-limit of $\zeta.$ 

We prove that the $\tau$-limit exists: 
the object $\T$ corresponds to an object $\Z\in\mC$ lying over some $\W\in \mO$ equipped with an inert map $\rho: \Y \to \W$ in $\mO$ that is not an equivalence. So $\rho$ lies over an inert map $[1] \to [\n]$ in $\Ass$ that is not the identity. So $\n=0$. Since $\mD \to \mO$ is an $\mO$-operad, $\mD_{\W} $ is contractible.
Let $\T' \in \mC^\X \coprod_{\mC_\X} \mC_\alpha$
and $\sigma: \T \to \psi(\T')$ a morphism in $\mD$.
Then $\T'$ corresponds to an object $\Z'\in \mC$ lying over some $\W'\in \mO$ equipped with an equivalence $\rho': \Y \to \W'$ in $\mO$. As $\rho$ is inert, $\sigma$ is sent by evaluation at the target to an inert map in $\mO$ lying over the unique map $[0] \to [1]$ in $\Ass$ that is not inert. So
$ \mC^\X \coprod_{\mC_\X} \mC_\alpha \times_{\mC^\alpha} \mC^\alpha_{\T/}$ is empty
and the $\tau$-limit of $\zeta$ is the $\tau$-final object of $\mD$ that corresponds to the unique object of $\mD_\W$.

We complete the proof by showing that the right adjoint of $\psi^*$ restricts to a right adjoint of $\psi'^*.$ For that it is enough to see that $\bar{\rH}$
restricts to a functor $\mC^\Y \to \mD$ over $\mO$ that sends maps whose image in $\mC$ is cocartesian over $\mO$, to maps cocartesian over $\mO$.
Let $\kappa: \T \to \T'$ be a map in $\mC^\Y$
lying over a map $\alpha: \Z \to \Z'$ in $\mC$ cocartesian over $\mO$ and a map $\beta: \W \to \W' $ in $\mO^\inert_{\Y/}$ that has to be inert.
We want to see that the image $\D \to \D'$ of $\bar{\rH}(\kappa)$ in $\mD$ is cocartesian over $\mO.$ If the map $\Y \to \W'$ is an equivalence, by the proof above the map $\Y \to \W$ is an equivalence so that $\beta$ and so $\alpha$ are equivalences. In this case there is nothing to show.
So we can assume that the map $\Y \to \W'$ is not an equivalence.
In this case as we have seen above, $\W'\in \mO_{[0]}$ and $\D' \in\mD_{\W'}.$
As $\beta$ is inert, $\D \to \D'$ factors as a $\varphi$-cocartesian lift of
$\beta$ followed by a map of the contractible fiber $\mD_{\W'}$ and so is $\varphi$-cocartesian.
\end{proof}

\vspace{2mm}

We also use Day-convolution in families of (generalized) $\infty$-operads:
let $\rS$ be an $\infty$-category and $\mO \to \Ass$ a cocartesian fibration relative to the collection of inert morphisms.
The composition $\mO \times \rS \to \mO \to \Ass$ is a cocartesian fibration relative to the collection of inert morphisms,
where a morphism of $\mO \times \rS$ is inert if its image in $\mO$ is inert and its image in $\rS$ is an equivalence.
Replacing $\mO \to \Ass$ by $\mO \times \rS \to \mO \to \Ass$ Lemma \ref{pulll},
Lemma \ref{bbbn} and Theorem \ref{proop} imply the following corollary: 

\begin{corollary}\label{remaq}

Let $\rS$ be an $\infty$-category, $\mO \to \Ass$ a cocartesian fibration relative to the collection of inert morphisms, $\mC \to \mO \times \rS$ a flat functor and cocartesian fibration relative to the collection of inert morphisms and $\mD \to \mO \times \rS$ a $\rS$-family of generalized $\mO$-operads.

\begin{enumerate}
\item 
The induced functor $(-)\times_{(\mO \times \rS)} \mC: \Cat^\mathrm{inert}_{\infty/\mO\times \rS} \to \Cat^\mathrm{inert}_{\infty/\mO\times \rS}$ admits a right adjoint $\Fun^{\mO\times \rS}(\mC,-).$

\item The functor $\Fun^{\mO \times \rS}(\mC, \mD) \to \mO \times \rS$ is a $\rS$-family of generalized $\mO$-operads whose fibers are $\mO$-operads if
the fibers of $\mD \to \mO \times \rS$ are $\mO$-operads.

\item For any functor $\rS' \to \rS$ the canonical map
$\rS' \times_\rS \Fun^{\mO\times \rS}(\mC,\mD) \to \Fun^{\mO\times \rS'}(\rS' \times_\rS\mC,\rS' \times_\rS\mD) $
of $\rS'$-families of (generalized) $\mO$-operads is an equivalence.

\item Let $\T$ be an $\infty$-category. If $\mC \to \mO\times \rS $ is a map of cocartesian fibrations over $\rS$
and $\mD \to \mO\times \T $ is a map of cocartesian fibrations over $\T$, $\Fun^{\mO\times \rS \times \T}(\mC \times \T,\rS \times \mD) \to \mO\times \rS \times \T$ is map of cartesian fibrations over $\rS$ and
cocartesian fibrations over $\T.$

\end{enumerate}
\end{corollary}

\appendix
\section{Enveloping cocartesian fibrations}

In this appendix we develop some basic constructions of enveloping cocartesian fibrations.

\vspace{2mm}

\begin{notation}
For any functor $\alpha: \mC \to \rS$ we write $$\Env(\mC):= \Fun([1],\rS) \times_{\Fun(\{0\},\rS)} \mC \to \Fun(\{1\},\rS)$$ for the pullback
of $\alpha$ along evaluation at the source $\Fun([1],\rS) \to \rS$.

\end{notation}

The diagonal embedding $\rS \to \Fun([1],\rS)$ yields an embedding over
$\Fun(\{1\},\rS)$:
$$\mC \to \Env(\mC)= \Fun([1],\rS) \times_{\Fun(\{0\},\rS)} \mC $$

\begin{proposition}\label{freefib}Let $ \mC \to \rS$ be a functor and $\mD \to \rS$ a cocartesian fibration.

\begin{enumerate}
\item The functor $$\Env(\mC)\to \Fun([1],\rS)\to \Fun(\{1\},\rS)$$ is a cocartesian fibration
whose cocartesian morphisms are those that get inverted in $\mC$.

\item The induced functor
$$\Fun_\rS(\Env(\mC), \mD) \to \Fun_\rS(\mC,\mD)$$
admits a fully faithful left adjoint that lands in the full subcategory
$\Fun_\rS^\cocart(\Env(\mC), \mD)$ of maps of cocartesian fibrations
$\Env(\mC)\to \mD$ over $\rS.$

\vspace{1mm}
\item The restriction $\Fun_\rS^\cocart(\Env(\mC), \mD) \to \Fun_\rS(\mC,\mD)$ is an equivalence.
\end{enumerate}
\end{proposition}

\begin{proof}
(1): For every $\infty$-category $\rS$ evaluation at the target
$\Fun([1],\rS) \to \rS$ is a cocartesian fibration
whose cocartesian morphisms are those maps inverted by evaluation at the source. This implies (1).

Next we remark that (3) follows from (2).	
For every object $\Z \in \Env(\mC)$ there is an object $\Y \in \mC$ and a morphism
$\Y \to \Z$ in $\Env(\mC)$ cocartesian over $\rS$.
Therefore the restriction $\Fun_\rS^\cocart(\Env(\mC), \mD) \to \Fun_\rS(\mC,\mD)$ is conservative and so an equivalence if (2) is shown.	

We complete the proof by showing (2):
By \cite[Lemma 2.7.]{Heine2022} there is an embedding $\mD \subset \mE:= \mP^\rS(\mD) $ of cocartesian fibrations over
$\rS$ into a cocartesian fibration whose fibers admit small colimits and whose fiber transports preserve small colimits.
By \cite[Corollary 4.3.2.14.]{lurie.HTT} the functor $\Fun_\rS(\Env(\mC), \mE) \to \Fun_\rS(\mC,\mE)$ admits a fully faithful left adjoint $\phi$.
We will prove that $\phi$ takes values in 
$\Fun_\rS^\cocart(\Env(\mC), \mE)$.
This implies that $\phi$ restricts to a functor
$\Fun_\rS(\mC, \mD) \to \Fun^\cocart_\rS(\Env(\mC),\mD)$
since the embedding $\mD \subset \mE $ preserves cocartesian morphisms and for any $\Z \in \Env(\mC)$ there is a $\Y \in \mC$ and a morphism
$\Y \to \Z$ in $\Env(\mC)$ cocartesian over $\rS$.
So we will prove that $\phi$ lands in $\Fun_\rS^\cocart(\Env(\mC), \mE).$
For any functor $\F: \mC \to \mE$ over $\rS$ and $\Z \in \Env(\mC)$ lying over $\s \in \rS$ the functor
$\kappa: \mC \times_{\Env(\mC)} (\Env(\mC))_{/\Z} \to \mC \xrightarrow{\F} \mE$ induces a functor $\mC \times_{\Env(\mC)} (\Env(\mC))_{/\Z}  \to \rS_{/\s} \times_\rS \mE $ corresponding to a natural transformation $\f$
from the composition $\mC \times_{\Env(\mC)} (\Env(\mC))_{/\Z} \xrightarrow{\kappa} \mE \to \rS$
to the constant functor with value $\s.$
Taking a cocartesian lift we obtain a natural transformation
$\kappa \to \f_*(\kappa),$ where $\f_*(\kappa)$ is a functor
$\mC \times_{\Env(\mC)} (\Env(\mC))_{/\Z} \to \mE_\s$ whose colimit is $\phi(\F)(\Z).$
The object $\Z \in \Env(\mC)$
corresponds to an object $\Y \in \mC$ lying over $\rt \in \rS$ equipped with a morphism
$\g: \rt \to \s $ in $\rS$. The projection $ \mC \times_{\Env(\mC)} (\Env(\mC))_{/\Z} \to \mC_{/\Y} $ factors as the chain of equivalences
$$ (\mC \times_{\Env(\mC)} \Env(\mC))_{/\Z} \simeq \rS \times_{\Fun([1],\rS)} \Fun([1],\rS)_{/\g} \times_{\rS_{\rt}} \mC_{/\Y} \simeq \mC_{/\Y} $$
since $\rS \times_{\Fun([1],\rS)} \Fun([1],\rS)_{/\g} \simeq \rS_{\rt}$.
Hence $\phi(\F)(\Z)$ is the colimit of the functor $\f_*(\kappa')$, where $\kappa': \mC_{/\Y} \to \mC \xrightarrow{\F} \mE $ and $\f$ is the canonical natural transformation from the composition $\mC_{/\Y} \xrightarrow{\kappa'} \mE \to \rS$ to the constant functor with value $\s$ using $\g: \rt \to \s.$
Thus there is a cocartesian lift $\F(\Y) \to \phi(\F)(\Z)$
of $\g:\rt \to \s.$ 
For any morphism $\Z \to \Z'$ in $\Env(\mC)$, where we use similar notation for
$\Z'$, there is a commutative square 
$$\begin{xy}
\xymatrix{
\F(\Y) \ar[d]
\ar[r]^{}
& \phi(\F)(\Z) \ar[d]^{} 
\\
\F(\Y') \ar[r]^{}  & \phi(\F)(\Z')}
\end{xy} $$
in $\rS$, where the horizontal morphisms are cocartesian over $\rS.$
The morphism $\Z \to \Z'$ is cocartesian over $\rS$ if and only if
the morphism $\Y \to \Y'$ is an equivalence, in which case the left vertical
morphism in the square is an equivalence so that the right vertical morphism is cocartesian over $\rS$.

\end{proof}

\begin{definition}\label{faccc}
Let $\rS$ be an $\infty$-category and $\L,\R \subset \Fun([1],\rS)$ full subcategories containing all equivalences and closed under composition.
We call $(\L,\R)$ a factorization system on $\rS$ if 
\begin{enumerate}
\item the embedding
$\R \subset \Fun([1],\rS)$ admits a left adjoint,
\item a morphism in
$\Fun([1],\rS)$ whose target belongs to $\R$, is a local equivalence if and only if its image under evaluation at the source belongs to $\L$ and under evaluation at the target is an equivalence.	
\end{enumerate}

\end{definition}

\begin{notation}
For every functor $\mC \to \rS$ and factorization system $(\L,\R)$ on $\rS$ we set $$\Env_\L(\mC):= \R \times_{\Fun(\{0\},\rS)} \mC \simeq \R \times_{\Fun([1],\rS)}\Env(\mC) \to \R \to \Fun(\{1\},\rS).$$ 
\end{notation}

\begin{lemma}\label{envvo}
Let $\rS$ be an $\infty$-category, $(\L,\R)$ a factorization system on $\rS$
and $\gamma: \mC \to \rS$ a cocartesian fibration relative to $\L$.
The embedding $\Env_\L(\mC)\subset \Env(\mC)$ admits a left adjoint,
where a morphism of $\Env(\mC)$ is a local equivalence if and only if its image in $\mC$ is $\gamma$-cocartesian, lies over a morphism of $\L$ and is inverted by evaluation at the target.

\end{lemma}

\begin{proof}
This follows from the fact that $\R \subset \Fun([1],\rS)$ is a localization with the described local equivalences (Definition \ref{faccc}) and $\gamma: \mC \to \rS$ is a cocartesian fibration relative to $\L$ so that local equivalences can be lifted.
\end{proof}

\begin{remark}\label{lababop}
Lemma \ref{envvo} implies that the functor $\Env_\L(\mC) \to \rS$ is a cocartesian fibration whose cocartesian morphisms are those whose image in $\mC$ is $\gamma$-cocartesian and lies over a morphism of $\L$. In particular, the embedding $\mC \subset \Env_\L(\mC)$ is a map of cocartesian fibrations relative to $\L.$		
\end{remark}

\begin{lemma}\label{loccc}
Let $(\L,\R)$ be a factorization system on an $\infty$-category $\rS$ and $\phi: \mC \to \rS$ a cocartesian fibration. 
The embedding $\mC \subset \Env_\L(\mC)$ admits a left adjoint relative to $\rS.$

\end{lemma}

\begin{proof}

Let $(\X, \f: \phi(\X) \to \Y) \in \Env(\mC)$ and let
$\g: \X \to \X'$ a $\phi$-cocartesian lift of $\f.$ 
The canonical embedding $\iota: \mC \subset \Env(\mC)$ sends 
$\Z \in \mC $ to $(\Z,\id_{\phi(\Z)}) \in \Env(\mC). $
The morphism $(\g, \id_\Y) : (\X,\f) \to \iota(\X') $ in $\Env(\mC)$ induces for any $\Z \in \mC$ a map
$ \mC(\X', \Z) \to \Env(\mC)((\X, \f), \iota(\Z) )$
that factors as 
$$ \mC(\X', \Z) \simeq \rS(\Y, \phi(\Z)) \times_{\rS(\phi(\X), \phi(\Z))} \mC(\X, \Z) \simeq  \Env(\mC)((\X, \f), \iota(\Z) ).$$

\end{proof}  

\begin{notation}
Let $(\L,\R), (\L',\R')$ be factorization systems on an $\infty$-category $\rS$
such that $\L \subset \L'$ and $\mC \to \rS$ a cocartesian fibration relative to $\L$.
Let $$\Env_\L^{\L'}(\mC) \subset \Env_\L(\mC)$$ be the full subcategory spanned by the objects $\Y \in \Env_\L(\mC)$ such that there is a cocartesian lift $\X \to \Y$ of a morphism of $\L'$ and $\X \in \mC.$

\end{notation}

\begin{remark}
For $\L'=\Fun([1],\rS)$ we have $	\Env_\L^{\L'}(\mC) = \Env_\L(\mC).$
\end{remark}

\begin{remark}\label{consa}
Since $\Env_\L(\mC) \to \rS$ is a cocartesian fibration, $\Env_\L^{\L'}(\mC) \to \rS$ is
a cocartesian fibration relative to $\L'$ using that the composite of two cocartesian morphisms is cocartesian.
\end{remark}

The embedding $\mC \subset \Env_\L(\mC)$ induces an embedding 
$\mC \subset \Env_\L^{\L'}(\mC)$ that is a map of cocartesian fibrations relative to $\L$ and for any cocartesian fibration $\mD \to \rS$ relative to $\L'$
the canonical functor
$\Fun^{\L'}_\rS(\Env_\L^{\L'}(\mC), \mD) \to \Fun_\rS^{\L}(\mC, \mD)$
is conservative, where $\L,\L'$ indicate maps of cocartesian fibrations relative to $\L,\L',$ respectively.

\begin{proposition}\label{Envelo}
Let $(\L,\R), (\L',\R')$ be factorization systems on an $\infty$-category $\rS$
such that $\L \subset \L'$ and $\mC \to \rS$ a cocartesian fibration relative to $\L$.	
\begin{enumerate}
\item For every cocartesian fibration $\mD \to \rS$ the functor $\Fun_\rS(\Env_\L^{\L'}(\mC), \mD) \to \Fun_\rS(\mC, \mD)$
admits a fully faithful left adjoint that restricts to a functor
$\Fun^\L_\rS(\mC, \mD) \to \Fun^{\L'}_\rS(\Env_\L^{\L'}(\mC), \mD).$

\item For any cocartesian fibration $\mD \to \rS$ relative to $\L'$ the functor
$\Fun^{\L'}_\rS(\Env_\L^{\L'}(\mC), \mD) \to \Fun^\L_\rS(\mC, \mD)$ is an equivalence.

\end{enumerate}
\end{proposition}
\begin{proof}
We first show that (1) implies (2).
The functor
$\Fun^{\L'}_\rS(\Env_\L^{\L'}(\mC), \mD) \to \Fun_\rS^{\L}(\mC, \mD)$
is the pullback of the functor
$\lambda: \Fun^{\L'}_\rS(\Env_\L^{\L'}(\mC), \Env_{\L'}(\mD)) \to \Fun_\rS^{\L}(\mC, \Env_{\L'}(\mD))$ since the embedding $\mD \subset \Env_{\L'}(\mD)$ is a map of cocartesian fibrations relative to $\L'$ by Remark \ref{lababop}.
So to show (2) it is enough to prove that $\lambda$ is an equivalence.
Since $ \Env_{\L'}(\mD) \to \rS$ is a cocartesian fibration, by (1) the functor $\lambda$ admits a fully faithful left adjoint. By Remark \ref{consa} $\lambda$ is conservative and so an equivalence.

%\vspace{1mm}
(1): By Lemma \ref{freefib} the functor $ \Fun_\rS(\Env(\mC), \mD) \to \Fun_\rS(\mC, \mD)$ admits a fully faithful left adjoint $\kappa$ taking values in $\Fun_\rS^\cocart(\Env(\mC), \mD)$.
We prove that the functor $\Fun_\rS(\Env_\L^{\L'}(\mC), \mD) \to \Fun_\rS(\mC, \mD)$ is right adjoint to the composition
$\theta_\mD: \Fun_\rS(\mC, \mD) \xrightarrow{\kappa} \Fun_\rS(\Env(\mC), \mD) \to \Fun_\rS(\Env_\L^{\L'}(\mC), \mD)$ and that $\theta_\mD$ is fully faithful
and restricts to a functor $\Fun^\L_\rS(\mC, \mD) \to \Fun^{\L'}_\rS(\Env_\L^{\L'}(\mC), \mD).$
The cocartesian fibration $\mD \to \rS$ embeds into the bicartesian fibration $\mP^\rS(\mD) \to \rS$ whose fibers admit small colimits and limits.
The functor $\theta_\mD$ is the restriction of the functor $\theta_{\mP^\rS(\mD)}$ since $\kappa$ takes values in $\Fun_\rS^\cocart(\Env(\mC), \mD)$ and the embedding $\mD \subset\mP^\rS(\mD)$
preserves cocartesian morphisms. 
Thus it is enough to show the corresponding statement for $\mP^\rS(\mD) \to \rS$
so that we can assume that $\mD \to \rS$ is a bicartesian fibration whose fibers admit small colimits and limits.
In this case the functor $\Fun_\rS(\Env_\L^{\L'}(\mC), \mD) \to \Fun_\rS(\mC, \mD)$ admits a fully faithful left adjoint and the functor $\Fun_\rS(\Env_\L(\mC), \mD) \to \Fun_\rS(\Env_\L^{\L'}(\mC), \mD) $ admits a fully faithful right adjoint $\omega.$

By Lemma \ref{envvo} the embedding $\iota: \Env_\L(\mC) \subset \Env(\mC)$
admits a left adjoint $\phi: \Env(\mC) \to \Env_\L(\mC)$ relative to $\rS$.
The resulting adjunction $\Env(\mC) \leftrightarrows \Env_\L(\mC)$ relative to $\rS$ 
induces an adjunction
$\iota^*: \Fun_\rS(\Env(\mC), \mD)\leftrightarrows  \Fun_\rS(\Env_\L(\mC), \mD) : \phi^*, $ where $\phi^*$ is fully faithful with essential image the functors $\Env(\mC) \to \mD$ over $\rS$ inverting local equivalences.
The functor $\Fun_\rS(\Env_\L^{\L'}(\mC), \mD) \to \Fun_\rS(\mC,\mD) $ restricting to $\mC$
factors as $$\Fun_\rS(\Env_\L^{\L'}(\mC), \mD) \xrightarrow{\omega} \Fun_\rS(\Env_\L(\mC), \mD) \xrightarrow{\phi^*} \Fun_\rS(\Env(\mC), \mD) \to \Fun_\rS(\mC,\mD) $$
and so is right adjoint to the functor
$$\theta_\mD: \Fun_\rS(\mC,\mD) \xrightarrow{\kappa} \Fun_\rS(\Env(\mC), \mD) \xrightarrow{\iota^*} \Fun_\rS(\Env_\L(\mC), \mD) \to \Fun_\rS(\Env_\L^{\L'}(\mC), \mD).$$

We complete the proof by showing that $\theta_\mD$ restricts. This follows from the following two statements:
\begin{itemize}
\item A functor $\Env(\mC)\to \mD$ over $\rS$ inverting local equivalences is a map of cocartesian fibrations over $\rS$ if and only if its restriction $ \Env(\mC) \subset \Env(\mC) \to \mD$ is.

\item A map $\psi: \Env(\mC) \to \mD$ of cocartesian fibrations over $\rS$ inverts local equivalences if its restriction 
$\tau: \mC \subset \Env(\mC) \to \mD$ is a map of cocartesian fibrations relative to $\L.$ 
\end{itemize}
The first statement follows from the characterization of local equivalences in Lemma \ref{envvo}. So we check the second point.
For any morphism $\h$ in $\Env(\mC)$ corresponding to a morphism $\f:\X \to \Y$ in $\mC$ 
and a commutative square 
$$\begin{xy}
\xymatrix{
\gamma(\X) \ar[d]
\ar[r]^{\gamma(\f)}
& \gamma(\Y) \ar[d]^{} 
\\
\X' \ar[r]^{\g}  & \Y'}
\end{xy} $$
in $\rS$ there is a unique commutative square
$$\begin{xy}
\xymatrix{
\tau(\X) \ar[d]
\ar[r]^{\tau(\f)}
& \tau(\Y) \ar[d]^{} 
\\
\psi(\X) \ar[r]^{\psi(\f)}  & \psi(\Y)}
\end{xy}$$
in $\mD$, where both vertical morphisms are cocartesian, lifting the former commutative square.

If $\h$ is a local equivalence, $\f: \X \to \Y$ is $\gamma$-cocartesian
and lies over a morphism of $\L$ and $\g$ is an equivalence.
So if $\tau: \mC \subset \Env(\mC) \to \mD$ is a map of cocartesian fibrations relative to $\L$, the image $\tau(\f)$ is cocartesian.
Therefore $\psi(\f)$ is also cocartesian as both vertical maps in the diagram are. Since $\psi(\f)$ lies over the equivalence $\g$, it is itself an 
equivalence.

\end{proof}

\begin{example}\label{envort}
Let $\rS, \T$ be $\infty$-categories, $(\L,\R)$ a factorization system on $\rS$ and $ \mC \to \rS \times \T$ a map of cocartesian fibrations relative to $\L$.
Let $$\L' := \L \times \T, \ \R':= \R \times \Fun([1],\T), \ \L'' := \Fun([1],\rS) \times \T, \ \R'':= \Fun([1],\rS) \times \Fun([1],\T).$$
Then $(\L',\R'), (\L'',\R'')$ are factorization systems. The diagonal embedding 
$\T \subset \Fun([1],\T)$ induces an embedding $\Env_\L(\mC) \subset \Env_{\L'}(\mC)$
that induces an equivalence $\Env_\L(\mC) \simeq \Env_{\L'}^{\L''}(\mC).$
So by Proposition \ref{Envelo}
for every map $\mD \to \rS \times \T$ of cocartesian fibrations over $\rS$ the functor $$\Fun^{\Fun([1],\rS) \times \T}_{\rS \times \T}(\Env_\L(\mC), \mD) \to \Fun^{\L\times \T}_{\rS \times \T}(\mC, \mD)$$ is an equivalence, where the left hand side are the maps of cocartesian fibrations over $\rS$ % over $\rS \times \T$
and the right hand side are the maps of cocartesian fibrations relative to $\L$. %over $\rS \times \T$.

\end{example}

\bibliographystyle{plain}
\bibliography{ma}

\begin{thebibliography}{10}

\bibitem{MR3190610}
Benjamin Antieau and David Gepner.
\newblock Brauer groups and \'{e}tale cohomology in derived algebraic geometry.
\newblock {\em Geom. Topol.}, 18(2):1149--1244, 2014.

\bibitem{MR3869643}
David Ayala and John Francis.
\newblock Flagged higher categories.
\newblock In {\em Topology and quantum theory in interaction}, volume 718 of
  {\em Contemp. Math.}, pages 137--173. Amer. Math. Soc., [Providence], RI,
  [2018] \copyright 2018.

\bibitem{MR4074276}
David Ayala and John Francis.
\newblock Fibrations of {$\infty$}-categories.
\newblock {\em High. Struct.}, 4(1):168--265, 2020.

\bibitem{MR1355899}
John~C. Baez and James Dolan.
\newblock Higher-dimensional algebra and topological quantum field theory.
\newblock {\em J. Math. Phys.}, 36(11):6073--6105, 1995.

\bibitem{MR2276611}
Julia~E. Bergner.
\newblock A model category structure on the category of simplicial categories.
\newblock {\em Trans. Amer. Math. Soc.}, 359(5):2043--2058, 2007.

\bibitem{berman2020enriched}
John~D. Berman.
\newblock Enriched infinity categories i: enriched presheaves, 2020.

\bibitem{article}
Andrew Blumberg, David Gepner, and Goncalo Tabuada.
\newblock A universal characterization of higher algebraic k-theory.

\bibitem{Shulman2009}
G.~Cruttwell and Michael Shulman.
\newblock A unified framework for generalized multicategories.
\newblock {\em Theory and Applications of Categories [electronic only]}, 24, 07
  2009.

\bibitem{GEPNER2015575}
David Gepner and Rune Haugseng.
\newblock Enriched $\infty$-categories via non-symmetric $\infty$-operads.
\newblock {\em Advances in Mathematics}, 279:575 -- 716, 2015.

\bibitem{MR3402334}
Rune Haugseng.
\newblock Rectification of enriched {$\infty$}-categories.
\newblock {\em Algebr. Geom. Topol.}, 15(4):1931--1982, 2015.

\bibitem{Rune}
Rune Haugseng.
\newblock $\infty$-operads via symmetric sequences.
\newblock {\em Mathematische Zeitschrift}, 301, 05 2022.

\bibitem{haugseng_melani_safronov_2020}
Rune Haugseng, Valerio Melani, and Pavel Safronov.
\newblock Shifted coisotropic correspondences.
\newblock {\em Journal of the Institute of Mathematics of Jussieu}, page
  1–65, 2020.

\bibitem{Heine2022}
Hadrian Heine.
\newblock A grothendieck construction for fibrations of $\infty$-categories,
 2022.

\bibitem{HINICH2020107129}
Vladimir Hinich.
\newblock Yoneda lemma for enriched $\infty$-categories.
\newblock {\em Advances in Mathematics}, 367:107129, 2020.

\bibitem{hinich2021colimits}
Vladimir Hinich.
\newblock Colimits in enriched $\infty$-categories and day convolution, 2021.

\bibitem{MR3607274}
Marc Hoyois, Sarah Scherotzke, and Nicol\`o Sibilla.
\newblock Higher traces, noncommutative motives, and the categorified {C}hern
  character.
\newblock {\em Adv. Math.}, 309:97--154, 2017.

\bibitem{Leinster2002}
Tom Leinster.
\newblock Generalized enrichment of categories.
\newblock {\em Journal of Pure and Applied Algebra}, 168:391--406, 05 2002.

\bibitem{leinster_2004}
Tom Leinster.
\newblock {\em Higher Operads, Higher Categories}.
\newblock London Mathematical Society Lecture Note Series. Cambridge University
  Press, 2004.

\bibitem{lurie.higheralgebra}
Jacob Lurie.
\newblock Higher {A}lgebra.
\newblock available at http://www.math.harvard.edu/~lurie/.

\bibitem{lurie.HTT}
Jacob Lurie.
\newblock {\em Higher topos theory}, volume 170 of {\em Annals of Mathematics
  Studies}.
\newblock Princeton University Press, Princeton, NJ, 2009.

\bibitem{MR4185309}
Andrew~W. Macpherson.
\newblock The operad that co-represents enrichment.
\newblock {\em Homology Homotopy Appl.}, 23(1):387--401, 2021.

\bibitem{stefanich2020presentable}
German Stefanich.
\newblock Presentable $(\infty, n)$-categories, 2020.

\end{thebibliography}

\begin{center}
EPFL, Lausanne, Switzerland \\ 
E-mail address: hadrian.heine@epfl.ch
\end{center}

\end{document}